\documentclass[10pt]{article}

\usepackage{titlesec}
\usepackage{authblk}
\usepackage[labelfont=sc,font=small,labelsep=period]{caption}
\usepackage{float}
\usepackage{comment}
\usepackage{a4wide}
\usepackage{tikz}
\date{}

\usepackage{amsmath} %[intlimits]
\usepackage{amssymb}
\usepackage{amsfonts}
\usepackage{latexsym}
\usepackage{amsthm}
\usepackage{amsxtra}
\usepackage{amscd}
\usepackage{bbm}
\usepackage{mathrsfs}
\usepackage{bm}

\usepackage{graphicx, color}
\usepackage{subcaption}
\usepackage[nottoc,notlof,notlot]{tocbibind} 
\usepackage{cite} % Enabling `[1-4]` - style citing

\usepackage[margin=1in]{geometry}

\usepackage[colorlinks=true,linkcolor=blue,citecolor=blue]{hyperref}
\usepackage{cleveref}

\numberwithin{equation}{section}
\numberwithin{figure}{section}
\theoremstyle{plain} %plain, definition, remark
\newtheorem{theorem}{Theorem}[section]
\newtheorem*{theorem*}{Theorem}
\newtheorem{lemma}[theorem]{Lemma}
\newtheorem*{lemma*}{Lemma}
\newtheorem{assumption}[theorem]{Assumption}
\newtheorem{corollary}[theorem]{Corollary}
\newtheorem*{corollary*}{Corollary}
\newtheorem{proposition}[theorem]{Proposition}
\newtheorem*{proposition*}{Proposition}
\newtheorem{definition}[theorem]{Definition}
\newtheorem*{definition*}{Definition}

\newtheorem*{conjecture*}{Conjecture}

\theoremstyle{definition} %plain, definition, remark

\newtheorem*{example*}{Example}
\newtheorem{remark}[theorem]{Remark}
\newtheorem*{remark*}{Remark}

\DeclareMathOperator{\OO}{O}
\DeclareMathOperator{\oo}{o}

\DeclareMathOperator{\tr}{Tr}

\renewcommand{\b}[1]{\boldsymbol{\mathrm{#1}}} %bold
\renewcommand{\cal}{\mathcal}

\newcommand{\ol}[1]{\overline{#1} \!\,} %overline
\newcommand{\wh}{\widehat}
\newcommand{\wt}{\widetilde}
\newcommand{\txt}[1]{\text{\rm{#1}}}

\newcommand{\sigG}{\mathsf{G}}
\newcommand{\sigH}{\mathsf{H}}

\newcommand{\fa}{\mathfrak a}
\newcommand{\fb}{\mathfrak b}
\newcommand{\fc}{\mathfrak c}
\newcommand{\fd}{\mathfrak d}
\newcommand{\sm}{\mathsf m}
\newcommand{\sfK}{\mathsf K}

\definecolor{darkred}{rgb}{0.9,0,0.3}
\definecolor{darkblue}{rgb}{0,0.3,0.9}
\newcommand{\cor}{\color{darkred}}
\newcommand{\cob}{\color{darkblue}}

\usepackage{ifthen}
\def\comment#1{\ifthenelse{\isodd{\value{page}}}{\marginpar{\raggedright\scriptsize{\textcolor{darkred}{#1}}}}{\marginpar{\raggedleft\scriptsize{\textcolor{darkred}{#1}}}}}

\renewcommand{\P}{\mathbb{P}}
\newcommand{\E}{\mathbb{E}}
\newcommand{\R}{\mathbb{R}}

\newcommand{\C}{\mathbb{C}}
\newcommand{\N}{\mathbb{N}}
\newcommand{\Z}{\mathbb{Z}}

\newcommand{\ri}{\mathrm{i}}
\newcommand{\sk}{\mathsf{k}}
\newcommand{\st}{\mathsf{t}}
\newcommand{\ft}{\mathfrak{t}}
\newcommand{\ii}{\mathrm{i}}
\newcommand{\dd}{\mathrm{d}}
\newcommand{\col}{\mathrel{\vcenter{\baselineskip0.75ex \lineskiplimit0pt \hbox{.}\hbox{.}}}}
\newcommand*{\deq}{\mathrel{\vcenter{\baselineskip0.65ex \lineskiplimit0pt \hbox{.}\hbox{.}}}=}

\renewcommand{\leq}{\leqslant}
\renewcommand{\geq}{\geqslant}
\renewcommand{\epsilon}{\varepsilon}

\newcommand{\qq}[1]{[\![{#1}]\!]}

\newcommand{\ind}[1]{\b 1 (#1)}
\newcommand{\indb}[1]{\b 1 \pb{#1}}

\newcommand{\pb}[1]{\bigl({#1}\bigr)}
\newcommand{\pB}[1]{\Bigl({#1}\Bigr)}
\newcommand{\pbb}[1]{\biggl({#1}\biggr)}
\newcommand{\pBB}[1]{\Biggl({#1}\Biggr)}

\newcommand{\qB}[1]{\Bigl[{#1}\Bigr]}

\newcommand{\hb}[1]{\bigl\{{#1}\bigr\}}

\newcommand{\hbb}[1]{\biggl\{{#1}\biggr\}}

\newcommand{\abs}[1]{\lvert #1 \rvert}

\newcommand{\avg}[1]{\langle #1 \rangle}

\newcommand{\scalar}[2]{\langle{#1} \mspace{2mu}, {#2}\rangle}

\DeclareMathOperator{\supp}{supp}
\DeclareMathOperator{\re}{Re}
\DeclareMathOperator{\im}{Im}

\DeclareMathOperator{\dist}{dist}

\newcommand{\bu}{{\bf{u}}}
\newcommand{\eb}{{\bf{e}}}
\newcommand{\bv}{{\bf{v}}}
\newcommand{\bw}{{\bf{w}}}
\newcommand{\sW}{{\mathcal{W}}}

\newcommand{\al}{\alpha}

\newcommand{\be}{\begin{equation}}
\newcommand{\ee}{\end{equation}}
\newcommand{\e}{{\varepsilon}}
\newcommand{\rd}{\mathrm{d}}

\title{Eigenvector distributions and optimal shrinkage estimators for large covariance and precision matrices}

\author[1]{Xiucai Ding \thanks{E-mail: xcading@ucdavis.edu.}}
\author[2]{ Yun Li \thanks{E-mail: liyun1723@mail.tsinghua.edu.cn}}
\author[2, 3]{Fan Yang  \thanks{E-mail: fyangmath@mail.tsinghua.edu.cn}}
\affil[1]{Department of Statistics, University of California, Davis}

\affil[2]{Yau Mathematical Sciences Center, Tsinghua University}

\affil[3]{Beijing Institute of Mathematical Sciences and Applications}

\begin{document}

\maketitle

\begin{abstract}
%In this paper, we study shrinkage estimators for large sample covariance matrices and precision matrices in the high dimensional regime. We consider models with almost arbitrary and possibly spiked population covariance matrices. Under some mild technical conditions, we obtain the asymptotic limits of the shrinkers for a broad family of loss functions. The main technical ingredient for deriving the asymptotic limits is a new result regarding the asymptotic distributions for the non-spiked eigenvectors of the sample covariance matrices, which can be of independent interest. 

This paper focuses on investigating Stein’s invariant shrinkage estimators for large sample covariance matrices and precision matrices in high-dimensional settings. We consider models that have nearly arbitrary population covariance matrices, including those with potential spikes. By imposing mild technical assumptions, we establish the asymptotic limits of the shrinkers for a wide range of loss functions. A key contribution of this work, enabling the derivation of the limits of the shrinkers, is a novel result concerning the asymptotic distributions of the non-spiked eigenvectors of the sample covariance matrices, which can be of independent interest.
\end{abstract}

%\tableofcontents

\section{Introduction}
Estimating large covariance matrices and their inverses, the precision matrices, is fundamental in modern data analysis. It is well-known that in the high-dimensional regime when the data dimension $p$ is comparable to or even larger than the sample size $n,$ the sample covariance matrices and their inverses are poor estimators \cite{yao2015large}. To obtain consistent estimators, many structural assumptions have been imposed, such as sparse or banded structures. Based on these assumptions, many regularization methods have been developed to obtain better estimators. We refer the readers to \cite{EJSMCRZ,fan2016overview,pourahmadi2013high} for a more comprehensive review.

Although these structural assumptions are useful for many applications, they may not be applicable in general scenarios. In this paper, we consider the estimation of the population covariance matrix $\Sigma$ and its inverse, denoted by $\Omega$, through Stein's (orthogonally or rotationally) invariant estimators \cite{steinlecture, Stein1986LecturesOT} without imposing (almost) any specific structural assumption. Given a $p \times n$ data matrix $Y$ and its sample covariance matrix $\mathcal{Q}=YY^\top/n$, we say $\widetilde{\Sigma} \equiv \widetilde{\Sigma}(\mathcal{Q})$ and $\widetilde{\Omega} \equiv \widetilde{\Omega}(\mathcal{Q})$ are invariant estimators for $\Sigma$ and $\Omega,$ respectively, if they satisfy that
\begin{equation}\label{eq_steincondition}
\mathsf{U}\widetilde{\Sigma}(\mathcal{Q}) \mathsf{U}^\top=\widetilde{\Sigma}(\mathsf{U} \mathcal{Q} \mathsf{U}^\top), \ \  \mathsf{U}\widetilde{\Omega}(\mathcal{Q}) \mathsf{U}^\top=\widetilde{\Omega}(\mathsf{U} \mathcal{Q} \mathsf{U}^\top),
\end{equation}   
for any $p \times p$ orthogonal matrix $\mathsf{U}.$ For an illustration, we take the covariance matrix estimation as an example. Denote the eigenvalues and corresponding eigenvectors of  $\mathcal{Q}$ by $\lambda_1\ge \lambda_2\ge \cdots \ge \lambda_p \ge 0$ and $\{\bm{u}_i\}_{i=1}^p$. Stein \cite{steinlecture, Stein1986LecturesOT} showed that the optimal invariant estimator for $\Sigma$ satisfy
\begin{equation}\label{eq_steincovarianceestimator}
\widetilde{\Sigma}=\sum_{i=1}^p \varphi_i \bm{u}_i \bm{u}_i^\top,
\end{equation}  
where $\varphi_i \equiv \varphi_i(\mathcal{Q}, \Sigma, \mathcal{L}), \ i=1,\ldots,p,$ commonly referred to as \emph{shrinkers}, are some nonlinear functionals depending on the choice of the loss function $\mathcal{L},$ the sample covariance matrix $\mathcal{Q}$, and the population covariance matrix $\Sigma.$ For many loss functions $\cal L$, the random quantities $\varphi_i$'s have closed forms (see Appendix \ref{appendix_lossfunctionsummary}) so that the estimation problem reduces to finding efficient estimators or estimable convergent limits for $\varphi_i$'s. As demonstrated in Appendix \ref{appendix_lossfunctionsummary}, $\varphi_i$ usually takes the following form:
\begin{equation}\label{eq_varphigeneralform}
\varphi_i=\bm{u}_i^\top \ell(\Sigma) \bm{u}_i,\quad 1\le i \le \min\{p,n\},
\end{equation}
where $\ell(x)$ is a function depending on $\mathcal{L}$, or the following form when $p>n$:
\begin{equation}\label{eq_varphigeneralform0}
\varphi_i= \frac{1}{p-n}\tr\left[U_0^\top \ell(\Sigma) U_0\right],\quad n+1\le i \le p. 
\end{equation}
Here, $U_0$ represents the eigenmatrix associated with the zero eigenvalues of $\cal Q$. 
In the classical setting when $p$ is fixed, Stein derived consistent estimators for \eqref{eq_varphigeneralform} under various loss functions; we refer the readers to \cite{steinsummary} in this regard. 

In the current paper, we investigate this problem in the high-dimensional regime when $p$ is comparable to $n$ in the sense that there exists a constant $\tau \in (0,1)$ such that 
\begin{equation}\label{eq_ratioassumption}
\tau\leq c_n:=\frac{p}{n} \leq \tau^{-1}.
\end{equation}   
Note that we allow for $p>n$ so that the inverse of the sample covariance matrix $\mathcal{Q}$ may not exist in general. We also treat our models with general and possibly spiked population covariance matrices $\Sigma$ (up to some mild technical assumptions). This substantially generalizes Johnstone's spiked covariance matrix model \cite{johnstone2001}, where the non-spiked eigenvalues are assumed to be unity (see (\ref{eq_simplemodel}) below). 

%Our assumptions on $\Sigma$ are general. Especially,  we will focus on a general and possibly spiked covariance matrix model (see Section \ref{sec_generallargecovariancematrixmodel}) which is a substantial generalization of   
%Johnstone's spiked covariance matrix model \cite{johnstone2001} where the non-spiked eigenvalues are assumed to be unity (see (\ref{eq_simplemodel}) below). 
%  
  
In the general model setup and under the condition (\ref{eq_ratioassumption}), we provide analytical and closed-form formulas to characterize the convergent limits of \eqref{eq_varphigeneralform} and \eqref{eq_varphigeneralform0} for both the spiked and non-spiked eigenvectors $\bm{u}_i$. 
%we provide analytical and closed-form formulas to characterize the convergent limits for all $\varphi_i, 1 \leq i \leq p$ for both the spikes and non-spikes. 
These limits can be further estimated consistently and adaptively for various choices of loss functions (or equivalently, the function $\ell$). One crucial step involves studying the asymptotic distributions for the sample eigenvectors $\bm{u}_i$, which can be of significant independent interest. In what follows, we first discuss some related works on the shrinkers' estimation for the regime (\ref{eq_ratioassumption}) in Section \ref{sec_shrinkageestimation}. Then, we review some relevant literature about the eigenvector distributions of some classical random matrix models in Section \ref{sec_eigenvectordistributionreview}. Finally, we provide an overview of the contributions of this paper and highlight some main novelties in Section \ref{sec_overviewresultsnovelties}.

\subsection{Related works on shrinker estimation}\label{sec_shrinkageestimation}

We now provide a brief review of some previous results concerning the estimation of the shrinkers $\varphi_i$ in the high-dimensional regime.  In \cite{donoho2018}, the authors considered Johnstone's spiked covariance matrix model with 
\begin{equation}\label{eq_simplemodel}
\Sigma=\sum_{i=1}^r \widetilde{\sigma}_i \bm{v}_i \bm{v}_i^\top+I,
\end{equation}
where $r\in \N$ is the rank of signals. By replacing Stein's original estimator (\ref{eq_steincovarianceestimator}) with the stronger rank-aware assumption $\varphi_i\equiv 1$ for $r+1\le i \le p$ (see equation (1.13) therein), \cite{donoho2018} provided analytical and closed-form convergent limits for the shrinkers $\varphi_j$, $1\le j\le r$, under various choices of loss functions when $p\le n$. These convergent limits can be consistently estimated using the first $r$ eigenvalues and eigenvectors of $\mathcal{Q}.$
An important insight conveyed by \cite{donoho2018} is that the selection of the loss function can have a significant impact on the estimation of shrinkers.
%  Their work conveys a significant insight, revealing the profound impact that the choice of a loss function can exert on the estimations of the shrinkers.
Nevertheless, the assumption $\varphi_i\equiv 1$ for $i > r$ is crucial for their theory, which fails when $I$ in (\ref{eq_simplemodel}) is replaced by a general positive definite covariance matrix $\Sigma_0$. 
%In particular, the rank-aware rule is not valid anymore. 
To address this issue, under Stein's setup (\ref{eq_steincovarianceestimator}) and using Frobenius norm as the loss function, Bun \cite{bun2018optimal} derived the convergent limits for shrinkers associated with the spikes of a general covariance matrix $\Sigma$ and provide an adaptive estimation of these limits.

%   in \cite{bun2018optimal}, under Stein's setup (\ref{eq_steincovarianceestimator}), using Frobenius norm as the loss function, the author obtained the convergent limits for the shrinkers associated with the spikes, i.e., $\varphi_j, j \leq r$ and those limits can also be estimated adaptively. 

On the other hand, the derivation of the convergent limits for shrinkers associated with non-spiked eigenvectors becomes much more challenging and is substantially different from \cite{donoho2018, bun2018optimal}. 
In this context, adaptive estimations for the shrinkers still can be provided, see e.g., \cite{MR2834718, ledoit2012, ledoit2022quadratic, mainreference},
%, the treatment of the existing literature is quite different from \cite{donoho2018, bun2018optimal}. Most of the literature, for example \cite{MR2834718, ledoit2012,ledoit2022quadratic,mainreference},  directly provide estimation for the shrinkers.
%Even though these estimators are adaptive, 
but theoretically, they are consistent only in an averaged sense. The main challenge lies in the fact that a key theoretical input is the limiting eigenvector empirical spectral distribution derived in \cite{MR2834718}, which is insufficient for deriving the convergent limits and providing consistent estimators for individual shrinkers. 

Inspired by these earlier works, in this paper, we will rigorously prove the convergent limits for all the (spiked or non-spiked) shrinkers associated with general spiked covariance matrix models in (\ref{eq_spikedcovariancematrix}) below for various loss functions. Roughly speaking, our generalized model extends (\ref{eq_simplemodel}) by replacing $I$ with an (almost) arbitrary positive definite covariance matrix, introducing a more realistic modeling approach. Note that the framework in \cite{donoho2018} is not applicable in this setting. We will also provide consistent estimators for individual shrinkers by approximating these convergent limits with sample quantities. From a technical viewpoint, our current paper improves all previous results in the existing literature.

\subsection{Related works on eigenvector distributions of sample covariance matrices}\label{sec_eigenvectordistributionreview}

Next, we briefly review some previous results concerning eigenvector distributions of random matrices, with a focus on the sample covariance matrices. On the global level, the eigenvector empirical spectral distribution (VESD) has been extensively investigated under various assumptions on the population covariance matrices, as explored in studies such as \cite{Baipaper,yangeigenvector,XYY,xqb}. Notably, these papers demonstrated that the limiting VESD is closely linked to the Marchenko-Pastur (MP) law \cite{marchenko1967distribution}. 
Building upon these results, investigations into the distributions of the linear spectral statistics (LSS) of the eigenvectors have been conducted under various settings in \cite{Baipaper, yangeigenvector}. Specifically, it has been established that the LSS of the eigenvectors exhibit asymptotic Gaussian behavior under almost arbitrary population covariance $\Sigma$. %For further details, we refer readers to the recent publication \cite{yangeigenvector}.

%Based on these results, the distributions of the linear spectral statistics (LSS) of the eigenvectors have been studied under various setups in \cite{Baipaper,yangeigenvector}. It turns out that the LSS of the eigenvectors are asymptotically Gaussian distributed. We refer the readers to the most recent work \cite{yangeigenvector} for further details.

On the local level, an important result was established in \cite{MR3606475} when $\Sigma=I$, demonstrating that the projection of any eigenvector onto an arbitrary deterministic unit vector converges in law to a standard normal distribution after a proper normalization. In particular, this implies that all the eigenvector entries are asymptotically Gaussian. Furthermore, in \cite{MR3606475}, the concept of \emph{quantum unique ergodicity} (QUE) for eigenvectors was also established based on these findings.  Similar results were obtained for the non-outlier eigenvectors under Johnstone's spiked covariance matrix model (\ref{eq_simplemodel}) in \cite{bloemendal2016principal}, while the distributions of outlier eigenvectors were studied in \cite{bao2022statistical,bao2022eigenvector,paul2007asymptotics}. 
For the case of a general \emph{diagonal} population covariance matrix $\Sigma$, the universality of eigenvector distributions was established both in the bulk and near the edge in \cite{ding2019singular}. However, the explicit distribution remains unknown.

Motivated by our specific applications in shrinkage estimation, our goal is to derive the explicit distributions for all eigenvectors of the non-spiked model, as well as the non-outlier eigenvectors of the spiked model, considering the assumption of a general population covariance matrix $\Sigma$. These results have been lacking in the existing literature, as indicated by the aforementioned overview. To achieve this, we draw inspiration from recent advancements in the analysis of eigenvectors of Wigner matrices, as presented in \cite{benigni2020eigenvectors,benigni2022optimal,MR3606475, knowles2013eigenvector,marcinek2022high}.

\subsection{An overview of our results and technical  novelties}\label{sec_overviewresultsnovelties}

In this subsection, we present an overview of our results and highlight the main novelties of our work. Our main contributions can be divided into two parts: the convergent limit of each individual shrinker for various loss functions under the general spiked covariance model (see Section \ref{sec_statisticsmainresults}); the asymptotic eigenvector distributions for sample covariance matrices with general covariance structures (see Section \ref{sec_mainresultque}).

For the convergence of shrinkers, under a general spiked covariance matrix model, we present explicit and closed-form formulas that characterize the convergent limit of each individual shrinker under various loss functions. The precise statement of this result can be found in Theorem \ref{thm_shrinkerestimate}. Notably, these formulas remain applicable even in the singular case with $p > n$. Corollary \ref{coro_simplifiedformula} provides simplified versions of these formulas for specific choices of the loss function. Leveraging these theoretical findings, we establish the asymptotic risks associated with shrinkage estimation of covariance and precision matrices under various loss functions in Corollary \ref{coro_lowbound}. To facilitate practical implementation, we also introduce adaptive and consistent estimators for the convergent limits of the shrinkers, as outlined in Theorem \ref{thm_consistentestimators}. An $\mathtt{R}$ package $\mathtt{RMT4DS}$\footnote{An online demo for the package can be found at \url{https://xcding1212.github.io/RMT4DS.html}} has been developed to implement our methods.

Now, we will delve into the technical details concerning the derivation of the convergent limits of \eqref{eq_varphigeneralform}; the derivation of the limit of \eqref{eq_varphigeneralform0} is considerably simpler. 
%We now focus on discussing some technical details for deriving the convergent limits \eqref{eq_varphigeneralform}---the derivation of the limit of \eqref{eq_varphigeneralform0} is much simpler. 
The proofs for the spiked and non-spiked shrinkers exhibit notable differences. In the case of spiked shrinkers, we employ a decomposition of (\ref{eq_varphigeneralform}) into two components: a low-rank part and a high-rank part, as illustrated in equation (\ref{eq_decomposition}) below. The low-rank portion can be effectively approximated by leveraging the convergence limits of the outlier eigenvalues and eigenvectors, as demonstrated in Lemma \ref{lemma_spikedlocation}. On the other hand, the high-rank component takes the form of $\sum_{j=r+1}^p w_j (\bm{u}_i^\top \bm{v}_j)^2$, $1 \leq i \leq r$, where $w_j$ are some deterministic weights and $\bm{v}_j$ are the eigenvectors of $\Sigma.$
The existing literature (e.g., \cite{MR3704770,DRMTA}) provides a "rough" delocalization bound like $(\bm{u}_i^\top \bm{v}_j)^2 \le n^{-1+\e}$ with high probability, for any small constant $\e > 0$. However, this bound is insufficient for our specific purpose. To address this issue, consider the resolvent (or Green's function) of $\cal Q$ defined as $G(z)=(\mathcal{Q}-z)^{-1}$ for $z\in \C$. 
%The low-rank part can be well approximated using the convergence limits of the outlier eigenvalues and eigenvectors (see Lemma \ref{lemma_spikedlocation} below). Existing literature (see e.g., \cite{MR3704770,DRMTA}) only provides a ``rough" delocalization bound like $(\bm{u}_i^\top \bm{v}_j)^2 \le n^{-1+\e}$ with high probability for any small constant $\e>0$, which is not sufficient for our purpose. 
%unlike in \cite{bloemendal2016principal,DRMTA}, 
Then, we will conduct a higher order resolvent expansion to obtain that (c.f.~Lemma \ref{lem_bulkproperty})
$$|\bm{u}_i^\top \bm{v}_j|^2=w'_{ij} G_{ij}(\mathfrak{a}_i) G_{ji}(\mathfrak{a}_i) +\OO(n^{-3/2+\e}) \quad \text{with high probability},$$
where $\fa_i$ is the convergent limit of the $i$-th outlier eigenvalue of $\cal Q$ and $w'_{ij}$ denotes certain deterministic weights. Now, our focus narrows down to estimating the random variable given by $\sum_{j=r+1}^p w'_{ij} G_{ji}(\mathfrak{a}_i) G_{ij}(\mathfrak{a}_i).$ In the special case where $w'_{ij}\equiv 1,$ Bun \cite{bun2018optimal} has provided an estimation of this quantity utilizing a concentration inequality. However, in this paper, we employ a variational approach that simplifies the estimation process and enables the computation of the limit for any sequence of weights $w_{ij}'$; see \eqref{eq_Lireducedform} for further details.

%In this paper, we use a variational approach that simplifies the estimation and can be used to compute the limit for any sequence of weights $w_{ij}'$; see \eqref{eq_Lireducedform} below for more details.

On the other hand, for the non-spiked shrinkers $\varphi_i$, $r+1\le i\le \min\{p,n\}$, the above approach fails. Instead, we need to directly evaluate $\sum_{j=1}^p w_j (\bm{u}_i^\top \bm{v}_j)^2$ by establishing what is known as a \emph{QUE estimate}. In the case where $\Sigma=I$, this form has been investigated in  \cite{bloemendal2016principal} utilizing the methodology developed in \cite{MR3606475}. In this work, we establish the QUE for the non-spiked eigenvectors of $\cal Q$ under an almost arbitrary population covariance matrix $\Sigma$. The key lies in establishing the asymptotic distributions for all non-outlier eigenvectors (c.f.~Theorem \ref{Thm: EE}), from which the QUE estimate can be derived (c.f.~\Cref{corollary_que}). The proof of the eigenvector distributions in this work builds upon and extends the eigenvector moment flow (EMF) approach introduced in \cite{MR3606475}, which involves three main steps. In the first step, we establish the local laws for the resolvent of $\mathcal{Q}$. In the second step, we study the EMF for the rectangular matrix Dyson Brownian motion (DBM) of $\cal Q$ defined as $\cal Q_t:=(Y+B_t)(Y+B_t)^\top/n$, where $Y$ represents the data matrix and $B_t$ is a matrix consisting of i.i.d.~Brownian motions of variance $t$. 
We demonstrate the relaxation of the EMF dynamics to equilibrium when $t \gg n^{-1/3}$, from which we can establish the Gaussian normality for the eigenvectors of $\cal Q_t$. Finally, in the last step, we show that the eigenvector distributions of $\mathcal{Q}$ are close to those of $\cal Q_t$, which is achieved through a standard Green's function comparison argument.
For a more detailed discussion of the above strategy, we direct readers to Section \ref{sec_generalproofdetail}. Notably, this approach has been recently employed to establish the asymptotic distributions of eigenvalues and eigenvectors for various random matrix ensembles featuring general population covariance or variance profiles, as evidenced in works such as \cite{benigni2020eigenvectors, benigni2022optimal,9779233,marcinek2022high,MR4156609}. 

%It has recently been used in establishing the asymptotic distributions of eigenvalues and eigenvectors for various random matrix ensembles with general population covariance or variance profiles in, to list but a few, \cite{benigni2020eigenvectors, benigni2022optimal,9779233,marcinek2022high,MR4156609}. 

In our setting, Step 1 has already been accomplished in \cite{MR3704770}. Moving on to Step 2, following the idea of \cite{MR3606475}, we establish in \Cref{main result: Jia} that a general functional $f_t(\xi)$, encoding the joint moments of the projections of the eigenvectors of $\cal Q_t$, relaxes to equilibrium with high probability on the time scale $t \gg n^{-1/3}$.
%for some general functional $f_t(\xi)$ encoding the joint moments of projections of the eigenvectors of $\cal Q_t$, it relaxes to equilibrium with high probability on the time scale $t \gg n^{-1/3}$. 
This equilibrium characterizes the joint moments of a multivariate normal distribution, which concludes Step 2 by the moment method (c.f.~Lemma \ref{lem: moment flow}). We remark that the proof of Theorem \ref{main result: Jia} relies on a probabilistic description of $f_t(\xi)$ as the solution to a system of coupled SDEs representing a specific interacting particle system (c.f.~Lemma \ref{lem: dyn}). Since the current paper is motivated by applications in shrinker estimation, we specifically focus on investigating the joint distribution of eigenvectors projected onto a single direction, as illustrated in (\ref{eq_lemma41equation}). Nevertheless, as discussed in \Cref{rmk:proj}, our results can be extended to scenarios where different eigenvectors are projected onto multiple distinct directions. In such cases, it becomes necessary to consider a more complex interacting particle system.

One technical innovation of our proof is the comparison argument in Step 3 (see Lemma \ref{lemma_universal}). In the existing literature (e.g., \cite{bloemendal2016principal, knowles2013eigenvector}), the comparison is typically made between two random matrices, say $W$ and $W^G_t$, where $W^G_t$ has the same covariance structure or variance profile as $W$, but with entries that are Gaussian divisible (roughly speaking, we will set $W=Y$ and $W^G_t=\sqrt{1-t}Y+B_t$). 
%which can be studied via a dynamical approach on a small time scale. However, applying these ideas directly, in addition to the relatively tedious calculations, 
However, this argument is applicable in our context only when the population covariance matrix $\Sigma$ is diagonal, as demonstrated in \cite{ding2019singular}. To address this issue, we introduce a new comparison approach by introducing an additional intermediate matrix $\mathcal{W}_t$ as defined in (\ref{eq_mathcalwt}). The matrix $\mathcal{W}_t$ possesses the same covariance structure as $W^G_t$ while incorporating the randomness in $W$. Moreover, for a fixed $t$, we can carefully choose $\mathcal{W}_t$ in such a way that it has the same law as $W$. 
Then, to conclude Step 3, we will introduce a novel interpolation between $\cal W_t$ and $W^G_t$. Specifically, we define a continuous sequence of matrices $W_t^s$, $ s\in [0,1]$ (see (\ref{eq_dtusx}) for the precise definition). These matrices are specifically selected so that $W_t^0$ follows the distribution of $\cal W_t$, while $W_t^1$ follows the distribution of $W^G_t$. Notably, unlike previous comparison arguments in the literature, during the transition from $W_t^0$ to $W_t^1$, only the deterministic covariance structure of the interpolating matrices varies with $s$. Finally, the comparison is conducted by controlling the derivative of the relevant quantities with respect to $s$. For a more comprehensive explanation, readers can refer to the discussion below (\ref{eq_parttwocompare}).

%so the comparison will be conducted on a deterministic matrix instead of the random matrix part. 
%For the actual proof, by properly choosing a relevant quantity, we use a continuous interpolation argument as developed in \cite{MR3704770}. 

% $Y^G+\sqrt{t} X^G.$ However, if we directly apply these ideas, in addition to the relatively tedious calculations,  as can be seen in \cite{ding2019singular}, we can only handle diagonal $\Sigma.$ To address this issue, we introduce a novel idea by working with another intermediate matrix in (\ref{eq_mathcalwt}). Then to prove the universality, it suffices to work with $Y_t^s=D_t^{1/2}U^s \mathsf{X}$ (see (\ref{eq_dtusx}) for more precise definition), where $\mathsf{X}$ is the random part and $s=0$ corresponds to $Y$ while $s=1$ corresponds to $Y^G+\sqrt{t} X^G$. Consequently, the comparison will be conducted on the deterministic  matrices $U^0$ and $U^1$ instead of the random part. For the actual proof, by properly choosing a relevant quantity, we use a continuous interpolation arguments as developed in \cite{MR3704770}. For more details, we refer the readers to the argument below (\ref{eq_parttwocompare}). 

\vspace{3pt}

\noindent{\bf Outline of the paper.} In this paper, we focus on the spiked covariance matrix model with a general population covariance. In Section \ref{s:main}, we introduce this model and state our main results, which are divided into two parts: in Section \ref{sec_statisticsmainresults}, we present the asymptotics of shrinkers; in Section \ref{sec_mainresultque}, we provide results concerning the eigenvector distributions. 
In Section \ref{sec_statisticalestimation}, we introduce adaptive and consistent estimators for the shrinkers and conduct numerical simulations to demonstrate the superior performance of our estimators. In Section \ref{sec_generalproofdetail}, we discuss our proof strategy and highlight the technical novelties. All the technical proofs of the main results are given in Appendices \ref{appendix_preliminary}--\ref{appendix_finalone}. For the convenience of readers, in Appendix \ref{appendix_lossfunctionsummary}, we summarize some commonly used loss functions and the related shrinkers.

\vspace{3pt}

\noindent{\bf Notations.} 
To facilitate the presentation, we introduce some necessary notations that will be used in this paper. We are interested in the asymptotic regime with $p, n\to \infty$. 
When we refer to a constant, it will not depend on $p$ or $n$. Unless otherwise specified, we will use $C$ to denote generic large positive constants, whose values may change from line to line. Similarly, we will use $\epsilon$, $\delta$, $\tau$, $c$ etc.~to denote generic small positive constants. 
For any two sequences $a_n$ and $b_n$ depending on $n$, $a_n = \OO(b_n)$, or $a_n \lesssim b_n$ means that $|a_n| \le C|b_n|$ for some constant $C>0$, whereas $a_n=\oo(b_n)$ or $|a_n|\ll |b_n|$ means that $|a_n| /|b_n| \to 0$ as $n\to \infty$. We say that $a_n\asymp b_n$ if $a_n = \OO(b_n)$ and $b_n = \OO(a_n)$. 
Moreover, for a sequence of random variables $x_n$ and non-negative quantities $a_n$, we use $x_n=\OO_{\mathbb{P}}(a_n)$ to mean that $x_n/a_n$ is stochastically bounded and use $x_n=\oo_{\mathbb{P}}(a_n)$ to mean that $x_n/a_n$ converges to zero in probability. Denote by $\mathbb{C}_+:=\{z\in\C:\im z>0\}$ the upper half complex plane and $\mathbb{R}_+:=\{x\in\R:x>0\}$ the positive real line. For an event $\Xi$, we let $\mathbf 1_\Xi$ or $\mathbf 1(\Xi)$ denote its indicator function.  We use $\{\bf e_k\}$ to denote the standard basis of certain Euclidean space, whose dimension will be clear from the context. For any $a, b \in \Z$, we denote $\qq{a, b}: = [a,b]\cap \Z$ and abbreviate $\qq{a}:=\qq{1, a}$. Given two (complex) vectors $\bu$ and $\bv$, we denote their inner product by $\langle \bu, \bv\rangle:=\bu^*\bv.$ We use $\|\mathbf v\|_q$, $q\ge 1$, to denote the $\ell_q$-norm of $\mathbf v$. Given a matrix $B = (B_{ij})$, we use $\|B\|$ and $\|B\|_{HS}$, and $\|B\|_{\max}:=\max_{i,j}|B_{ij}|$ to denote the operator, Hilbert-Schmidt, and maximum norms, respectively. We also use $B_{\bu\bv}:=\bu^*B\bv$ to denote the ``generalized entries" of $B$.
%Given two vectors $\bu$ and $\bv$, we denote their inner product by $\langle \bu, \bv\rangle:=\bu^*\bv.$  Given a vector $\mathbf v$, $|\mathbf v|\equiv \|\mathbf v\|_2$ denotes the Euclidean norm. % and $\|\mathbf v\|_p$ denotes the $\ell_p$-norm. 

\vspace{3pt}

\noindent{\bf Acknowledgments.} The authors want to thank Jun Yin for many helpful discussions. XCD is partially supported by NSF DMS-2114379 and DMS-2306439. 
YL and FY are partially supported by the National Key R\&D Program of China (No. 2023YFA1010400).

%\newpage 

\section{Main results}\label{s:main}
%In this section, we present the main results. We first introduce our framework in Section \ref{sec_generallargecovariancematrixmodel}. Then, we present the main results on the convergent limits for shrinkers and the eigenvector distribution of general sample covariance matrices in Sections \ref{sec_statisticsmainresults} and \ref{sec_mainresultque}, respectively. 
%We first prepare some notations. For any probability measure $\nu$ on $\mathbb{R}$  its \emph{Stieltjes transform} is defined as {
%	\begin{equation}\label{eq_defnstitlesjtransform}
%	m_{\nu}(z)=\int_{\mathbb{R}} \frac{1}{x-z} \nu (\rd x), \qquad z \in \mathbb{C}_+. 
%	\end{equation} 
%	When $x\in\mathbb{R}$ we also adopt the convention that
%	\begin{equation*}
%	m_\nu(x)=\lim_{\eta \downarrow 0} m_\nu(x+\mathrm{i} \eta). 
%	\end{equation*} }
%For any $n \times n$ symmetric matrix $H$, we define its \emph{empirical spectral distribution} (ESD) as
%\begin{equation*}
%\mu_H=\frac{1}{n} \sum_{i=1}^n  \delta_{\lambda_i(H)}.
%\end{equation*}

\subsection{The model and main assumptions}\label{sec_generallargecovariancematrixmodel}

%We begin by defining our models and preparing some notations.

In this paper, we consider the \emph{non-spiked population covariance matrix} $\Sigma_0$, which is a deterministic $p \times p$ positive definite  matrix. Suppose it has a spectral decomposition  
\begin{equation}\label{eq_defnsigma0}
\Sigma_0=\sum_{i=1}^p \sigma_i \bm{v}_i \bm{v}_i^\top = V \Lambda_0 V^\top,\quad \Lambda_0=\operatorname{diag}\{\sigma_1, \cdots, \sigma_p\},
\end{equation}
where $\Lambda_0$ is the diagonal matrix of eigenvalues $0<\sigma_p \leq \cdots \leq \sigma_2 \leq \sigma_1<\infty$ arranged in the descending order, and $V$ denotes the eigenmatrix, i.e., the collection of the eigenvectors $\{\bm{v}_i\}$. 
% In what follows, we assume that the eigenvalues are arranged in the descending order:
% \begin{equation}\label{eq_evassumptionone}
% 0<\sigma_p \leq \sigma_{p-1} \leq \cdots \leq \sigma_2 \leq \sigma_1<\infty. 
% \end{equation}
%Throughout the paper, we call $\Sigma_0$ the \emph{non-spiked population covariance matrix}. 
In the statistical literature, people are also interested in finite rank deformations of $\Sigma_0$ by adding a fixed number of spikes %(i.e., eigenvalues that are detached from the bulk of the spectrum) 
to $\Lambda_0$. Following \cite{bun2017cleaning,DRMTA,ding2021spiked}, we define the \emph{spiked population covariance matrix} $\Sigma$ as 
\begin{equation}\label{eq_spikedcovariancematrix}
\Sigma=\sum_{i=1}^p \widetilde{\sigma}_i \bm{v}_i \bm{v}_i^\top  \equiv V \Lambda V^\top, \quad \Lambda=\operatorname{diag}\{\wt\sigma_1, \cdots, \wt\sigma_p\}.
\end{equation}
Here, for the sequence of eigenvalues $0<\widetilde{\sigma}_p \leq \cdots \leq \widetilde{\sigma}_2 \leq \widetilde{\sigma}_1<\infty$, 
% \begin{equation}\label{eq_spikebound}
% 0<\widetilde{\sigma}_p \leq \widetilde{\sigma}_{p-1} \leq \cdots \leq \widetilde{\sigma}_2 \leq \widetilde{\sigma}_1<\infty 
% \end{equation}   
% such that the following 
%For definiteness,
we assume that there exists a fixed $r\in \N$ and a non-negative sequence $\{d_i\}_{i=1}^r$ such that %$\{\widetilde{\sigma}_k\}_{k=1}^p$ satisfy
\begin{equation}\label{eq_defnspikes}
\widetilde{\sigma}_i:=
\begin{cases}
(1+d_i) \sigma_i, & i \leq r \\
\sigma_i, & i \geq r+1
\end{cases}.
\end{equation}
In other words, the first $r$ eigenvalues $\widetilde{\sigma}_i$, $i\in \qq{r}$, are the spiked eigenvalues with spiked strengths characterized by $d_i$. 
%With the above notations, we now introduce our models. In the current paper, we consider both the non-spiked and spiked population covariance matrices $\Sigma_0$ and $\Sigma.$ 
Accordingly, we define the non-spiked and spiked sample covariance matrices considered in this paper as
\begin{equation}\label{eq_samplecovariancematrixdefinition}
\mathcal{Q}_1:=\Sigma_0^{1/2} XX^\top \Sigma_0^{1/2}, \qquad  \widetilde{\mathcal{Q}}_1:=\Sigma^{1/2} XX^\top \Sigma^{1/2},
\end{equation}  
where $X \in \mathbb{R}^{p \times n}$ is a $p\times n$ random data matrix with i.i.d.~entries of mean zero and variance $n^{-1}.$ In our proof, we will also frequently use their companions:  
\begin{equation}\label{eq_grammatrixdefinition}
\mathcal{Q}_2:= X^\top \Sigma_0 X, \qquad  \widetilde{\mathcal{Q}}_2:=X^\top \Sigma X. 
\end{equation}  
Note $\mathcal Q_1$ (resp.~$\widetilde{\mathcal{Q}}_1$) has the same nonzero eigenvalues as $\mathcal Q_2$ (resp.~$\widetilde{\mathcal{Q}}_2$). In this paper, we will often regard the non-spiked model as a special case of the spiked one with $r=0$.

In this paper, we study the asymptotic behavior of the models $\mathcal{Q}_1$ and $\widetilde{\mathcal{Q}}_1$ using the \emph{Stieltjes transform method}. Given any $n \times n$ symmetric matrix $M$, we denote its \emph{empirical spectral density} (ESD) as 
	\begin{equation*}
\mu_M:=\frac{1}{n} \sum_{i=1}^n  \delta_{\lambda_i(M)},
\end{equation*}
where $\lambda_i(M)$, $i\in \qq{n}$, denote the eigenvalues of $M$. 
% via their \emph{resolvents} (or \emph{Green's functions}). For a  symmetric $n \times n$ matrix $H$ and $z\in\C_+$, we denote its resolvent and \emph{empirical spectral measure} (ESD) as 
% 	\begin{equation*}
% \mathcal{G}(H,z):=(H-z)^{-1},\qquad \mu_H=\frac{1}{n} \sum_{i=1}^n  \delta_{\lambda_i(H)}.
% \end{equation*}
Given a probability measure $\nu$ on $\mathbb{R}$, we define its \emph{Stieltjes transform} as
		\begin{equation}\label{eq_defnstitlesjtransform}
		m_{\nu}(z):=\int_{\mathbb{R}} \frac{1}{x-z} \nu (\rd x), \qquad z \in \mathbb{C}_+.
		\end{equation} 
When $z=x\in \R$, we adopt the convention that
		%\begin{equation*}
		$m_\nu(x):=\lim_{\eta \downarrow 0} m_\nu(x+\mathrm{i} \eta)$.
  % for $,\qquad x\in\mathbb{R}.$ 
%		\end{equation*} 
It is well-known that the Stieltjes transform of the ESD of $\mathcal{Q}_2$ has the same asymptotics as the so-called \emph{deformed Marchenko-Pastur (MP) law} $\varrho \equiv \varrho_{\Sigma_0, n}$ \cite{marchenko1967distribution}. It may include a Dirac delta mass $\delta_0$ at zero, for example, when $p > n$ or when $\sigma_i = 0$ for some $i$. Furthermore, $\varrho$ on $\R_+$ can be determined by its Stieltjes transform $m(z)$. 
%It may have a delta mass $\delta_0$ at zero (e.g., when $p>n$ or $\sigma_i=0$ for some $i$), and its $\varrho$ on $\R_+$ can be defined via its Stieltjes transform $m(z)$. 
More precisely, define the function $h:\C\to \C$ as
\begin{equation}\label{eq_defnf}
h(x)\equiv h(x,\Sigma_0):=-\frac{1}{x}+\frac{1}{n} \sum_{i=1}^p \frac{1}{x+\sigma_i^{-1}}. 
\end{equation}
(As a convention, we let $(x+\sigma_i^{-1})^{-1}=0$ when $\sigma_i=0$.) 
%then under the assumption of (\ref{eq_evassumptionone}), 
For any fixed $z\in\C_+$, there exists a unique solution $m \equiv m(z,\Sigma_0) \in \mathbb{C}_+$ to the self-consistent equation  
\begin{equation}\label{eq_defnmc}
z=h(m,\Sigma_0), \quad \text{with} \quad \operatorname{Im} m>0; 
\end{equation}  
see \cite[Lemma 2.2]{MR3704770} or the book \cite{baibook} for more details.  Then, we can determine the MP density $\varrho$ from $m$ as
\begin{equation}\label{ST_inverse}
\varrho(E) = {\pi}^{-1}\lim_{\eta\downarrow 0} \im m (E+\ii\eta),\quad E>0.
\end{equation}
%In the literature, $\varrho$ is referred to as the {deformed Marchenco-Pastur law}. 
We summarize some basic facts about the support of $\varrho$ in the following lemma.
\begin{lemma}[Lemma 2.5 of \cite{MR3704770} and Lemma 2.4 of \cite{ding2018}]\label{lem_property} 
The support of $\varrho$ is a disjoint union of connected components on ${\mathbb{R}}_+:$
\begin{equation}\label{eq_supportdmp}
{\mathbb{R}}_+\cap \operatorname{supp} \varrho  = {\mathbb{R}}_+\cap \bigcup_{k=1}^q [a_{2k}, a_{2k-1}]   ,
\end{equation}
where $q$ is an integer that depends only on the \textup{ESD} of $\Sigma_0$, and we shall call $a_{1} > a_2 >\cdots > a_{2q}\ge 0$ the \emph{spectral edges} of $\varrho$. Here, $\{a_k\}_{k=1}^{2q}$ can be characterized as follows: there exists a real sequence $\{b_k\}_{k=1}^{2q}$ such that $(x,m)=(a_k,b_k)$ are real solutions to the equations 
\begin{equation*}
x=h(m), \qquad  \text{and} \qquad h'(m)=0.
\end{equation*}
\end{lemma}  

Following the convention in the random matrix theory literature, we denote the rightmost and leftmost edges by $\lambda_+:=a_1$ and $\lambda_-:=a_{2q}$, respectively. For any $1\le k\le q$, we define
\begin{equation}\label{Nk}
n_k := \sum_{l\le k} n\int_{(a_{2l},a_{2l-1}]} \varrho (x)\dd x,
\end{equation}
which is the classical number of eigenvalues in $(a_{2k},\lambda_+]$. As a convention, we set $n_0=0$. It has been shown in \cite[Lemma A.1]{MR3704770} that $n_k\in \N$ for $k\in \qq{q}$. We now introduce the quantiles $\gamma_k$ of $\varrho$, which are indeed the classical locations for the eigenvalues of $\mathcal Q_2$.

\begin{definition}%[Generalized eigenvalue locations]
\label{defn_generaleigenvaluelocation} %We denote $\gamma_k$, $k\in \qq{p}$, as follows. For $1 \leq j \leq \mathsf{K},$ 
We define the \emph{classical eigenvalue locations} (or the quantiles) $\gamma_k \equiv \gamma_k(n)$ of the deformed Marchenco-Pastur law $\varrho$ as the unique solution to the equation 
% Then, we define the classical locations $\gamma_j$ for the eigenvalues of $\mathcal Q_2$ through
\begin{equation}\label{def:gamma_alpha}
 \int_{\gamma_k}^{\infty} \varrho(x)\dd x= \frac{k-1/2}{n},  \quad k\in \qq{\mathsf K},
\end{equation}
where we have abbreviated $\mathsf{K}:=n\wedge p$. 
Note that $\gamma_k$ is well-defined since the $n_k$'s are integers. As a convention, we set $\gamma_k=0$ for $k\in \qq{\mathsf{K}+1, n\vee p}$.

% \begin{equation}\label{def:gamma_alpha}
% \gamma_k:=\sup_{x}\left\{\int_{x}^{+\infty} \varrho(u)\dd u > \frac{k-1}{n}\right\},\quad k\in \qq{\mathsf K}.
% \end{equation}
% We also denote 
% \begin{equation}\label{eq_Kdefinition}
% \mathsf{K}:=\min\{p,n\},
% \end{equation} 
% In particular, we have $\gamma_1 = \lambda_+$. Moreover, when $n>p$, we set $\gamma_k=0$ for $k\in \qq{\mathsf{K}+1, n}$ as a convention.
% For the deformed Marchenco-Pastur law $\varrho$ as in Lemma \ref{lem_property}, we denote its \emph{classical eigenvalue locations} (i.e., $k/n$-quantiles) $\gamma_k \equiv \gamma_k(n)$ as
% 	\begin{equation}\label{def:gamma_alpha}
% 	\int_{\gamma_k}^{\infty} \varrho (\dd x)=\frac{k}{n}, \ 1 \leq k \leq \mathsf{K}.
% 	\end{equation}
% 	Moreover, when $p>n$ we denote $\gamma_k, \mathsf{K}+1 \leq k \leq p$ as
% 	\begin{equation}\label{eq_generalization}
% 	\mathbf{1}(p>n)\gamma_k \equiv 0, \  \mathsf{K}+1 \leq k \leq p.
% 	\end{equation}
\end{definition}

%To establish our main results, we require additional assumptions on our models. We summarize them as follows to avoid repetition.
%Before concluding this subsection, to avoid repetition, we summarize the assumptions as follows. 
%We will need the following mild assumptions \cite{MR3704770} to rule out the existence of spikes in $\Sigma_0$ and guarantee the regularity behavior of $\varrho.$  

Now, we are ready to state the main assumptions for our results. 
\begin{assumption}\label{main_assumption} 
We assume the following assumptions hold. % throughout the paper:
\begin{enumerate}
\item{\bf On dimensionality.} There exists a small constant $\tau\in (0,1)$ such that \eqref{eq_ratioassumption} holds and
\be\label{eq_ratioassumption2}
|c_n-1|\ge \tau.
\ee

% \begin{equation*}
% %	\label{eq_ratioassumption}
% \tau\leq c_n:=\frac{p}{n} \leq \tau^{-1}.
% \end{equation*}
\item{\bf On $X$ in (\ref{eq_samplecovariancematrixdefinition}).} For $X=(x_{ij}),$ suppose that $x_{ij} ,\ i \in \qq{p}, \ j \in \qq{n},$ are i.i.d.~real random variables with
%\begin{equation*}
$\mathbb{E} x_{ij}=0$ and $\mathbb{E} x_{ij}^2= {n}^{-1}.$  
%\end{equation*}
In addition, we assume that the entries of $X$ have arbitrarily high moments: for each fixed $k\in \N$, there exists a constant $C_k>0$ such that 
\begin{equation}\label{eq_momentassumption}
\mathbb{E} | \sqrt{n} x_{ij}|^k \leq C_k. 
\end{equation}
Finally, we assume that the entries of $X$ have vanishing third moments, i.e., $\mathbb{E} x_{ij}^3=0$. %  and for all $k \in \mathbb{N},$ there exists some constant $C_k>0$ such that 

\item{\bf On $\Sigma_0$ in (\ref{eq_defnsigma0}).} There exists a small constant $\tau_1 \in (0,1)$ such that  
\begin{equation}\label{eq_evassumptionone}
\tau_1 \leq \sigma_p \leq  \cdots \leq \sigma_2 \leq \sigma_1 \leq \tau_1^{-1}. 
\end{equation}
Moreover, for the two sequences of $\{a_k\}$ and $\{b_k\}$ given in Lemma \ref{lem_property}, we assume that 
\begin{equation}\label{eq:edgeregular}
a_k \geq \tau_1,  \quad \min_{l \neq k} |a_k-a_l|\geq \tau_1, \quad \min_{i}|\sigma_i^{-1}+b_k| \geq \tau_1. 
\end{equation}
Finally, for any small constant $\tau_2\in (0,1),$ there exists a constant $\varsigma \equiv \varsigma_{\tau_1, \tau_2}>0$ such that 
\begin{equation}\label{eq:bulkregular}
\varrho(x)\ge \varsigma \quad \text{for}\quad x\in [a_{2k}+\tau_2, a_{2k-1}-\tau_2], \ k\in \qq{q}.
\end{equation}
%the density $\varrho$ is bounded from below by $\varsigma$ inside each bulk component $[a_{2k}+\tau_2, a_{2k-1}-\tau_2]$.

\item{\bf On the spikes in (\ref{eq_defnspikes}).}  %we assume %(\ref{eq_spikebound}) holds 
There exists a fixed integer $r$ and a constant $\varpi\in (0,1)$ such that 
\begin{equation}\label{eq_outlierassumption}
\widetilde{\sigma}_i>-b_1^{-1}+\varpi \quad \text{for}\quad i\in \qq{r};\quad  {\sigma}_i<-b_1^{-1}-\varpi \quad \text{for}\quad i\in \qq{r+1,p}.  
\end{equation}
Moreover, the spikes are distinct in the sense that $\min_{i \neq j \in \qq{r}}| \widetilde{\sigma}_i-\widetilde{\sigma}_j|>\varpi.$ We also assume that the largest spike $\widetilde{\sigma}_1$ is bounded from above by $\varpi^{-1}$. %$\widetilde{\sigma}_i<\infty.$

% let $\mathsf{s}$ be the number of distinct values of $\{ \widetilde{\sigma}_1,\cdots, \widetilde{\sigma}_r\}$ and denote them as $\theta_1, \cdots, \theta_{\mathsf{s}}.$ We assume that $\min_{ 1 \leq i \neq j \leq \mathsf{s}}| \theta_i-\theta_j|>\varpi.$

\end{enumerate}
\end{assumption}

Let us now provide a brief discussion on the above assumption. 
Condition \eqref{eq_ratioassumption} of part (i) means that we are considering the high-dimensional regime in this paper. We remark that in related works such as \cite{donoho2018,ledoit2012,mainreference}, it is required that $c_n$ converges to a fixed constant $c\in (0,\infty)$ as $n\to \infty$. In comparison, our assumption (\ref{eq_ratioassumption}) is more flexible and purely data-dependent. The condition \eqref{eq_ratioassumption2} is introduced to avoid the occurrence of the ``hard edge" phenomenon in the deformed MP law $\rho$,
that is, $\lambda_-\to 0$ as $c_n\to 1$. In fact, by the properties of the MP law \cite{marchenko1967distribution}, it is known that 
\be\label{eq:lambda-}
\lambda_+\le \sigma_1 (1+\sqrt{c_n})^2, \quad \lambda_-\ge \sigma_p (1-\sqrt{c_n})^2. 
\ee
The condition \eqref{eq_ratioassumption2} is relatively mild since when $c_n=1+\oo(1)$, we can always omit certain samples to ensure that $p/n\ge 1+\tau$. Additionally, it is worth mentioning that even when $c_n=1+\oo(1)$, all our results remain valid for eigenvectors with indices up to $(1-\tau)p$, which correspond to the eigenvalues away from the hard edge $\lambda_-$.
%It is seen that this condition is mild since when $c_n=1+\oo(1)$, we can always drop some samples to ensure that $p/n\ge 1-\tau$. Furthermore, we also remark that even when $c_n=1+\oo(1)$, all our results remain valid for the eigenvectors with indices at most $(1-\tau)p$, which correspond to the eigenvalues away from the hard edge $\lambda_-$. 

Part (ii) imposes some moment conditions on the entries of $X$. We remark that the high moment condition (\ref{eq_momentassumption}) can be relaxed to a certain extent---we may only assume that \eqref{eq_momentassumption} holds for $k \leq C$ with $C$ being an absolute constant (e.g., $C=8$). 
Additionally, the vanishing third-moment condition is not essential, and it can be eliminated using an alternative approach distinct from our current proof. However, due to limitations in length, we will not explore this direction in the present paper, leaving it for future investigation.

For part (iii), the condition \eqref{eq_evassumptionone} is seen to be mild. The conditions \eqref{eq:edgeregular} and \eqref{eq:bulkregular} are some technical regularity conditions imposed on the ESD of $\Sigma_0$. The edge regularity condition \eqref{eq:edgeregular} has previously appeared (in slightly different forms) in several works on sample covariance matrices \cite{BPZ1, Karoui,HHN,MR3704770,LS,Regularity4}, and it ensures a regular square-root behavior of the density $\varrho$ near the spectral edges $a_k$. The bulk regularity condition \eqref{eq:bulkregular} was introduced in \cite{MR3704770}, and it imposes a lower bound on the asymptotic density of eigenvalues away from the edges. Both of these conditions are satisfied by ``generic" population covariance matrices $\Sigma_0$; see e.g., the discussions in \cite[Examples 2.8 and 2.9]{MR3704770}.

%is standard in random matrix theory. These conditions ensure that $\varrho$ behaves regularly near the edges and in the bulk. These conditions follow \cite{MR3704770}. They are generic for $\Sigma_0$ and  can be satisfied by general $\Sigma_0$, see Examples 2.8 and 2.9 of \cite{MR3704770}. 

Finally, condition \eqref{eq_outlierassumption} in part (iv) means that $\wt\sigma_i$, $i\in \qq{r}$, are the supercritical spikes concerning the BBP transition \cite{BBP} of the largest few  eigenvalues of \smash{$\widetilde{\mathcal{Q}}_1$}. In particular, they will give rise to outlier sample eigenvalues beyond the right edge $\lambda_+$ of $\supp \varrho$; see e.g., Theorem 3.6 of \cite{ding2021spiked}. 
It is possible to extend our results to the more general case with the sharper condition: for a small constant $\e>0$,
$$\widetilde{\sigma}_i>-b_1^{-1}+n^{-1/3+\e},\quad i\in \qq{r}.$$
(It is known that $-b_1^{-1}+\OO(n^{-1/3})$ is the critical regime for BBP transition.) 
We can also allow for degenerate spikes and spikes $\widetilde{\sigma}_i \equiv \widetilde{\sigma}_i(n)$ that diverge with $n$ as $n\to \infty$. However, for brevity, we do not pursue such generalizations in this paper. 

%Part (iv) of \Cref{main_assumption} also assumes that the spikes are non-degenerate and finite. imposes the condition that the spikes and the support of $\varrho$ are well-separated. In fact, we can replace $\varpi$ with $\OO(n^{-1/3}),$ allowing for the consideration of degenerate spikes (i.e.~spikes with algebraic multiplicity greater than one) and spikes $\widetilde{\sigma}_i \equiv \widetilde{\sigma}_i(n)$ that diverge with $n$ under a weak separation condition. For brevity, We do not pursue such generalizations in this paper. 

\subsection{Asymptotics of optimal shrinkage estimators}\label{sec_statisticsmainresults} 

In this subsection, we state the main statistical results on the analytical formulas and asymptotic risks for the optimal shrinkage estimators. Our first result concerns the asymptotics of the shrinkers. To state it, we now introduce more notations. For $h$ defined in (\ref{eq_defnf}) and $i \in \qq{r},$ we denote 
\begin{equation}\label{eq_importantdefinition}
\mathfrak{a}_i= h(-\widetilde{\sigma}_i^{-1}), \qquad \mathfrak{b}_i={h'(-\widetilde{\sigma}_i^{-1})}/{h(-\widetilde{\sigma}_i^{-1})}. 
%, \ \mathfrak{c}_i:= \frac{d_i^2}{\widetilde{\sigma}_i^2} \mathfrak{b}_i. 
\end{equation}
(We will see in \Cref{lemma_spikedlocation} that $\fa_i$ is the classical location (i.e., convergent limit) of the $i$-th outlier eigenvalue $\lambda_i(\wt{\cal Q}_1)$.)  
As noted in \eqref{eq_varphigeneralform} and \eqref{eq_varphigeneralform0}, the optimal shrinkers usually depend on the loss function $\mathcal{L}$ via some continuous function $\ell(x)$ on $(0,\infty)$. Given such a function $\ell$, for $\mathsf{t}>0$ small, and $x>0$, we define 
\begin{equation}\label{eq:sigmait}
\sigma_{i,\mathsf t}(x):= \frac{\sigma_i}{1+\mathsf{t}\ell(\sigma_i)/x},\qquad i\in \qq{p}.
\end{equation}
% \begin{equation}\label{eq:sigmait}
% \sigma_{i,\mathsf t}\equiv \sigma_{i,\mathsf{t}}(z) := \frac{\sigma_i}{1+\mathsf{t}\ell(\sigma_i)/(z\sigma_i)},\qquad i\in \qq{p}.
% \end{equation}
%The sequence $\{\widetilde{\sigma}_i\}_{i=1}^p$  will be used in deriving the asymptotics of the shrinkers via a variational argument, see XXX for more details. 
We further define $\wt{m}_{\mathsf{t}}(z,x)$ as in (\ref{eq_defnmc}) by replacing $\{\sigma_i\}_{i=1}^p$ with $\{{\sigma}_{i,\mathsf t}(x)\}_{i=1}^p$, that is,  $\widetilde m_{\st}(z,x)$ solves the following self-consistent equation: 
\begin{equation}\label{eq_selfconsistent_mt}
  z = -\frac{1}{\widetilde m_{\mathsf{t}}(z,x)} +\frac{1}{n}\sum_{i=1}^p \frac{1}{\widetilde m_{\mathsf{t}}(z,x)+\sigma_i^{-1}(1+\mathsf{t}\ell(\sigma_i)/x)}. %\quad \im \widetilde m_\st>0.
\end{equation}
Note when $\mathsf{t}=0$, we have ${\sigma}_{i,0}(x)=\sigma_i$ and $\widetilde{m}_0(z,x)=m(z)$. 
%With a bit of abuse of notation, 
We then denote the partial derivative of $\wt m_\st(z,x)$ in $\mathsf{t}$ as 
\begin{equation}\label{eq_mtderivative}
\dot m_0(z,x)=\left. \partial_\st \wt m_\st(z,x)\right|_{\st=0}.  
\end{equation}
By differentiating the equation \eqref{eq_selfconsistent_mt} with respect to $z$ and $\mathsf{t}$ at $\st=0$, we obtain that 
\be\label{eq_m'm0} \frac{m'(z)}{m(z)^2}-\frac{1}{n}\sum_{i=1}^p \frac{m'(z)}{[m(z)+\sigma_i^{-1}]^2}=1,\quad \frac{\dot m_0(z,x)}{m(z)^2}-\frac{1}{n}\sum_{i=1}^p \frac{\dot m_0(z,x)+\ell(\sigma_i)/(\sigma_i x)}{[m(z)+\sigma_i^{-1}]^2}=0.
\ee
From these two equations, we can derive the following explicit expression for $\dot m_0(z,x)$:
\be\label{eq_formofm0}
\dot m_0(z,x)=\frac{m'(z)}{nx}\sum_{i=1}^p \frac{\ell(\sigma_i)/\sigma_i}{[m(z)+\sigma_i^{-1}]^2}.
\ee

%  For $1 \leq i \leq r,$ we denote {\color{red} the notations need to be changed here.
%\begin{equation}
%\theta (d_i):=h(-\widetilde{\sigma}_i^{-1}) \equiv h(-\sigma_i^{-1} (1+d_i)^{-1}).
%\end{equation}  

By abbreviating $\dot m_0(\mathfrak{a}_i)\equiv \dot m_0(\mathfrak{a}_i,\fa_i)$, we now define 
\begin{equation}\label{eq_defnpsi}
\psi_i\equiv \psi_i(\ell):=\mathfrak{b}_i \left( \frac{\ell(\widetilde{\sigma}_i)}{\widetilde{\sigma}_i} +\mathfrak{a}_i \dot m_0(\mathfrak{a}_i)
\right), \quad i \in \qq{r},
\end{equation}
%%where $\widetilde m'_0(\mathfrak{a}_i)$ is defined as follows
%%\begin{equation}\label{eq_mtderivative}
%%\widetilde m'_0(\mathfrak{a}_i)=\frac{\partial \widetilde m_\st(z)}{ \partial \st}\Big|_{\st=0, z=\mathfrak{a}_i}.  
%%\end{equation}
%%\begin{equation}\label{eq_defnpsi}
%%\psi_i:=\frac{\ell(\widetilde{\sigma}_i)}{\widetilde{\sigma}_i} \mathfrak{b}_i+\frac{\mathfrak{c}_i}{(\sigma_i^{-1}+m(\mathfrak{a}_i))^2} \frac{1}{n} \sum_{j=r+1}^p \frac{\ell(\sigma_j)}{\sigma_j} \frac{m'(\mathfrak{a}_i)}{(\sigma_j^{-1}+m(\mathfrak{a}_i))^2}.
%%\end{equation}`
%%In addition, 
%
%
%%\begin{equation*}
%%\psi_i:=\frac{g(\widetilde{\sigma}_i)}{\widetilde{\sigma}_i} \frac{h'(-\widetilde{\sigma}_i^{-1})}{h(-\widetilde{\sigma}_i^{-1})}+\frac{d_i^2 \theta'(d_i)}{ \sigma_i \theta(d_i) (\sigma_i^{-1}+m(\theta(d_i)))^2} \frac{1}{n} \sum_{j=r+1}^p \frac{g(\widetilde{\sigma}_j) m'(\theta(d_i))}{\widetilde{\sigma}_j (\widetilde{\sigma}_j^{-1}+m(\theta(d_i)))^2}.
%%\end{equation*}
%
%%For notational convenience, 
% 
%
%
%and the diagonal matrix $\Lambda_0$ in (\ref{eq_defnsigma0}), 
%and for $i\in \qq{r+1,p}$,  %given an integer $1\le \upsilon\le p$, 
and the function %we set 
\begin{equation}\label{eq_keydefinition1111}
\vartheta(x) \equiv \vartheta(\ell,x):=\frac{1}{p} \sum_{j=r+1}^{p} \ell(\wt\sigma_j) \phi(\bm{v}_j, \bm{v}_j, x), 
\end{equation} 
%\begin{equation}\label{eq_keydefinition1111}
%\vartheta_i=\frac{1}{p} \left(\sum_{j=1}^{\mathsf{K}} g(\sigma_j) \phi(\bm{v}_j, \bm{v}_j, \gamma_i)+\mathbf{1}(c_n>1) \sum_{j=\mathsf{K}+1}^p g(\sigma_j) \phi(\bm{v}_j, \bm{v}_j, 0) \right), \  
%\end{equation} 
where for any two vectors $\bm{w}_1, \bm{w}_2\in \R^p$, we denote
\begin{equation}\label{eq_phidefinitionintheend}
\phi(\bm{w}_1, \bm{w}_2, x):=
\begin{cases}
c_n \bm{w}_1^\top \left[\Sigma \left(x |1+m(x)\Sigma|^2\right)^{-1} \right] \bm{w}_2, & x>0 \\
 \left(1-c_n^{-1}\right)^{-1} \bm{w}_1^\top \left(1+m(0)\Sigma\right)^{-1}  \bm{w}_2, & x=0 \ \text{and} \ c_n>1
\end{cases}. 
\end{equation}
Here, we adopt the convention that $\phi(\bm{w}_1, \bm{w}_2, x):=\lim_{\eta \downarrow 0} \phi(\bm{w}_1, \bm{w}_2, x+\ii \eta)$ for $x\in \R$. 

In this paper, we denote the eigenvalues of the non-spiked sample covariance matrix $\cal Q_1$ by $\lambda_1\ge \lambda_2\ge \cdots \ge \lambda_p$ and those of the spiked version $\widetilde{\mathcal{Q}}_1$ by $\wt\lambda_1\ge \wt\lambda_2 \ge \cdots \ge \wt\lambda_p$.
%For ease of presentation, and with a slight abuse of notations, we will denote the sample eigenvectors of either $\cal Q_1$ or \smash{$\widetilde{\mathcal{Q}}_1$} as $\{\bm{u}_i\}_{i=1}^p$. In addition, we denote by $U_0=(\bm{u}_{\mathsf K+1},\ldots, \bm{u}_p)\in \R^{p-n}$ the eigenmatrix associated with the zero eigenvalues of $\cal Q_1$ or $\wt{\cal Q}_1$. 
To facilitate the presentation and with a slight abuse of notation, we will use the notation $\{\bm{u}_i\}_{i=1}^p$ to represent the sample eigenvectors of either $\cal Q_1$ or \smash{$\widetilde{\mathcal{Q}}_1$}. Additionally, we denote by $U_0=(\bm{u}_{\mathsf K+1},\ldots, \bm{u}_p)\in \R^{p\times (p-n)}$ the eigenmatrix associated with the zero eigenvalues of $\cal Q_1$ or $\wt{\cal Q}_1$. The specific model being referred to will be evident from the context. 
%With a bit of abuse of notation, we use $\{\bm{u}_i\}_{i=1}^p$ for the eigenvectors of the spiked or non-spiked sample covariance matrices associated with the eigenvalues in a decreasing fashion. 
Now, we are ready to state the main results. First, we provide the convergent limits of the shrinkers.

\begin{theorem}\label{thm_shrinkerestimate} Suppose Assumption \ref{main_assumption} holds. Recall that $\gamma_i$ are the quantiles of $\varrho$ defined in \Cref{defn_generaleigenvaluelocation}. For any continuous function $\ell(x)$ defined on $(0, \infty)$, the following estimates hold:
%For $i \in \qq{r}$, we have that 
\begin{align}
    \mathbb{E} \left| \bm{u}_i^\top \ell(\Sigma) \bm{u}_i-\psi_i(\ell) \right|=\oo(1),\quad\quad &i\in \qq{r}; \label{eq_outliercase}\\
    \mathbb{E} \left| \bm{u}_i^\top \ell(\Sigma) \bm{u}_i- \vartheta(\ell,\gamma_{i-r}) \right|=\oo(1),\quad \quad &i\in \qq{r+1,\sfK};\label{eq_nonspikedshrinkeranalytical}\\
    \mathbb{E} \left| (p-n)^{-1}\tr\left[U_0^\top \ell(\Sigma) U_0\right] - \vartheta (\ell,0)\right|=\oo(1),\quad\quad &i\in \qq{\sfK+1,p}. \label{eq_nonspikedshrinkeranalytical0}
\end{align}
% \begin{equation}\label{eq_outliercase}
% \mathbb{E} \left| \bm{u}_i^\top \ell(\Sigma) \bm{u}_i-\psi_i(\ell) \right|=\oo(1),\quad i\in \qq{r};  
% \end{equation} 
% %For $i\in \qq{r+1,\sfK}$, we have that 
% \begin{equation}\label{eq_nonspikedshrinkeranalytical}
% \mathbb{E} \left| \bm{u}_i^\top \ell(\Sigma) \bm{u}_i- \vartheta(\ell,\gamma_{i-r}) \right|=\oo(1),\quad i\in \qq{r+1,\sfK};
% \end{equation} 
% %Finally, for $i\in \qq{\sfK+1,p}$, we have that 
% \begin{equation}\label{eq_nonspikedshrinkeranalytical0}
% \mathbb{E} \left| (p-n)^{-1}\tr\left[U_0^\top \ell(\Sigma) U_0\right] - \vartheta (\ell,0)\right|=\oo(1),\quad i\in \qq{\sfK+1,p}. 
% \end{equation} 
The above results (\ref{eq_nonspikedshrinkeranalytical}) and (\ref{eq_nonspikedshrinkeranalytical0}) also extend to the non-spiked model with $r=0$.

% \begin{enumerate}
% \item For the non-spiked covariance model, 

% \item For the spiked covariance model,
% Moreover, when $r+1 \leq i \leq p,$ (\ref{eq_nonspikedshrinkeranalytical}) and \eqref{eq_nonspikedshrinkeranalytical0} still hold with $\vartheta_i \equiv \vartheta_i(\ell, r+1)$. 
% % in (\ref{eq_keydefinition1111}).
% % by replacing $\vartheta_i$ with $\vartheta_{i-r}$. 
% \end{enumerate}
\end{theorem}

Theorem \ref{thm_shrinkerestimate} provides a closed-form, analytic, and deterministic formula for the convergent limits of the shrinkage estimators. As summarized in Appendix \ref{appendix_lossfunctionsummary}, commonly used loss functions include $\ell(x)=x,\ x^2,$ $x^{-1},\ x^{-2},\ \log x$, etc. Our result gives the asymptotics of the shrinkers with exact dependence on the loss functions, the population eigenvalues, the classical eigenvalue locations, and the aspect ratio $c_n$. In Section \ref{sec_combinedalgo}, we will provide numerical algorithms to estimate the aforementioned quantities consistently with sample quantities.
%by showing that $\vartheta_i$ or $\psi_i$ can be approximated well by sample quantities. 
Shrinkage estimators for other random matrix models have also been studied to some extent; see e.g., \cite{benaych2023optimal,bun2017cleaning, 7820186,leeb2022optimal}. We believe that our strategies and arguments can be extended and applied to these models with some modifications. We will explore this direction in future studies.

In the special case with $\ell(x)=x$ (corresponding to the loss functions of  Frobenius, minimal variance, disutility, and inverse Stein norms), the formulas of $\vartheta$ or $\psi_i$ will be significantly simplified, as shown in the following result.
\begin{corollary}\label{coro_simplifiedformula}
Suppose Assumption \ref{main_assumption} holds. Define functions $ \xi(x)$ and $\zeta(x)$ as
\begin{equation*}
\xi(x)\deq
\begin{cases}
\frac{1}{x|m(x)|^2}, & x>0 \\
\frac{1}{(c_n-1)m(0)}, & x=0 
\end{cases}, \quad \quad \zeta(x)\deq \frac{m'(x)}{|m(x)|^2}.
\end{equation*}
% and $\zeta(x)=  \mathfrak{b}_i\frac{m'(x)}{|m(x)|^2}$.
Then, the following estimates hold:
\begin{align}
    \mathbb{E}\left|\bm{u}_i^\top \Sigma \bm{u}_i-\mathfrak{b}_i\zeta (\mathfrak{a}_i)\right|=\oo(1), \quad\quad  &i\in \qq{r};\label{eq_spikedform}\\
    \mathbb{E}\left|\bm{u}_i^\top \Sigma \bm{u}_i-\xi(\gamma_{i-r})\right|=\oo(1),\quad\quad &i \in \qq{r+1,\sfK};\label{eq_nonspikeduniform}\\
    \mathbb{E} \left| (p-n)^{-1}\tr\left[U_0^\top \Sigma U_0\right] - \xi(0) \right|=\oo(1),\quad\quad &i\in\qq{\sfK+1,p}. \label{eq_nonspikeduniform0}
\end{align}
% \begin{equation}\label{eq_spikedform}
% \mathbb{E}\left|\bm{u}_i^\top \Sigma \bm{u}_i-\mathfrak{b}_i\zeta (\mathfrak{a}_i)\right|=\oo(1), \quad  i\in \qq{r};
% \end{equation} 
% \begin{equation}\label{eq_nonspikeduniform}
% \mathbb{E}\left|\bm{u}_i^\top \Sigma \bm{u}_i-\xi(\gamma_{i-r})\right|=\oo(1),\quad i \in \qq{r+1,\sfK};
% \end{equation}
% \begin{equation}\label{eq_nonspikeduniform0}
% \mathbb{E} \left| (p-n)^{-1}\tr\left[U_0^\top \Sigma U_0\right] - \xi(0) \right|=\oo(1),\quad i\in\qq{\sfK+1,p}. 
% \end{equation}
The above results (\ref{eq_nonspikeduniform}) and (\ref{eq_nonspikeduniform0}) also extend to the non-spiked model with $r=0$. 
% Then, for the non-spiked covariance model, we have that for $ i \in \qq{\sfK},$

% and for $i\in\qq{\sfK+1,p}$,
% \begin{equation}\label{eq_nonspikeduniform0}
% \mathbb{E} \left| (p-n)^{-1}\tr\left[U_0^\top \Sigma_0 U_0\right] - \xi(\gamma_i) \right|=\oo(1). 
% \end{equation} 
% For the spiked covariance model, we have that
% \begin{equation}\label{eq_spikedform}
% \mathbb{E}|\bm{u}_i^\top \Sigma \bm{u}_i-\mathfrak{b}_i\zeta (\mathfrak{a}_i)|=\oo(1), \qquad  i\in \qq{r}.
% %\ \text{where} \ \zeta_i= \frac{1}{\mathfrak{a}_i|m(\mathfrak{a}_i)|^2}
% % \mathfrak{b}_i \left(1+\mathfrak{a}_i \widetilde m_0'(\mathfrak{a}_i) \right).
% \end{equation} 
%\begin{equation}\label{eq_spikedform}
%\mathbb{E}|\bm{u}_i^\top \Sigma \bm{u}_i-\xi(\mathfrak{a}_i)|=\oo(1),
%\end{equation}
%Moreover, when $r+1\le i\le p$, (\ref{eq_nonspikeduniform}) and \eqref{eq_nonspikeduniform0} still hold true. 
%with $\gamma_i$ replaced with $\gamma_{i-r}.$ 

\end{corollary}

Compared with Theorem \ref{thm_shrinkerestimate}, we have simpler and more implementable formulas in the case of $\ell(x)=x$. 
%For the non-spiked model, these formulas no longer involve the population eigenvalues. 
Notably, the unobserved quantities depend only on the Stieltjes transform $m(z)$ and the classical eigenvalue locations $\fa_i$ or $\gamma_i$. As we will see in Section \ref{sec_combinedalgo} and \Cref{appendix_preliminary}, $m(z)$ can be well approximated by the Stieltjes transform of the ESD of $\cal Q_2$ or \smash{$\wt {\cal Q}_2$} due to the local laws of their resolvents, and $\fa_i$ and $\gamma_i$ can be efficiently estimated by their sample counterparts using the convergence of outlier eigenvalues and the rigidity of non-outlier eigenvalues. 
%the convergence of spikes and rigidity of eigenvalues can be efficiently estimated using their sample versions according to the local laws. 
These estimations only rely on the sample eigenvalues of $\cal Q_2$ or $\wt {\cal Q}_2$, which significantly simplifies our numerical algorithms.

%Compared to the general formulas in Theorem \ref{thm_shrinkerestimate}, the case $\ell(x)=x$ leads to more simple and implementable formulas. First of all, the formulas are unified and indifferent regardless of the whether the model is spiked or not. Second, the formulas do not contain the population eigenvalues any more. Third,  the only non-observed quantities are the Stieltjes transforms and the eigenvalue locations. As will be seen in Section \ref{sec_combinedalgo}, according to the local laws, convergence of spikes and rigidity of eigenvalues,  both of them can be efficiently estimated using their sample versions which  only rely on the eigenvalues of the sample covariance matrices. This significantly simplifies   our numerical algorithms. 

%\begin{remark}
%{\color{blue} list some examples of $\ell(x)$ based on the loss functions in Appendix \ref{appendix_lossfunctionsummary}. Also mention that (\ref{eq_outliercase}) holds with high probability and sharp rate $n^{-1/2}$. Also discuss how the choice of $\ell(x)$ influence our results. Especially when $\ell(x)=x,$ $\psi_i$ will be largely simplified. Discuss the estimation of the spectrum, refer to Section \ref{sec_combinedalgo}. }
%\end{remark}

Before concluding this section, we establish the asymptotic risks (or generalization errors) for Stein's estimator under 12 commonly used loss functions as summarized in Appendix \ref{appendix_lossfunctionsummary}.

\begin{corollary}\label{coro_lowbound} Suppose Assumption \ref{main_assumption} holds. Recall (\ref{eq_defnpsi}) and (\ref{eq_keydefinition1111}). Given a  function $\ell$, 
%for the non-spiked model, define 
%$\theta_i(\ell):=\vartheta_i(\ell,1)$, $i\in \qq{p}$; for the spiked model, 
define $\theta_i(\ell):=\psi_i(\ell)$ for $i\in \qq{r}$, $\theta_i(\ell):=\vartheta(\ell,\gamma_{i-r})$ for $i\in \qq{r+1,\sfK}$, and $\theta_i(\ell):=\vartheta(\ell,0)$ for $i\in \qq{\sfK+1,p}$. 
Let \smash{$\wt\Sigma$} denote the optimal invariant
estimators for $\Sigma$ as defined in \eqref{eq_steincovarianceestimator}.
For the loss functions $\mathcal{L}(\cdot,\cdot)$ given in Table \ref{table_summarylossfunction} below, the following estimates hold with probability $1-\oo(1)$. 
\begin{enumerate}
\item For %the non-spiked model %$\theta_i\equiv \theta_i(\ell)$ and 
$\theta_i\equiv \theta_i(\ell)$ with $\ell(x)=x,$ we have that with probability $1-\oo(1)$, 
\begin{equation*}
\begin{aligned}
&\mathcal{L}_{\operatorname{Fro}}(\Sigma, \wt{\Sigma})=\sqrt{\frac{1}{p} \sum_{i=1}^p \left( \sigma_i^2-\theta_i^2 \right) }+\oo(1), \ \ \mathcal{L}_{\operatorname{inStein}}(\Sigma, \wt{\Sigma})=\frac{1}{p} \sum_{i=1}^p \log \frac{\theta_i}{\sigma_i} +\oo(1),\\
&\mathcal{L}_{\operatorname{disu}}(\Sigma, \wt{\Sigma})=1-\frac{\sum_{i=1}^p \theta_i^{-1}}{\sum_{i=1}^p \sigma_i^{-1}}+\oo(1), \ \ \mathcal{L}_{\operatorname{MV}}(\Sigma, \wt{\Sigma})=p \left( \frac{1}{\sum_{i=1}^p \theta_i^{-1}}-\frac{1}{\sum_{i=1}^p \sigma_i^{-1}} \right)+\oo(1), 
\end{aligned}
\end{equation*}
where $\operatorname{Fro,\ inStein,\ disu}$, and $\operatorname{MV}$ are the shorthand notations for the Frobenius, inverse Stein, disutility, and minimum variance norms, respectively.

\item For $\theta_i\equiv \theta_i(\ell)$ with  $\ell(x)=x^{-1},$ we have that with probability $1-\oo(1)$,
\begin{equation*}
\begin{aligned}
&\mathcal{L}_{\operatorname{inFro}}(\Sigma, \wt{\Sigma})=\sqrt{\frac{1}{p}\sum_{i=1}^p \left( \sigma_i^{-2}-\theta_i^{2} \right)}+\oo(1), \  \  
\mathcal{L}_{\operatorname{Stein}}(\Sigma, \wt{\Sigma})=\frac{1}{p} \sum_{i=1}^p \log (\sigma_i \theta_i)+\oo(1), \\
&\mathcal{L}_{\operatorname{wFro}}(\Sigma, \wt{\Sigma})=1-\frac{\sum_{i=1}^p \theta_i^{-1}}{\sum_{i=1}^p \sigma_i}+\oo(1),
\end{aligned}
\end{equation*}
where $\operatorname{inFro,\ Stein,}$ and $\operatorname{wFro}$ are the shorthand notations for the inverse Frobenius, Stein, and weighted Frobenius norms, respectively. 

\item For $\theta_i\equiv \theta_i(\ell)$ with  $\ell(x)=\sqrt{x},$ we have that with probability $1-\oo(1)$, 
\begin{equation*}
\mathcal{L}_{\operatorname{Fre}}(\Sigma, \wt{\Sigma})=\sqrt{\frac{1}{p} \sum_{i=1}^p (\sigma_i-\theta_i^2)}+\oo(1),
\end{equation*}
where $\operatorname{Fre}$ means the Fr{\'e}chet norm. 

\item For $\theta_i\equiv \theta_i(\ell)$ with  $\ell(x)=\log x,$ we have that with probability $1-\oo(1)$, 
\begin{equation*}
\mathcal{L}_{\operatorname{LE}}(\Sigma, \wt{\Sigma})=\sqrt{\frac{1}{p} \sum_{i=1}^p \left[(\log \sigma_i)^2 -  {\theta_i^2}\right]}+\oo(1),
\end{equation*}
where $\operatorname{LE}$ means the Log-Euclidean norm. 

\item Define $\theta_{i,k}\equiv \theta_i(\ell_k)$, $k=1,2,3,4,$ with $\ell_1(x)=x^{-2}$, $\ell_2(x)=x^{-1}$, $\ell_3(x)=x$, and $\ell_4(x)=x^2$. Then, we have that with probability $1-\oo(1)$, 
\begin{equation*}
\begin{aligned}
&\mathcal{L}_{\operatorname{Qu}}(\Sigma, \wt{\Sigma})=\sqrt{1-\frac{1}{p} \sum_{i=1}^p \frac{(\theta_{i,2})^2}{\theta_{i,1}}}+\oo(1), \ \ \mathcal{L}_{\operatorname{inQu}}(\Sigma, \wt{\Sigma})=\sqrt{1-\frac{1}{p} \sum_{i=1}^p \frac{(\theta_{i,3})^2}{\theta_{i,4}}}+\oo(1), \\
&\mathcal{L}_{\operatorname{symStein}}(\Sigma, \wt{\Sigma})=2 \left( \frac{1}{p} \sum_{i=1}^p\sqrt{\theta_{i,2} \theta_{i,3}} -1\right)+\oo(1),
\end{aligned}
\end{equation*}
where $\operatorname{Qu,\ inQu,}$ and $\operatorname{symStein}$ are the shorthand notations for the quadratic, inverse quadratic, and symmetrized Stein norms, respectively. 
\end{enumerate}
%All the above results remain valid as long as we replace $\sigma_i$ and $\theta_i$ by $\wt\sigma_i$ and $\wt\theta_i$, respectively. 
%All of the above equations remain valid for $\cal L(\Sigma,\wt\Sigma)$ as long as we substitute $\sigma_i$ and $\theta_i$ with $\widetilde{\sigma}_i$ and $\widetilde{\theta}_i$, respectively.
\end{corollary}

\subsection{Eigenvector distributions for sample covariance matrices}\label{sec_mainresultque}

In this subsection, we present the theoretical results regarding the asymptotic distributions of the non-outlier eigenvectors of (non-spiked or spiked) sample covariance matrices. For simplicity of presentation, we adopt the notion of ``asymptotic equality in distribution", as defined in \cite[Definition 2.2]{bao2022statistical}.

\begin{definition}\label{defn_asymptoticweakconvergence}
Two sequences of random vectors $\bm{x}_n, \bm{y}_n \in \mathbb{R}^{k}$ are \emph{asymptotically equal in distribution}, denoted by $\bm{x}_n \simeq \bm{y}_n,$ if they are tight (i.e., for any $\epsilon>0,$ there exists a $C_\e>0$ such that $\sup_n \mathbb{P}(\| \bm{x}_n\| \geq C_\e) \leq \epsilon$) and satisfy
\begin{equation*}
\lim_{n \rightarrow \infty} \left[ \mathbb{E} \mathfrak{h}(\bm{x}_n)-\mathbb{E} \mathfrak{h}(\bm{y}_n) \right]=0,
\end{equation*}
for any bounded smooth function $\mathfrak{h}: \mathbb{R}^{k} \rightarrow \mathbb{R}.$ 
\end{definition}

\begin{theorem}[Eigenvector distribution]\label{Thm: EE}
Suppose Assumption \ref{main_assumption} holds. Given a deterministic unit vector $\bm{v} \in \mathbb{R}^p$ and any fixed integer $L$, take a subset of indices %$\{i_k\}_{k=1}^L \subset \qq{1,p}$ for the non-spiked model and 
$\{i_k+r\}_{k=1}^L \subset \qq{r+1,\sfK}$ for the spiked model. Define the $L\times L$ diagonal matrix 
\begin{equation}\label{eq_Vdefinition}
\Xi_L:=\operatorname{diag}\left\{ \phi(\bm{v}, \bm{v},\gamma_{i_1}), \cdots, \phi(\bm{v}, \bm{v},\gamma_{i_L}) \right\},
\end{equation}
where $\phi$ is defined in (\ref{eq_phidefinitionintheend}) and $\{\gamma_{i_k}\}_{k=1}^L$ are defined in Definition \ref{defn_generaleigenvaluelocation}. 
Then, we have that  
\be\label{eq:conv_indis}
(p|\avg{\bm{v},\bm{u}_{i_1+r}}|^2,\ldots, p|\avg{\bm{v},\bm{u}_{i_L+r}}|^2) \simeq (|\cal N_1|^2,\ldots,|\cal N_L|^2),
\ee
where $\{\cal N_k\}_{k=1}^L$ are independent centered Gaussian random variables with variances $\phi(\bm{v}, \bm{v},\gamma_{i_k})$. Equivalently, we can express \eqref{eq:conv_indis} as 
\begin{equation}\label{eq:conv_indis2}
% \sqrt{p} 
% \begin{pmatrix}
%  \avg{\bm{v},\bm{u}_{i_1}} \\
%  \vdots \\
%  \avg{\bm{v}, \bm{u}_{i_L}}
% \end{pmatrix}
% \simeq \mathcal{N}(\bm{0}, \Xi_L), \quad \text{or}\quad 
\sqrt{p} 
\begin{pmatrix}
 \xi_1 \avg{\bm{v},\bm{u}_{i_1+r}} \\
 \vdots \\
 \xi_L \avg{\bm{v}, \bm{u}_{i_L+r}}
\end{pmatrix}
\simeq \mathcal{N}(\bm{0}, \Xi_L) ,
\end{equation}
where $\xi_1,\ldots, \xi_L\in\{\pm 1\}$ are i.i.d.~uniformly random signs independent of $X$. 
%where the first equation is for the non-spiked model and the second one is for the spiked model. 
The above result also extends to the non-spiked model with $r=0$. % (recall that $\{\bm{u}_i\}_{i=1}^p$ represent the sample eigenvectors of either $\cal Q_1$ or \smash{$\widetilde{\mathcal{Q}}_1$}). 
%\left(\left\langle \bv, \bu_k\right\rangle,  \left\langle \bv, \bu_l\right\rangle\right) \to  \cal N\left(0,   \varphi_0\left(\lim_{N}\gamma_{k}\right)\right)\times \cal N\left(0,   \varphi_0\left(\lim_{N}\gamma_{l}\right)\right)
\end{theorem}
%
%\begin{theorem}\label{Thm: BE}[Bulk eigenvectors] There exists some $\e>0$ such that for any  
%\be\label{klbulk}
%k,l\in[ N^{1-\e}, K-N^{1-\e}]
%\ee
%and fixed $\bv \in\R^M$,  we  have 
%$$
%\left(\left\langle \bv, \bu_k\right\rangle,  \left\langle \bv, \bu_l\right\rangle\right) \to  \cal N\left(0,   \varphi_0\left(\lim_{N}\gamma_{k_N}^{(N)}\right)\right)\times \cal N\left(0,   \varphi_0\left(\lim_{N}\gamma_{l_N}^{(N)}\right)\right)
%$$
%in distribution.  
%
%\end{theorem}

Theorem \ref{Thm: EE} provides a characterization of the joint distribution of the projections $\avg{\bm{v},\bm{u}_{i_k}}$ (referred to as ``generalized components") of the non-outlier eigenvectors from the non-spiked or spiked sample covariance matrices. It demonstrates that the generalized components of different eigenvectors are asymptotically independent Gaussian variables, and their variances can be explicitly computed. Given that eigenvectors are defined only up to a phase, we express the distributions of eigenvectors using the forms presented in \eqref{eq:conv_indis} or \eqref{eq:conv_indis2}. 
As a specific example, when $\Sigma=I_p$, we can verify that $\phi(\bm{v},\bm{v},\gamma_{i_k})=1$ by using the explicit form of $m(z)$ solved from (\ref{eq_defnmc}). % and Corollary \ref{coro_simplifiedformula}. 
This reduces to the result in \cite[Theorem 2.20]{bloemendal2016principal}. We provide numerical illustrations of \Cref{Thm: EE} in Figure \ref{fig_compareone}. 
%and the $\bm{u}_{i_k}^\top \bm{u}_{i_k}=1$ (or the discussion near (\ref{eq_varphitz}) below). 
%recovers and generalizes the discussion of \cite[Appendix C]{MR3606475}. 

\begin{figure}%[H]
\begin{center}
\begin{subfigure}{0.4\textwidth}
\includegraphics[width=6.5cm,height=4.4cm]{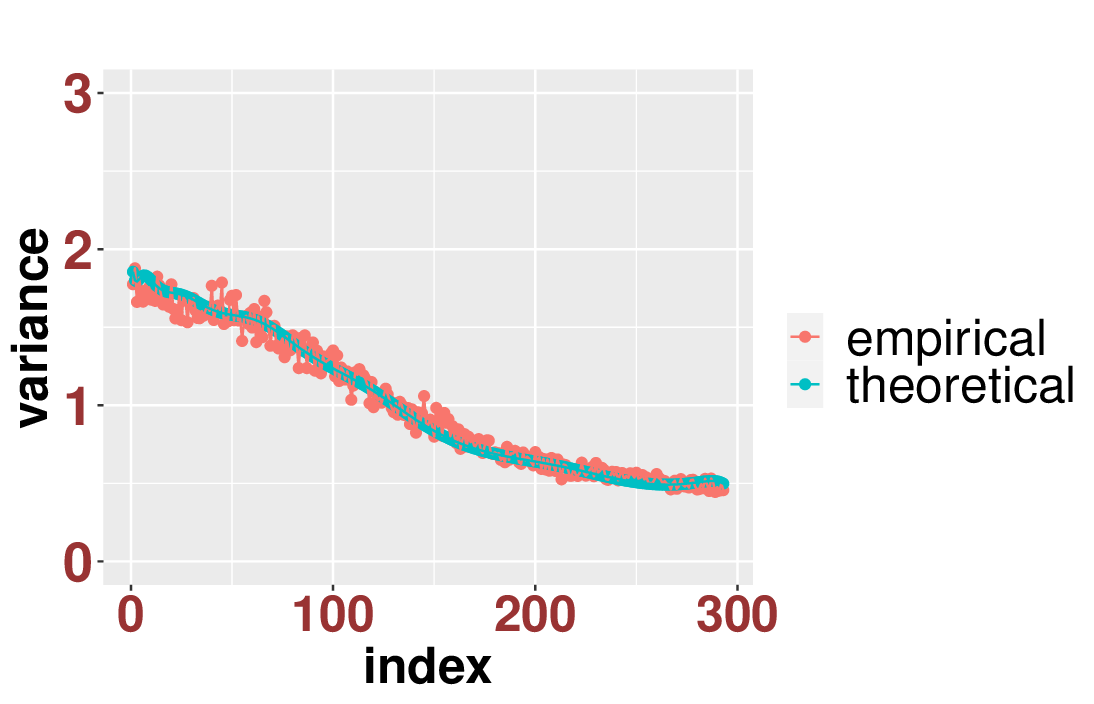}
\end{subfigure}
\hspace*{1.5cm}
\begin{subfigure}{0.4\textwidth}
\includegraphics[width=6.5cm,height=4.4cm]{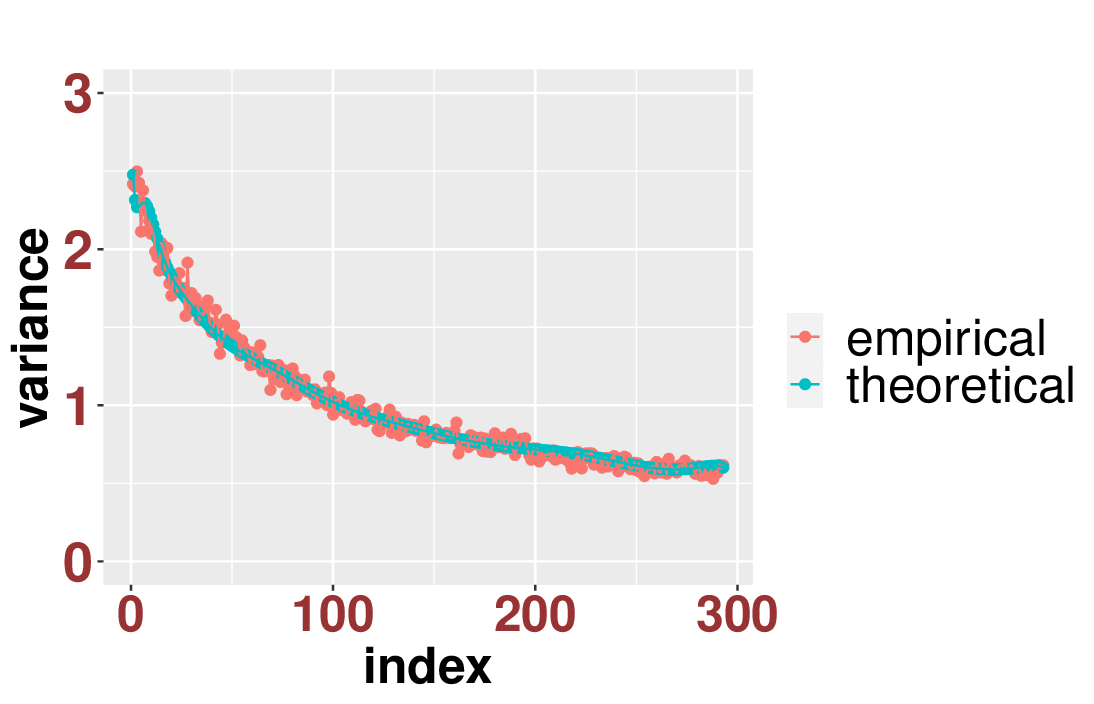}
\end{subfigure}
\end{center}
\caption{  Variances of the generalized components of non-outlier eigenvectors. We demonstrate the validity of Theorem \ref{Thm: EE} by comparing the empirical variances with the formula presented in (\ref{eq_Vdefinition}). We consider $\mathcal{Q}_1$ in (\ref{eq_samplecovariancematrixdefinition}) with $p=300$ and $n=600$. The random matrix $X$ consists of i.i.d.~$\mathcal{N}(0,1)$ entries. 
In the left panel, we take the ESD of $\Sigma_0$ to be $\mu_{\Sigma_0}=0.5 \delta_3+0.5 \delta_1$. In the right panel, we take $\mu_{\Sigma_0}=p^{-1}\sum_{i=1}^p\delta_{1+i/p}$. The empirical results are based on 1,000 repetitions and $\bm{v}=\bm{v}_1$, the leading eigenvector of $\Sigma_0$. The theoretical variances are evaluated through the estimated $\phi$, which will be discussed in detail in \Cref{sec_statisticalestimation} (see \Cref{thm_consistentestimators}). To facilitate implementation, users can directly use the function $\mathsf{MP}_{-}\mathsf{vector}_{-}\mathsf{dist}$ provided in our R package $\mathsf{RMT4DS}.$
%For simplicity, we assume that the spectrum of $\Sigma_0$ follows some probability distribution $\pi.$ 
}
\label{fig_compareone}
\end{figure}

\begin{remark}\label{rmk:proj}
For the sake of statistical applications and to maintain simplicity in our presentation, we have focused on a scenario where the sample eigenvectors are projected onto the same direction $\bm{v}$. However, as discussed in \Cref{rem:color} below, our results can be extended to situations where different sample eigenvectors are projected onto multiple distinct directions. More precisely, our method can be generalized to show that for deterministic unit vectors $\bm{v}_{i_k} \in \mathbb{R}^p$, $k \in \qq{L}$,     
\begin{equation}\label{eq_coloredresult}
\sqrt{p} 
\begin{pmatrix}
\xi_1 \avg{\bm{v}_{i_1}, \bm{u}_{i_1+r}} \\
 \vdots \\
\xi_L \avg{\bm{v}_{i_L}, \bm{u}_{i_L+r}}
\end{pmatrix}
\simeq \mathcal{N}(\bm{0}, \Xi'_L), \quad \text{with} \quad   \Xi'_L:=\operatorname{diag}\left\{ \phi(\bm{v}_{i_1}, \bm{v}_{i_1},\gamma_{i_1}), \cdots, \phi(\bm{v}_{i_L}, \bm{v}_{i_L},\gamma_{i_L}) \right\}. 
\end{equation} 
Such a result has recently been established for Wigner matrices in \cite{marcinek2022high}, and we believe similar arguments can be applied to our setting. However, since (\ref{eq_coloredresult}) falls beyond the scope of the current paper, we intend to pursue it in future research. 
\end{remark}

\begin{remark}
In the current paper, our focus regarding the spiked model has been on the distribution of its non-outlier eigenvectors. As for the outlier eigenvectors, we only examine their first-order asymptotics (see Lemmas \ref{lemma_spikedlocation} and \ref{lem_bulkproperty}), which suffice for our statistical applications. It is important to note that the derivation of their asymptotic distributions relies on a totally different approach from that for Theorem \ref{Thm: EE}, which will be explained in more detail in Section \ref{sec_proofstrategy}. 
In the case where $\Sigma_0=I$, the distributions of outlier eigenvectors have been extensively studied in \cite{bao2022statistical}. Unlike Theorem \ref{Thm: EE}, the results in \cite{bao2022statistical} demonstrate that the distribution of the generalized components of outlier eigenvectors generally involves a linear combination of Gaussian and Chi-square random variables. Nevertheless, we believe that by following the approach in \cite{bao2022statistical} and utilizing some tools developed in our paper, we can derive the distribution of all outlier eigenvectors for general $\Sigma_0$.
Finally, we remark that the methods presented in our paper can be applied to study the eigenvector distribution in other important statistical models, such as those discussed in \cite{fan2021principal,zhou2019eigenvalue}. Exploring these directions will be the focus of our future work. %We will explore these directions in future works.
%For the spiked model, we have only focused on the distribution of its non-outlier eigenvectors in the current paper. For the outlier eigenvectors, we only need their first-order asymptotics for our statistical applications, as summarized in Lemmas  \ref{lemma_spikedlocation} and \ref{lem_bulkproperty} below. The derivation of their asymptotic distributions relies on a totally different approach from that of Theorem \ref{Thm: EE}, as we will explain in Section \ref{sec_proofstrategy} in more detail. In fact, when $\Sigma_0=I,$ the distributions of the outlier eigenvectors have been extensively studied in \cite{bao2022statistical}. Unlike Theorem \ref{Thm: EE}, the results in \cite{bao2022statistical} show that the distribution of the generalized components of outlier eigenvectors is generally a linear combination of several Gaussian and Chi-square random variables. 
\end{remark}

\begin{remark}
Motivated by the applications in covariance matrix estimation, we have focused on demonstrating the eigenvector distributions for the sample covariance matrices defined in (\ref{eq_samplecovariancematrixdefinition}), specifically the distribution of the left singular vectors of the data matrix.  
However, in applications such as spectral clustering, there is also interest in the right singular vectors. Fortunately, by virtue of symmetry, our arguments can be readily extended to study the eigenvectors of $\cal Q_2$ and $\wt{\cal Q}_2$ in (\ref{eq_grammatrixdefinition}) with minor modifications. For the sake of simplicity, we will not explore this direction in the present paper.
%For more details, we direct readers to Remark \ref{remark_generalizationleft} below.
\end{remark}

During the proof of Theorem \ref{Thm: EE}, we can derive the following concentration inequality for the weighted average of the generalized components of eigenvectors, known as the quantum unique ergodicity estimate \cite{benigni2020eigenvectors, MR3606475}. This result plays a key role in the proof of Theorem \ref{thm_shrinkerestimate}.  
%Recall $\phi(\cdot,\cdot,\cdot)$ defined in (\ref{eq_phidefinitionintheend}).

\begin{theorem}[Quantum unique ergodicity]\label{corollary_que}
Let $\{w_j\}_{j=1}^{p}$ be a deterministic sequence of real values such that $|w_j|\le 1$ for all $ j \in \qq{p}$. Recall that \smash{$\{\bm{v}_j\}_{j=1}^{p}$} are the eigenvectors of $\Sigma$. Then, there exists a constant $\fd>0$ such that for any $\e>0$ and each $i \in \qq{r+1,\sfK}$, %there exists a constant $C>0$ such that for each $i \in \qq{p}$, 
the following estimate holds:
\begin{equation}\label{eq:weakQUE}
\P\bigg( \Big|\sum_{j=1}^{p} w_j \left| \avg{\bm{u}_i, \bm{v}_j} \right|^2-p^{-1}\sum_{j=1}^{p} w_j \phi(\bm{v}_j, \bm{v}_j, \gamma_{i-r})  \Big|>\e \bigg) \leq  n^{-\mathfrak d}/ \e^{2}.
\end{equation}
The above result also extends to the non-spiked model with $r=0$. 
%This estimate remains valid under the spiked model when $i\in \qq{r+1,\sfK}$.  
%\textcolor{blue}{should this be $P(|\sum_{j=1}^p w_j | \bm{u}_i^\top \bm{w}_j |^2-p^{-1}\sum_{j=1}^{p} w_j \phi(\bm{w}_j, \bm{w}_j, \gamma_i)|>c)<\dots$?}
\end{theorem}

%Two remarks are in order. 
%\begin{remark}
In the literature, \Cref{corollary_que} is sometimes called a ``weak" form of QUE since both the probability bound and the rate in estimate \eqref{eq:weakQUE} are non-optimal. Recently, a stronger notion of QUE called the \emph{eigenstate thermalization hypothesis} has also been established for Wigner matrices  \cite{cipolloni2021eigenstate,adhikari2023eigenstate,cipolloni2022normal,CIPOLLONI2022109394,CEH2023,CEJK2023}. We conjecture that a similar form should also hold for sample covariance matrices, in the sense that 
\begin{equation*}
\sum_{j=1}^p w_j \left| \avg{\bm{u}_i, \bm{v}_j} \right|^2=p^{-1}\sum_{j=1}^{p} w_j \phi(\bm{v}_j, \bm{v}_j, \gamma_{i-r})+\OO_{\prec}(n^{-1/2}).
\end{equation*}
We plan to explore this direction in future works.
%However, the strategy developed for Wigner matrices in \cite{cipolloni2021eigenstate} does not apply to general covariance matrices directly and a substantial modification is required.
%\end{remark}

%\begin{proof}
%This is the QUE part. In fact the proof is more similar to Corollary 1.5 of \cite{benigni2020eigenvectors}. 
%\end{proof}

\section{Adaptive and consistent estimators for shrinkers}\label{sec_statisticalestimation}

In Section \ref{sec_statisticsmainresults}, we have provided the formulas for the convergent limits of the shrinkers $\varphi_i$. However, in practical applications, the quantities involved in (\ref{eq_importantdefinition})--(\ref{eq_phidefinitionintheend}) are typically unknown and need to be estimated. 
In this section, we propose adaptive and consistent estimators for these quantities, which in turn provide consistent estimators for the shrinkers $\varphi_i$.  Our focus will be on the possibly spiked model, which encompasses the non-spiked model as a special case with $r=0$.
%Throughout the rest of this section,   
%Let $\mu_i, 1 \leq i \leq p$ be the eigenvalues of $\widetilde{\mathcal{Q}}_1.$
%\subsection{Data-driven estimators and their consistency}\label{sec_combinedalgo}

%\subsection{Consistent estimation of the spectrum and related quantities} 

\subsection{Data-driven estimators for the shrinkers}\label{sec_combinedalgo}

As one can see, estimating the quantities in (\ref{eq_importantdefinition})--(\ref{eq_phidefinitionintheend}) requires consistent estimation of the spectrum of the \emph{non-spiked eigenvalues} of $\Sigma$. In the literature, this problem has been addressed for non-spiked sample covariance matrices in \cite{NEK, KV} based on the sample eigenvalues of $\cal Q_1$. 
For the spiked sample covariance matrix model, we know that the non-outlier eigenvalues of \smash{$\widetilde{\mathcal{Q}}_1$} stick to those of $\cal Q_1$ as indicated by equation (\ref{eq_eigenvaluesticking}) below. Consequently, we can substitute the non-outlier eigenvalues \smash{$\wt\lambda_i$}, $i\in \qq{r+1,p}$, into the algorithms from \cite{NEK,KV} to obtain estimators for the eigenvalues of $\Sigma_0$, denoted as 
\begin{equation}\label{eq_sigmaestimator}
\widehat{\sigma}_{1} \geq \widehat{\sigma}_{2} \geq \cdots \geq \widehat{\sigma}_{p} > 0.
\end{equation}
In particular, they provide consistent estimation of the spectrum of the non-spiked eigenvalues $\wt\sigma_j=\sigma_j$, $j\ge r+1$. To facilitate ease of use for users, these algorithms can be implemented using the functions $\mathtt{MPEst}$ (for the method in \cite{NEK})  or $\mathtt{MomentEst}$ (for the method in \cite{KV}) in our $\mathtt{R}$ package. Next, we turn to the estimation of the quantities associated with the spikes of $\Sigma$. For $i\in \qq{r}$, we let 
\begin{equation}\label{eq:aibi}
\widehat{\mathfrak{a}}_i=\wt\lambda_i, \quad \widehat{\mathfrak{b}}_i=(\widehat{\mathfrak{a}}_i\widehat{m}_i')^{-1}, \quad  \widehat{\widetilde{\sigma}}_i=-(\widehat{m}_i)^{-1}, 
%\widehat{\mathfrak{a}}_i^{-1}\bigg( \frac{1}{n} \sum_{j=r+1}^n \frac{1}{(\wt\lambda_j-\widehat{\mathfrak{a}}_i)^2} \bigg)^{-1} . 
\end{equation}
% \begin{equation}\label{eq_initialdefinitionlambdat}
% \quad =\frac{1}{\st} \re \bigg(\frac1n  \sum_{j=1}^n \frac{1}{\lambda_{j,\st}-\widehat{\mathfrak{a}}_i+\ii \st}-\frac1n \sum_{j=r+1}^n \frac{1}{\wt\lambda_j-\widehat{\mathfrak{a}}_i}  \bigg),\quad i\in \qq{r}.
% \end{equation} 
where $\widehat{m}_i$ and $\widehat{m}_i'$ are defined as 
\begin{equation} \label{eq:aibi2}
\widehat{m}_i\deq \frac{1}{n}  \sum_{j=r+1}^n \frac{1}{ \wt{\lambda}_j-\wh\fa_i},\quad \widehat{m}_i'\deq \frac{1}{n}  \sum_{j=r+1}^n \frac{1}{ (\wt{\lambda}_j-\wh\fa_i)^2} .
\end{equation}
For a small constant $\e>0$, we define
\be\label{eq:kepm}
K_\e^+:=\max_{j}\left\{j: \wh\sigma_j \ge \e\right\},\quad K_\e^-:=\min_{j}\left\{j: \wh\sigma_j \le -\wh m_r^{-1}-\e\right\}.
\ee
Then, we introduce the following estimator of $\dot m_0(\mathfrak{a}_i)$:
$$ \widehat{\mathsf{m}}'_{i,0}(\e)=\frac{\widehat{m}_i'}{n\wh\fa_i}\sum_{j=(r+1)\vee K_\e^-}^{p\wedge K_\e^+} \frac{\ell(\wh \sigma_j)\wh \sigma_j}{(1+\wh m_i \wh\sigma_j)^2}.$$
%$$\mathbf 1\left( |\ell(\wh \sigma_j)| +|\wh \sigma_j|+ |1+\wh m_i \wh\sigma_j|^{-1}\le \e^{-1} \right).$$
% {\cor For a small $\st>0$, let $\widehat{\Sigma}_\st=\operatorname{diag}\{\widehat{\sigma}_{1,\st},\widehat{\sigma}_{2,\st},\ldots,\widehat{\sigma}_{p,\st}\}$ be the diagonal matrix with entries (recall \eqref{eq:sigmait})
% \begin{equation}\label{eq:whsigmait}
% \widehat{\sigma}_{i,\st}=\frac{\widehat{\sigma}_i}{1+\st \ell (\widehat{\sigma}_i)/\widehat{\mathfrak{a}}_i }.
% \end{equation}
% Sample a $p \times n$ random matrix $X^G$ i.i.d.~Gaussian entries of mean zero and variance $n^{-1}$. Let $\{\lambda_{j,\st} \}_{j=1}^p$ denote the eigenvalues of $\widehat{\Sigma}_{\st}^{1/2} X^G (X^G)^\top \widehat{\Sigma}_{\st}^{1/2}$. Then, we define
% \begin{equation}\label{eq_initialdefinitionlambdat}
% \widehat{\mathsf{m}}'_{i,\st}=\frac{1}{\st} \re \bigg(\frac1n  \sum_{j=1}^n \frac{1}{\lambda_{j,\st}-\widehat{\mathfrak{a}}_i+\ii \st}-\frac1n \sum_{j=r+1}^n \frac{1}{\wt\lambda_j-\widehat{\mathfrak{a}}_i}  \bigg),\quad i\in \qq{r}.
% \end{equation} 
% }

%For a $p \times n$ matrix $X$ with i.i.d.~Gaussian entries with mean zero and variance $n^{-1},$ denote the eigenvalues of $\widehat{\Sigma}_{\st}^{1/2} XX^\top \widehat{\Sigma}_{\st}^{1/2}$ as $\{\lambda_j^\st \}.$ 

Note all the above quantities can be computed adaptively using the observed sample covariance matrix $\widetilde{\mathcal{Q}}_1.$ Moreover, the next result shows that 
%Then the following result shows that all the above quantities can be computed adaptively with the observed sample covariance matrix $\widetilde{\mathcal{Q}}_1.$ Moreover, in the following lemma, we show that 
they are consistent estimators for the relevant quantities. %In the next subsection, we will use them to construct our estimators for the shrinkers $\varphi_i$, $i\in \qq{p}$.

\begin{lemma}\label{lem_usefulquantitiesconsistentestimator}
Under Assumption \ref{main_assumption}, the ESD $\wh\mu$ of $\{\widehat{\sigma}_{i}\}$ converges weakly to $\mu_{\Sigma_0}$, that is, for each $x\ge 0$, 
\begin{equation}\label{eq_consistentestimatorsigma}
\left|\wh \mu((-\infty,x])-\mu_{\Sigma_0}((-\infty,x])\right|=\oo_{\mathbb{P}}(1).
%\widehat{\sigma}_i=\wt\sigma_i+\oo_{\mathbb{P}}(1),  %1 \leq i \leq p.
\end{equation}
For any continuous function $\ell$ defined on $(0, \infty)$, there exists a small constant $\e_0>0$ such that for any $i \in \qq{r}$ and $\e\le \e_0$,
\begin{equation}\label{eq_consistenestimatorotherquantities}
\widehat{\mathfrak{a}}_i=\mathfrak{a}_i+\oo_{\mathbb{P}}(1), \quad  \widehat{\mathfrak{b}}_i=\mathfrak{b}_i+\oo_{\mathbb{P}}(1), \quad \widehat{\widetilde{\sigma}}_i=\widetilde{\sigma}_i+\oo_{\mathbb{P}}(1), \quad  \widehat{\mathsf{m}}'_{i,0}(\e)=\dot m_0(\mathfrak{a}_i)+\oo_{\mathbb{P}}(1).
\end{equation}
%where $\oo_{\st}(1)$ denotes an error that converges to 0 as $\st\downarrow 0$.
\end{lemma} 

%The above lemma indicates that $\widehat{\sigma}_i, \widehat{\mathfrak{a}}_i, \widehat{\mathfrak{b}}_i, $ and $\widehat{\mathsf{m}}_i$ with sufficiently small $\st$ are consistent estimators for $\sigma_i, \mathfrak{a}_i, \mathfrak{b}_i, \widetilde{\sigma}_i$ and $m_0'(\mathfrak{a}_i),$ respectively. All these quantities will be used to construct our estimators for the shrinkers $\varphi_i, 1 \leq i \leq p.$

% {\color{red}[add more here]}
%
%
%...... {\color{red}[need to summarize a little bit of \cite{NEK} or \cite{KV}], spiked or non-spiked, discuss the estimation of these values. Then discuss how to }

%\subsection{Data-driven estimators for the shrinkers}\label{sec_combinedalgo2}

% In this section, we will propose some consistent estimators for the quantities in \eqref{eq_defnpsi} and \eqref{eq_keydefinition1111}, and then propose our estimators for the shrinkers. 
%  For simplicity, we adopt the commonly used assumption that $\{\sigma_i\}_{i=1}^p$ follows certain distribution $\pi$  \cite{baibook,MR3704770}. Let $\{\widehat{\sigma}_i\}_{i=1}^p$ be the sequence of consistent estimators computed either from \cite{NEK} or \cite{KV}. 
 
%  Recall that $\mu_i,1\le i\le p$ denote the eigenvalues of $\mathcal{Q}_1$ and $\mathsf{K}=\min\{n,p\}$. 

%{\cor From \eqref{eq_consistenestimatorotherquantities}, we see that $\widehat{\mathsf{m}}'_{i,t}$ is a consistent estimator of $\dot m_0(\mathfrak{a}_i)$ if we choose $n\to \infty$ followed by $\st\to 0$.}
With the above lemma, we can propose a consistent estimator for $\psi_i$ in \eqref{eq_defnpsi} as
\begin{equation}\label{eq:estpsi}
\widehat{\psi}_i(\ell,\e):=\widehat{\mathfrak{b}}_i \left( \frac{\ell(\widehat{\widetilde{\sigma}}_i)}{\widehat{\widetilde{\sigma}}_i}+\widehat{\mathfrak{a}}_i \widehat{\mathsf{m}}'_{i,0}(\e) \right),\quad   i \in \qq{r}.
\end{equation}
Next, we propose a consistent estimator for $\vartheta$ in \eqref{eq_keydefinition1111}. 
% For the quantiles defined in Definition \ref{defn_generaleigenvaluelocation}, we introduce its estimator 
%  \begin{equation}\label{eq:whgammai}
%  \widehat{\gamma}_i=
%  \begin{cases}
%  \wt\lambda_{i+r}, &\quad  1 \leq i \leq \mathsf{K}-r \\
%  0 , &\quad  \mathsf{K}-r+1 \leq i \leq p\vee n
%  \end{cases}.
%  \end{equation}
%For the non-spiked shrinker estimators $\vartheta_i, r+1\le i\le p$, 
Corresponding to $\phi$ in \eqref{eq_phidefinitionintheend}, we define
\begin{equation}\label{eq_phiestimator}
\widehat{\phi}_j(x)=
\begin{cases}
c_n \widehat{\sigma}_j ( x|1+ \widehat{m}(x) \widehat{\sigma}_j |^2 )^{-1}, &\quad r+1 \leq j \leq p, \ x>0 \\
(1-c_n^{-1})^{-1}(1+\widehat{m}_0 \widehat{\sigma}_j)^{-1}, &\quad r+1 \leq j \leq p,\ x=0
\end{cases},
\end{equation} 
where $\wh m(x)$ and $\wh m_0$ are defined as follows with $\eta=n^{-1/2}$:
%Here for some positive value $\eta=\mathrm{o}(1),$ we define 
%\begin{equation*}
%\widehat{m}_0=\left|\frac{1}{n} \sum_{j=r+1}^n \frac{1}{ \widehat{\gamma}_j-\mathrm{i} \eta} \right|,
%\end{equation*}
%and 
\begin{equation}\label{eq_choiceofeta_m}
\widehat{m}(x)\deq \frac{1}{n}  \sum_{j=r+1}^n \frac{1}{ \wt{\lambda}_j-x-\mathrm{i} \eta} , \quad x>0; \quad \wh m_0:=\re \bigg(\frac{1}{n} \sum_{j=r+1}^n \frac{1}{ \wt{\lambda}_j-\mathrm{i} \eta} \bigg).
% \begin{cases}{\displaystyle  ,\\
% \Big|{\displaystyle ,\quad &\mathsf K+1\le i\le p.
% \end{cases}
\end{equation}
%where in light of (\ref{defn_generaleigenvaluelocation}), we denote 
%\begin{equation*}
%\widehat{\gamma}_i=
%\begin{cases}
%\mu_i, & r+1 \leq i \leq \mathsf{K} \\
%0 , & \mathsf{K}+1 \leq i \leq p
%\end{cases}.
%\end{equation*}
With the above notations, we then define that for $i \in \qq{r+1 , \sfK}$ and small $\e>0$, 
\begin{equation}\label{eq:estvartheta}
\begin{aligned}
&\widehat{\vartheta}_i(\ell,\e):=\frac{1}{p}\sum_{j=r+1}^{p\wedge K_\e^+}  \ell(\widehat{\sigma}_j) \widehat{\phi}_j(\wt\lambda_{i}) \mathbf 1\left(\left|1+ \widehat{m}(\wt\lambda_i) \widehat{\sigma}_j \right|\ge \e\right), \quad \widehat{\vartheta}_0(\ell,\e):=\frac{1}{p}\sum_{j=r+1}^{p\wedge K_\e^+}  \ell(\widehat{\sigma}_j) \widehat{\phi}_j(0) .
\end{aligned}
\end{equation}
%For the spiked shrinkers and loss function $\ell$, according to Lemma \ref{lem_usefulquantitiesconsistentestimator} and (\ref{eq_defnpsi}) we define 
%We are now ready to state the consistency result for our proposed shrinker estimators. 
Now, we are prepared to present the consistency result for our proposed estimators of the shrinkers.

\begin{theorem} \label{thm_consistentestimators} 
Under Assumption \ref{main_assumption}, %the estimators proposed in \eqref{eq:estpsi} and \eqref{eq:estvartheta} satisfy that 
%holds.  Let $\{\widehat{\sigma}_i\}_{i=1}^p$ be the sequence of consistent estimators for $\{\sigma_i\}_{i=1}^p$ computed either from \cite{NEK} or \cite{KV}. Then we have that for when $\st \downarrow 0$ in (\ref{eq_initialdefinitionlambdat})
for any continuous function $\ell$ defined on $(0, \infty)$, there exists a small constant $\e_0>0$ such that for any $\e\le \e_0$, 
\be
\widehat{\psi}_i(\ell,\e)=\psi_i(\ell)+\oo_{\mathbb{P}}(1), \quad i \in \qq{r}; \label{eq_spikedshrinkerestimationandconvergence}
\ee
%and for $1 \leq i \leq \sfK-r$ or ,
\be
\widehat{\vartheta}_i(\ell,\e)=\vartheta(\ell,\gamma_{i-r})+\mathrm{o}_{\mathbb{P}}(1), \quad i \in \qq{r+1 , \sfK}; \quad \widehat{\vartheta}_0(\ell,\e)=\vartheta(\ell,0)+\oo_{\mathbb{P}}(1). %\widehat{\phi}_i=\phi(\bm{v}_i,\bm{v}_i,\gamma_i)+\mathrm{o}_{\mathbb{P}}(1). 
\label{eq_nonspikedshrinkerestimationandconvergence}
\ee
\end{theorem}

Combining \Cref{thm_consistentestimators} and \Cref{thm_shrinkerestimate}, we see that $\widehat{\vartheta}_i$, $\widehat{\vartheta}_0$, and $\widehat{\psi}_i$ can be used to consistently estimate the shrinkers and their associated asymptotic risks for various loss functions (see Corollary \ref{coro_lowbound}). These estimators are constructed in a data-driven manner and can be implemented easily. In Section \ref{sec_simulations}, we conduct extensive numerical simulations to demonstrate the effectiveness of our proposed estimators. Also notice that by employing (\ref{eq_phiestimator}), we are able to estimate the variances in (\ref{eq_Vdefinition}) for the eigenvector distribution stated in Theorem \ref{Thm: EE}.

%Additionally, with (\ref{eq_phiestimator}), we are also be able to estimate the variances in (\ref{eq_Vdefinition}) for the eigenvector distribution in Theorem \ref{Thm: EE}. 
%Two remarks are worth noting. 

%When $\ell(x)=x,$ following Theorem \ref{thm_consistentestimators}, 
% For $i\in \qq{r}$, we define $\wh\zeta_i\deq {\wh m'_i}/{|\wh m_i|^2}$
% \begin{equation*} 
% \wh\zeta_i\deq \frac{\wh m'_i}{|\wh m_i|^2}.
% \end{equation*}
% \smash{$\wh\fa_i$ and $\wh \fb_i$} as in \eqref{eq:aibi}, and define {\cor move definitions}
% \begin{equation*} 
% \widehat{m}_i\deq \frac{1}{n}  \sum_{j=r+1}^n \frac{1}{ \wt{\lambda}_j-\fa_i},\quad \widehat{m}_i'\deq \frac{1}{n}  \sum_{j=r+1}^n \frac{1}{ (\wt{\lambda}_j-\fa_i)^2},\quad \wh\zeta_i\deq \frac{\wh m'_i}{|\wh m_i|^2}.
% \end{equation*}
% Then, with $\wh m$ and $\wh m_0$ in \eqref{eq_choiceofeta_m}, we introduce
When $\ell(x)=x,$ we can provide more straightforward estimators for the simplified formulas presented in Corollary \ref{coro_simplifiedformula}. 
With the notations in \eqref{eq:aibi2} and \eqref{eq_choiceofeta_m}, we define
\begin{equation*}
\wh\zeta_i\deq \frac{\wh m'_i}{|\wh m_i|^2},\quad i\in \qq{r};\quad \wh\xi_i\deq
\begin{cases}
\frac{1}{\wt\lambda_i|\wh m(\wt\lambda_i)|^2}, & i\in\qq{r+1 , \mathsf{K}} \\
\frac{1}{(c_n-1)\wh m_0}, & i \in \qq{\mathsf{K}+1,p}
\end{cases}.
\end{equation*}
Notably, all these estimators do not involve the estimated non-outlier eigenvalues $\wh\sigma_i$ in \eqref{eq_sigmaestimator}.

\begin{lemma}\label{rem_otherestimators}
Under Assumption \ref{main_assumption}, we have that  
\be
\wh \fb_i\widehat{\zeta}_i=\fb_i\zeta(\fa_i)+\mathrm{o}_{\mathbb{P}}(1), \quad i \in \qq{r};
\label{eq_nonspikedshrinkerestimationandconvergence1}
\ee
%and for $r+1 \leq i \leq p$,
\be\label{eq_nonspikedshrinkerestimationandconvergence2}
\widehat{\xi}_i=\xi(\gamma_{i-r})+\oo_{\mathbb{P}}(1), \quad i \in \qq{r+1,\sfK};\quad  \widehat{\xi}_0 =\xi(0)+\oo_{\mathbb{P}}(1).
\ee
\end{lemma}  

\begin{remark}\label{rem_rank}
In the construction of the estimators for shrinkers, we assume that the number of spikes $r$ is known. However, in practical applications, this assumption is often unrealistic. Fortunately, it has been demonstrated in \cite{ding2021spiked,9779233,ke2023estimation, passemier2014estimation} that $r$ can be estimated consistently using the eigenvalues $\{\wt\lambda_i\}$ under item (iv) of Assumption \ref{main_assumption}. To facilitate the convenience of our readers, we have included a function called $\mathtt{GetRank}$ in the $\mathtt{RMT4DS}$ package, which can be used to estimate $r$.
\end{remark}

\subsection{Numerical simulations}\label{sec_simulations}

In this section, we employ Monte-Carlo simulations to demonstrate the effectiveness of our proposed estimators for the loss functions corresponding to $\ell(x)=x$ or $\ell(x)=x^{-1}$. % (which is related to the precision matrices).
% For simplicity, we focus on the covariance matrices, and similar conclusions can be made for the precision matrices. 

%\begin{enumerate}
%\item show our own using $\ell(x)=x$ and $\ell(x)=\sqrt{x}$ (plots of shrinkers)
%\item comparing our performance $\ell(x)=x$ with existing two methods, both plots and time consuming table; also asymptotic risk and PRIAL
%\item comparing ours with the only existing method to check risk and PRIAL for $\ell(x)=\sqrt{x}$ 
%\end{enumerate}

\subsubsection{Setup}\label{sec_setup}
%We now outline the setup of our simulations. 

%In the simulations, we sample $X$ such that its entries are i.i.d.~Gaussian random variables satisfying item (ii) of Assumption \ref{main_assumption}. Regarding the population covariance matrix $\Sigma,$ %for the sake of concreteness, we examine the following four choices:

In the simulations, we generate $X$ with i.i.d.~Gaussian random variables that satisfy item (ii) of Assumption \ref{main_assumption}. As for the population covariance matrix $\Sigma$, we consider the following four alternatives:
\begin{enumerate}
\item $\Sigma$ takes the form (\ref{eq_spikedcovariancematrix}), where $V$ is an orthogonal matrix generated from the $\mathtt{R}$ function $\mathtt{randortho},$ and  $\Lambda$ is defined as
\begin{equation*}
\Lambda=\operatorname{diag} \{9, \underbrace{3,\cdots,3}_\text{$p/2-1$}, \underbrace{1,\cdots,1}_\text{$p/2$} \}.
\end{equation*} 
\item $\Sigma$ takes the form  (\ref{eq_spikedcovariancematrix}), where $V$ is an orthogonal matrix generated from the $\mathtt{R}$ function $\mathtt{randortho}$, and $\Lambda$ is defined as
%\begin{equation*}
$\Lambda=\operatorname{diag} \{9, g_2, \cdots, g_p \},$
%\end{equation*}
where $\{g_k\}$ represents real numbers evenly distributed within the interval $[1, 2]$.
%are real numbers evenly distributed on the interval $[1, 2]$.
%follows uniform distribution on the interval $[1, 2]$. 

\item $\Sigma$ is a spiked matrix with a single spike equal to 9, and $\Sigma_0$ is a Toeplitz matrix with $(\Sigma_0)_{ij}= 0.4^{|i-j|}$ for $i,j \in \qq{p}$.
%in the sense that the $(i,j)$ entry follows that
% \begin{equation*}
% (\Sigma_0)_{ij}= 0.4^{|i-j|}, \quad  i,j \in \qq{p}.   
% \end{equation*}

\item $\Sigma$ takes the form  (\ref{eq_spikedcovariancematrix}), where $V$ is an orthogonal matrix generated from the $\mathtt{R}$ function $\mathtt{randortho}$, and $\Lambda$ is defined as
\begin{equation*}
\Lambda=\operatorname{diag} \{15, \underbrace{8,\cdots,8}_\text{$p/2-1$}, \underbrace{1,\cdots,1}_\text{$p/2$} \}.
\end{equation*}
\end{enumerate}
We would like to highlight that in our simulations, the support of the MP law $\varrho$ consists of a single component (i.e., $q=1$ in (\ref{eq_supportdmp})) for settings (i)--(iii), and two bulk components (i.e., $q=2$ in (\ref{eq_supportdmp})) for setting (iv).

Regarding the loss functions, we consider those associated with $\ell(x)=x$ and $\ell(x)={x}^{-1}.$ As stated in \Cref{coro_lowbound} (or \Cref{lem_explicityformula} below), $\ell(x)=x$ corresponds to the Frobenius, inverse Stein, disutility, and minimum variance loss functions, while $\ell(x)=x^{-1}$ corresponds to the Stein, weighted Frobenius, and inverse Frobenius loss functions.

\subsubsection{Performance of our estimators}

We proceed to evaluate the performance of our estimators for the shrinkers $\varphi_i$. To compute the quantity in (\ref{eq_choiceofeta_m}), we select $\eta \asymp n^{-1/2}$.
In order to facilitate visual interpretation, we present our results for $\ell(x)$ in Figure \ref{fig_compareoursonlyone} and for $\ell(x)=x^{-1}$ in Figure \ref{fig_compareoursonlytwo}. It is evident that the estimators outlined in Theorem \ref{thm_consistentestimators} and \Cref{rem_otherestimators} yield accurate predictions for the shrinkers across all simulation scenarios, encompassing different loss functions associated with $\ell(x)=x$ or $\ell(x)=x^{-1}$.

%For better visualization, we present our results for $\ell(x)$ in Figure \ref{fig_compareoursonlyone} and for $\ell(x)=x^{-1}$ in Figure \ref{fig_compareoursonlytwo}. We see that the estimators in Theorem \ref{thm_consistentestimators} provide accurate predictions for the shrinkers in all simulation settings for different loss functions corresponding to $\ell(x)=x$ or $\ell(x)=x^{-1}$. 

%we need to choose a small value $\eta=\mathrm{o}(1).$ In the simulations below, we set $\eta \asymp n^{-1/2}$ for concreteness.

\begin{figure}[h]
\hspace*{-0.1cm}
\begin{subfigure}{0.2\textwidth}
\includegraphics[width=5cm,height=4.4cm]{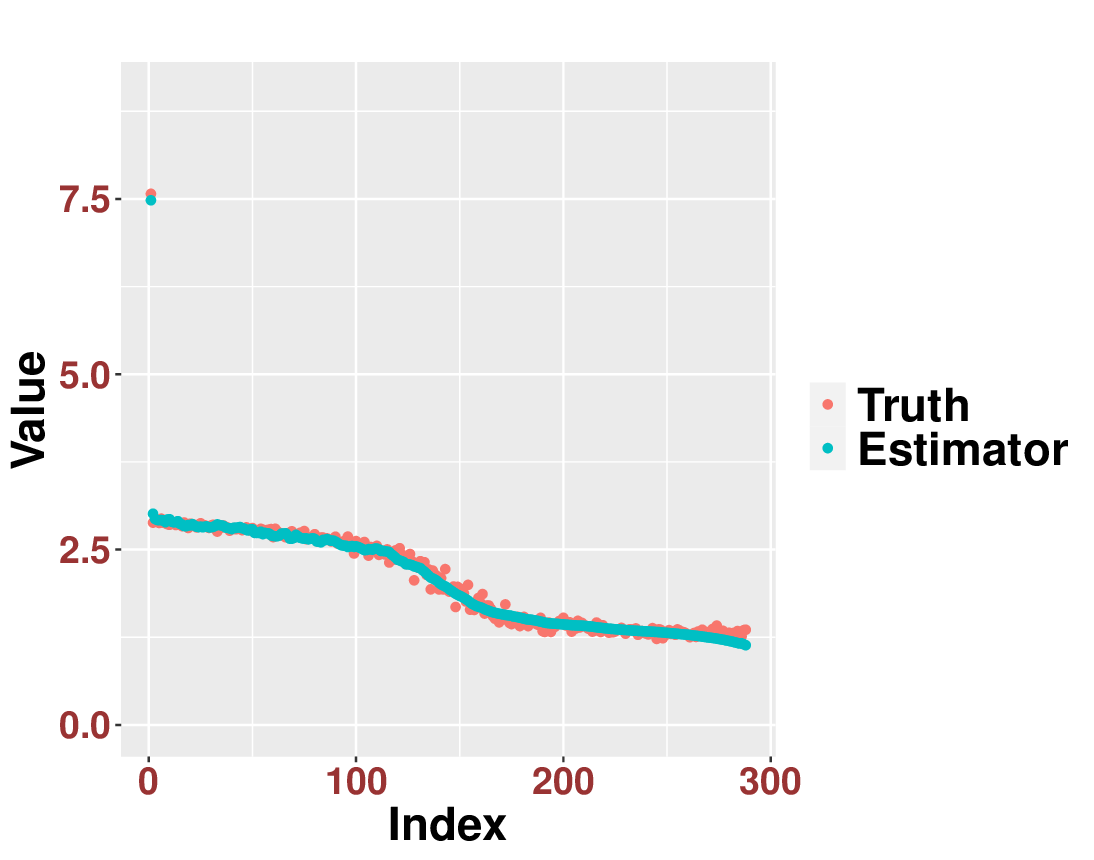}
\end{subfigure}
\hspace{0.18cm}
\begin{subfigure}{0.2\textwidth}
\includegraphics[width=5cm,height=4.4cm]{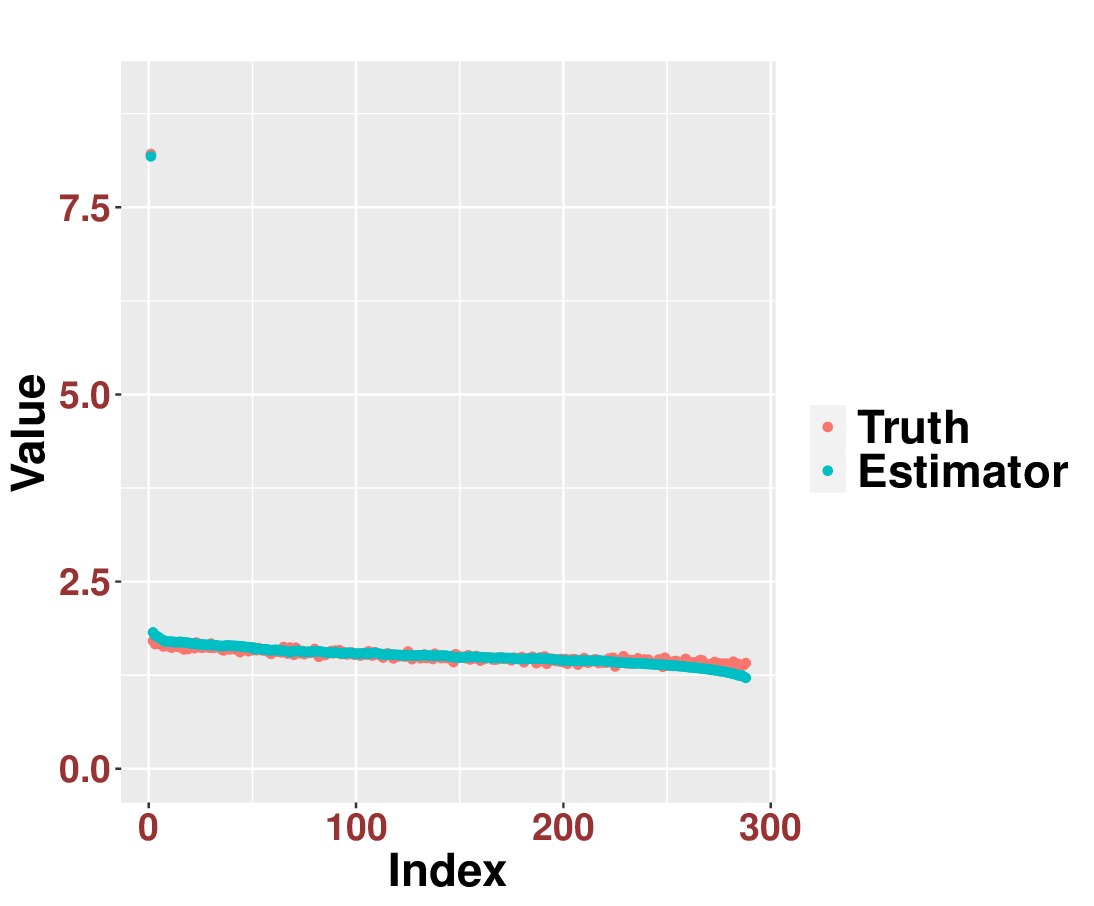}
\end{subfigure}
\hspace*{0.17cm}
\begin{subfigure}{0.2\textwidth}
\includegraphics[width=5cm,height=4.4cm]{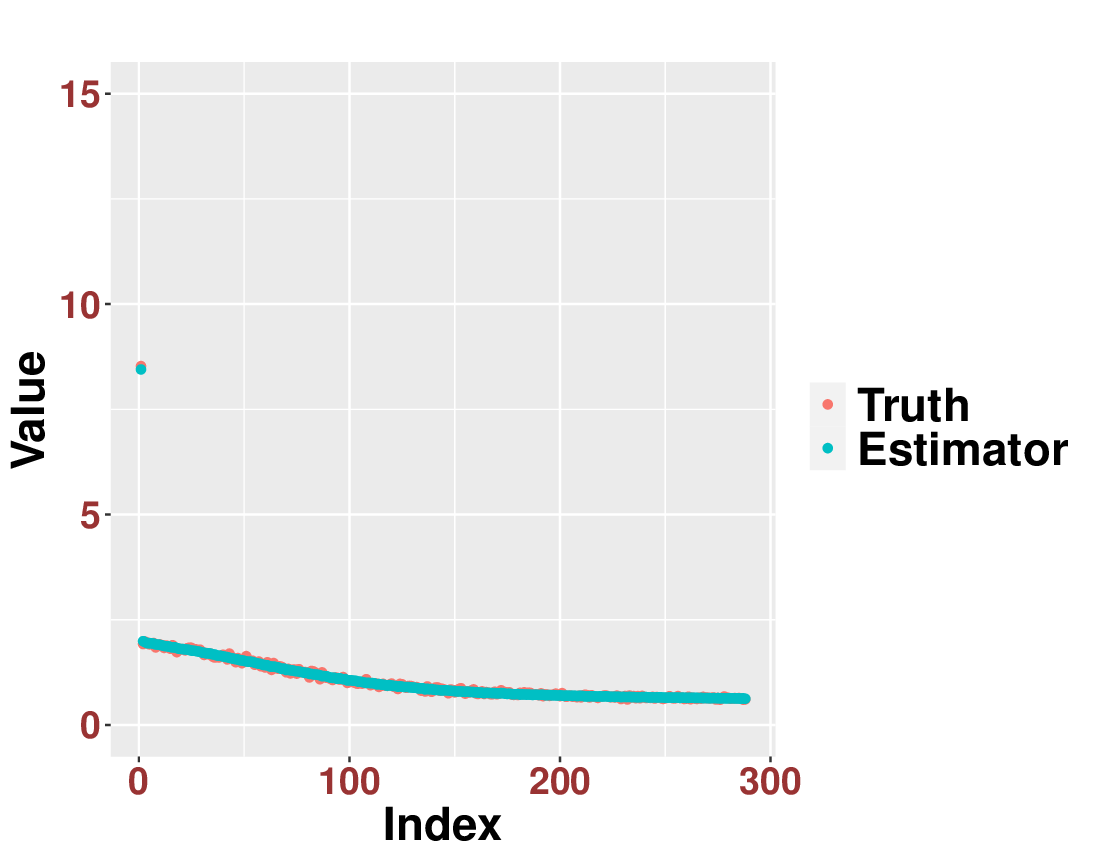}
\end{subfigure}
\hspace*{0.17cm}
\begin{subfigure}{0.2\textwidth}
\includegraphics[width=5cm,height=4.4cm]{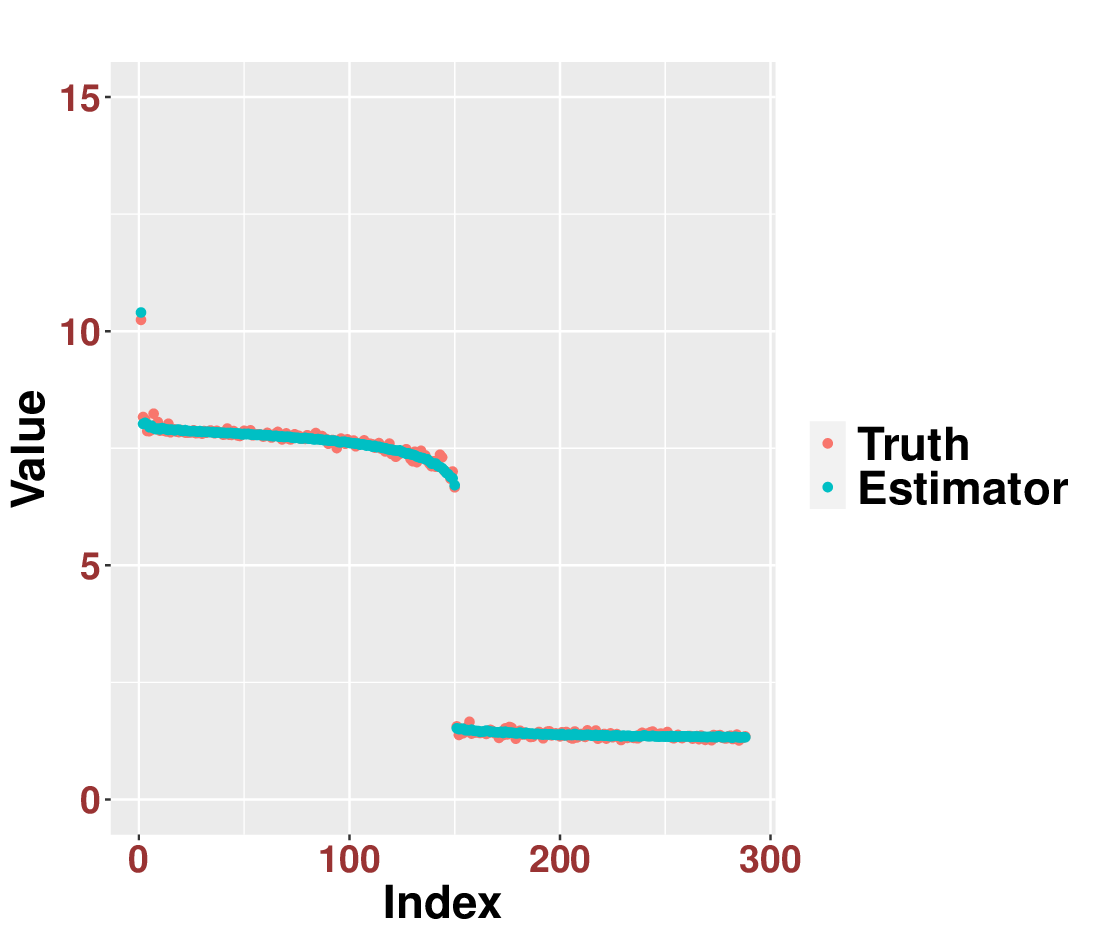}
\end{subfigure}
\caption{Performance of our estimators for all $\varphi_i , \  i \in \qq{p},$ for $\ell(x)=x$. From left to right, we provide the results for the simulation settings (i)--(iv) as outlined in Section \ref{sec_setup}. In the simulations, we set $p=300$ and $n=600$, and we have used the consistent estimators presented in \Cref{rem_otherestimators}.
%Corollary \ref{coro_simplifiedformula} and its consistent estimators as in \Cref{rem_otherestimators}. 
}
\label{fig_compareoursonlyone}
\end{figure}

\begin{figure}[h]
\hspace*{-0.1cm}
\begin{subfigure}{0.2\textwidth}
\includegraphics[width=5cm,height=4.4cm]{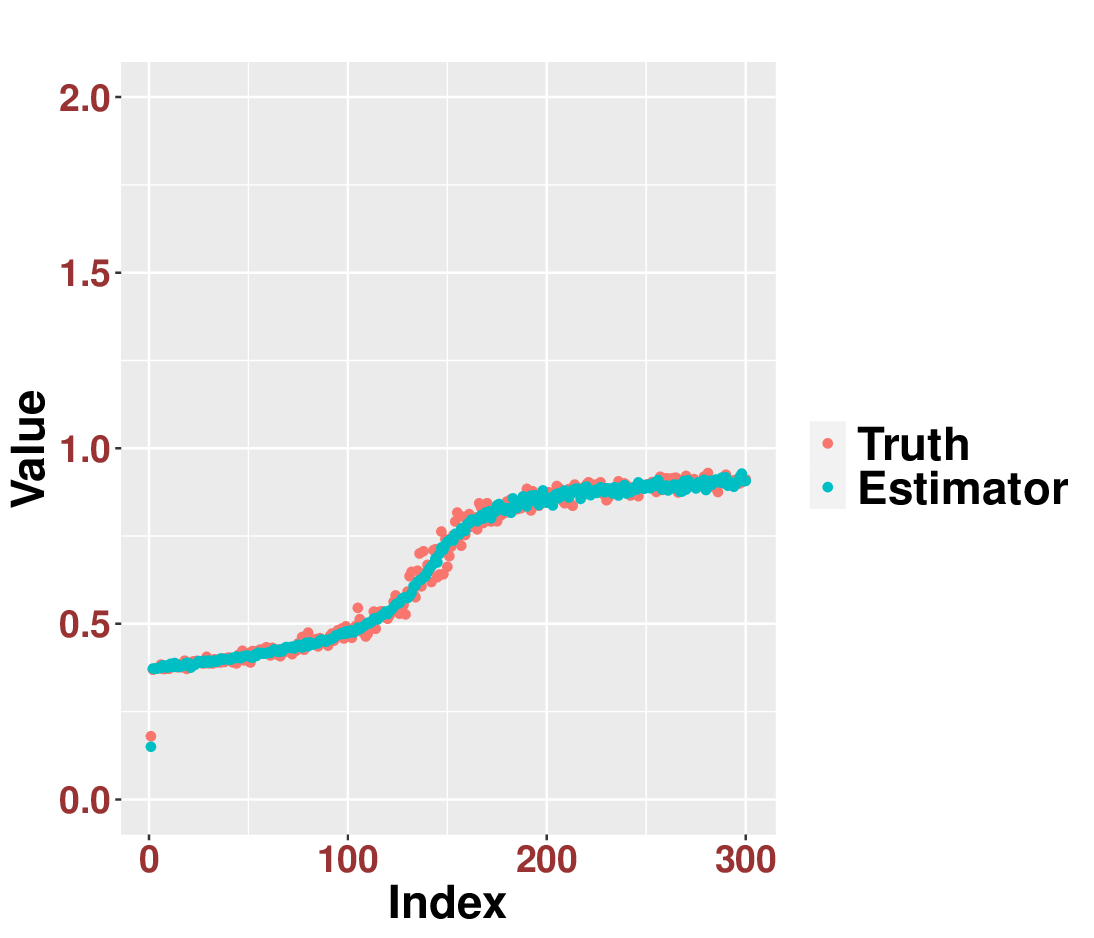}
\end{subfigure}
\hspace{0.18cm}
\begin{subfigure}{0.2\textwidth}
\includegraphics[width=5cm,height=4.4cm]{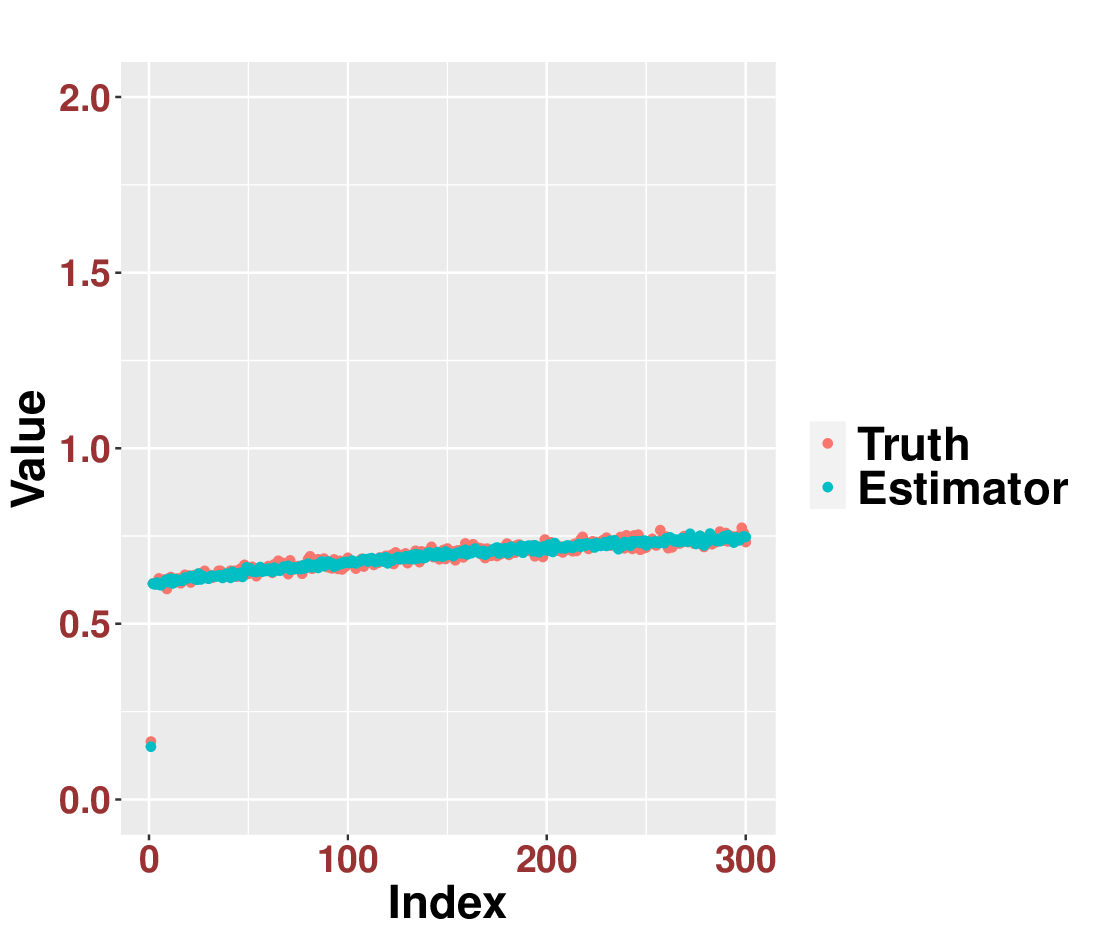}
\end{subfigure}
\hspace{0.18cm}
\begin{subfigure}{0.2\textwidth}
\includegraphics[width=5cm,height=4.4cm]{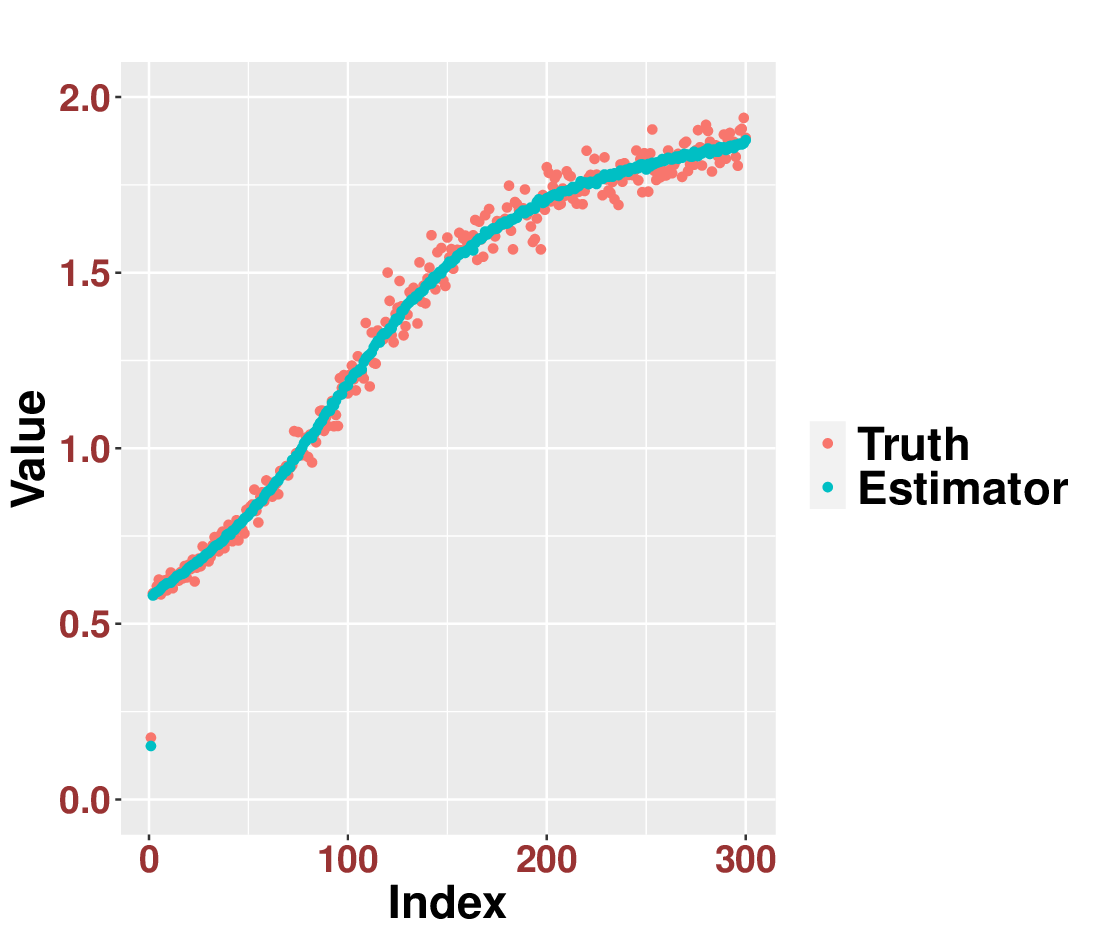}
\end{subfigure}
\hspace{0.18cm}
\begin{subfigure}{0.2\textwidth}
\includegraphics[width=5cm,height=4.4cm]{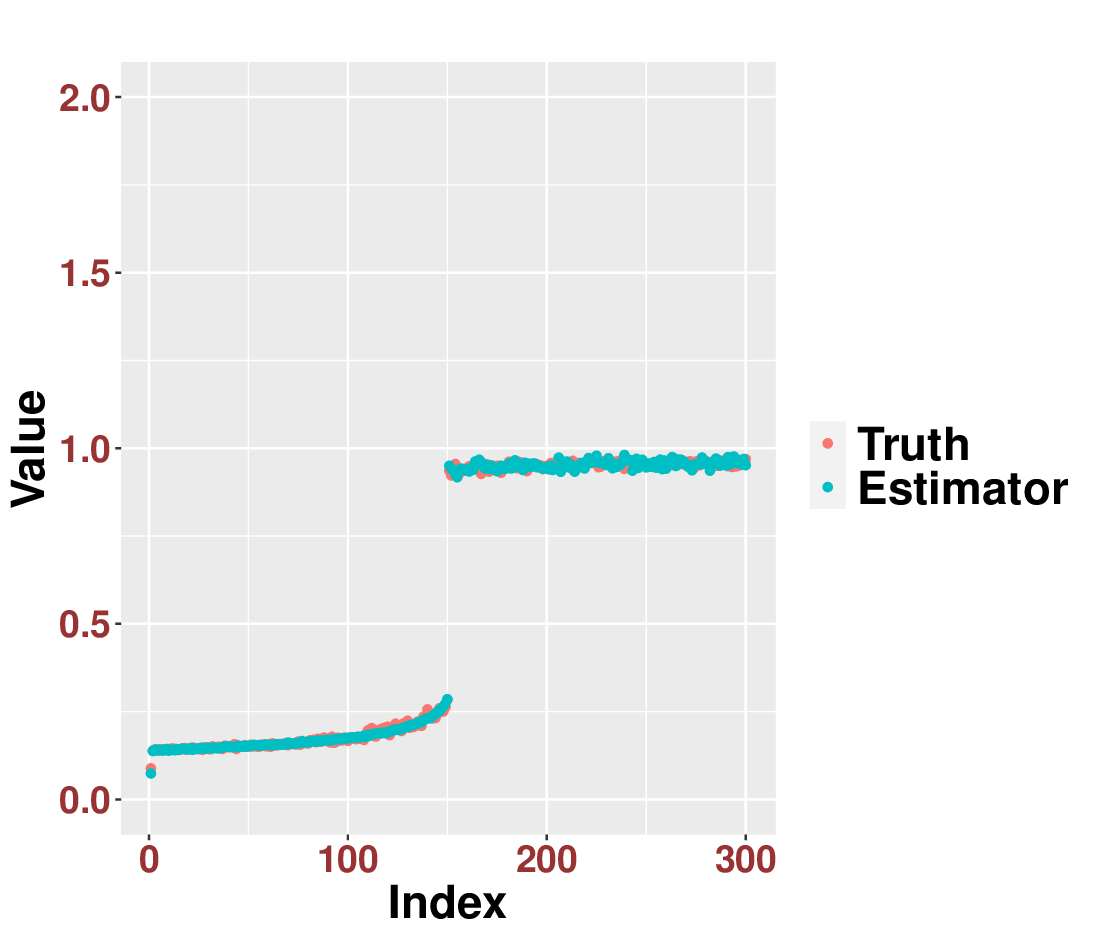}
\end{subfigure}
\caption{Performance of our estimators for all $\varphi_i ,\ i \in \qq{p}$, for $\ell(x)=x^{-1}$. From left to right,  we provide the results for the simulation settings (i)--(iv) as outlined in Section \ref{sec_setup}. In the simulations, we set $p=300$ and $n=600.$ }
\label{fig_compareoursonlytwo}
\end{figure}  

\subsubsection{Comparison with other methods}

Next, we compare our proposed estimators with the existing methods for estimating shrinkers proposed by Ledoit and Wolf. For $\ell(x)=x,$ we compare our method (Estimator) with the recently proposed quadratic-inverse shrinkage estimator (QIS) \cite{ledoit2022quadratic}\footnote{The $\mathtt{R}$ codes can be found at \url{https://github.com/MikeWolf007/covShrinkage/blob/main/qis.R}} and the QuEST method based on numerically solving equation (\ref{eq_defnmc}) \cite{LEDOIT2015360}.\footnote{The QuEST method can be implemented using the $\mathtt{R}$ package $\mathtt{nlshrink}.$} For $\ell(x)=x^{-1},$ the only existing method is the LIS method proposed in \cite{ledoit2022quadratic}.\footnote{The $\mathtt{R}$ codes can be found at \url{https://github.com/MikeWolf007/covShrinkage/blob/main/lis.R}} Since all these methods focus on non-spiked models, we remove the spikes from the four settings described in Section \ref{sec_setup} for comparison.
%for comparison, we still adopt the four settings in \ref{sec_setup} but remove their spikes. 

First, we visually compare the plots of different estimators of the shrinkers in Figures \ref{fig_compareothersone} and \ref{fig_compareotherstwo}. Our proposed estimators demonstrate superior accuracy in all settings and for both forms of $\ell(x)$. The performance of the other estimators depends on the specific underlying population covariance matrix and the chosen loss function.
Second, in Figure \ref{fig_compareotherriskone}, we assess the accuracy of different estimators in predicting the generalization errors, as stated in Corollary \ref{coro_lowbound}. Our proposed estimators provide the most accurate predictions. It is worth noting that the competing methods perform well when $\ell(x)=x$, but their performance deteriorates when $\ell(x)=x^{-1}$.

\begin{figure}[h]
\hspace*{-0.1cm}
\begin{subfigure}{0.2\textwidth}
\includegraphics[width=5cm,height=4.4cm]{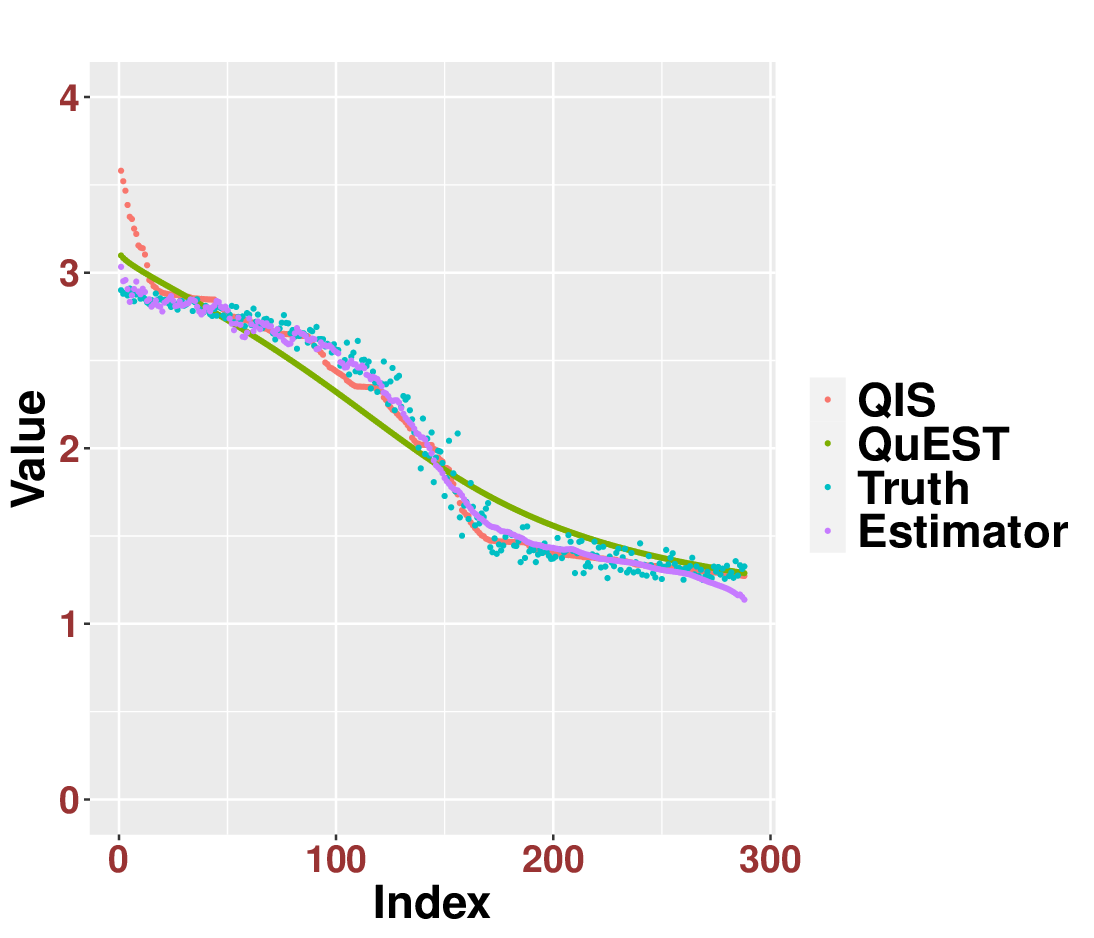}
\end{subfigure}
\hspace{0.18cm}
\begin{subfigure}{0.2\textwidth}
\includegraphics[width=5cm,height=4.4cm]{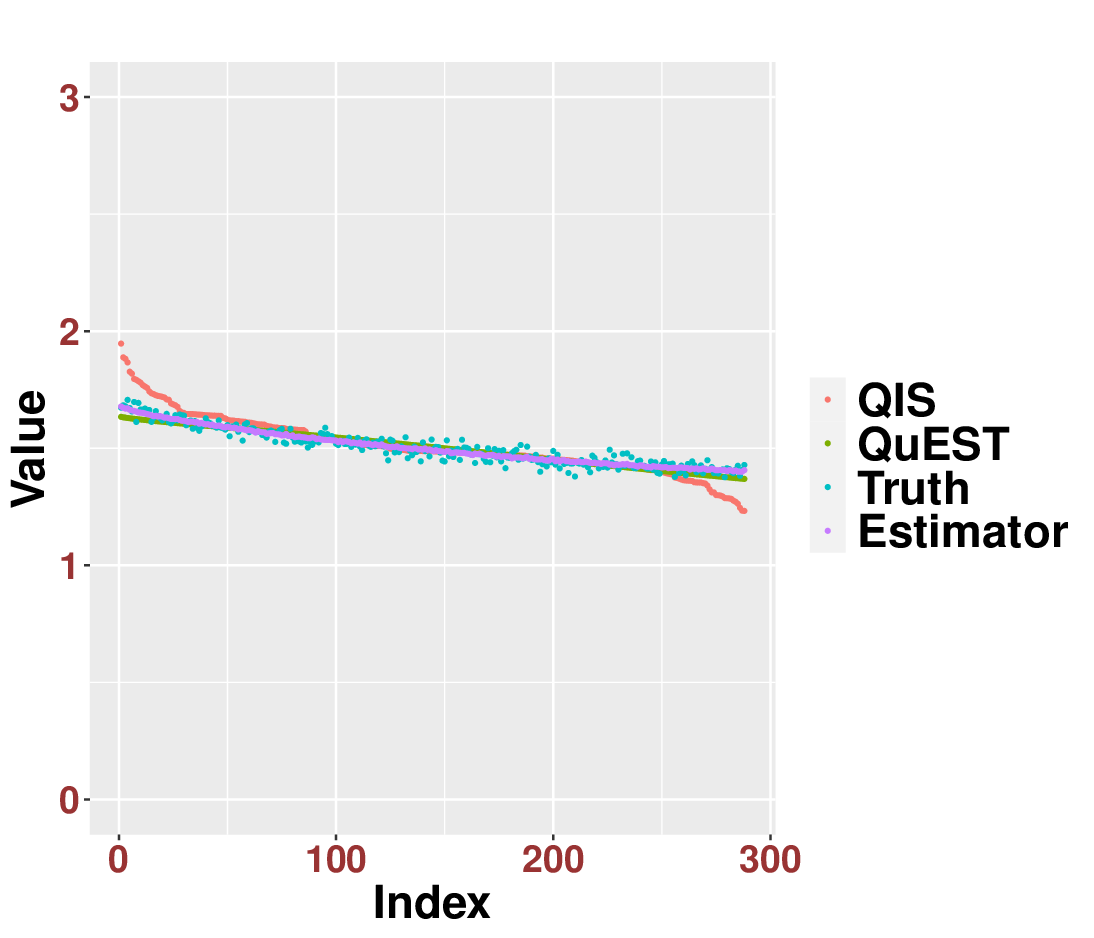}
\end{subfigure}
\hspace{0.18cm}
\begin{subfigure}{0.2\textwidth}
\includegraphics[width=5cm,height=4.4cm]{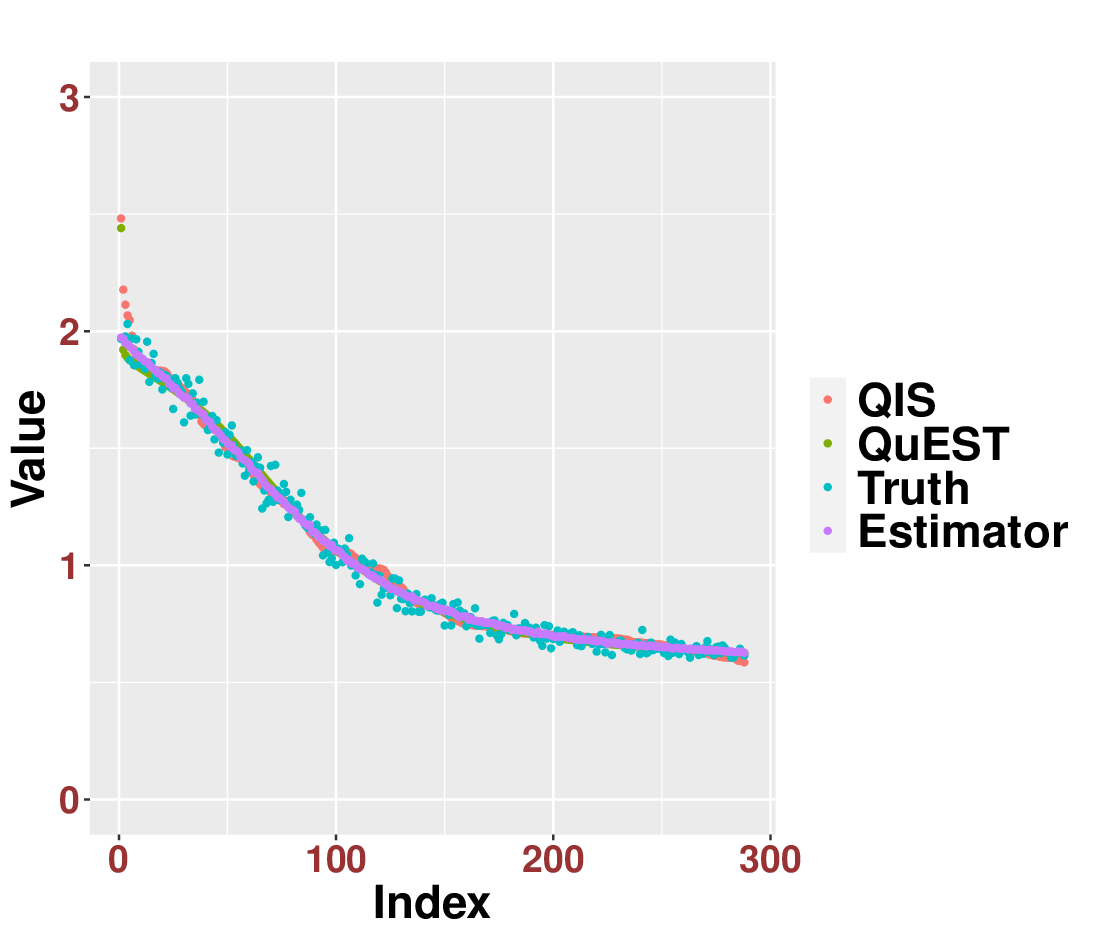}
\end{subfigure}
\hspace{0.18cm}
\begin{subfigure}{0.2\textwidth}
\includegraphics[width=5cm,height=4.4cm]{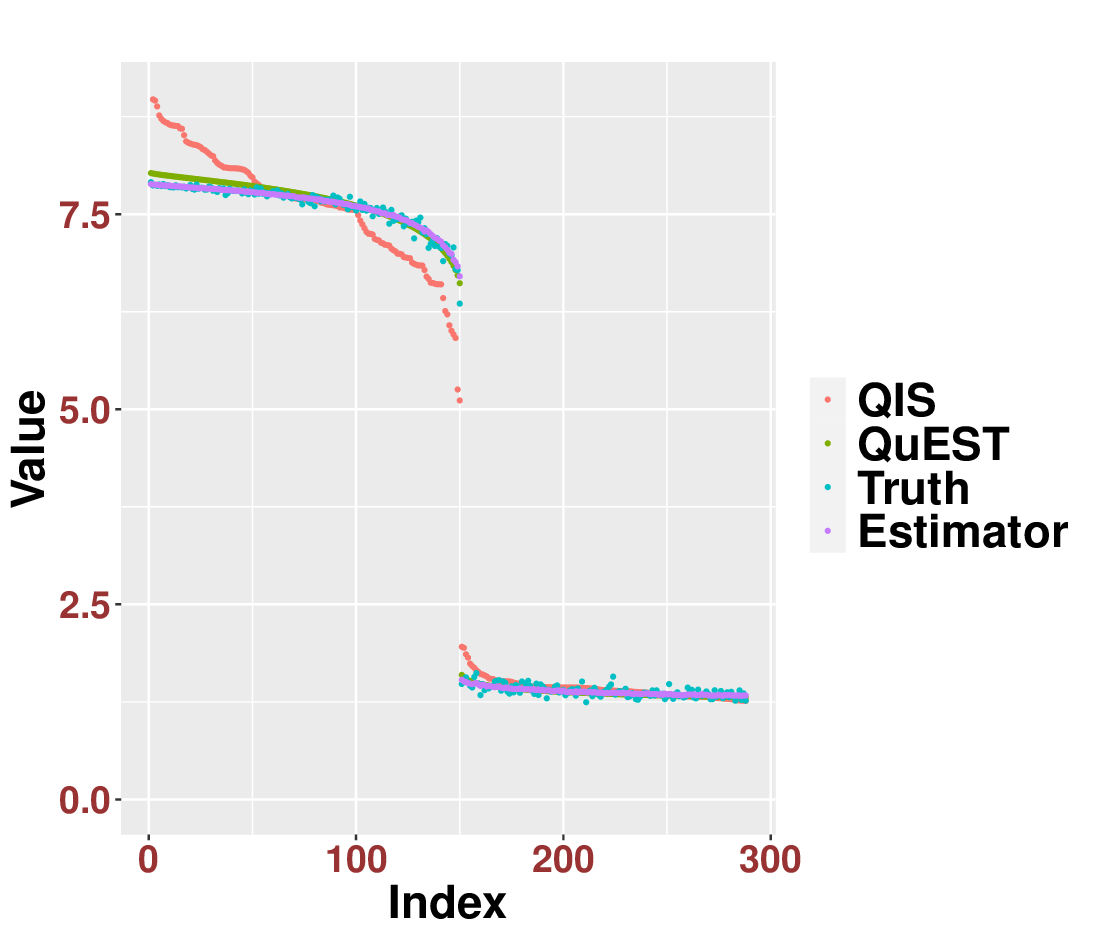}
\end{subfigure}
\caption{Comparison of different methods for $\ell(x)=x$. From left to right, we provide the results for the simulation settings (i)--(iv) as outlined in Section \ref{sec_setup}. In the simulations, we set $p=300$ and $n=600.$}
\label{fig_compareothersone}
\end{figure}

\begin{figure}[h]
\hspace*{-0.1cm}
\begin{subfigure}{0.2\textwidth}
\includegraphics[width=5cm,height=4.4cm]{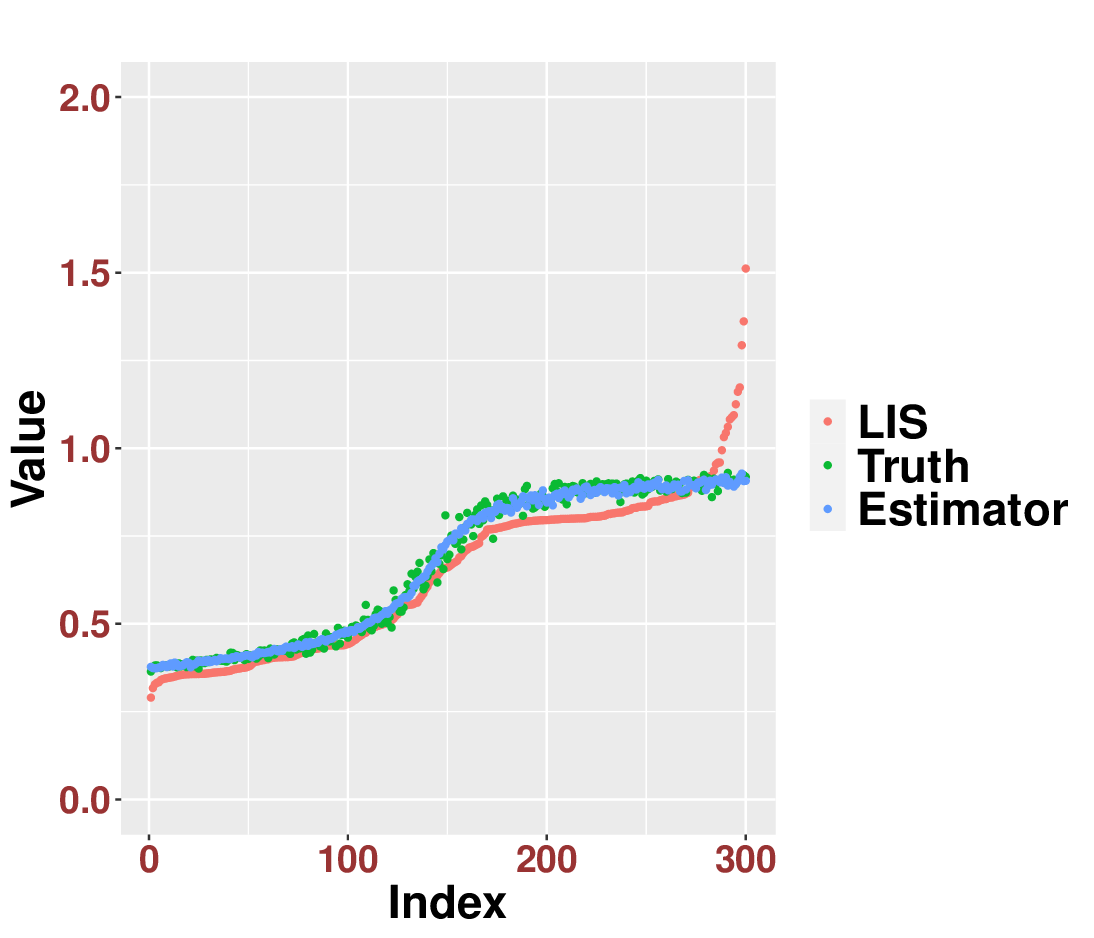}
\end{subfigure}
\hspace{0.18cm}
\begin{subfigure}{0.2\textwidth}
\includegraphics[width=5cm,height=4.4cm]{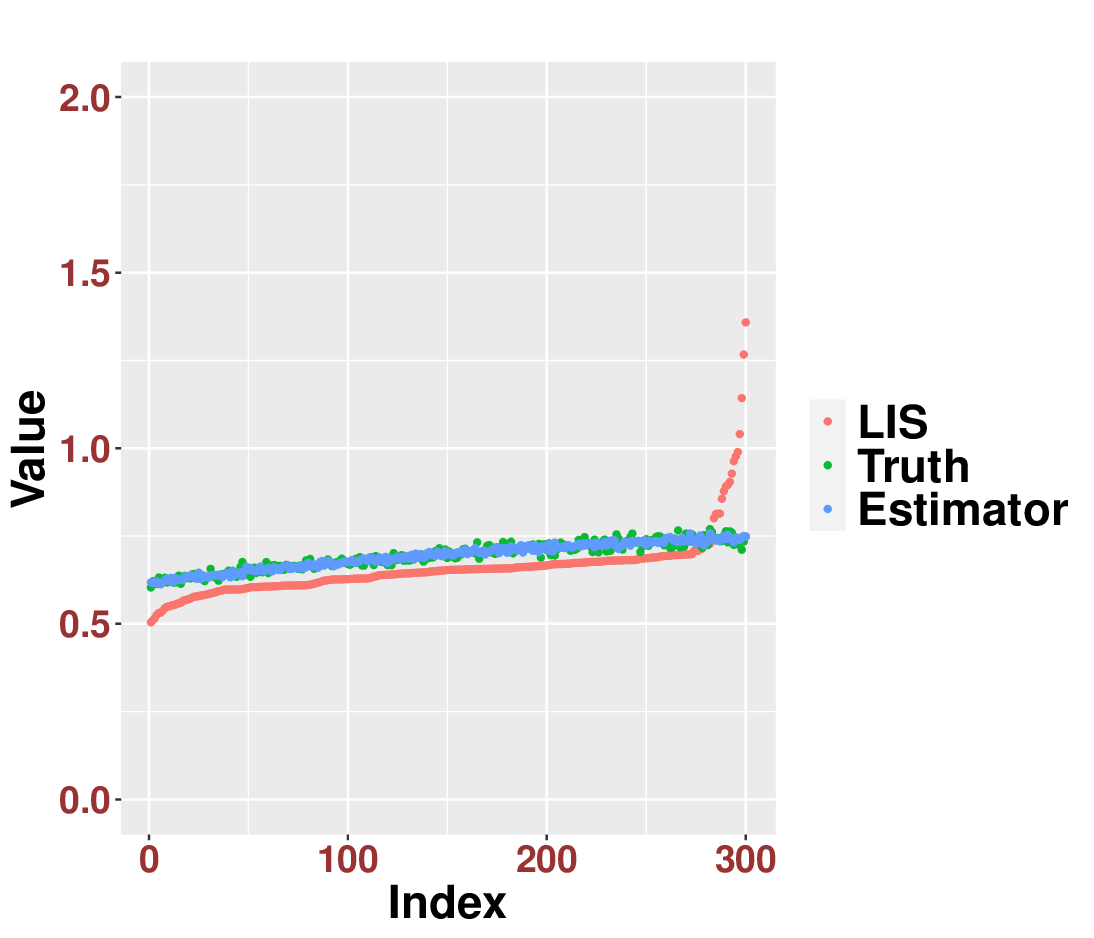}
\end{subfigure}
\hspace{0.18cm}
\begin{subfigure}{0.2\textwidth}
\includegraphics[width=5cm,height=4.4cm]{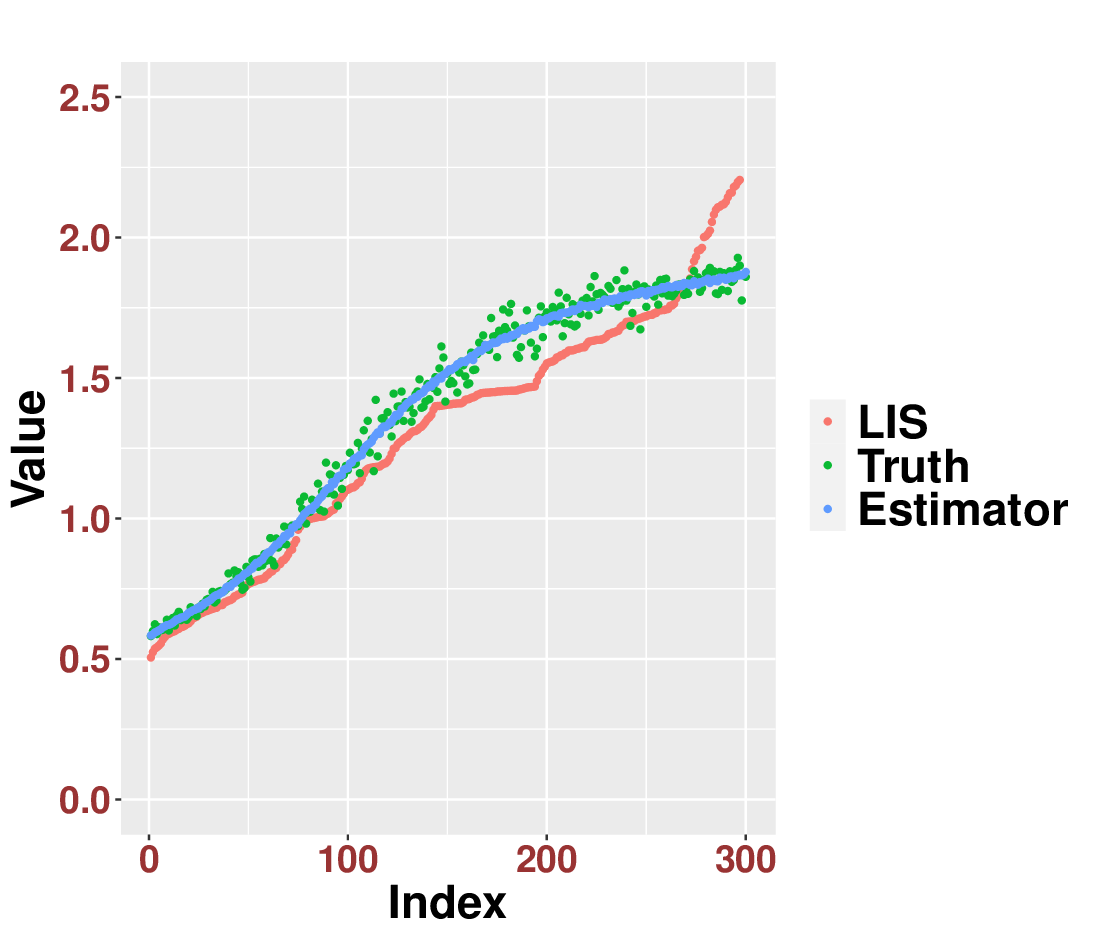}
\end{subfigure}
\hspace{0.18cm}
\begin{subfigure}{0.2\textwidth}
\includegraphics[width=5cm,height=4.4cm]{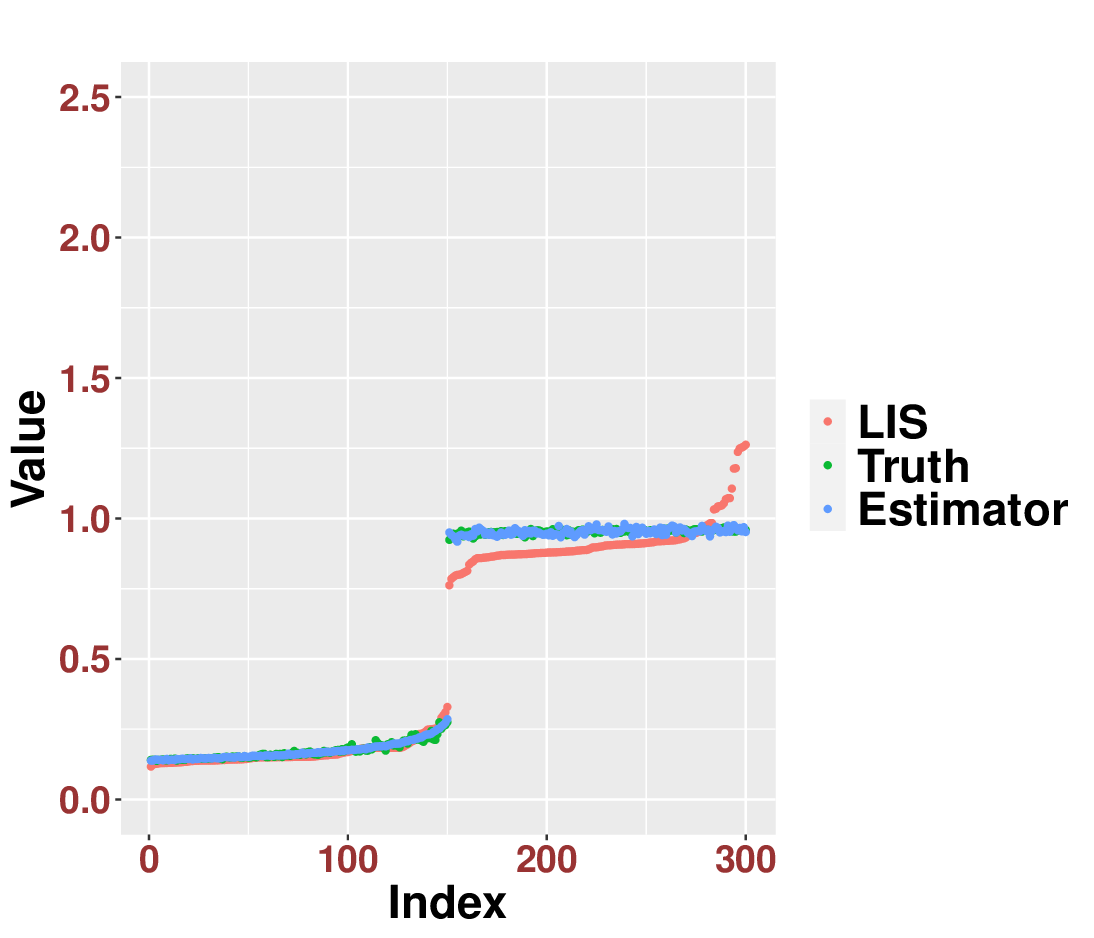}
\end{subfigure}
\caption{Comparison of different methods for $\ell(x)=x^{-1}$. From left to right, we provide the results for the simulation settings (i)--(iv) as outlined in Section \ref{sec_setup}. In the simulations, we set $p=300$ and $n=600.$ }
\label{fig_compareotherstwo}
\end{figure}

%Our comparisons are twofold. First, in Figures \ref{fig_compareothersone} and \ref{fig_compareotherstwo}, we visually compare the plots of different estimators of the shrinkers. We conclude that our proposed estimators demonstrate superior accuracy in all settings and for both forms of $\ell(x)$. As for the other estimators, their performance depends crucially on the specific underlying population covariance matrix and the chosen loss function. Second, in Figure \ref{fig_compareotherriskone}, we check the accuracy of different estimators in predicting the generalization errors as in Corollary \ref{coro_lowbound}. We find that our proposed estimators provide the most accurate predictions. It is also worth noting that the competing methods perform well when $\ell(x)=x$ but get worse when $\ell(x)=x^{-1}.$ 

\begin{figure}[h]
\begin{subfigure}{0.22\textwidth}
\includegraphics[width=5cm,height=4.4cm]{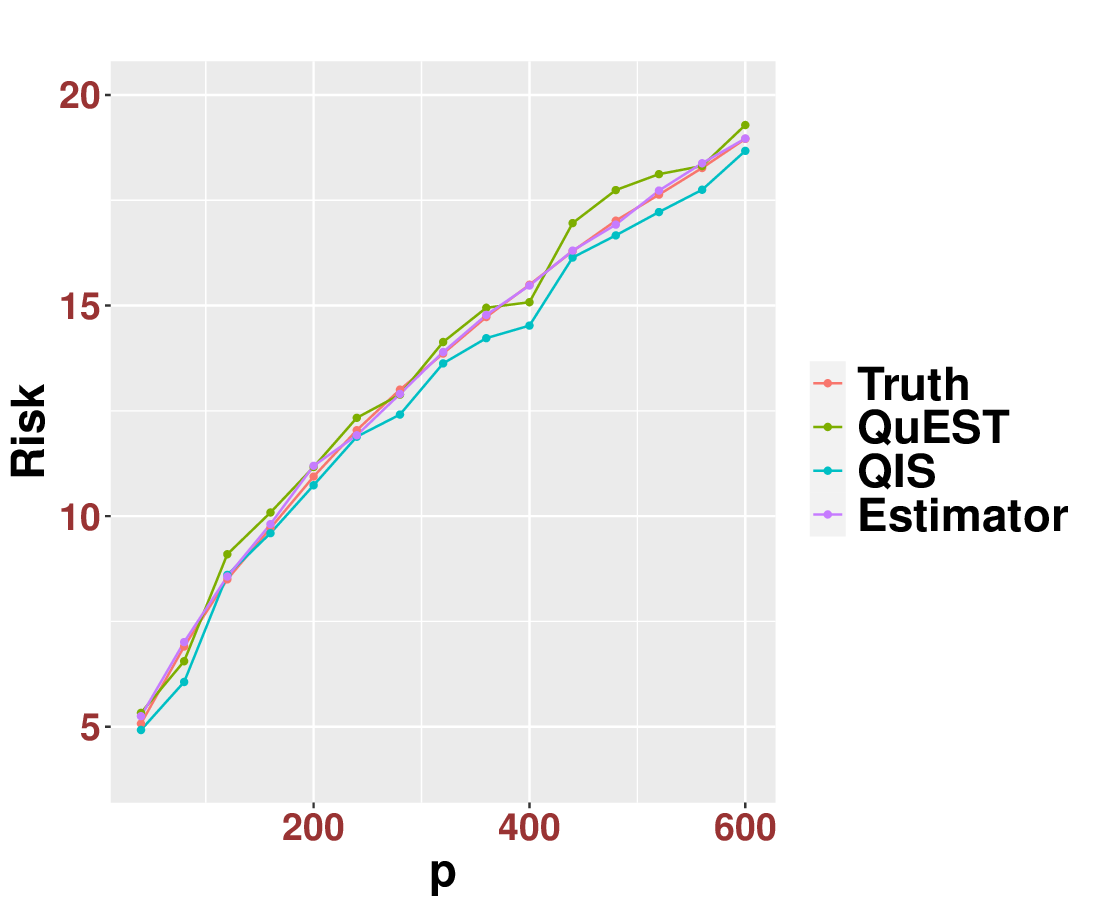}
\end{subfigure}
\begin{subfigure}{0.21\textwidth}
\includegraphics[width=5cm,height=4.4cm]{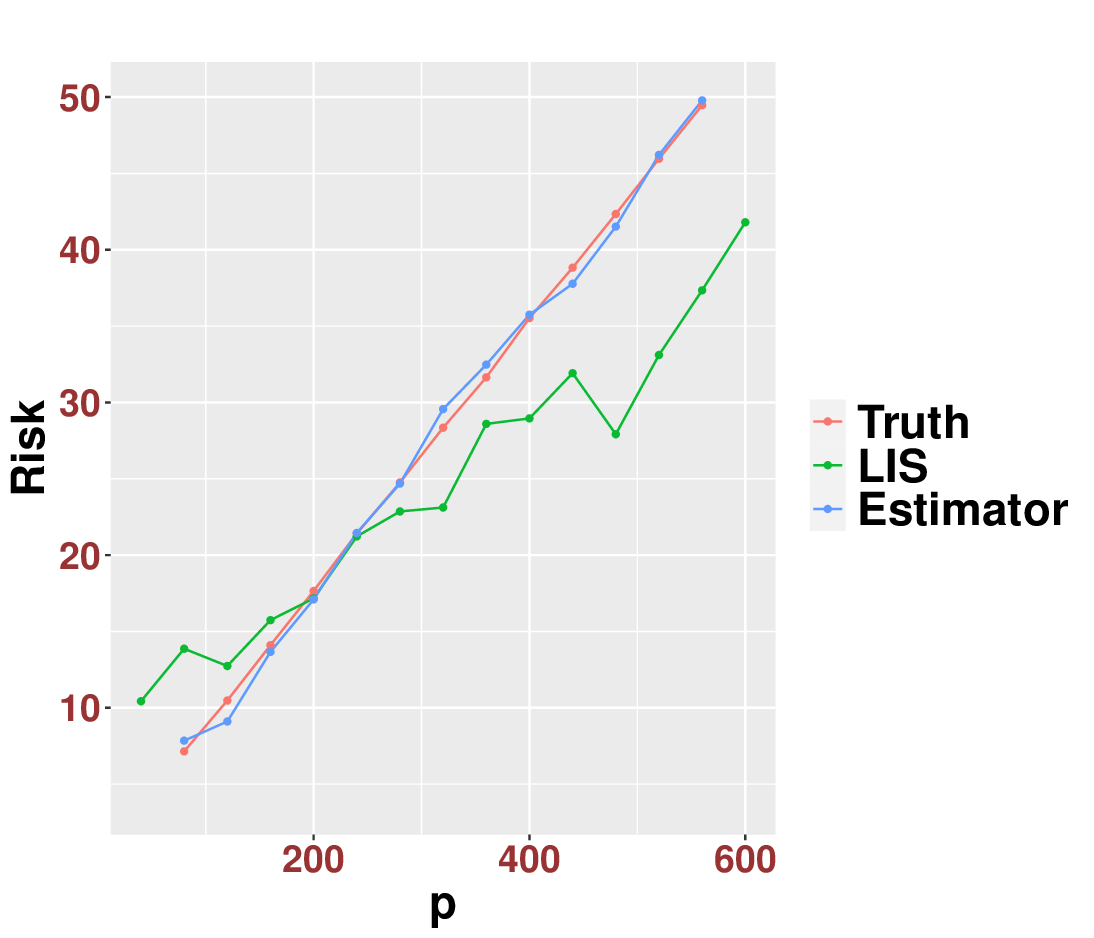}
\end{subfigure}
\begin{subfigure}{0.21\textwidth}
	\includegraphics[width=5cm,height=4.4cm]{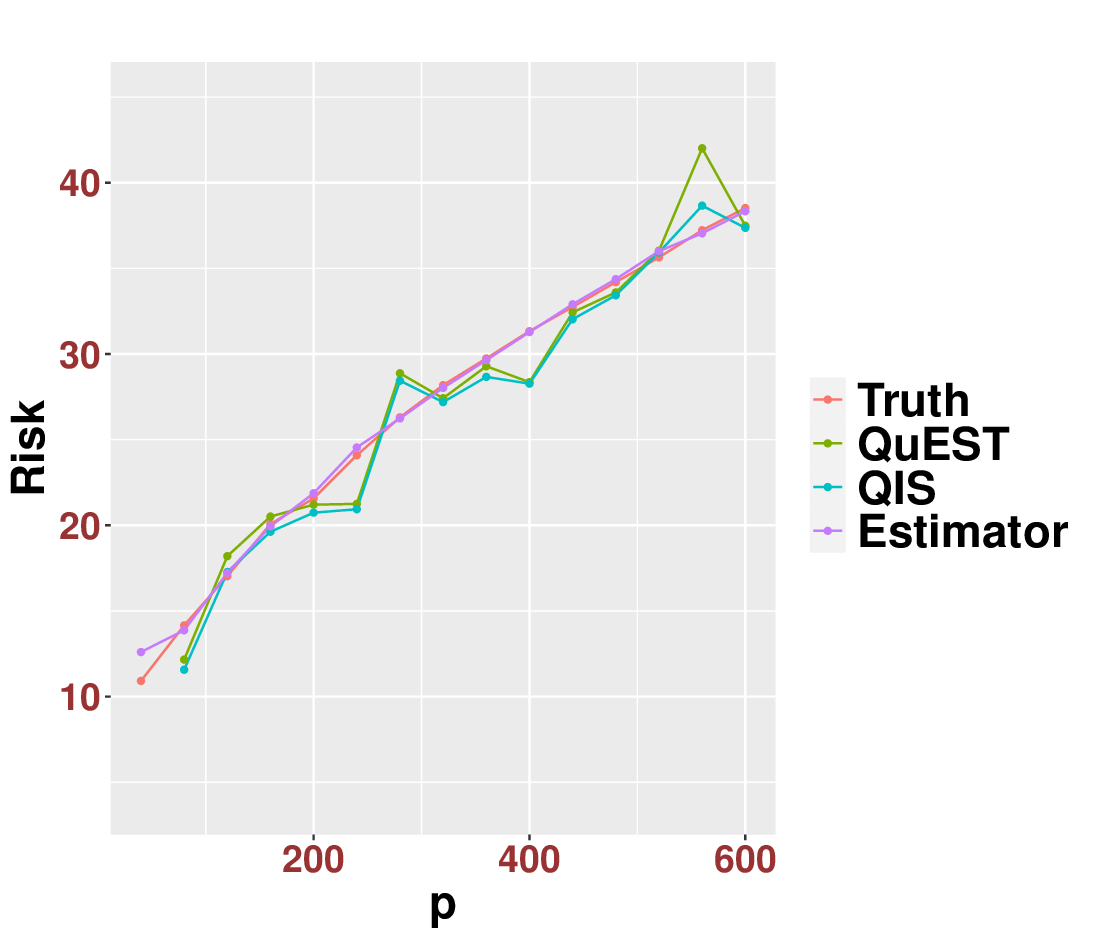}
\end{subfigure}
\begin{subfigure}{0.21\textwidth}
	\includegraphics[width=5cm,height=4.4cm]{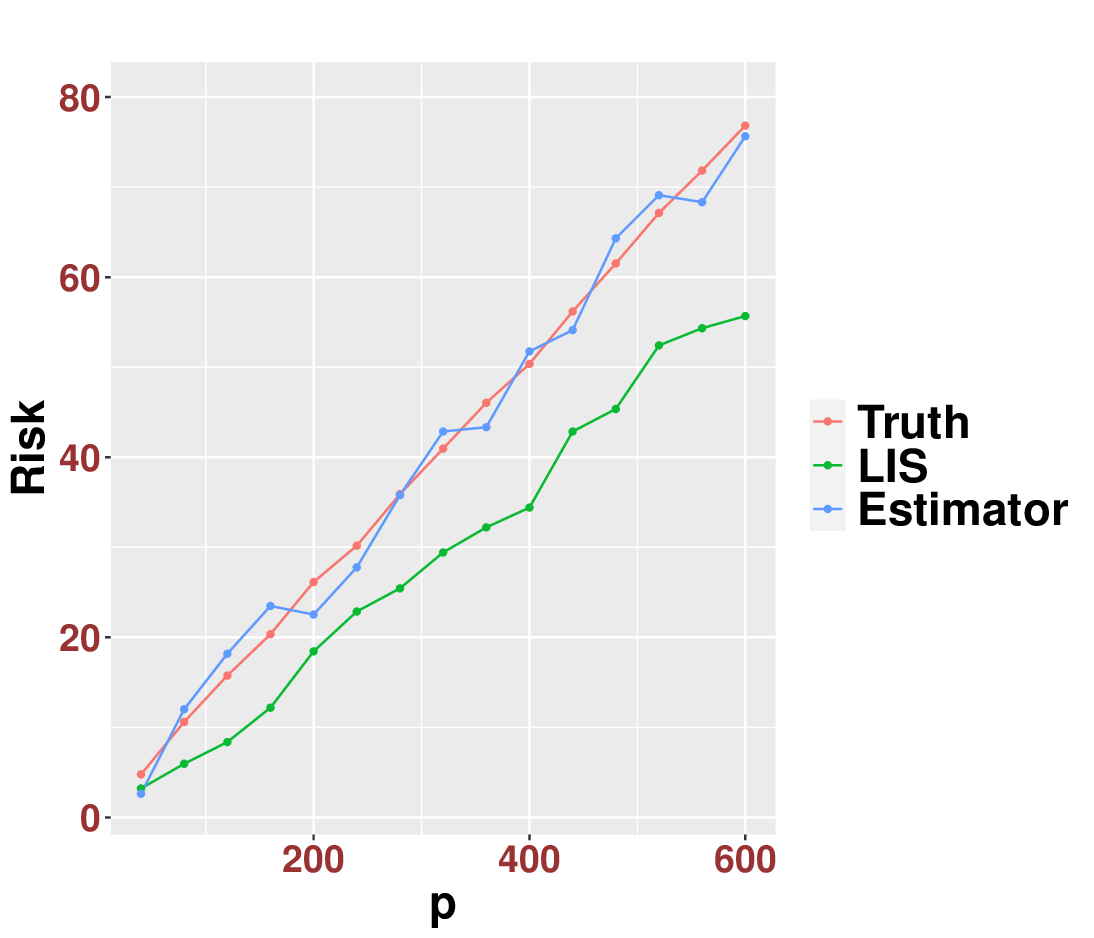}
\end{subfigure}
\caption{Comparison of different methods for predicting the risks under simulation settings (i) and (iv) as outlined in \Cref{sec_setup}. From left to right, the comparisons are made for four cases: setting (i) with the Frobenius norm ($\ell(x)=x$), setting (i) with the Stein norm ($\ell(x)=x^{-1}$), setting (iv) with the Frobenius norm, and setting (iv) with the Stein norm. In the simulations, we set $p=300$ and $n=600$. To enhance visualization, the reported risks are multiplied by $p$.}
%	For the left panel, we use the Frobenius norm which corresponds to $\ell(x)=x.$ For the right panel, we consider the Stein norm which corresponds to $\ell(x)=x^{-1}.$ 
\label{fig_compareotherriskone}
\end{figure}

%\begin{figure}
%\begin{subfigure}{0.5\textwidth}
%\includegraphics[width=8cm,height=5cm]{Figure/riskx2.eps}
%\end{subfigure}
%\begin{subfigure}{0.5\textwidth}
%\includegraphics[width=8cm,height=5cm]{Figure/riskx12.eps}
%\end{subfigure}
%\caption{Comparison of different methods for predicting the risks under setting (iv). Other simulation settings are the same as those in the caption of Figure \ref{fig_compareotherriskone}. }
%\label{fig_compareotherrisktwo}
%\end{figure}  

\section{Proof strategy and key technical ingredients}\label{sec_generalproofdetail}

\subsection{Outline of the proof strategy}\label{sec_proofstrategy}

In this subsection, we outline the proof strategy. We focus on the more challenging term in \eqref{eq_varphigeneralform}, while the term in \eqref{eq_varphigeneralform0} is considerably simpler to deal with. We rewrite \eqref{eq_varphigeneralform} as
\begin{equation}\label{eq:ellSigma}
\bm{u}_i^\top \ell(\Sigma) \bm{u}_i=\sum_{j=1}^p \ell(\wt\sigma_j) |\avg{\bm{u}_i, \bm{v}_j}|^2, \quad   i\in \qq{\sfK},
\end{equation}
where recall that $\{\bm{v}_i\}_{i=1}^p$ and $\{\bm{u}_i\}_{i=1}^p$ are the eigenvectors of the population and sample covariance matrices $\Sigma$ and \smash{$\widetilde{\mathcal{Q}}_1$}, respectively. The analysis of outlier and non-outlier eigenvectors in the spiked model relies on different strategies. 
The study of outlier eigenvectors is based on a standard perturbation argument. 
Under item (iv) of Assumption \ref{main_assumption}, the outlier eigenvalues are well separated from the bulk of the spectrum (see Lemma \ref{lemma_spikedlocation} below). Therefore, using Cauchy's integral formula, the generalized components of the outlier eigenvectors can be represented as a contour integral of the generalized components of the resolvents, which can be well approximated using the anisotropic local law (c.f.~Section \ref{appendix_locallawsection}).

Specifically, for $i\in \qq{r}$, we decompose \eqref{eq:ellSigma} as   
\begin{equation}\label{eq_decomposition}
\bm{u}_i^\top \ell(\Sigma) \bm{u}_i =\sum_{j=1}^r \ell(\widetilde{\sigma}_j) |\avg{\bm{u}_i, \bm{v}_j}|^2 +\sum_{j=r+1}^p \ell(\widetilde{\sigma}_j) |\avg{\bm{u}_i, \bm{v}_j}|^2. 
\end{equation}
On one hand, according to Lemma \ref{lemma_spikedlocation} below, the first term on the right-hand side (RHS) of \eqref{eq_decomposition} is dominated by the term $\ell(\widetilde{\sigma}_i) |\avg{\bm{u}_i, \bm{v}_i}|^2$ since $|\avg{\bm{u}_i, \bm{v}_j}|^2$ is negligible for $j\ne i$. The term $\ell(\widetilde{\sigma}_i) |\avg{\bm{u}_i, \bm{v}_i}|^2$ can be further estimated using the asymptotic behavior of \smash{$\widetilde \lambda_i$} and $|\avg{\bm{u}_i, \bm{v}_i}|^2$ described in Lemma \ref{lemma_spikedlocation}.  
On the other hand, for the second term on the RHS of \eqref{eq_decomposition}, since there are $p-r$ terms in the summation, existing delocalization results in the literature---$|\avg{\bm{u}_i, \bm{v}_j}|^2 \prec n^{-1}$, $j \in \qq{r+1, p}$---are no longer sufficient for our proof. Instead, we need to derive higher-order asymptotics for $|\avg{\bm{u}_i, \bm{v}_j}|^2$. In (\ref{eq_intermediateresult}) below, we will show that with high probability,
\begin{equation*}
|\avg{\bm{u}_i, \bm{v}_j}|^2=\frac{\mathfrak{b}_id_i^2}{\widetilde{\sigma}_i^2\wt\sigma_j} G_{ij}(\mathfrak{a}_i) G_{ji}(\mathfrak{a}_i)+\OO (n^{-3/2+\e}),\quad j\in \qq{r+1,p},
\end{equation*}
where $G(z)$ represents the resolvent associated with the non-spiked model (see \Cref{rmk_keymodification} below for the precise definition). By estimating the first term on the RHS using a variational argument and the local laws of $G(z)$,  we obtain the asymptotic limit for the second term on the RHS of \eqref{eq_decomposition} and complete the proof. Further details can be found in the discussion below \eqref{eq_w2iform}. 

Next, we describe the strategy for studying the non-outlier eigenvectors, which differs significantly from the approach used for the outlier eigenvectors and is essentially non-perturbative. 
Let  $\{\bu_i\}_{i=1}^p=\{V^\top\bm{u}_i\}_{i=1}^p$ be the eigenvectors of $\Lambda^{1/2}V^\top XX^\top V \Lambda^{1/2}$. By noting that $\bm{u}_i^\top \ell(\Sigma) \bm{u}_i=\bu_i^\top \ell(\Lambda) \bu_i$, we can focus on the model $W:=\Lambda^{1/2}V^\top X$ and study the asymptotic properties of its singular vectors.
% Let $D \equiv \Lambda_0$ for the non-spiked model and $D \equiv \Lambda$ for the spiked model and denote
% \begin{equation}\label{eq_W0definition}
% W(0) \equiv W:=D^{1/2} V^\top X, 
% \end{equation}
%where $D \equiv \Lambda_0$ for the non-spiked model and $D \equiv \Lambda$ for the spiked model. 
Let $X^G$ be a $p \times n$ Gaussian random matrix that is independent of $X$. Its entries are i.i.d.~centered Gaussian random variables with variance $n^{-1}$. For $t \geq 0,$ we define a \emph{matrix Dyson Brownian motion} $W(t)$ as
\begin{equation}\label{defn_wt}
W(t):=W(0) +\sqrt{t} X^G. %,\quad \text{with}\quad W(0)=W. 
\end{equation} 
The matrix $W(0)$ might not be precisely equal to $W$ (although in the proof below, we will consider it as a perturbation of $W$). To account for this, we introduce a new notation for the initial condition, denoting it as \smash{$W(0):=D_0^{1/2}V^\top X$}, where $D_0$ will be chosen later and is a diagonal matrix with eigenvalues that satisfy the conditions for $\{\wt\sigma_i\}$ in \Cref{main_assumption}. 
In other words, $D_0$ consists of a non-spiked part $D_{00}$ with eigenvalues satisfying item (iii) of \Cref{main_assumption}, while the spikes of $D_0$ satisfy item (iv) of \Cref{main_assumption}. Then, we denote the sample covariance matrix of $W(t)$ by  
\begin{equation}\label{eq_qt}
Q(t) \;\deq\; W(t) W(t)^\top\, ,
\end{equation}
and let $\{\lambda_i(t)\}_{i=1}^p$ and $\{\bu_i(t)\}_{i=1}^p$ be the associated eigenvalues and eigenvectors of $Q(t)$, respectively. %Note that specifically, we have $\lambda_i(0)=\lambda_i $ and $\bu_i(0)=\bu_i $ following the above notations. 
As shown in \cite{9779233,ding2022edge}, the non-outlier eigenvalues of $Q(t)$ follow a \emph{rectangular Dyson Brownian motion}, whose limiting density $\varrho_t$ is given by \emph{the rectangular free convolution} of the ESD of $D_{00}$ with the MP law at time $t$. Similar to Definition \ref{defn_generaleigenvaluelocation}, we denote by $\{\gamma_k(t)\}_{k=1}^p$ the classical eigenvalue locations for $\varrho_t$.  (For the definitions of $\varrho_t$ and $\gamma_k(t)$, we refer readers to the discussion around \eqref{koux} in the appendix, where $\Lambda_0$ corresponds to $D_{00}$.)
%(We refer readers to the discussion around \eqref{koux} for the definitions of $\varrho_t$ and $\gamma_k(t)$, where $\Lambda_0$ is taken to be $D_{00}$.)  
Corresponding to (\ref{eq_phidefinitionintheend}), for $x>0$ and vectors $\bv, \bw\in \R^p,$ we define the function 
\be\label{defvarphi1}
\varphi_0(\bv,\bw,x)\equiv \varphi(\bv,\bw,x,D_0) := c_n\bv^\top D_0(x|1+m(x,D_{00})D_0|^2)^{-1} \bw,
\ee
where $m(x,D_{00})$ is defined as in (\ref{eq_defnmc}) with $\Sigma_0$ replaced by $D_{00}$. 
% \be\label{defvarphi1}
% %:=\sum_k  \frac{\phi \,  \sigma_k }{E \abs{1 + m (E^+)  \sigma_k  }^2} \, |\bv(k)|^2
%  \varphi_0( \bv,\bw,x ): = 
% \begin{cases} 
%  c_n\bv^\top \Lambda_0(x|1+m(x)\Lambda_0|^2)^{-1} \bw, &  x>0 \\
%   (1-c_n^{-1})^{-1} \bv^\top (1+m(0) \Lambda_0 )^{-1} \bw, & x=0 \ \text{and} \ c_n>1.
%  \end{cases}
% \ee

Our proof strategy for the non-outlier eigenvectors consists of two steps. In the first step, we establish the eigenvector distribution for $Q(t)$ on a small time scale $t \asymp n^{-1/3+\mathsf{c}}$, where $\mathsf c$ is a small positive constant. In the second step, we introduce a novel comparison argument to show that the eigenvector distribution of $Q(t)$ at time $t \asymp n^{-1/3+\mathsf{c}}$ is asymptotically equal to the distribution at $t=0$. These two steps together establish the asymptotic distribution of the eigenvectors of the original matrix $Q(0)= WW^\top$ for properly chosen $D_0$. For the first step, we will prove the following counterpart of \Cref{Thm: EE} for $t \asymp n^{-1/3+\mathsf{c}}$.

%see the discussion below Lemma \ref{lemma_universal} for more details. Combining these two steps establishes the asymptotic distribution of the eigenvectors of $Q(0) \equiv WW^\top$.
% {\color{red}[maybe we need to add general results for all $c_n$ in the end here]}
\begin{lemma}\label{lem: moment flow} 
Suppose Assumption \ref{main_assumption} holds (when $\Lambda=D_0$). Take $n^{-1/3+\e}\le t \le n^{-1/6-\e}$ for an arbitrary constant $\e\in (0,1/12)$. Given a deterministic unit vector $\bv \in \mathbb{R}^p$ and a subset of indices $\{i_k+r\}_{k=1}^L \subset \qq{r+1,\sfK}$ for a fixed integer $L$, define the $L\times L$ diagonal matrix 
% For any deterministic unit vector $\bm{v} \in \mathbb{R}^p$  and fixed $L\in\Z_+$, 
% let $\{i_k\}_{k=1}^L \subset \qq{1,p}$ for the non-spiked model and $\{i_k\}_{k=1}^L \subset \qq{r+1,p}$ for the spiked model be a sequence of distinct indices, and define
\begin{equation}\label{eq:XiL}
\Xi_L(t):=\operatorname{diag}\left\{ \varphi_0(\bv, \bv,\gamma_{i_1}(t)), \cdots, \varphi_0(\bv, \bv,\gamma_{i_L}(t)) \right\}.
\end{equation}
%where $\varphi_0$ is defined in (\ref{defvarphi1}). 
Then, we have that  
\begin{equation}\label{eq_lemma41equation}
\sqrt{p} 
\begin{pmatrix}
 \xi_1\avg{\bv, \bu_{i_1+r}(t)} \\
 \vdots \\
\xi_L\avg{\bv, \bu_{i_L+r}(t)}
\end{pmatrix}
\simeq \mathcal{N}(\bm{0}, \Xi_L(t)), 
\end{equation}
where $\xi_1,\ldots, \xi_L\in\{\pm 1\}$ are i.i.d.~uniformly random signs independent of $W(t)$. 
\end{lemma}

The proof of Lemma \ref{lem: moment flow} will be outlined in Section \ref{sec_maintwolemmas}. It is based on a careful analysis of the eigenvector moment flow for $Q(t)$, and the main result of the analysis is summarized in \Cref{main result: Jia}.

%First, we will demonstrate that {\color{red}(\ref{eq_lemma41equation}) holds, with the covariance matrix $\Xi_L(t)$ replaced by $\Xi_L(0)$}, by analyzing the eigenvector moment flow for $Q(t)$. This step will be summarized in Theorem \ref{main result: Jia}. Subsequently, we will establish that $\Xi_L(t)$ is well approximated by $\Xi_L(0)$ by examining the local density function $\varrho_t$, as detailed in Lemma \ref{thm_localdensityestimate} below.

%In the proof, we will first show that  (\ref{eq_lemma41equation}) holds by replacing $\mathrm{V}_d(0)$ with $\mathrm{V}_d(t)$ for $t \asymp n^{-1/3+\mathsf{c}}$ in Theorem \ref{main result: Jia} by analyzing the eigenvector moment flow. Then we will show that $\mathrm{V}_d(t)$ can be well approximated by $\mathrm{V}_d(0)$ by analyzing the local density function as in Lemma \ref{thm_localdensityestimate}.

%{\color{blue}[add some discussions here: also two steps here. The first step is to prove something for $\mathrm{V}_d(t)$ as in Theorem \ref{main result: Jia}. The second step is to prove local density estimate so that for $t=\oo(1)$ $\mathrm{V}_d(t)$ is close to $\mathrm{V}_d(0)$ as in Lemma \ref{thm_localdensityestimate}]. }
%{\color{red} mention the model reduction here, i.e., why it suffices to study $XUD$ instead of $X\Sigma^{1/2}.$}

%Next, in the second step in deriving the eigenvector distribution of $Q(0)$,
% we show that the eigenvector distribution of $Q(t)$ at time $t \asymp n^{-1/3+\mathsf{c}}$ is the same as that at $t=0$ so that the results hold for $Q(0) \equiv WW^\top$ which completes the proof. More specifically,

For the second step, we will establish the following comparison result, which shows that the distributions of the non-outlier eigenvectors of $W(t)$ coincide with those of $W$ when we carefully select $D_0$. Note that the proof of Theorem \ref{Thm: EE} is completed by combining \Cref{lemma_universal} with Lemma \ref{lem: moment flow}.

%show that (\ref{eq_lemma41equation}) remains valid for $t=0$, as stated in the following lemma. It is worth noting that the proof of Theorem \ref{Thm: EE} is completed by combining \Cref{lemma_universal} with Lemma \ref{lem: moment flow}. See Appendix \ref{appendix_proof_ev} for the formal proof.

%{\bf Step 2}: Comparison:

\begin{lemma}\label{lemma_universal}  
Suppose Assumption \ref{main_assumption} holds. Fix a time $\ft = n^{-1/3+\mathsf{c}}$ for a small constant $\mathsf{c} \in (0,1/6)$. We choose $D_0=\Lambda-\ft$. For a fixed integer $L$, let $\theta:\mathbb{R}^L \rightarrow \mathbb R$ be a smooth function satisfying that 
$$ \partial^{\mathbf k}\theta(x)\le C(1+\|x\|_2)^C$$
for a constant $C>0$ and all $\mathbf k\in \N^L$ satisfying $\|\mathbf k\|_1\le 5$. Then, given any deterministic unit vector $\bv \in \mathbb{R}^p$ and an arbitrary subset of indices $\{i_k+r\}_{k=1}^L \subset \qq{r+1,\sfK}$, there exists a constant $\nu>0$ such that 
\begin{equation*}
\mathbb{E} \left( \theta\left(p|\avg{\bv, \bu_{i_1+r}(\ft)}|^2, \ldots, p|\avg{\bv, \bu_{i_L+r}(\ft)}|^2\right) \right)=\mathbb{E} \left( \theta\left(p|\avg{\bv, \bu_{i_1+r}}|^2, \ldots, p|\avg{\bv, \bu_{i_L+r}}|^2\right) \right) +\OO(n^{-\nu}), 
\end{equation*}
where $\{\bu_i(\ft)\}_{i=1}^p$ and $\{\bu_i\}_{i=1}^p$ are the eigenvectors of $Q(\ft)$ and $WW^\top$, respectively. 

%Under the assumption of Thm. \ref{Thm: EE}, we have 
% $$
%\left(\left\langle \bv, \bu_k\right\rangle,  \left\langle \bv, \bu_l\right\rangle\right) \to 
%\left(\left\langle \bv, \bu^G_k\right\rangle,  \left\langle \bv, \bu^G_l\right\rangle\right) $$ 
%here $\bu^G_l$ is denoted as the $l$-th eigenvector in the Gaussian case, i.e., the entries of $X$ are Gaussian i.i.d random variables.  
%
%%\end{lemma}
% 
%%\begin{lemma}\label{compbulk}  Define
% $$
%W_t:=D_t^{1/2} U X, \quad D_t=D+t
%%, \quad D^t:= \begin{pmatrix}
%% D'+t I & 0\\
% %0& 0
%% \end{pmatrix}
%$$
%Note: not $W(t)$. Let $\bu_k^t$ and $\bu_l^t$ be the eigenvectors of $W_tW_t^*$.  If $t\le N^{-9/10}$ then for any $k$ and $l$
% $$
%\left(\left\langle \bv, \bu_k^t\right\rangle,  \left\langle \bv, \bu_l^t\right\rangle\right) \to 
%\left(\left\langle \bv, \bu_k(t)\right\rangle,  \left\langle \bv, \bu_l(t)\right\rangle\right) $$ 
 \end{lemma}

The proof of Lemma \ref{lemma_universal} will be outlined in Section \ref{subsec_functionalrepresentation}. In the proof, we introduce the auxiliary matrix 
\begin{equation}\label{eq_mathcalwt}
\mathcal{W}_t :=D_t^{1/2} V^\top X, \quad \text{with}\quad D_t:=D_0 +t .
\end{equation}  
We choose the non-spiked parts of $D_0$ and $D_t$ as $D_{00}=\Lambda_0-\ft$ and $D_{t0}=D_{00}+t$, respectively. 
%where $D_0$ is a diagonal matrix {\color{red}to be determined later} whose eigenvalues satisfy the conditions  in \Cref{main_assumption}.
%for $\{\wt\sigma_i\}_{i=1}^p$
We denote the sample covariance matrix of $\cal W_t$ as $\mathcal{Q}_t:=\mathcal{W}_t \mathcal{W}_t^\top$ and its eigenvectors as $\{\bw_i(t) \}_{i=1}^p$. We aim to prove that for any $t\asymp n^{-1/3+\fc}$, 
%Denote the eigenvectors of $\mathcal{Q}_t:=\mathcal{W}_t \mathcal{W}_t^\top$ as $\{\bw_i(t) \}_{i=1}^p.$  
%Then, with a slight abuse of notations, we can define $W(t)$ as in \eqref{defn_wt} but with the initial condition $W(0)=W_0$, where \smash{$W_0:=D_0^{1/2}V^\top X$}. We will prove that {\color{red}[I suggest we change the notation of LHS of (4.9). They are now different because of the change of initial conditions]}
\begin{equation}\label{eq_parttwocompare}
\mathbb{E} \left( \theta\left(p|\avg{\bv, \bu_{i_1+r}(t)}|^2, \ldots, p|\avg{\bv, \bu_{i_L+r}(t)}|^2\right) \right)=\mathbb{E} \left( \theta\left(p|\avg{\bv, \bw_{i_1+r}(t)}|^2, \cdots, p|\avg{\bv, \bw_{i_L+r}(t)}|^2\right) \right)+\oo(1). 
\end{equation}
This completes the proof of \Cref{lemma_universal} by taking $t=\ft$, in which case we have $\cal W_{\ft}=W$ as $D_0=\Lambda-\ft$. 
%As a special case, when $t=\ft$ and $D_0=\Lambda-\ft$, we have $\cal W_{\ft}=W$. Hence, 

\iffalse
Combining \eqref{eq_parttwocompare} with Lemma \ref{lem: moment flow}, we obtain that $(\sqrt{p}\avg{\bv, \bw_{i_k+r}(t)})_{1 \leq k \leq L}$ is also an asymptotically Gaussian random vector with covariance matrix $\Xi_L(t)$.  
%{\color{red}[If we do this, one comparison step (probably very trivial) is missing. The RHS of (4.9) matches the RHS of the equation in Lemma 4.2 but the LHS is not due to the change of initial conditions. But since we are using EXACTLY the same $X^G$ as in (\ref{defn_wt}) so they should be very close following the same comparision argument. But We need to mention this.]}
%defined with $\Lambda$ replaced by $\Lambda_0$. 
As a special case, when $t=\ft$ and $D_0:=\Lambda-t$, we have $\cal W_t=W$. Since $t=\oo(1)$, it is straightforward to see that the eigenvalues of $D_0$ also satisfy the conditions for \smash{$\{\wt\sigma_i\}$} in \Cref{main_assumption} and $\Xi_L(D_0,t)$ asymptotically matches $\Xi_L(\Lambda,0)$. This completes the proof of \Cref{lemma_universal}.
\fi

%Then, we can complete the proof  of Lemma \ref{lemma_universal} as $t=\oo(1).$  

To prove (\ref{eq_parttwocompare}), we will utilize a functional integral representation formula in terms of the resolvents, as presented in Lemma \ref{7.2}, and a novel comparison argument in Lemma \ref{nanzi} below. Now, we discuss briefly the new comparison strategy. In the literature (e.g., \cite{bloemendal2016principal, ding2019singular, knowles2013eigenvector}), when establishing the universality for the singular vector distributions of $\cal W_t$, people typically compare the representation formula for $\cal W_t$ directly with that for \smash{$D_t^{1/2}V^\top X^G$}. Let \smash{$Y^G$} be an independent copy of $X^G$. Due to the rotational invariance of the distribution of \smash{$X^G$}, \smash{$D_t^{1/2}V^\top X^G$} has the same law as $D_0^{1/2}V^\top X^G + \sqrt{t} Y^G$, whose non-outlier eigenvector distributions have been provided in \Cref{lem: moment flow}.
%By adding a small Gaussian component, one can study the flow $D^{1/2}V^\top X^G+\sqrt{t}X^G$ via a dynamical approach. 
In the previous comparison approach between $\cal W_t$ and \smash{$D_t^{1/2}V^\top X^G$}, people often applied the Lindeberg replacement (or other interpolations) to replace the entries of $X$ step by step with the entries of $X^G$ in the representation formula, controlling the error at each step using the local laws of the related resolvents. 
%and control the error appearing in each step by using the local laws of the related resolvents. 
%with the local laws to compare $X$ and $X^G$ entrywise. 
However, this approach fails for our setting---a direct application of the comparison idea (as in \cite{ding2019singular, knowles2013eigenvector}) leads to uncontrollable errors. 
%both tedious calculations and limitations on our assumptions about $\Sigma$. 
%novel approach by working with the auxiliary matrix (\ref{eq_mathcalwt}). Then we connect 
To address this issue, we propose a novel interpolation method that establishes a connection between $\mathcal{W}_t$ and $W(t)$ as $s$ varies from 0 to 1: 
\be\label{eq_dtusx}
W_t^s:=D_t^{1/2}\left[U^0 \mathsf{X}+(U^1-U^0) (\chi^s\odot \mathsf{X})\right]
, \quad \text{with}\quad \mathsf{X}:=\begin{pmatrix}
 X  \\
 X^G
\end{pmatrix}.
\ee
% $${\cor Y_t^s:=D_t^{1/2}U^s \mathsf{X}, \quad \text{with}\quad \mathsf{X}:=\begin{pmatrix}
%  X  \\
%  X^G
% \end{pmatrix}.}$$ 
Here, $\chi^s$ represents a continuous sequence of $2p\times n$ Bernoulli random matrices with i.i.d.~Bernoulli$(s)$ random entries, $\odot$ represents the Hadamard product, $U^0:=(V^\top,\mathbf 0_{p\times p})$, and $U^1:=((D_0/D_t)^{1/2} V^\top,  (t/D_t)^{1/2} V^\top )$. Under this choice, we can verify that $W_t^0=\mathcal{W}_t$, and $W_t^1$ has the same law as $W(t)$, leveraging the rotational invariance of the distribution of $X^G$. As a result, our analysis will focus on the deterministic covariance structure 
%matrix $U^s$ (not defined) 
rather than the random components $X$ and $X^G$.
%Consequently, our comparison will be conducted on the deterministic matrices $U^s$ instead of the random parts $X$ and $X^G$. 
%Here in the construction, $\mathsf{X}$ collects the randomness of our models, see \eqref{eq_rewriteone}--(\ref{eq_dtusx}) for the precise definition. 

%In the actual proof of (\ref{eq_parttwocompare}), we use a continuous interpolation argument on a properly chosen quantity (c.f.~(\ref{eq_keyquantitytouse})), as developed in \cite{MR3704770}. 

\subsection{Eigenvector moment flow and proof of Lemma \ref{lem: moment flow}}\label{sec_maintwolemmas}  

%we will analyze the eigenvector moment flow of $Q(t)$ and match the joint moments for the generalized components with the proper Gaussian random vector. 

The objective of this section is to prove Lemma \ref{lem: moment flow} by analyzing the eigenvector moment flow (EMF) of $Q(t)$ in (\ref{eq_qt}), building upon the idea in \cite{MR3606475}. We first introduce some new notations before discussing the key ideas for the proof. 

For any deterministic unit vector $\bv\in \R^{p }$ as in Lemma \ref{lem: moment flow}, we define $z_k (t) \deq \sqrt{p} \scalar{\b v}{\b u_k(t)}.$ Moreover, we use the notation $\bm{\lambda}(s)\equiv (\lambda_i(s))_{1 \leq i \leq p}$ and denote the filtration of $\sigma$-algebras up to time $t$ by
\begin{equation}\label{eq_mathcalft}
\mathcal{F}_t=\sigma\left( W(0), (\bm{\lambda}(s))_{0 \leq s \leq t} \right).  
\end{equation}
Inspired by \cite{MR3606475,9779233}, %given any configuration $\xi \col \qq{1,{  p}} \to \N$, 
we consider observables of the form
\begin{equation}\label{eq_momentflowflowequation}
f_t(\xi) \deq \E \pBB{\prod_{k = i_0}^{{ \sfK}} z_k(t)^{2 \xi(k)} \bigg\vert \cal F_t} \prod_{k = i_0}^{{ \sfK}} \frac{1}{(2\xi(k)-1)!!}\,,
\end{equation}
where we set $i_0=1$ for the non-spiked model and $i_0=r+1$ for the spiked model. Here, $\xi \col \qq{i_0, \sfK} \to \N$ represents a fixed configuration, and we can interpret $\xi(j)$ as the number of particles at site $j$. This interpretation allows us to view the EMF as an interacting particle system later.
%Above, $\xi(j)$ can be regarded as the number of particles at the site $j$, so later we can interpret the EMF as an interacting system. 
The factor $(2 \xi(k)-1)!!$ corresponds precisely to the $\xi(k)$-th moment of a standard Gaussian random variable, and we adopt the convention that $(-1)!!=1$. Clearly, $f_t(\xi)$ is a functional encoding joint moments of the generalized components of the non-outlier eigenvectors.

%Similar quantities have been utilized for generalized Wigner matrices in \cite{MR3606475} (see equations (3.3) and (3.4) therein).  
%{\cob  Here we give an heuristic argument for $f_t$. 
Before delving into the analysis of $f_t(\xi),$ let us provide some heuristics as to why the asymptotic variance takes the form presented in \eqref{eq:XiL}. For simplicity, we consider the single-particle case where $\xi=\delta_{k_0}$ for some $k_0\in \qq{i_0,\sfK}$. In this scenario, we have $f_t(\xi)=\E \left(|z_{k_0}(t)|^2\big \vert \cal F_t\right). $ For $z=E+\ii \eta \in \mathbb{C}_+$, using the spectral decomposition of $W(t)W(t)^\top$, we find that on the local scale of order $\eta,$
\be\label{paidu}
\frac{p \im  \big(\bv^\top \left(W(t) W(t)^\top-z\right)^{-1} \bv\big)}{\im \tr\left(W(t) W(t)^\top -z\right)^{-1}}
\approx \text {local average of }|z_k(t)|^2 \text{ for }{\lambda_k(t)-E=\OO (\eta)} .
\ee
% observe from the spectral decompositions that, 
%\begin{align*}
%&\im \tr\left(W(t)W^*(t)-z\right)^{-1}=\sum_{k=1}^p \frac{\eta}{(\lambda_k(t)-E)^2+\eta^2} , \\
% & p \im  \left(\bv^* \left(W(t) W^*(t)-z\right)^{-1} \bv\right)=\sum_{k=1}^p \frac{\eta |z_k(t)|^2}{(\lambda_k(t)-E)^2+\eta^2}.  
%\end{align*}
%Roughly speaking, 
 %the following should hold 
Drawing inspiration from the results for Wigner matrices \cite{MR3606475}, we expect that the random variables $z_k(t)$ are asymptotically independent, and that the distributions of $z_k(t)$ and $z_{k'}(t)$ are close to each other when $|k-k'|\ll n$. Guided by this intuition, we expect that that $f_t(\xi)$ should closely resemble the averaged quantity in (\ref{paidu}) due to the law of large numbers, as long as we choose $n^{-1}\ll \eta \ll 1$.
%should be close to each individual $k$ if $\lambda_k(t)-E=\OO (\eta).$ 

%Next, We now make the connection between (\ref{paidu}) and the formulas in (\ref{eq_phidefinitionintheend}). 

Next, we show that the left-hand side (LHS) of \eqref{paidu} can be effectively approximated by $\varphi_0$ as defined in \eqref{defvarphi1}. First, using the local laws presented in \Cref{sare} and \Cref{rmk_m1m2law} below, the LHS of (\ref{paidu}) has a deterministic limit, denoted as 
 \begin{align}\label{eq_varphitz}
 \varphi_t(\bv, \bv, z)\equiv \varphi(\bv, \bv, z,D_t):=&\, -\frac{  \im \big(\bv^\top [z(1+m(z,D_{t0})D_t)]^{-1} \bv\big) }{ |z|^{-2} \eta (1-c_n^{-1})+c_n^{-1} \im m(z,D_{t0})}, 
 %\nonumber \\   =&\, \cor \frac{\eta +(\eta \re m(z))\bv^\top \Lambda_t (|z|^2|1+m(z)\Lambda_t|^2)^{-1} \bv +(E\im m(z)) \bv^\top \Lambda_t (|z|^2|1+m(z)\Lambda_t|^2)^{-1} \bv }{\left[|z|^{-2} \eta (c_n^{-1}-1)+c_n^{-1} \im m(z) \right] }
 \end{align}
for $z=E+\ii \eta$, where $D_{t0}:=D_{00}+t$ denotes the non-spiked part of $D_t$. 
%{\cor where we recall $D_t$ defined in (\ref{eq_mathcalwt}) and $m_t(z):=m(z,D_t)$ defined in \eqref{eq_defnmc}.} {\color{red}[need to be made consistently. t=0 we mean the initial condition but not really zero.]} 
By setting $t=0$, $E=x$, and $\eta \downarrow 0$, we obtain the formula in \eqref{defvarphi1} with $\bw=\bv$. (This is also why we use the same notation $\varphi$ in \eqref{defvarphi1} and \eqref{eq_varphitz}.)
%(this is also why we use the same notation $\varphi$ in \eqref{defvarphi1} and \eqref{eq_varphitz}). 
%{\cor Letting $E=0$ and $\eta \downarrow 0$, we obtain the second formula in \eqref{defvarphi1} when $c_n>1$. (how?)}
%case of (\ref{eq_phidefinitionintheend}). 
%When $E=0$, we can use a continuity argument by expanding the denominator of $\varphi_t$ as in (\ref{eq_varphitz}). 
%More specifically,  
%\begin{equation*}
%\lim_{\eta \downarrow 0} \lim_{E \downarrow 0}\frac{-\eta(1+\bv^\top D_t \bv \re m(z))-E \bv^\top D_t \bv \im m(z)}{\eta(1-c_n^{-1})|1+m(z)D_t|^2-c_n^{-1} |z|^2 \im m(z)|1+m(z)D_t|^2}=
%\end{equation*}
Moreover, by the eigenvalue rigidity property of $W(t)W(t)^\top$ stated in Lemma \ref{sare} below, we know that  $\lambda_k(t)$ is well-approximated by $\gamma_k(t)$ with high probability. Therefore, by choosing $E=\gamma_{k_0}(t)$ and $\eta=\oo(1)$, according to the above discussion, we expect that when $t=\oo(1)$,
\begin{equation*}
\E  |z_{k_0}(t)|^2\approx \varphi_t(\bv,\bv,\gamma_{k_0}(t)+\ri \eta )\approx \varphi_t(\bv,\bv,\gamma_{k_0}(t))\approx  \varphi_0(\bv,\bv,\gamma_{k_0}(t)).
\end{equation*}
% Here we recall the convention that
% $\varphi_t^{+}(x)\deq\lim_{\eta \downarrow 0} \varphi_t(x+\ri\eta)$ for $ t \geq 0. $ 
This gives the desired result for the single-particle case. 
For general configuration, using the above argument we heuristically expect that  
\begin{equation}\label{eq_approximatecool}
f_t(\xi)\approx \prod_{k=i_0}^\sfK  \varphi_0(\bv,\bv,\gamma_k(t))^{\xi(k)}.
\end{equation} 
% $\eta=n^\delta\left(\gamma_{k_0}(t)- \gamma_{{k_0}\pm 1} (t) \right)$ for $\delta>0$ sufficiently small, we expect from the anisotropic local law Lemma \ref{sare} and the above discussion that when $t=\oo(1)$ the following approximation holds,
%$$
%\E |z_k(t)|^2\approx \frac{M \im  \left(\bv, \left(WW^*(t)-z\right)^{-1} \bv\right)}{\im \tr\left(WW^*(t)-z\right)^{-1}}\approx \varphi_t(z)
%$$
%where (as suggested by Lemma \ref{sare})  
%
%
%
%
%
%Then choose $E=\gamma_k(t)$ and $\eta=N^\delta\left(\gamma_k(t)- \gamma_{k\pm 1} (t) \right)$, Roughly speaking we have 
In general, establishing the asymptotic distribution of $\{|z_{k}(t)|^2\}_{k=i_0}^{\sfK}$ can be accomplished by demonstrating that its finite-dimensional moments asymptotically match those of a random vector composed of independent Chi-squared random variables with the desired variances.
%is an asymptotically Gaussian vector, it suffices to show that its finite-dimensional moments match those of a multivariate normal random vector with the desired covariance. 
To make the above heuristics rigorous, we will work with the following smoothed version of $\varphi_0(\cdot)$ in the actual proof. 
% {\cob (With the explicit definition of $g_t(\xi)$ below, we can prove \eqref{eq:XiL} directly now.) In other words, we will extend the above argument to show that for a general particle configuration, 
% %For general configuration, using the above argument we claim that  
% \begin{equation}\label{eq_approximatecool}
% f_t(\xi)\approx \prod_{k=1}^p  \varphi_0(\gamma_k(0))^{\xi(k)}.
% \end{equation}
% We remark that in the actual proof, for technical reasons, we will work with a smoothed version of the RHS of (\ref{eq_approximatecool}), denoted as  $g_t(\xi)$, which is obtained as a convolution of $\varphi_0$ with a smooth bump function.}
%where $\varphi_0$ will be convolved with a certain smooth function and evaluated at $\lambda_k(t)$. 

\begin{definition}\label{defn_gt(xi)_single}
Given the deterministic unit vector $\bv \in \mathbb{R}^p$, let $\varphi_0$ be defined as in \eqref{defvarphi1}. Let $s(x)$ be a smooth, non-negative function supported in $[-1,1]$ such that $\int_\mathbb{R} s(x)dx=1$. For any constant $\epsilon >0$, we define $s_{\epsilon}(x)\deq n^{\epsilon}s(n^{\epsilon} x)$ and the function $\phi(x, \epsilon)$ as the convolution
    \begin{equation}\label{eq_defnphiepsilon_single}
\phi(x, \epsilon):=s_{\epsilon} *\varphi_0(\bv,\bv,x). 
    \end{equation}
    Then, for any fixed configuration $\xi:\qq{i_0,p}\mapsto \mathbb{N}$, we define % and small constant $\epsilon_g>0$, we define 
    \begin{equation}\label{eq_onecolorgt}
	g_t(\xi,\epsilon)\deq \prod_{k=i_0}^\sfK \phi(\lambda_k(t),\epsilon)^{\xi(k)}. %\quad \text{with}\quad g_k(t,\e)\deq \phi(\lambda_k(t),\epsilon).
    \end{equation}
\end{definition}

\begin{remark}
We remark that the function $g_t$ is defined in terms of $\varphi_0$ and $\lambda_k(t)$ rather than $\varphi_t$ and $\gamma_k(t)$. The usage of $\lambda_k(t)$ instead of $\gamma_k(t)$ is due to certain technical considerations, which will become apparent in the technical derivation (\ref{eq_whygtrandom}) below.
%We have used $\lambda_k(t)$ for some technical reason as will be evident in  (\ref{eq_whygtrandom}) below. 
Using the definition of $\phi$ and the estimate \eqref{eq:denominator} in the appendix, it can be readily observed that for any small constant $\tau>0$,
 \be\label{xiaoy}
 \left|\partial_x^s\phi^{ij}(x, \e) \right|=\OO(n^{s\e}), \quad \forall \ \ s\in \{0,1,2\}, \quad x\in [\tau, \tau^{-1}].
\ee
%Here the key input is for $E: \rd(E, \supp \rho)\le  c_0$, we have $\min_k  |1+m(E^+)\sigma_k|=\Theta(1)$.
\end{remark}

% {\cob The formal definition of $g_t(\xi)$ will be presented in  (\ref{eq_onecolorgt}) once the necessary notations have been introduced. } 
The following EMF result shows that $f_t(\xi)$ converges to the equilibrium $g_t(\xi,\e)$ at a time scale of  $n^{-1/3+\mathsf c}$ as long as we choose $\e$ sufficiently small. 
%provides suitable control of the difference between $f_t(\xi)$ and $g_t(\xi)$. 

% provide suitable control of the following difference 
%\begin{equation}\label{eq_gtxifirstintroduce}
%\sum_{\xi} \left( f_t(\xi)-g_t(\xi) \right),
%\end{equation}
%where $g_t(\xi)$ will be formally defined in (\ref{eq_onecolorgt}) after some necessary notations are introduced. Roughly speaking,  $g_t(\xi)$ is a smoothed version of the right-hand side of (\ref{eq_approximatecool}) and has the form of  $\prod_{k} (h_\e*\varphi_0^+)(\lambda_k (t))^{\xi(k)}, $
%where $h_\e$ is a smooth function and with random argument $\lambda_k(t)$. In particular, Step one aims to prove the following result. 

%\be\label{goodbulkxi}
%S^{Bulk}_{n,\; \e_B} := \left\{\xi  \col \qq{1,{  M}} \to \N: \quad \max_{k\notin [N^{1-\e_B}, K-N^{1-\e_B}]}\xi (k)=0, \quad |\xi|=n \right\}, \quad |\xi|: =\sum _k \xi_k
%\ee

%\begin{theorem}[{\cor Bulk eigenvectors of $W(t )$ where $t\sim N^{-1+\e_T }$}]\label{main result: Yi}With some parameter $\e_B$, w e define 
% 
%
%
%Then for small enough $\e_T$ (like $1/10$) and $\e_B\le \e_T/C$ with large number $C$,  there exists some (small) positive $\e_0$ and $\rho$ (depends on $\e_B$ and $\e_T$, and $n$) such that for fixed $C_1<C_2$, letting $T_i=C_i N^{-1+\e_T}$, $i=1,2$,   we have 
% \be\label{moni}
% \P\left( \sup_{ T_1\le t\le T_2} \max_{\xi \in S^{Bulk}_{n,\; \e_B}} \left |f_t(\xi)-g_t(\xi)\right| \le N^{-\e_0}  \right)\ge 1-N^{-\rho} 
%  \ee
%      \end{theorem}   

\begin{theorem}\label{main result: Jia} 
Under the setting of Lemma \ref{lem: moment flow}, choose $T_i=C_i n^{-1/3+\mathsf{c}}$ for some constants $0<C_1<C_2$ and $\mathsf{c} \in (0,1/6)$.
%For some constant $0<\mathsf{c} \leq 1/6$ and $i=1,2$, denote 
Then, for any fixed $L\in \N$ and $\nu>0$, there exist (small) positive constants $\e_0,\e_1>0$ such that 
 $$
 \P\left(\sup_{ T_1\le t\le T_2} \max_{|\xi|=L} \left |f_t(\xi)-g_t(\xi,\e_1)\right| \ge n^{-\e_0}\right)\le n^{-\nu},  
  $$
  where $|\xi|:=\sum_{j=i_0}^{\sfK} \xi(j)$ denotes the total number of particles. 
\end{theorem}   
%  \begin{proof}
%  See Appendix \ref{sec_proof44}. 
% \end{proof}   
% {\cob Define the generalized version of $g_t$ in Appendix XXX below, and comment on the multi-colored case.}
\begin{lemma}\label{thm_localdensityestimate0}
Under the setting of Lemma \ref{lem: moment flow}, take $n^{-1/3+\e}\le t \le n^{-1/6-\e}$ for an arbitrary constant $\e\in (0,1/12)$. Then, for any fixed $m\in\mathbb{N}$, there exist (small) positive constants $\e_0,\e_1,\nu>0$ such that 
\begin{align*}
    \mathbb{P}\left(\max_{|\xi|=m}\Big|g_t(\xi,\e_1) - \prod_{k=i_0}^\sfK  \varphi_0(\bv,\bv,\gamma_k(t))^{\xi(k)}\Big|\ge n^{-\e_0}\right) \le n^{-\nu}.
\end{align*}
% A lemma showing that $g_t(\xi)$ match the RHS of \eqref{eq_approximatecool}.
\end{lemma}
\begin{proof}
%By definition, the function $g_t(\xi)$ is a smoothed version of the product $\prod_{k=i_0}^p\varphi_0(\bv,\bv,\lambda_k(t))^{\xi(k)}$. The proof then follows from the standard fact about convolutions with approximate identities and the rigidity of $\lambda_k(t)$, as stated in Lemma \ref{sare} below (see in particular \eqref{TAZ2}).
The function $g_t(\xi,\e_1)$ represents a smoothed version of the product $\prod_{k=i_0}^\sfK\varphi_0(\bv,\bv,\lambda_k(t))^{\xi(k)}$.
The proof then follows from the standard fact concerning convolutions with approximate identities and the rigidity of $\lambda_k(t)$, as stated in Lemma \ref{sare} below (specifically, refer to \eqref{TAZ2}).
\end{proof}

We are now ready to prove Lemma \ref{lem: moment flow} by combining the results in \Cref{main result: Jia} and Lemma \ref{thm_localdensityestimate0}. 
\begin{proof}[\bf Proof of Lemma \ref{lem: moment flow}] 
%Denote temporarily by $\mathbf{V}$ the random vector on the LHS of \eqref{eq_lemma41equation}. 
Let $\xi$ be a configuration supported on $\{i_k+r\}_{k=1}^L\subset\qq{r+1,\sfK}$. % such that $\xi(i_k+r)\in\mathbb{Z}_+$ for $k \in\qq{L}$. 
Note that the functional $f_t(\xi)$ in \eqref{eq_momentflowflowequation} encodes the joint moments of the random variables $ p|\avg{\bv, \bu_{i_1+r}(t)}|^2,\ldots,  p|\avg{\bv, \bu_{i_L+r}(t)}|^2$. By employing Theorem \ref{main result: Jia} and Lemma \ref{thm_localdensityestimate0}, we establish that these joint moments match the corresponding finite-dimensional joint moments of $|\cal N_1|^2,\ldots,|\cal N_L|^2,$ where $\cal N_1,\ldots, \cal N_L$ are i.i.d.~centered Gaussian random variables having the desired variances $\varphi_0(\bv, \bv,\gamma_{i_1}(t)),\ldots, \varphi_0(\bv, \bv,\gamma_{i_L}(t))$. This concludes the proof using the standard method of moments. 
%conclude that $\mathbf{V}$ is asymptotically Gaussian since $f_t(\xi)$ matches with the corresponding finite-dimensional joint moments of the entries of a centered Gaussian vector with covariance matrix given by $\Xi_L(t)$. This completes the proof of the statement.{\color{red} Moreover,  using the definition of $g_t(\xi)$ in (\ref{eq_onecolorgt}), combining with the arguments around (\ref{eq_varphitz}) and the rigidity result in Lemma \ref{sare}, one can prove (\ref{eq_lemma41equation}). 
%when replacing $\mathrm{V}_d(0)$ with $\mathrm{V}_d(t)$ for $t \asymp n^{-1/3+\mathsf{c}}.$ Finally, Lemma \ref{thm_localdensityestimate} allows us to replace $\mathrm{V}_d(t)$ with $\mathrm{V}_d(0)$ on the small time scale.} 
% This completes the proof of the statement.} 
%{\color{blue} Finish the proof here: one only need to estimate the moments of $\left(\left\langle \bv, \bu_k(t)\right\rangle,  \left\langle \bv, \bu_l(t)\right\rangle\right)$. They can be easily calculated with Lemma \ref{main result: Yi} and \ref{main result: Jia}, and the delocalization results. }
\end{proof}

\begin{remark}
In \Cref{main result: Jia}, we have chosen $t \asymp n^{-1/3+\mathsf{c}}$, which yields an almost sharp speed of convergence to local equilibrium for the edge eigenvalues and eigenvectors (see e.g., \cite{MR3606475,landon2017edge,9779233}). However, this rate is suboptimal for the convergence of bulk eigenvalues and eigenvectors. Indeed, it has been demonstrated in  \cite{MR3606475,benigni2020eigenvectors} that the bulk eigenvectors of Wigner matrices converge to local equilibrium once $t \gg n^{-1}$, and we anticipate that similar results extend to sample covariance matrices as well. Since this aspect is not the primary focus of the current paper, we will refrain from delving into this direction here.

%hence the rate in Lemma \ref{thm_localdensityestimate} can be updated to $\mathrm{o}(n^{-1})$. 

%in order to unify the results across all eigenvectors. While this choice is sufficient for our purpose and almost optimal for the edge eigenvectors, it is not optimal for the bulk eigenvectors. It is not hard to see from \cite{MR3606475} that the bulk eigenvectors will have a faster convergence once $t \gg n^{-1}$, and hence the rate in Lemma \ref{thm_localdensityestimate} can be updated to $\mathrm{o}(n^{-1})$. See analogous results for Wigner matrices in Lemma 4.1 of \cite{landon2017convergence}. As this is not the focus of the current paper, we will explore this direction in  future work.   
\end{remark}

The proof of \Cref{main result: Jia} can be found in Appendix \ref{appendix_maximumestimation}. The main ingredient is a probabilistic description of the EMF of $f_t(\xi)$ as a multi-particle random walk in a random environment. The random environment is described by the well-studied Dyson Brownian motion of the eigenvalues. %Let $Q(t)$ be defined as in \eqref{eq_qt}. 
Let $\mathsf B(t)=(\mathsf B_{ij}(t))$ be a $p \times n$ matrix where $\mathsf B_{ij}$, $i\in \qq{p}$ and $j\in \qq{n}$, are independent standard Brownian motions. For any $t\ge 0$, $W(t)$ in \eqref{defn_wt} can be rewritten as 
\be\label{eq:WtDBM}
W(t)=W(0)+n^{-1/2}\mathsf B_t. 
\ee
We refer to this type of dynamics as a rectangular DBM. Under \eqref{eq:WtDBM}, \Cref{lem:Dyson_ev} provides the stochastic differential equations (SDEs) that describe the evolution of eigenvalues and eigenvectors of $Q(t)$.  
Note that $\lambda_{k}(t) \equiv 0$ when $k\in \qq{\mathsf{K}+1, p}$. To account for this, we introduce the following equivalence relation. 

\begin{definition}[Equivalence relation]\label{defn_equivalentrelation}
%	Recall $\mathsf K=\min\{p,n\}.$ 
 We define the following equivalence relation on $\qq{1,p}$: $k \sim l$ if and only if $k = l$ or $k,l \geq \mathsf K + 1$. In particular, when $p \leq n$, $k \sim l$ simply means $k = l$. 
\end{definition}

%a simple consequence of Lemma \ref{lem:dyson} below.
\begin{lemma}\label{lem:Dyson_ev}
 %Set $Q(0)=WW^\top$ and let $Q(t)$ be defined as in \eqref{eq_qt} using $B(t)$. Then 
The eigenvalues $\{\lambda_k(t)\}_{k=1}^p$ and eigenvectors $\{\bu_k(t)\}_{k=1}^p$ of $Q(t)$ satisfy the following SDEs: 
\begin{equation}\label{sLa}
\dd \lambda_k(t) \;=\; 2 \sqrt{\lambda_k (t)} \frac{\dd B_{kk}(t)}{\sqrt{n}}  + \bigg( 1+ \frac{1}{n} \sum_{l:l\not\sim  k } \frac{\lambda_k(t) + \lambda_l(t)}{\lambda_k(t) - \lambda_l(t)}  \bigg) \dd t,  
\end{equation}
\begin{equation} \label{dyson_evector}
\dd \b u_k(t) \;=\; \frac{1}{\sqrt{n}}\sum_{l:l \not\sim k}  \frac{\sqrt{\lambda_k(t)} \ \dd B_{lk}(t) + \sqrt{\lambda_l(t)} \, \dd B_{kl}(t)}{\lambda_k(t) - \lambda_l(t)} \, \b u_l(t)  - \frac{1}{2n} \sum_{l:l \not\sim k} \frac{\lambda_k(t) + \lambda_l(t)}{(\lambda_k(t) - \lambda_l(t))^2} \, \b u_k(t) \, \dd t\,,
\end{equation}
where $k \in \qq{p}$ and $B_{ij}$, $i,j\in \qq{p}$, are independent standard Brownian motions. %(Note that $B$ has the same law as $\mathsf B$, but they may be different matrix Brownian motions.
(Note that $B$ has the same distribution as $\mathsf B$, although they generally correspond to different matrix Brownian motions.) We will refer to \eqref{sLa} and \eqref{dyson_evector} as DBM and eigenvector flow, respectively. Their initial conditions at $t=0$ are given by \smash{$\{\wt\lambda_k\}_{k=1}^p$} and $\{\bu_k\}_{k=1}^p$.
\end{lemma}
 \begin{proof}
See \cite{bru1989diffusions} for the derivations. Here, we adopt the formulation presented in \cite[Appendix C]{MR3606475}.
%See Lemma \ref{lem:dyson}.
 \end{proof}

%Recall (\ref{eq_Kdefinition}), we have that $\lambda_{k}(t) \equiv 0,  \ \mathsf{K}+1 \leq k \leq p$ when $p>n$. To accommodate this, we introduce the equivalence relation as follows. 
%\begin{definition}[Equivalent relation]\label{defn_equivalentrelation}
%Recall $\mathsf K=\min\{p,n\}.$ We define the equivalence relation on $\qq{1,p}$ by setting $k \sim l$ if and only if $k = l$ or $k,l \geq \mathsf K + 1$. In particular, if $p \leq n$ then $k \sim l$ means $k = l$. 
%\end{definition}

We extend the definition \eqref{eq_momentflowflowequation} to 
\begin{equation}\label{eq_momentflowflowequation2}
f_t(\xi) \deq \E \pBB{\prod_{k = 1}^{p} z_k(t)^{2 \xi(k)} \bigg\vert \cal F_t} \prod_{k = 1}^{p} \frac{1}{(2\xi(k)-1)!!}\,
\end{equation}
for the particle configuration $\xi :\qq{p}\to \N$. For $\xi = (\xi(1),\ldots, \xi(p))$, we introduce the notation
\begin{equation}\label{eq_simplifiedcomeandgo}
\xi^{k\to l} \;\deq\; \xi + \ind{\xi(k) \geq 1} (\b e_l - \b e_k),
\end{equation}
which denotes the particle configuration obtained by moving one particle from site $k$ to site $l$, given that $\xi(k)\ge 1$. The following lemma defines the main object of study---the eigenvector moment flow. Its proof will be presented in Appendix \ref{appendix_flowequation}. 
\begin{lemma}\label{lem: dyn}
Suppose the SDEs \eqref{sLa} and \eqref{dyson_evector} in Lemma \ref{lem:Dyson_ev} hold. Let $f_t(\xi)$ be defined as in \eqref{eq_momentflowflowequation2}. Define the generator $\mathscr  B (t)$ as
\begin{equation}\label{eq_generatorb}
\mathscr  B (t) f(\xi) \;\deq\; \frac12\sum_{l \not\sim k} \Upsilon_{kl}(t) \, 2 \xi(k) (1 + 2 \xi(l)) \pb{f(\xi^{k \to l}) - f(\xi)},
\end{equation}
where the matrix $\Upsilon$ is defined as 
\begin{equation}\label{eq_Upsilonkldefinition}
\Upsilon_{kl}(t) \;\deq\; \ind{k \not\sim l} \frac{\lambda_k(t) + \lambda_l(t)}{2n(\lambda_k(t) - \lambda_l(t))^2}\,.
\end{equation}
Then, $f_t(\xi)$ satisfies the equation
\begin{equation} \label{moment_flow}
\partial_t f_t(\xi) \;=\;  \mathscr B (t) f_t(\xi)\,.
\end{equation}
%The equation \eqref{moment_flow} may be regarded as a system of coupled equations for the quantities $f_t(\xi)$. {\cob Note: at $t=0$ it is free to choose the eigenvectors in the subspace for $\lambda=0$, this formula holds for all different choices.  }
\end{lemma}
% \begin{proof}
% See Appendix \ref{appendix_flowequation}.
% \end{proof}

%The proof of Lemmas \ref{lem:Dyson_ev} and \ref{lem: dyn}, along with a brief discussion of the Dyson eigenvector flow will be presented in Appendix \ref{appendix_flowequation}. The dynamics described by these equations play a crucial role in proving Theorem \ref{main result: Jia}.  For example, (\ref{sLa}) and (\ref{moment_flow}) will be used to derive the SDEs for certain quantities related to $g_t(\xi)$ and $f_t(\xi)$. See Appendix \ref{appendix_maximumestimation} for more details.  

The proof of Theorem \ref{main result: Jia} is based on a careful analysis of the equation \eqref{moment_flow}. For more details, readers can refer to Appendix \ref{appendix_maximumestimation}. 

%see (\ref{eq_sde1})--(\ref{eq_sde3}). As another example, (\ref{moment_flow}) will be used to calculate the dynamics of quantities involving $f_t(\xi).$  
%These equations are the key gradients to the proof of Theorem \ref{main result: Jia} via It\^o's lemma. For example, (\ref{sLa}) will be used to obtain the SDEs for some further quantities related to $g_t(\xi)$; see (\ref{eq_sde1})--(\ref{eq_sde3}). For another example, (\ref{moment_flow}) will be used to calculate the dynamics of quantities involving $f_t(\xi).$  

\subsection{Green's function comparison and proof of Lemma  \ref{lemma_universal}}\label{subsec_functionalrepresentation} 

%In this section, we extend the result of Lemma \ref{lem: moment flow} to the non-Gaussian case by removing the Gaussian component through a new comparison argument. 

As discussed below \Cref{lemma_universal}, in order to complete the proof, it suffices to prove the comparison estimate \eqref{eq_parttwocompare}. Inspired by the works \cite{bloemendal2016principal, ding2019singular, knowles2013eigenvector}, we first establish a functional integral representation formula for the generalized components of the non-outlier eigenvectors, expressed in terms of the resolvents of either the non-spiked or spiked model. 
%Inspired by the arguments of \cite{bloemendal2016principal, ding2019singular, knowles2013eigenvector}, for the eigenvectors of the non-spiked model or the non-outlier eigenvectors of the spiked model, we write the generalized components as a functional integral in terms of the resolvents.
\iffalse
Recall $i_0$ in (\ref{eq_momentflowflowequation}).  For simplicity, we focus on the case when there is only one connected component in the support of the deformed MP law $\varrho$ (i.e.~$q=1$ in Lemma \ref{lem_property}). Without loss of generality, we assume the indices $i_k, i_0 \leq k \leq m$ with $i_0\le i_k\le \mathsf{K}$ are closer to the upper edge $\lambda_+$  in the following statements. The general case will be discussed in Remark \ref{remark_loweredgecase} below. 
\fi

We first introduce some new notations. For any $E_2\ge E_1 >0,$ let $ \mathfrak{f}_{E_1, E_2, \eta}(x)$ denote an indicator function of the interval $[E_1, E_2]$ smoothed on the scale $\eta.$ More precisely,  $\mathfrak{f}_{E_1, E_2, \eta}$ is a smooth function satisfying the following properties: (1) $\mathfrak{f}_{E_1, E_2, \eta}(x) \equiv 1$ for $x \in [E_1, E_2]$, (2) $\mathfrak{f}(x) \equiv 0$ for $x \notin  [E_1-\eta, E_2+\eta]$, and (3) %for some constant $C>0$, 
\begin{equation}\label{eq_generalfdefinition}
\sup_{x\in \R} |\mathfrak{f}_{E_1, E_2, \eta}'(x)| \lesssim \eta^{-1}  , \qquad  \sup_{x\in \R} |\mathfrak{f}_{E_1, E_2, \eta}''(x)| \lesssim \eta^{-2} . 
\end{equation}
For any positive integer $k\in \N$, let $\mathfrak{q}_k: \mathbb{R} \rightarrow \mathbb{R}^{+}$ be a smooth cutoff function such that 
\begin{equation}\label{eq_smoothqfunction}
\mathfrak{q}_k(x) \equiv 1, \ \ \text{if} \ |x-k| \leq \frac{1}{3}; \qquad \mathfrak{q}_k(x) \equiv 0, \ \ \text{if} \ |x-k| \geq \frac{2}{3}; \qquad \sup_{x\in \R}|\mathfrak q_k'(x)|\lesssim 1. 
\end{equation}
%In addition, for notational simplicity, we set 
%\begin{equation}\label{eq_mathsfKdefinition}
%\Delta_k:=n^{-2/3} k^{-1/3}. 
%\end{equation}
For any $k\in \qq{r+1,\sfK}$, if $\gamma_k \in [a_{2j},a_{2j-1}]$, then we define 
\begin{equation}\label{eq_mathsfKdefinition}
	\Delta_k:=[(n_{j}+1-k)\wedge (k-n_{j-1})]^{-1/3}n^{-{2}/{3}}, 
\end{equation}
where recall that $n_j$ is defined in \eqref{Nk} and the choice of $\Delta_k$ is due to the rigidity of eigenvalues in \eqref{TAZ2} below. Now, for any constants $\epsilon, \delta_1,\delta_2>0$ and $E\in \R$, we define
\begin{equation}\label{eq_keyquantitiesdefinition}
E^-_k(\e,\delta_1):=E- n^{-2\epsilon+\delta_1}\Delta_k,\quad E^+(\e)\deq \lambda_++2n^{-2/3+\epsilon}, \quad \eta_k(\e):=\Delta_k n^{-2\epsilon}, \quad \widetilde{\eta}_k(\e):=\Delta_k n^{-3\epsilon},
\end{equation} 
%Let $\lambda_+ \equiv a_1.$ 
 % (\ref{eq_generalfdefinition}) and $E^U\deq \lambda_++2n^{-2/3+\epsilon}$ (see (\ref{eq_defneuk}) below), we abbreviate 
%For any constant $\delta_2>0$, define
and the intervals (recall that $\gamma_k(t)$ are defined above \eqref{defvarphi1})
 \begin{equation}\label{eq:Ik}
 	I_k\equiv I_k(t,\delta_2)\deq[\gamma_k(t)-n^{\delta_2}\Delta_k, \gamma_k(t)+n^{\delta_2}\Delta_k].
 	%\quad \text{for $i_0\le k\le \mathsf{K}$,}\qquad I_k=[-n^{-L},n^{-L}] \quad \text{for $k\ge \mathsf{K}$.}
 \end{equation}
Corresponding to the function $\mathfrak{f}_{E_1, E_2,\eta}$ defined above, we abbreviate 
\begin{equation}\label{eq_sequecenoffunction}
\mathfrak{f}_{k}(x)\equiv \mathfrak{f}_{k}(x,\e,\delta_1):=\mathfrak{f}_{E_{k}^-(\e,\delta_1),E^+(\e), \widetilde{\eta}_{k}(\e)}(x). %\quad k\in \qq{r+1,\sfK}.
\end{equation}
%Recall $Q(t)$ in (\ref{eq_qt}). 
Finally, for any $E \in \mathbb{R}$, we define 
\begin{equation}\label{eq_xedefinition}
x_{k}^1(E,t):=\bv^\top (Q(t)-E-\ri \eta_{k})^{-1} \bv,\quad  x_{k}^0(E,t):=\bv^\top (\cal Q_t-E-\ri \eta_{k})^{-1} \bv, 
\end{equation} 
where recall that $Q(t)$ and $\cal Q_t$ were defined in \eqref{eq_qt} and \eqref{eq_mathcalwt}, respectively. 

Now, we are ready to state the key functional representation formula for the generalized components of non-outlier eigenvectors. The proof of this formula will be presented in Appendix \ref{sec_proofof72}.

 \begin{lemma}\label{7.2}
Under the setting of Lemma \ref{lemma_universal}, for any $t\asymp n^{-1/3+\mathsf c}$, there exists a small constant $\e>0$ such that the following estimate holds for sufficiently small constants $\delta_1\equiv \delta_1(\e)$, $\delta_2\equiv \delta_2(\e,\delta_1)$, and $\nu>0$:    
 \begin{align*}
	\E \theta \left(p |\avg{\bv, \bu_{i_1+r}(t)}|^2, \ldots, p |\avg{\bv, \bu_{i_L+r}(t)}|^2 \right)= \E \theta \Big( & \frac{p}{\pi}\int_{I_{i_1}} \left[\im x^1_{i_1}(E,t) \right]
	\mathfrak q_{i_1}\left[\tr \mathfrak f_{i_1}(Q(t)) \right]\dd E,  \ldots, \\
	&  \frac{p}{\pi}\int_{I_{i_L}} \left[\im x^1_{i_L}(E,t) \right]
	\mathfrak q_{i_L}\left[\tr \mathfrak f_{i_L}(Q(t)) \right]\dd E\Big)+\OO(n^{-\nu}).
\end{align*}
A similar estimate holds for $\{p|\avg{\bv, \bw_{i_k+r}(t)}|^2\}_{k=1}^L$, $\{x^0_{i_k}(E,t)\}_{k=1}^L$, and $\mathcal{Q}_t$. % defined via (\ref{eq_mathcalwt}). 
%Suppose the assumptions of Lemma \ref{lemma_universal} hold. 
%For $\epsilon, \delta_1>0,$ denote the sequence of functions $\mathfrak{f}_{i_k}(x),  \ 1 \leq k \leq m$ as in (\ref{eq_sequecenoffunction}). %{\cor Recall $i_0$ in (\ref{eq_momentflowflowequation}).} 
% $I_k\deq[\gamma_k-n^{\delta_2}\Delta_k, \gamma_k+n^{\delta_2}\Delta_k]$ for $i_0\le k\le \mathsf{K}$ and $I_k=[-n^{-L},n^{-L}]$ for $k\ge \mathsf{K}$. 
%Then for any smooth function $\theta$ we have that 
%Moreover, for some small constant $\delta_2>0$ and the sequences of intervals $I_{i_k}\equiv I_{i_k}(\delta_2)$ defined in Definition \ref{defn_definitionofintervals}, and 
\end{lemma}
%\begin{proof}
%See Appendix \ref{sec_proofof72}. 
%\end{proof}

%Lemma \ref{7.2} represents the generalized components as a functional integral of the resolvents of the associated covariance matrices. According to the discussions below Lemma \ref{lemma_universal},  
%it suffices to prove that the functionals of the resolvents of $Q(t)$ and $\mathcal{Q}_t$ are sufficiently close. 

With \Cref{7.2}, proving \eqref{eq_parttwocompare} boils down to showing that the two representation formulas involving the resolvents of $Q(t)$ and $\mathcal{Q}_t$ are asymptotically equal to each other.  As discussed below \eqref{eq_parttwocompare}, this is achieved through a novel comparison argument based on the interpolation \eqref{eq_dtusx}. Correspondingly, we define the family of interpolating sample covariance matrices as 
\begin{equation}\label{eq_dtusx0}
	\mathsf{Q}_t^s= W_t^s (W_t^s)^\top,
 \quad \text{with}\quad \mathsf{Q}^1_t\stackrel{d}{=}Q(t),\quad \mathsf{Q}^0_t=\mathcal{Q}_t.
\end{equation}  
Here, ``$\stackrel{d}{=}$" means ``equal in distribution".
Let $x^s_{k}(E,t)$ be defined as in (\ref{eq_xedefinition}) using $\mathsf{Q}_t^s.$ The following lemma provides the Green's function comparison result, which, together with \Cref{7.2}, concludes the proof of \eqref{eq_parttwocompare}.  

 \begin{lemma}\label{nanzi} 
 Under the setting of Lemma \ref{lemma_universal}, define $W_t^s$ as in \eqref{eq_dtusx} for some semi-orthogonal matrices ${U}^0, U^1 \in \mathbb{R}^{p \times 2p}$. %Suppose $U^0$ and $U^1$ satisfy  $U^0(U^0)^{\top}=U^1(U^1)^{\top}=I_{p},$, 
Suppose there exist constants $C,\mathsf{c}'>0$ such that
\be\label{tianya}
\big[(U^0-U^1)^\top (U^0-U^1)\big]^{1/2}_{jj}\le Cn^{-\mathsf c'},\quad j\in \qq{2p}. 
\ee 
%\\be\label{tianya}
%\left(n^{-2}+|\E x^3_{jj} |\right)\big[(U^0-U^1)^\top (U^0-U^1)\big]^{1/2}_{jj}\le n^{-2-\mathsf c'}. 
%\ee 
Then, there exists a constant $\nu>0$ such that
\begin{align}\label{eq_comparisonkeykeykeykey}
	&\E \theta \left( \mathcal U_{i_1}(\mathsf Q_t^0),  \ldots, \mathcal U_{i_L}(\mathsf Q_t^0)\right) =\E \theta \left( \mathcal U_{i_1}(\mathsf Q_t^1),  \ldots, \mathcal U_{i_L}(\mathsf Q_t^1)\right)+\OO(n^{-\nu}),
\end{align}
%\begin{align}\label{eq_comparisonkeykeykeykey}
% &\E \theta \Big(  \frac{p}{\pi}\int_{I_{i_1+r}} \left[\im x^0_{i_1+r}(E) \right]
% \mathfrak q_{i_1+r}\left(\tr \mathfrak f_{i_1+r}(\mathsf{Q}_t^0) \right)\dd E,  \ldots, \frac{p}{\pi}\int_{I_{i_L+r}} \left[\im x^0_{i_L+r}(E) \right]
% \mathfrak q_{i_L+r}\left(\tr \mathfrak f_{i_L+r}(\mathsf{Q}_t^0) \right)\dd E\Big) \\
% &=  \E  \theta \Big( \frac{p}{\pi}\int_{I_{i_1}} \left[\im x^1_{i_1}(E) \right]
% \mathfrak q_{i_1}\left(\tr \mathfrak f_{i_1}(\mathsf{Q}_t^1) \right)\dd E,  \cdots, \frac{p}{\pi}\int_{I_{i_m}} \left[\im x^1_{i_m}(E) \right]
% \mathfrak q_{i_m}\left(\tr \mathfrak f_{i_m}(\mathsf{Q}_t^1) \right)\dd E\Big)+\oo(1). \nonumber
%\end{align}
where we used the simplified notation
$$\mathcal U_k(\mathsf Q_t^s):= \frac{p}{\pi}\int_{I_{k}} \left[\im x^s_{k}(E,t) \right]\mathfrak q_{k}\left(\tr \mathfrak f_{k}(\mathsf{Q}_t^s) \right)\dd E,\quad k\in \qq{r+1,\sfK}.$$

%With $\e, \delta$, $\delta_1$ and $\delta_2$ and $d$ defined for \eqref{where},  
%we have 
%\be\label{magic}
% \E  h\left(\frac{M}{\pi}\int_{I_k}\im x^0(E) q(y^0(E)) \right)- \E  h\left(\frac{M}{\pi}\int_{I_k}\im x^1(E) q(y^1(E)) \right)=o(1)
%\ee
%where $x^i$, $y^i$ are defined as above with $G^i$, $i=0,1$.   
%
%Furthermore,
%combining with lemma \ref{7.2},  for the $k$ th eigenvector:
%\be\label{jiansuan}
%\E h |(\bu_k(W^0), \bv )|^2-\E h |(\bu_k(W^1), \bv )|^2=o(1)
%\ee
  \end{lemma}  
%  \begin{proof}
%  See Appendix \ref{sec_prooflemma410}.
% \end{proof}  

To establish \eqref{eq_comparisonkeykeykeykey}, we will control the derivative of $\E \theta \left( \mathcal U_{i_1+r}(\mathsf Q_t^s),  \ldots, \mathcal U_{i_L+r}(\mathsf Q_t^s)\right)$ with respect to $s$. More details can be found in Appendix \ref{sec_prooflemma410}. 
Now, we are ready to conclude the proof of Lemma \ref{lemma_universal} by combining the above two lemmas. 

\begin{proof}[\bf Proof of Lemma \ref{lemma_universal}]
% By Assumption \ref{main_assumption} and with 
When $t \asymp n^{-1/3+\mathsf{c}}$, it is straightforward to check that the condition (\ref{tianya}) holds with $c'=1/3-\mathsf c$ for $U^0$ and $U^1$ defined  below \eqref{eq_dtusx}. Then, by applying \Cref{7.2} and \Cref{nanzi}, we obtain \eqref{eq_parttwocompare}, which further concludes the proof of \Cref{lemma_universal} by setting $t=\ft$ and $D_0=\Lambda-\ft$. 
% Therefore, we can apply Lemma \ref{nanzi} to obtain  (\ref{eq_parttwocompare}). Combined with Lemma \ref{lem: moment flow}, we conclude that $(\sqrt{p} \bv^\top \bu_{i_k}^t)_{1 \leq k \leq m}$ is asymptotically a Gaussian random vector. The proof of the statement is then completed by replacing $D$ with $D-t$ in (\ref{eq_mathcalwt}) and using $t \asymp n^{-1/3+\mathsf{c}}$ again.
\end{proof}

\bibliographystyle{abbrv}
\bibliography{shrinkage} 

\appendix 

\section{Summary of loss functions and shrinkers}\label{appendix_lossfunctionsummary}
According to \cite{donoho2018, mainreference}, some special loss functions have been used frequently in the literature owing to their significance in applications. For the convenience of readers, we list them in Table \ref{table_summarylossfunction}. Hereafter, we consistently denote the loss functions by $\mathcal{L}(\cdot, \cdot): \mathbb{R}^{p \times p} \times \mathbb{R}^{p \times p} \rightarrow [0, \infty)$. 
%\begin{definition}\label{defn_lossfunctions} For the population covariance matrix $\Sigma$ and its estimator $\widehat{\Sigma}$, we consider the following loss functions presented in Table \ref{table_summarylossfunction}. 

\begin{table}[!ht]
\def\arraystretch{2}
\begin{center}
\hspace*{-0.8cm}
\begin{tabular}{|c|c|c|c|}
\hline
\bf Loss function  &  ${\bf \mathcal{L}(\Sigma, \widehat{\Sigma})}$ & \bf Loss function  &  ${\bf \mathcal{L}(\Sigma, \widehat{\Sigma})}$  \\ \hline
Frobenius & $\| \Sigma-\widehat{\Sigma} \|_F/\sqrt{p}$ & Inverse Quadratic &  $\|  \widehat{\Sigma}^{-1}\Sigma-I \|_F/\sqrt{p}$ \\ \hline
Disutility  & $\tr[(\widehat{\Sigma}^{-1}-\Sigma^{-1})^2 \Sigma]/\tr(\Sigma^{-1}) $ & Quadratic & $\| \Sigma^{-1} \widehat{\Sigma}-I \|_F/\sqrt{p}$ \\ \hline
Inverse Stein & $ \left( \tr[\widehat{\Sigma}^{-1} \Sigma-I]-\log[\det (\widehat{\Sigma}^{-1} \Sigma) ] \right)/p $ & Fr{\'e}chet  & $\| \widehat{\Sigma}^{1/2}-\Sigma^{1/2}\|_F/\sqrt{p}$ \\ \hline
Minimum Variance & $p\tr[\widehat{\Sigma}^{-1} \Sigma \widehat{\Sigma}^{-1}] \Big/ \left( \tr [\widehat{\Sigma}^{-1}] \right)^2-p/\tr(\Sigma^{-1}) $ & Log-Euclidean &   $\| \log(\widehat{\Sigma})-\log (\Sigma) \|_F/\sqrt{p}$ \\ \hline
Stein & $ \left( \tr[\widehat{\Sigma} \Sigma^{-1}-I]-\log[\det (\widehat{\Sigma} \Sigma^{-1}) ] \right)/p $ & Symmetrized Stein &  $\tr [\widehat{\Sigma} \Sigma^{-1}+\widehat{\Sigma}^{-1} \Sigma-2I]/p$ \\ \hline
Weighted Frobenius & $\tr[(\widehat{\Sigma}-\Sigma)^2 \Sigma^{-1}]/\tr(\Sigma)$ & Inverse Frobenius &  $\| \Sigma^{-1}-\widehat{\Sigma}^{-1} \|_F/\sqrt{p}$\\ \hline
\end{tabular}
\caption{Summary of commonly used loss functions for the population covariance matrix $\Sigma$ and its estimator $\widehat{\Sigma}$. }\label{table_summarylossfunction}
\end{center}
\end{table} 
%\end{definition}

For the loss functions considered in \Cref{table_summarylossfunction}, the analytical forms of the shrinkers $\varphi_i$ can be computed explicitly. We summarize them in Lemma \ref{lem_explicityformula} below. 

\begin{lemma}\label{lem_explicityformula} For the loss functions in Table \ref{table_summarylossfunction}, the optimal shrinkers $\varphi_i$, $i\in \qq{\mathsf K}$, are given as follows. 
\begin{enumerate}
\item For the Frobenius, inverse Stein, disutility, and minimum variance norms, we have that $\varphi_i=\bm{u}_i^\top \Sigma \bm{u}_i. $
\item For Stein, weighted Frobenius, and inverse Frobenius norms, we have that $\varphi_i=(\bm{u}_i^\top \Sigma^{-1} \bm{u}_i)^{-1}.$ 
\item For the symmetrized Stein norm, we have that $\varphi_i=\sqrt{\frac{\bm{u}_i^\top \Sigma \bm{u}_i}{\bm{u}_i^\top \Sigma^{-1} \bm{u}_i}}.$
\item For the Log-Euclidean norm, we have that $\varphi_i=\exp(\bm{u}_i^\top \log(\Sigma)\bm{u}_i).$
\item For the Fr{\'e}chet norm, we have that $\varphi_i=(\bm{u}_i ^\top \Sigma^{1/2} \bm{u}_i)^{2}.$
\item For the quadratic norm, we have that  $\varphi_i=\frac{\bm{u}_i^\top \Sigma^{-1} \bm{u}_i}{\bm{u}_i^\top \Sigma^{-2} \bm{u}_i}.$ 
\item For the inverse quadratic norm, we have that $\varphi_i=\frac{\bm{u}_i^\top \Sigma^{2} \bm{u}_i}{\bm{u}_i^\top \Sigma \bm{u}_i}.$ 
\end{enumerate}
When $c_n>1$, the above results remain valid for $i\in \qq{\mathsf K+1,p}$ by replacing the factors $\bm{u}_i^\top \ell(\Sigma) \bm{u}_i$ (for $\ell(x)=x, \ x^{-1}, \ \log x, \ \sqrt{x}, \ x^{-2}, \ x^2$) with $ \tr[U_0^\top \ell(\Sigma)U_0]/(p-n).$ 
\end{lemma}   
\begin{proof}
The proof follows from straightforward calculations; see \cite{donoho2018, mainreference}. 
\end{proof}

For the above loss functions, we have the following decomposition of their associated risks. Recall the the optimal invariant estimator $\wt{\Sigma}$ in  (\ref{eq_steincovarianceestimator}) and define the diagonal matrix
\begin{equation*}
\Phi:=\operatorname{diag}\left\{ \varphi_1, \cdots, \varphi_p \right\}.
\end{equation*}
\begin{lemma}\label{lem_lossdecomposition}
For the loss functions in Table \ref{table_summarylossfunction} and their optimal solutions $\{\varphi_i\}$ given in Lemma \ref{lem_explicityformula}, we have the following identities. 
\begin{enumerate}
\item For the Frobenius norm,  we have that 
\begin{equation*}
\| \Sigma-\widetilde{\Sigma}   \|_F^2= \left\| \Sigma \right\|_F^2-\left\| \Phi \right\|_F^2. 
\end{equation*}
\item For the inverse Frobenius norm, we have that 
\begin{equation*}
\| \Sigma^{-1}-\widetilde{\Sigma}^{-1} \|_F^2=\| \Sigma^{-1} \|_F^2-\| \Phi^{-1} \|_F^2. 
\end{equation*}
\item For the weighted Frobenius norm, we have that 
\begin{equation*}
\tr\big[(\widetilde{\Sigma}-\Sigma)^2 \Sigma^{-1}\big]=\tr(\Sigma)-\tr(\Phi). 
\end{equation*}
\item For the disutility norm, we have that 
\begin{equation*}
\tr\big[(\widetilde{\Sigma}^{-1}-\Sigma^{-1})^2 \Sigma\big]=\tr(\Sigma^{-1})-\tr(\Phi^{-1}). 
\end{equation*}
\item For the inverse Stein norm,  we have that 
\begin{equation*}
 \tr\big[\widetilde{\Sigma}^{-1} \Sigma-I\big]-\log[\det (\widetilde{\Sigma}^{-1} \Sigma) ] =\log \left( \frac{\det (\Phi)}{\det (\Sigma)} \right). 
\end{equation*}
\item For the Stein norm, we have that 
\begin{equation*}
 \tr\big[\widetilde{\Sigma} \Sigma^{-1}-I\big]-\log[\det (\widetilde{\Sigma} \Sigma^{-1}) ] =\log \left( \frac{\det (\Sigma)}{\det (\Phi)} \right). 
\end{equation*}
\item For the Fr{\'e}chet norm,  we have that 
\begin{equation*}
\big\| \widetilde{\Sigma}^{1/2}-\Sigma^{1/2} \big\|_F^2=\tr(\Sigma)-\tr(\Phi).
%\left\| \Sigma^{1/2} \right\|_F^2-\left\| \Phi^{1/2} \right\|_F^2.
\end{equation*}
\item For the minimal variance norm, we have that 
\begin{equation*}
p\tr[\widetilde{\Sigma}^{-1} \Sigma \widetilde{\Sigma}^{-1}] \Big/ \left( \tr [\widetilde{\Sigma}^{-1}] \right)^2 =p/\tr(\Phi^{-1}) .
%=p\left( \| \Phi^{-1} \|_F^{-1}-\| \Sigma^{-1} \|_F^{-1} \right). 
\end{equation*}
\item For the quadratic norm, we have that 
\begin{equation*}
\big\| \Sigma^{-1} \wt{\Sigma}-I \big\|_F^2=p-\sum_{i=1}^{\mathsf K} \frac{(\bu_i^\top \Sigma^{-1} \bu_i)^2}{\bu_i^\top \Sigma^{-2} \bu_i} -  \frac{[\tr(U_0^\top \Sigma^{-1} U_0)]^2}{\tr(U_0^\top \Sigma^{-2} U_0)}.
\end{equation*}
\item For the inverse quadratic norm, we have that 
\begin{equation*}
\big\| \wt{\Sigma}^{-1} \Sigma-I \big\|_F^2=p-\sum_{i=1}^{\mathsf K} \frac{(\bu_i^\top \Sigma \bu_i)^2}{\bu_i^\top \Sigma^2 \bu_i}-  \frac{[\tr(U_0^\top \Sigma U_0)]^2}{\tr(U_0^\top \Sigma^{2} U_0)}.
\end{equation*}
\item For the Log-Euclidean norm, we have that
\begin{equation*}
\big\| \log (\widetilde{\Sigma})-\log (\Sigma) \big\|_F^2=\| \log(\Sigma) \|_F^2-\| \log(\Phi) \|_F^2.
\end{equation*}
\item For the symmetrized Stein norm, we have that
\begin{equation*}
\frac{1}{2}\tr \big[\wt{\Sigma} \Sigma^{-1}+\wt{\Sigma}^{-1} \Sigma-2I\big]=\sum_{i=1}^{\sfK} \sqrt{(\bu_i^\top \Sigma^{-1} \bu_i)\cdot ( \bu_i^\top \Sigma \bu_i)}+ \sqrt{ \tr (U_0^\top \Sigma^{-1} U_0)\cdot \tr( U_0^\top \Sigma U_0)} -p. 
\end{equation*}
\end{enumerate}
\end{lemma}

\begin{proof}
The proof follows from straightforward calculations using Lemma \ref{lem_explicityformula}. We omit the details.  
\end{proof}

%\subsection{Risk decomposition for different loss functions}

\section{Some preliminary results}\label{appendix_preliminary}
In this section, we present some results that will be used in the technical proofs of the main results. For ease of presentation, in this paper, we consistently use the following notion of \emph{stochastic domination} introduced in \cite{Average_fluc}.
It simplifies the presentation of the results and their proofs by systematizing statements of the form ``$\xi$ is bounded by $\zeta$ with high probability up to a small power of $n$".

\begin{definition}[Stochastic domination]\label{stoch_domination}
{\rm{(i)}} Let
	\[\xi=\left(\xi^{(n)}(u):n\in\mathbb N, u\in U^{(n)}\right),\hskip 10pt \zeta=\left(\zeta^{(n)}(u):n\in\mathbb N, u\in U^{(n)}\right),\]
	be two families of non-negative random variables, where $U^{(n)}$ is a possibly $n$-dependent parameter set. We say $\xi$ is stochastically dominated by $\zeta$, uniformly in $u$, if for any fixed (small) $\tau>0$ and (large) $D>0$, 
	\[\mathbb P\bigg(\bigcup_{u\in U^{(n)}}\left\{\xi^{(n)}(u)>n^\tau\zeta^{(n)}(u)\right\}\bigg)\le n^{-D}\]
	for large enough $n\ge n_0(\tau, D)$, and we will use the notation $\xi\prec\zeta$. If for some complex family $\xi$ we have $|\xi|\prec\zeta$, then we will also write $\xi \prec \zeta$ or $\xi=\OO_\prec(\zeta)$. 
	
	\vspace{5pt}
	\noindent {\rm{(ii)}} As a convention, for two deterministic non-negative quantities $\xi$ and $\zeta$, we will write $\xi\prec\zeta$ if and only if $\xi\le n^\tau \zeta$ for any constant $\tau>0$. 

 %\vspace{5pt}
%	\noindent {\rm{(iii)}} Let $A$ be a family of random matrices and $\zeta$ be a family of nonnegative random variables. Then, we use $A=\OO_\prec(\zeta)$ to mean that $\|A\|\prec \xi$, where $\|\cdot\|$ denotes the operator norm. 
	
	\vspace{5pt}
	\noindent {\rm{(iii)}} We say that an event $\Xi$ holds with high probability (w.h.p.) if for any constant $D>0$, $\mathbb P(\Xi)\ge 1- n^{-D}$ for large enough $n$. More generally, we say that an event $\Omega$ holds $w.h.p.$ in $\Xi$ if for any constant $D>0$,
	$\P( \Xi\setminus \Omega)\le n^{-D}$ for large enough $n$.
\end{definition}

% \begin{definition}[Stochastic domination]

% (i) Let
% \[\xi=\left(\xi^{(n)}(u):n\in\bbN, u\in U^{(n)}\right),\hskip 10pt \zeta=\left(\zeta^{(n)}(u):n\in\bbN, u\in U^{(n)}\right),\]
% be two families of nonnegative random variables, where $U^{(n)}$ is a possibly $n$-dependent parameter set. We say $\xi$ is stochastically dominated by $\zeta$, uniformly in $u$, if for any fixed (small) $\epsilon>0$ and (large) $D>0$, 
% \[\sup_{u\in U^{(n)}}\bbP\left(\xi^{(n)}(u)>n^\epsilon\zeta^{(n)}(u)\right)\le n^{-D},\]
% for large enough $n \ge n_0(\epsilon, D)$, and we shall use the notation $\xi\prec\zeta$. Throughout this paper, the stochastic domination will always be uniform in all parameters that are not explicitly fixed, such as the matrix indices and the spectral parameter $z$.  If for some complex family $\xi$ we have $|\xi|\prec\zeta$, then we will also write $\xi \prec \zeta$ or $\xi=\OO_\prec(\zeta)$.
% %\item[(ii)] 

% \vspace{5pt}

% \noindent (ii) We say an event $\Xi$ holds with high probability if for any constant $D>0$, $\mathbb P(\Xi)\ge 1- n^{-D}$ for large enough $n$.

% \end{definition} 

%We next provide some notations. 
We denote the resolvents of the sample covariance matrices in (\ref{eq_samplecovariancematrixdefinition}) and (\ref{eq_grammatrixdefinition}) by
\begin{equation}\label{eq_originalG}
\mathsf{G}_{a}(z):=\left(\mathcal{Q}_{a}-z\right)^{-1}, \quad  \widetilde{\mathsf{G}}_{a}(z):=\big(\widetilde{\mathcal{Q}}_{a}-z\big)^{-1},\quad a\in \{1,2\}. 
\end{equation}
For our proofs, it is more convenient to work with the following linearized block matrix $\sigH$, whose inverse $\sigG$ is also called the resolvent (of $\sigH$): 
%Denote the linearizing block matrix and its inverse as
\begin{equation}\label{eq_orginallinearzation}
\sigH(z):=
\begin{pmatrix}
-\Sigma_0^{-1} & X \\
X^\top & -zI_n
\end{pmatrix}, \quad \sigG(z):=\sigH(z)^{-1}. 
\end{equation}
By Schur's complement formula (see e.g., Lemma 4.4 of \cite{MR3704770}), we have that 
\begin{equation}\label{eq_gzschurcomplement}
\sigG(z):=
\begin{pmatrix}
z \Sigma_0^{1/2} \mathsf{G}_1(z) \Sigma_0^{1/2} & \Sigma_0 X \mathsf{G}_2 \\
\mathsf{G}_2X^\top \Sigma_0  & \mathsf{G}_2
\end{pmatrix}.
\end{equation} 
Hence, a control of the resolvent $\sigG$ yields directly a control of both $\sigG_1$ and $\sigG_2$.
For notational convenience, we will consistently use the following notations of index sets
\begin{equation}\label{eq_indexset} 
\mathcal I_1:=\qq{1,p}, \quad \ \mathcal I_2:=\qq{p+1, p+n}, \quad \ \mathcal I:=\mathcal I_1\cup\mathcal I_2=\qq{1,p+n}.
\end{equation}
Then, we label the indices of the blocks of $\sigH$ according to 
$X= (X_{i\mu}:i\in \mathcal I_1, \mu \in \mathcal I_2).$  
%In the rest of this paper, whenever referring to the entries of $\sigH$ and $\sigG$, we will consistently use the Latin letters $i,j\in\mathcal I_1$, Greek letters $\mu,\nu\in\mathcal I_2$, and $a,b\in\mathcal I$. 
In what follows, we often omit the dependence on $z$ and simply write $\mathsf{H}, \mathsf{G}$, etc. Finally, we adopt the following convention for matrix multiplication: 
for matrices of the form $A:=(A_{st}: s \in \mathsf{l}(A), t \in \mathsf{r}(A))$ and $B:=(B_{st}: s \in \mathsf{l}(B), t \in \mathsf{r}(B))$, whose entries are indexed by some subsets $\mathsf{l}(A), \mathsf{r}(A),\mathsf{l}(B), \mathsf{r}(B) \subset \mathbb{N},$ the matrix multiplication $AB$ is understood as 
\begin{equation}\label{eq_matrixnotation}
(AB)_{st}:=\sum_{k \in \mathsf{r}(A) \cap \mathsf{l}(B)} A_{sk} B_{kt}
\end{equation} 
for $s \in \mathsf{l}(A)$ and $t \in \mathsf{r}(B)$.

\subsection{Local laws for sample covariance matrices}\label{appendix_locallawsection} 

In this subsection, we present a key input for our proofs, namely the local laws for the resolvent $\mathsf{G}$ \cite{MR3704770}.
% we provide the results of the local laws for the matrices $\mathcal{Q}_{1(2)}$ in (\ref{eq_samplecovariancematrixdefinition}) and (\ref{eq_grammatrixdefinition}). 
%According to (\ref{eq_gzschurcomplement}), it suffices to control $\sigG(z).$ 
% These results have been proved in \cite{MR3704770}. 
Recall the Stieltjes transform of the deformed Marchenco-Pastur law $m(z)$ defined in (\ref{eq_defnmc}).
% For $m(z)$ defined in (\ref{eq_defnmc}), 
We introduce the following (deterministic) matrix limit of $\sigG$ defined as % $(p+n)\times (p+n)$ block matrix as
\begin{equation}\label{eq_Pigz}
% \Pi_{\mathrm{g}}(z):=
\Pi_{\sigG}(z):=
\begin{pmatrix}
-\Sigma_0 (1+m(z) \Sigma_0)^{-1} & 0\\
0 & m(z)I_n
\end{pmatrix}.
% \in \mathbb{R}^{(p+n) \times (p+n)}.
\end{equation}
Moreover, for an arbitrarily small constant $\tau\in (0,1)$, we define the spectral parameter domain 
\begin{equation}\label{eq_generalspectraldomain}
\mathbf{D} \equiv \mathbf{D}(\tau, n):=\left\{z=E+\ri \eta \in \mathbb{C}_+: \tau \leq E \leq \tau^{-1}, \ n^{-1+\tau} \leq \eta \leq \tau^{-1}    \right\},
\end{equation}
a spectral domain outside the support of $\varrho$,
\begin{equation*}
\mathbf{D}_{o} \equiv \mathbf{D}_{o}(\tau, n):=\left\{ z=E+\ri \eta \in \mathbb{C}_+: \tau\le E\le \tau^{-1}, 0<\eta\le \tau^{-1}, \ \operatorname{dist} (E, \operatorname{supp}(\varrho)) \geq n^{-2/3+\tau} \right\},
\end{equation*}
and a spectral domain around the origin,
\begin{equation*}
\mathbf{D}_0 \equiv \mathbf{D}_0(n):=\left\{ z=E+\ri \eta \in \mathbb{C}: |z|\le (\log n)^{-1} \right\},
\end{equation*}
Define $\sm_n(z):=n^{-1} \operatorname{Tr} \mathsf{G}_2(z).$ The following lemma states the local laws for $\sigG(z)$ and $\sm_n(z)$. 
\begin{lemma}[Theorems 3.6 and 3.7 of \cite{MR3704770}]
\label{lem_standardaniso}
%Suppose (i)--(iii) of Assumption \ref{main_assumption} hold. Let $\mathbf{u}, \mathbf{v}$ be two deterministic vectors in $\mathbb{R}^{p+n}.$ Then for any $z \in \mathbf{D}$ uniformly we have that 

Suppose (i)--(iii) of Assumption \ref{main_assumption} hold. Then, the following estimates hold uniformly in $z=E+\ii \eta\in \mathbf D$:
\begin{itemize}
\item {\bf Anisotropic local law}: For any two  deterministic unit vectors $\mathbf{u}, \mathbf{v}\in \mathbb{R}^{p+n}$, we have 
\begin{equation}\label{eq:aniso}
\left| \mathbf{u}^\top \sigG(z) \mathbf{v}-\mathbf{u}^\top \Pi_{\sigG}(z) \mathbf{v} \right| \prec \Psi(z),
\end{equation}
where $\Psi$ is an error control parameter defined as
$$\Psi(z):=\sqrt{\frac{\im m(z)}{n \eta}}+\frac{1}{n \eta}.$$

\item {\bf Averaged local law}: We have
\begin{equation}\label{eq:aver}
|\sm_n(z)-m(z)| \prec \frac{1}{n \eta}.
\end{equation}
\end{itemize}
In addition, outside the support of $\varrho$, we have a stronger {\bf anisotropic local law} uniformly in $z=E+\ii \eta\in \mathbf D_{o}$:
\begin{equation}\label{eq:aniso_out}
\left| \mathbf{u}^\top \sigG(z) \mathbf{v}-\mathbf{u}^\top \Pi_{\sigG}(z) \mathbf{v} \right| \prec n^{-1/2}(\kappa+\eta)^{-1/4}, % \sqrt{\frac{\im m(z)}{n \eta}} 
\end{equation}
for any two  deterministic unit vectors $\mathbf{u}, \mathbf{v}\in \mathbb{R}^{p+n}$, where $\kappa$ is defined as $\kappa:= \operatorname{dist}(E, \supp \varrho ).$ 
% \begin{itemize}
% \item {\bf Anisotropic local law}: For any two  deterministic unit vectors $\mathbf{u}, \mathbf{v}\in \mathbb{R}^{p+n}$, we have 
% \begin{equation}\label{eq:ansio_out}
% \left| \mathbf{u}^\top \sigG(z) \mathbf{v}-\mathbf{u}^\top \Pi_{\sigG}(z) \mathbf{v} \right| \prec \sqrt{\frac{\im m(z)}{n \eta}}. 
% \end{equation}
% \item {\bf Averaged local law}: Let $\kappa:=\operatorname{dist}(\re z, \supp(\varrho))$. We have {\cor check}
% \begin{equation}\label{eq:aver_out}
% |\sm_n(z)-m(z)| \prec \frac{1}{n(\kappa+\eta)} . 
% \end{equation}
% \end{itemize}
\end{lemma}
%  \begin{proof}
% See Theorems 3.6 and 3.7 of \cite{MR3704770}. 
%  \end{proof}
% We first prepare some notations and {\color{red} from here, state a more general form of rectangular version; add the results of control of the imaginary parts.}

Similar to \eqref{eq_orginallinearzation} and \eqref{eq_Pigz}, we define the (linearized) resolvent $\wt\sigG$ for the spiked model and the corresponding deterministic limit as
\begin{equation}\label{eq_resol}
\wt\sigG(z):=
\begin{pmatrix}
-\Sigma^{-1} & X \\
X^\top & -zI_n
\end{pmatrix}^{-1},\quad \wt\Pi_{\sigG}(z):=
\begin{pmatrix}
-\Sigma (1+m(z) \Sigma)^{-1} & 0\\
0 & m(z)I_n
\end{pmatrix}.
\end{equation}

\begin{lemma}[Theorem C.4 of \cite{AOIT_2020}]\label{lem_standardaniso0}
Suppose Assumption \ref{main_assumption} holds and $c_n\ge 1+\tau$. Then, the following {\bf anisotropic local law} holds for $\sigG$ and $\wt{\sigG}$ uniformly in $z=E+\ii \eta\in \mathbf D_0$:
\begin{equation}\label{eq:aniso0}
\left| \mathbf{u}^\top \sigG(z) \mathbf{v}-\mathbf{u}^\top \Pi_{\sigG}(z) \mathbf{v} \right| \prec n^{-1/2},\quad \big| \mathbf{u}^\top \wt\sigG(z) \mathbf{v}-\mathbf{u}^\top \wt\Pi_{\sigG}(z) \mathbf{v} \big| \prec n^{-1/2}.
\end{equation}
\end{lemma}

\begin{remark}\label{rmk_keymodification}
In the proofs, it is slightly more convenient to use the following version of Lemma \ref{lem_standardaniso}. Recall the eigenmatrix $V \in \mathbb{R}^{p \times p}$ defined in (\ref{eq_defnsigma0}). Let 
\begin{equation*}
\mathbf{V}:=
\begin{pmatrix}
V & 0 \\
0 & I_n 
\end{pmatrix}\in \mathbb{R}^{(n+p)\times(n+p)}.
\end{equation*}
We then rotate $\sigH$ and $\sigG$ in (\ref{eq_orginallinearzation}) as 
\begin{equation}\label{eq_defnfinalG}
H(z):=\mathbf{V}^\top \sigH(z) \mathbf{V},  \quad  G(z):=H(z)^{-1}.
\end{equation}
Under this definition, with \eqref{eq_gzschurcomplement}, we have that 
%the above notations, the analogs of (\ref{eq_orginallinearzation}) and (\ref{eq_gzschurcomplement}) are as follows
\be\label{aniso}
H=\begin{pmatrix}
 -\Lambda_0^{-1}&  V^\top X\\
 X^\top V & -z
 \end{pmatrix}, \quad 
 G=\begin{pmatrix}
z \Lambda_0^{1/2} \mathcal{G}_1 \Lambda_0^{1/2} & \Lambda_0 V^\top X \mathcal{G}_2 \\
\mathcal{G}_2X^\top V \Lambda_0  & \mathcal{G}_2
\end{pmatrix}.
 \ee
where the resolvents $\cal G_1$ and $\cal G_2$ are defined as
\begin{equation}\label{eq_mathcalG1mathcalG2}
\mathcal{G}_1(z):=(\Lambda_0^{1/2} V^\top X X^\top V \Lambda_0^{1/2}-z)^{-1}, \quad  \mathcal{G}_2(z):=\mathsf{G}_2(z).
\end{equation}
In this case, it is easy to see that Lemmas \ref{lem_standardaniso} and \ref{lem_standardaniso0} hold for $G(z)$ if we replace $\Pi_{\sigG}(z)$ by 
\begin{equation}\label{eq:defnPi}
\Pi(z)=
\begin{pmatrix}
-\Lambda_0 (1+m(z) \Lambda_0)^{-1} & 0\\
0 & m(z)
\end{pmatrix}.
\end{equation}
The main advantage of $\Pi$ over $\Pi_{\sigG}$ is that $\Pi$ is a diagonal matrix. Similarly, we can define $\wt{\cal G}_1$, $\wt{\cal G}_2$, $\wt G$, and $\wt \Pi$ by replacing $\Lambda_0$ with $\Lambda$ in the above definitions. 
\end{remark}

The following lemma outlines some basic estimates regarding $m(z)$ and $\im m(z)$. Notably, it shows that the entries of $\Pi$ in \eqref{eq:defnPi} are of order 1. These estimates will be used tacitly in the subsequent proofs.

\begin{lemma}\label{remark_controlmplaw}
Suppose items (i) and (iii) of Assumption \ref{main_assumption} hold. Then, for any small constant $c>0$, the following estimates hold uniformly for $z$ satisfying $c \le |z|\le c^{-1}$:
\begin{equation}\label{eq:Imm}
|m(z)|\asymp 1, \quad \im m(z) \asymp 
\begin{cases}
\sqrt{\kappa+\eta}, \ & \text{if} \ E \in  \supp(\varrho) \\
\frac{\eta}{\sqrt{\kappa+\eta}},\  & \text{if} \ E \notin  \supp(\varrho)
\end{cases},
\end{equation}
\be\label{eq:denominator}
\min_{j=1}^p|1+m(z)\sigma_j|\gtrsim 1.
\ee
% where $\kappa= \operatorname{dist}(\re z, \supp(\varrho)).$
For sufficiently small constant $c > 0$, the following estimate holds uniformly for $z$ satisfying $|z|\le c$:
\be\label{eq:mzero}
|zm(z)| \asymp \begin{cases} 1, & \ \text{if} \ c_n\le 1-\tau \\ |z|, & \ \text{if} \ c_n\ge 1+\tau \end{cases}.
\ee
\end{lemma}
\begin{proof}
The estimates in \eqref{eq:Imm} are proved in Lemma 4.10 and equation (A.7) of \cite{MR3704770}. The estimate \eqref{eq:denominator} is proved in Lemma A.6 of \cite{ding2018}. Finally, the estimate \eqref{eq:mzero} follows from the fact that the deformed MP law has support within $[\lambda_-,\lambda_+]$ with a delta mass $(1-c_n)\delta_0$ when $c_n<1$. 
\end{proof}

As an important consequence of the averaged local laws in Lemma \ref{lem_standardaniso}, we have the following rigidity estimate of the eigenvalues $\lambda_j\equiv \lambda_j(\cal Q_2)$.

%For convenience, we denote $\gamma_0:=+\infty$ and $\gamma_{n\wedge N+1}:=0$. 

\begin{lemma}[Theorem 3.12 of \cite{MR3704770}] \label{lem_rigidity}
%Suppose \eqref{aver_law} and the regularity conditions in Definition \ref{def_regular} hold. Then, 
Under the assumptions of  \Cref{lem_standardaniso},  for $\gamma_j \in [a_{2k},a_{2k-1}]$, we have 
\begin{equation}\label{rigidity2} 
| \lambda_j - \gamma_j| \prec [(n_{k}+1-j)\wedge (j-n_{k-1})]^{-1/3}n^{-{2}/{3}},
\end{equation}
where we recall the notations in \eqref{Nk} and \Cref{defn_generaleigenvaluelocation}.
\end{lemma} 

%Recall (\ref{TAZ2}). 
% \begin{lemma}\label{lem_rigidity} Suppose the assumptions of Lemma \ref{lem_standardaniso} hold. Then we have that for $1 \leq i \leq \mathsf{K}$
% \begin{equation*}
% |\lambda_i-\gamma_i|\prec n^{-2/3} (\min \{i, n-i+1\})^{-1/3}.
% \end{equation*}
% \end{lemma}

\subsection{Eigenvalues and eigenvectors of the spiked covariance model} 

In this subsection, we present a collection of useful results regarding the eigenvalues and eigenvectors of the spiked model $ \widetilde{\mathcal{Q}}_1$ in (\ref{eq_samplecovariancematrixdefinition}). Recall that \smash{$\wt\lambda_i$} and $\lambda_i$ represent the eigenvalues of \smash{$\widetilde{\mathcal{Q}}_1$} and $\cal Q_1$, respectively, while $\fa_i$ and $\fb_i$ are defined in (\ref{eq_importantdefinition}). Furthermore, we emphasize that the ${\bm{u}_i}$ mentioned in Lemmas \ref{lemma_spikedlocation} and \ref{lem_bulkproperty} below refers to the eigenvectors of \smash{$\widetilde{\mathcal{Q}}_1$}, rather than $\mathcal{Q}_1$.
First, \Cref{lemma_spikedlocation} gives the first-order limits of the outlier eigenvalues and eigenvectors, along with nearly optimal convergence rates.

%Most of the results can be found in \cite{bun2018optimal,DRMTA}. 
% Recall $h(x)$ defined in (\ref{eq_defnf})
% and the spiked covariance matrices in (\ref{eq_spikedcovariancematrix}) and (\ref{eq_defnspikes}).  

%locations and rates of the outlier eigenvalues and eigenvectors. Recall $h(x)$ defined in (\ref{eq_defnf}) and (\ref{eq_importantdefinition}).

\begin{lemma}[Theorems 3.2 and 3.3 of \cite{DRMTA}]\label{lemma_spikedlocation} Suppose Assumption \ref{main_assumption} holds. For any $ i \in \qq{r}$, we have that
\begin{equation}\label{eq:outlier_location}
  \left|\wt\lambda_i-\mathfrak{a}_i\right|  \prec n^{-1/2}.
\end{equation}
Moreover, for the corresponding outlier eigenvectors, we have that for any $i,j\in \qq{r}$,
\begin{equation}\label{eq:outlier_proj}
\left| \langle \bm{u}_i, \bm{v}_j \rangle \right|^2- \frac{\mathfrak{b}_i }{\widetilde{\sigma}_i} \delta_{ij} \prec n^{-1/2}. 
\end{equation} 
\end{lemma}
% \begin{proof}
% See Theorems 3.2 and 3.3 of \cite{DRMTA}. 
% \end{proof}

Next, we state some estimates for the non-outlier eigenvalues and eigenvectors. 
% Recall (\ref{eq_importantdefinition}) and define
% \begin{equation*}
% \mathfrak{c}_i:= {d_i^2\mathfrak{b}_i}/{\widetilde{\sigma}_i^2} .
% \end{equation*} 
%Moreover, recall the resolvent matrix $G(z)$ in \eqref{eq_defnfinalG}.
% (\ref{eq_importantdefinition}). 
\begin{lemma}\label{lem_bulkproperty}
Suppose Assumption \ref{main_assumption} holds. For $ i \in \qq{\sfK-r}$, we have the eigenvalue sticking estimate:
\begin{equation}\label{eq_eigenvaluesticking}
|\wt\lambda_{i+r}-\lambda_i|\prec n^{-1}, 
\end{equation}
and the eigenvector delocalization estimate: 
\begin{equation}\label{eq_eigenvectordelocalization}
\left| \langle \bm{u}_{i+r}, \bv \rangle \right|^2 \prec n^{-1}
\end{equation}
for any deterministic unit vector $\bv\in \R^p$. When $c_n>1$, we have that
\begin{equation}\label{eq_eigenvectordelocalization_at0}
\|U_0^\top \bv \|_2^2 \prec 1.
\end{equation}
%where we recall that $\{\lambda_i\}$ denote the eigenvalues of $\mathcal{Q}_1$. 
Finally, for $i \in \qq{r}$ and $j \in \qq{r+1, p},$ we have that 
\begin{equation}\label{eq_intermediateresult}
\left| \langle \bm{u}_i, \bm{v}_j \rangle\right|^2- \frac{\mathfrak{c}_i}{\sigma_j} G_{ij}(\mathfrak{a}_i) G_{ji}(\mathfrak{a}_i)   \prec n^{-3/2},\quad \text{with}\quad \mathfrak{c}_i:= \frac{d_i^2\mathfrak{b}_i}{\widetilde{\sigma}_i^2} .
\end{equation}
%where $\fc_i$ is defined as $\mathfrak{c}_i:= {d_i^2\mathfrak{b}_i}/{\widetilde{\sigma}_i^2} $. 

%In \eqref{eq_intermediateresult} and \eqref{eq_eigenvectordelocalization}, it should be noted that $\{\bm{u}_i\}$ denote the eigenvectors of $\widetilde{\cal Q}_1$ rather than $\cal Q_1$.

%Above in \eqref{eq_intermediateresult} and \eqref{eq_eigenvectordelocalization}, $\{\bm{u}_i\}$ refer to the eigenvectors of $\wt{\cal Q}_1$ instead of $\cal Q_1$.
\end{lemma}
\begin{proof}
The results (\ref{eq_eigenvaluesticking}) and (\ref{eq_eigenvectordelocalization}) have been proved in Theorems 3.7 and 3.14 of \cite{ding2021spiked} for $i \leq \tau p$ with $\tau\in (0,1)$ being a small constant. The reason for considering a subset of indices is that only the edge regularity condition around $\lambda_+=a_1$ was assumed in \cite{ding2021spiked}. However, with the stronger regularity conditions in \eqref{eq:edgeregular} and \eqref{eq:bulkregular}, we can extend (\ref{eq_eigenvaluesticking}) and (\ref{eq_eigenvectordelocalization}) to all $ i \in \qq{\sfK-r}$ following the proof in \cite{ding2021spiked}.
%The reason why we only handle a subset of indices is that it was only assumed the edge regularity condition around $\lambda_+=a_1$ in \cite{ding2021spiked}. With the stronger regularity condition stated in item (iii) of \Cref{main_assumption}, following the proof in \cite{ding2021spiked}, we can prove that (\ref{eq_eigenvaluesticking}) and (\ref{eq_eigenvectordelocalization}) extend to all $ i \in \qq{\sfK-r}$. 
We omit the details. As for \eqref{eq_eigenvectordelocalization_at0}, we have the trivial bound
$${\|U_0^\top \bv \|_2^2}\le \eta\im [\bv^\top \wt{\sigG}_1(\ii \eta) \bv],$$
which follows from the spectral decomposition of $\wt{\sigG}_1$.
The RHS can be written as 
$$\bv^\top \wt{\sigG}_1(\ii \eta) \bv = \left(\bv^\top\Sigma^{-1/2},\mathbf 0_n^\top\right) \frac{\wt\sigG(\ii \eta)}{\ii \eta}\begin{pmatrix}
     \Sigma^{-1/2}\bv\\ \mathbf 0_n 
\end{pmatrix}.$$
Then, \eqref{eq_eigenvectordelocalization_at0} follows immediately from the local law on $\wt\sigG$ in \eqref{eq:aniso0} by taking $\eta=(\log n)^{-1}$.

It remains to prove (\ref{eq_intermediateresult}). The proof relies on a representation of $\left| \langle \bm{u}_i, \bm{v}_j \rangle\right|^2$ in terms of contour integrals of the resolvent $G$ (see e.g., \cite{bloemendal2016principal,bun2018optimal,DRMTA,ding2021spiked}).
%along with a continuous argument as mentioned in equation (7.18) of \cite{DRMTA}. 
We only sketch the proof since many arguments below are similar to those in \cite{bun2018optimal,DRMTA}. 
Let  $\{\bu_i\}_{i=1}^p=\{V^\top\bm{u}_i\}_{i=1}^p$ be the left singular eigenvectors of $\Lambda^{1/2}V^\top X$. 
% Observe that the matrices $\Lambda^{1/2} V^\top X$ and $\Sigma^{1/2}X$ share the same singular values and right singular vectors. We denote the left singular vectors associated with $\Lambda^{1/2}V^\top X$ as $\{\bu_i\}.$ It is direct to see from singular value decomposition that 
% \begin{equation}\label{eq_keyconnection}
%  \bm{u}_i= V \bu_i.
% \end{equation}
%Consequently, we only need to estimate the entries of $\bu_i$ 
Due to the relation
\begin{equation}\label{eq_keyconnection}
\avg{\bm{v}_j, \bm{u}_i}=(\bm{v}_j^\top V) \bu_i=\eb_j^\top \bu_i,  
\end{equation}
it suffices to focus on the entries $\bu_i(j)$ of $\bu_i$ in the following proof. 
%Here for the $(r+1)$ indices, we use the following convention  
Let $\Gamma_i$ denote the boundary of $B_\rho(\fa_i)$, which represents a disk of radius $\rho$ centered at \smash{$\fa_i$}. Here, we choose $\rho$ to be sufficiently small so that $B_{2\rho}(\fa_i)$ does not contain $\lambda_+$ or any other outlier $\fa_k$, $k\in \qq{r}\setminus\{i\}$. Then, by \eqref{eq:outlier_location} and item (iv) of Assumption \ref{main_assumption}, we know that with high probability, $\Gamma_i$ only encloses $\wt\lambda_i$ and no other eigenvalues of $G(z)$. By applying Cauchy's integral formula, we obtain that 
\be\label{eq:uij_perturb} |\bu_i(j)|^2=-\frac{1}{2\pi \ii}\oint_{\Gamma_i}(\wt{\cal G}_1)_{jj}(z)\dd z=-\lim_{\delta\downarrow 0} \frac{1}{2\pi \ii}\oint_{\Gamma_i}(\wt{\cal G}_1^\delta)_{jj}(z)\dd z,
\ee
where $\wt{\cal G}_1^\delta$ is defined by introducing a perturbation to the $j$-th eigenvalue of $\Lambda$: 
$$ \wt{\cal G}_1^\delta(z):=\big(\Lambda_\delta^{1/2} V^\top X X^\top V \Lambda_\delta^{1/2}-z\big)^{-1},\quad \text{with}\quad \Lambda_\delta:=\Lambda+\delta \wt\sigma_j \mathbf e_j \mathbf e_j^\top. $$
%for some sufficiently small constant $\rho>0$, and set $\Gamma_i=h(\gamma_i).$ Under (iv) of Assumption \ref{main_assumption}, We can estimate the RHS by using the Woodbury matrix identity to write $\wt G$ in terms of $G$, and then controlling the $G$ entries with the local laws stated in \Cref{lem_standardaniso}. 
%For the RHS of \eqref{eq:uij_perturb}, utilizing the Woodbury matrix identity to express $\wt{\cal G}^\delta_1$ in terms of $\cal G_1$ and using the fact that $\cal G_1$ contains no pole inside $\Gamma_i$, we obtain that
%then control the entries of $G$ using the local laws described in \Cref{lem_standardaniso}. After performing this calculation, we obtain 
To compute the right-hand side (RHS) of \eqref{eq:uij_perturb}, we utilize the Woodbury matrix identity to express $\wt{\cal G}^\delta_1$ in terms of $\cal G_1$. Using the fact that $\cal G_1$ contains no pole inside $\Gamma_i$, we then obtain the following expression:
\begin{equation}\label{eq:outuij}
\langle \bu_i, \eb_j \rangle^2=\lim_{\delta \downarrow 0} \frac{1+\delta}{\delta^2} \frac{1}{2 \pi \ri} \oint_{\Gamma_i} \left[ \mathbf{D}_{r,j}(\delta)^{-1}+1+z\mathbf{E}_{r,j}^\top \mathcal{G}_1(z) \mathbf{E}_{r,j} \right]_{r+1,r+1}^{-1} \frac{\dd z}{z}. 
\end{equation}
where we introduce the matrices
\begin{equation*}
\mathbf{E}_{r,j}:=(\eb_1, \cdots, \eb_r,  \eb_{j}) \in \mathbb{R}^{p \times (r+1)}, \quad  \mathbf{D}_{r,j}(\delta):=\operatorname{diag}\left\{d_1, d_2, \cdots, d_r, \delta \right\} \in \mathbb{R}^{(r+1) \times (r+1)}.
\end{equation*}
For a detailed derivation of \eqref{eq:outuij}, we refer readers to Lemma 5.7 and equations (7.5) and (7.6) of \cite{DRMTA}.

% using the convention
% \begin{equation*}
% \mathsf{r}=r+1,
% \end{equation*}

%In what follows, we will focus on $\eb_j^\top \bu_i$ using the model $\Lambda^{1/2}V^\top X.$ Denote 
% In fact, analogous results have been proved for a slightly different model in \cite{bun2018optimal}.

Denote $\Lambda_{0,j}:=\operatorname{diag}\{\sigma_1, \cdots, \sigma_r, \sigma_j \}$, and  %and $\Delta(z) \equiv \Delta_j(z)$ as 
\begin{equation}\label{eq_deltaz}
\Delta_j(z):=-\Upsilon_{j}(z)-z\mathbf{E}_{r,j}^\top \mathcal{G}_1(z) \mathbf{E}_{r,j}, \quad \Upsilon_{j}(z):=\left(1+m(z) \Lambda_{0,j}\right)^{-1}.
\end{equation}   
By the local law \eqref{eq:aniso_out} and the estimate on $\im m(z)$ in \eqref{eq:Imm}, we have that uniformly in $z\in \Gamma_i$,
\be\label{eq:bddDeltaj}
\left\|\Delta_j(z)\right\|\prec n^{-1/2}.
\ee
Denote $\Xi_j(z):=\mathbf{D}_{r,j}^{-1}+1-\Upsilon_j(z).$ By item (iv) of \Cref{main_assumption} and the choice of $\rho$, we have that $\|\Xi_j(z)^{-1}\|\lesssim 1$ uniformly in $z\in \Gamma_i$. Then, applying Taylor expansion to 
$$\left( \mathbf{D}_{r,j}^{-1}+1 +z\mathbf{E}_{r,j}^\top \mathcal{G}_1(z) \mathbf{E}_{r,j} \right)^{-1}=\left( \Xi_j(z)-\Delta_j(z) \right)^{-1}$$ 
till order three, we can expand \eqref{eq:outuij} as
\begin{equation*}
 \langle \bu_i, \eb_j \rangle^2=s_1+s_2+s_3+s_4, 
\end{equation*}   
where $s_i$, $i=1,2,3,4,$ are defined as 
\begin{align*}
& s_1:= \lim_{\delta \downarrow 0} \frac{1+\delta}{\delta^2} \frac{1}{2 \pi \ri} \oint_{\Gamma_i} \frac{1}{1+\delta^{-1}-(1+m(z) \sigma_j)^{-1}}  \frac{\dd z}{z}, \\
& s_2:=\lim_{\delta \downarrow 0} \frac{1+\delta}{\delta^2} \frac{1}{2 \pi \ri} \oint_{\Gamma_i} \frac{\left( \Delta(z) \right)_{r+1,r+1} }{[1+\delta^{-1}-(1+m(z) \sigma_j)^{-1}]^2} \frac{\dd z}{z}, \\
& s_3:=\lim_{\delta \downarrow 0} \frac{1+\delta}{\delta^2} \frac{1}{2 \pi \ri} \oint_{\Gamma_i} \left[ \Xi_j^{-1}(z) \Delta(z) \Xi_j^{-1}(z) \Delta(z)  \Xi_j^{-1}(z) \right]_{r+1,r+1} \frac{\dd z}{z}, \\
& s_4:=\lim_{\delta \downarrow 0} \frac{1+\delta}{\delta^2} \frac{1}{2 \pi \ri} \oint_{\Gamma_i} \left[\left( \Xi_j^{-1}(z) \Delta(z)\right)^3 \left( \Xi_j(z)-\Delta(z) \right)^{-1} \right]_{r+1,r+1} \frac{\dd z}{z}.
\end{align*}
% Above we have used the short-hand notation for the diagonal matrix 
% \begin{equation*}
% \Xi_j(z):=\mathbf{D}_{r,j}^{-1}+1-\Upsilon_j(z).
% \end{equation*}
%For $s_4,$ the discussion is similar to the controlling of the error term in Proposition 7.2 of \cite{DRMTA}, or equivalently, Lemma 4.33 of \cite{bun2018optimal}. We omit the detail and conclude that under (iv) of Assumption \ref{main_assumption} that 
Using \eqref{eq:bddDeltaj}, it is easy to see that 
\begin{equation*}
|s_4| \prec n^{-3/2}. 
\end{equation*} 
%This shows (\ref{eq_intermediateresult}) and completes the proof of the lemma.
%Now we control the above four terms. 
For $s_1$ and $s_2,$ by following a similar argument as in \cite[Proposition 7.2]{DRMTA} or \cite[Lemma 4.31]{bun2018optimal}, we can show that for sufficiently small $\delta,$ the functions $z^{-1}$, $[1+\delta^{-1}-(1+m(z) \sigma_j)^{-1}]^{-1}$, and $(\Delta(z))_{r+1,r+1}$ are all holomorphic on and inside $\Gamma_i$ with high probability. Hence, applying Cauchy's integral theorem, we get that 
\begin{equation}\label{eq:s1s2}
s_1=s_2=0\quad \text{with high probability}.
\end{equation}

To estimate the leading term $s_3$, we employ a modified proof strategy similar to that of \cite[Lemma 4.32]{bun2018optimal}. By definition, we can write that 
\begin{align*}
\left[ \Xi_j^{-1}(z) \Delta(z) \Xi_j^{-1}(z) \Delta(z)  \Xi_j^{-1}(z) \right]_{r+1,r+1}= \frac{1}{ \left[ 1+\delta^{-1}-(1+m(z) \sigma_j)^{-1} \right]^2  } \sum_{k=1}^{r+1} \Delta_{r+1, k} \Delta_{k, r+1} \Xi_j(k,k)^{-1},
\end{align*}
where we have introduced the abbreviation $\Delta \equiv \Delta(z)$ and used $\Xi_j(k,k)$ to denote the $k$-th diagonal entry of $\Xi_j(z)$. We decompose the sum into a diagonal part with $k=r+1$ and an off-diagonal part with $k\in \qq{r}$:
$$ \sum_{k=1}^{r+1} \Delta_{r+1, k} \Delta_{k, r+1} \Xi_j(k,k)^{-1} =\mathsf{E}_d+\mathsf{E}_o , $$
% We point out that according to the definition of $\Delta(z)$ in (\ref{eq_deltaz}), we see that for $1 \leq k \leq r$
% \begin{equation*}
% \Delta_{r+1,k}=-z \eb_{j}^\top \mathcal{G}_1(z) \eb_k. 
% \end{equation*}
% Consequently, the above decomposition can be rewritten as 
% \begin{align*}
% \left( \Xi_j^{-1}(z) \Delta(z) \Xi_j^{-1}(z) \Delta(z)  \Xi_j^{-1}(z) \right)_{r+1,r+1}=\frac{1}{ \left( 1+\delta^{-1}-(1+m(z) \sigma_j)^{-1} \right)^2 } \left( \mathsf{E}_d+\mathsf{E}_o \right),
% \end{align*}
where $\mathsf{E}_d$ and $\mathsf{E}_o$ are defined as
%respectively the diagonal and off-diagonal terms, denoted as 
\begin{align*}
& \mathsf{E}_d:= \Delta^2_{r+1, r+1} \Xi_j(r+1,r+1), \quad  \mathsf{E}_o:=z^2 \sum_{k=1}^r \frac{ (\eb_{j}^\top \mathcal{G}_1(z) \eb_k)(\eb_{k}^\top \mathcal{G}_1(z) \eb_{j}) }{  1+d_k^{-1}-(1+m(z) \sigma_k)^{-1} }. 
\end{align*} 
By the discussion above \eqref{eq:s1s2}, we have that for sufficiently small $\delta$, $\mathsf{E}_d$ is holomorphic on and inside $\Gamma_i$ with high probability. This implies that 
\begin{equation*}
\frac{1}{2\pi \ii}\oint_{\Gamma_i} \frac{\mathsf{E}_d}{ \left[ 1+\delta^{-1}-(1+m(z) \sigma_j)^{-1} \right]^2} \frac{\dd z}{z}=0 \quad \text{with high probability}. 
\end{equation*}
%Recall $G(z)$ defined in (\ref{eq_defnfinalG}).
For the off-diagonal part $\mathsf{E}_o$, using (\ref{aniso}), we can rewrite it as
% by (\ref{aniso}), Remark \ref{rmk_keymodification} and spectral decomposition, we have
% \begin{equation*}
% G_{jk}(z) G_{kj}(z)=z^2 \sigma_j \sigma_k \eb_{j}^\top \mathcal{G}_1(z) \eb_k \eb_{k}^\top \mathcal{G}_1(z) \eb_{j}.  
% \end{equation*} 
% Consequently, for the off-diagonal term $\mathsf{E}_o,$ we have 
%In particular, we have
\begin{equation*}
\mathsf{E}_o= \frac{1}{\sigma_j}\sum_{k=1}^{r} \frac{1}{\sigma_k} \frac{G_{jk}(z) G_{kj}(z)}{ 1+d_k^{-1}-(1+m(z) \sigma_k)^{-1}}.
\end{equation*}

Under item (iv) of Assumption \ref{main_assumption}, using the definition of $m(z)$ in \eqref{eq_defnmc} and the definition of $\fa_i$ in (\ref{eq_importantdefinition}), we find that for $k\in \qq{r}\setminus \{i\}$,
$$1+d_k^{-1}-(1+m(\fa_i) \sigma_k)^{-1}=1+d_k^{-1}-\left(1- {\sigma_k}/\wt\sigma_i\right)^{-1}\gtrsim 1.$$
%$$|\fa_i-\fa_k|=|h(-\widetilde{\sigma}_i^{-1})-h(-\widetilde{\sigma}_k^{-1})| \gtrsim 1.$$ 
Consequently, due to the continuity of $m(z)$ near $\fa_i$, we can choose $\rho$ sufficiently small such that $1+d_k^{-1}-(1+m(z) \sigma_k)^{-1}$ has no zero on and inside $\Gamma_i$. Then, by Cauchy's integral theorem, we have that for $k\in \qq{r}\setminus \{i\}$,
\begin{equation*}
\oint_{\Gamma_i} \frac{G_{jk}(z) G_{kj}(z)}{ \left[ 1+\delta^{-1}-(1+m(z) \sigma_j)^{-1} \right]^2 \left[1+d_k^{-1}-(1+m(z) \sigma_k)^{-1}\right]}\frac{\dd z}{z}=0
\end{equation*}
with high probability. Finally, we only need to consider the $k=i$ term in $\mathsf E_o$. Applying the parametrization $z=h(\zeta),$ we obtain that with high probability, 
\begin{align*}
\frac{1+\delta}{\delta^2}\frac{1}{2\pi \ii}\oint_{\Gamma_i} \frac{\mathsf{E}_o}{ \left[ 1+\delta^{-1}-(1+m(z) \sigma_j)^{-1} \right]^2} \frac{\dd z}{z}&=\frac{(1+\delta)d_i}{2\pi \ii \widetilde{\sigma}_i \sigma_j \sigma_i }\oint_{\gamma_i} \frac{ h'(\zeta) G_{ji}(h (\zeta)) G_{ij} (h (\zeta))}{ h(\zeta) \left[ 1+\delta-\delta(1+\zeta \sigma_j)^{-1} \right]^2} \frac{1+\zeta \sigma_i }{\zeta+\widetilde{\sigma}_i^{-1}} \dd \zeta, \\
&=  (1+\delta)\frac{d_i^2\fb_i}{\widetilde{\sigma}_i^2 \sigma_j}   \frac{ G_{ji}(\fa_i) G_{ij} (\fa_i)}{\left[1+\delta-\delta (1-\sigma_j/\widetilde{\sigma}_i)^{-1}\right]^2},
%\\&=  \frac{d_i^2}{\widetilde{\sigma}_i^2 \sigma_j} \frac{h'(-\widetilde{\sigma}_i^{-1})}{h(-\widetilde{\sigma}_i^{-1})}  \frac{ G_{ji}(h (-\widetilde{\sigma}_i^{-1})) G_{ij} (h (-\widetilde{\sigma}_i^{-1}))}{\left(1+\delta^{-1}+(1-\sigma_j/\widetilde{\sigma}_i)^{-1} \right)^2},
\end{align*}
where we used the relation $m(h(\zeta))=\zeta$ as in (\ref{eq_defnmc}) in the first step, and the residual theorem at $\zeta=-\wt\sigma_i^{-1}$ in the second step. Finally, taking $\delta\to 0$ and using the dominated convergence theorem, we conclude that
\begin{equation*}
s_3= \frac{d_i^2\fb_i}{\widetilde{\sigma}_i^2 \sigma_j}  G_{ji}(\fa_i) G_{ij} (\fa_i),
\end{equation*} 
% \begin{equation*}
% s_3= \frac{d_i^2}{\widetilde{\sigma}_i^2 \sigma_j} \frac{h'(-\widetilde{\sigma}_i^{-1})}{h(-\widetilde{\sigma}_i^{-1})}  G_{ji}(h (-\widetilde{\sigma}_i^{-1})) G_{ij} (h (-\widetilde{\sigma}_i^{-1})). 
% \end{equation*} 
which completes the proof of \eqref{eq_intermediateresult}.
% For $s_4,$ the discussion is similar to the controlling of the error term in Proposition 7.2 of \cite{DRMTA}, or equivalently, Lemma 4.33 of \cite{bun2018optimal}. We omit the detail and conclude that under (iv) of Assumption \ref{main_assumption} that 
% \begin{equation*}
% |s_4| \prec n^{-3/2}. 
% \end{equation*} 
% This shows (\ref{eq_intermediateresult}) and completes the proof of the lemma.
% \textcolor{blue}{check the index more carefully, some $r+1$ should be repelaced by $j$.}
%Using the definition   (\ref{eq_defnmc}), we notice that $$ (1+\sigma_i (h(-\widetilde{\sigma}_i^{-1})))^{-1}=1+d_i^{-1}. $$ 
%Finally, we prove . {\color{red} finish here}
\end{proof}

\subsection{Local laws for the rectangular matrix Dyson Brownian motion}\label{sec: iso of Wt} 

% rectangular matrix DBM for the non-spiked model.  
% In this subsection, we prove the results for the rectangular matrix DBM as introduced in (\ref{defn_wt}) and (\ref{eq_qt}) for the non-spiked model that $D=\Lambda_0$. 
% We rewrite the equation (\ref{eq_defnmc}) as 
%  \be\label{eq_selfconsistent} 
% z=-\frac{1}{m }+\frac1n\tr\frac{\Lambda_0}{1+m \Lambda_0}. 
% \ee 

In this subsection, we consider the rectangular matrix DBM of the {\bf non-spiked} model:
$$ W_0(t):=W_0+\sqrt{t} X^G,\quad \text{with}\quad W_0:=\Lambda_0^{1/2}V^\top X,$$
where $X^G$ is defined above \eqref{defn_wt}. In this subsection, we present the local laws for its resolvent, which will be utilized in our proof. From these local laws, we can also derive the local laws for the resolvent of the (spiked) rectangular matrix DBM defined in (\ref{defn_wt}) and (\ref{eq_qt}).  

Let us begin by introducing some notations. We define $\Lambda_t:=\Lambda_0+t$ and $m_t(z):=m(z,\Lambda_t)$ as given by \eqref{eq_defnmc}. In other words, $m_t(z)$ is the unique solution to the equation 
\be\label{koux}
z=- \frac{1}{m_t}+\frac1n\tr\frac{\Lambda_t}{1+m_t  \Lambda_t} 
\ee
subject to the condition that $\im m_t(z)\ge 0$ whenever $\im z\ge 0$. Subsequently, we denote the density function associated with $m_t$ as $\varrho_t$, and refer to the spectral edges and quantiles of $\varrho_t$ as $a_k(t)$ and $\gamma_k(t)$, following the definitions in \Cref{lem_property} and in \Cref{defn_generaleigenvaluelocation}. 
Using the stability of the self-consistent equation \eqref{koux} established in \cite{MR3704770}, it is not hard to check that \be\label{eq:perturbrho}
\sup_{x\in \R} |\varrho_t(x)-\varrho(x)|=\oo(1)
\ee
when $t=\oo(1)$.  
%Then, we let $\varrho_t$ be the probability density function associated with $m_t$, and denote by $\gamma_k(t)$ the quantiles of $\varrho_t$ defined as in \Cref{defn_generaleigenvaluelocation}.  
%with respect to the rectangular matrix DBM in (\ref{def:gamma_alpha}) using $\varrho_t$.  Let $\lambda_{+,t} \equiv \gamma_{1}(t)$ and $\lambda_{-,t} \equiv \gamma_{\mathsf{K}}(t)$. 
Then, we define $\cal G_{1,t}$, $\cal G_{2,t}$, $G_t$, and $\Pi_t$ as described in \Cref{rmk_keymodification}, with the replacements of $W_0$, $\Lambda_0$, and $m(z)$ by $W_0(t)$, $\Lambda_t$, and $m_t(z)$, respectively.  
Let $\{\lambda_k(t)\}_{k=1}^p$ and $\{\bu_k(t)\}_{k=1}^p$ denote the eigenvalues and corresponding eigenvectors of $W_0(t)W_0(t)^\top$. 
%Let $U_0(t):=(\bu_{\mathsf K+1}(t),\ldots, \bu_p(t))\in \R^{p\times (p-n)}$.
Finally, we introduce the error control parameter
\begin{equation*} 
  \Psi_t(z):= \sqrt{\frac{\im m_t(z)}{n\eta}}+\frac{1}{n\eta},
\end{equation*}
Now, we are ready to state the main result of this subsection.   
%  We also define
% \begin{equation}\label{eq_mnzdefinition}
% m_{n,t}(z):=\frac{1}{n} \operatorname{Tr} (W(t)^\top W(t)-z)^{-1}.
% \end{equation}
% and $\mathcal{G}_{1,t}(z):=(Q(t)-z)^{-1}$ for $Q(t)$ defined in (\ref{eq_qt}). 

%The above self-consistent equation \eqref{eq_selfconsistent} naturally generalizes to the rectangular matrix DBM, namely at time $t$ we define 
% \be\label{eq_selfconsistent} 
%m:=m_{D}:=  m_{D} (t, z): \quad \frac{1}{m }=-z+\frac1N\tr\frac{D }{1+m  D } ,\quad \Pi_D:= - D(1+m D)^{-1}
%\ee 
%We denote
%\begin{equation*}
%\phi \;\deq\; \frac{M}{N}\,, \qquad K \;\deq\; M \wedge N\,.
%\end{equation*}
% \section{Anisotropic law and rigidity of $W(t)$}
 %As above, in this section,  Let $W(t)$ be an $M \times N$ matrix of independent standard Brownian motions starting at $W_0$, and set $X(t) \deq W(t) / \sqrt{N}$. 
% $m_t=m_t(z)$ via  
% \be\label{koux}
%  \quad \frac{1}{m_t }=-z+\frac1n\tr\frac{D_t}{1+m_t  D_t}, \quad \quad D_t=D +t. 
% \ee
%We point out that (\ref{eq_selfconsistent}) and (\ref{koux}) differs only in $D_t.$  This can be easily understood since when $X$ is Gaussian, $D_t$ is the covariance matrix of $W(t).$ 
%The following estimates in Lemma \ref{sare} for $W(t)$ will be frequently used in later calculations. 
% At a generic time $t,$ define a filtration as follows
%\begin{equation*}
%\cal F_t \;\deq\; \sigma \pb{W(0),\,  (\bm{\lambda}(s))_{s \in [0,t]}}\, ,  \ \bm{\lambda}(s)=(\lambda_i(s))_{1 \leq i \leq K}.
%\end{equation*}

\begin{lemma} \label{sare} 
Suppose (i)--(iii) of Assumption \ref{main_assumption} hold. Fix a time $T \asymp n^{-1/3+\mathsf{c}}$ for some constant $ \mathsf{c} \in (0,1/6).$  The following estimates hold uniformly in $z\in  \mathbf{D}$ and $ t\in [0,T]$.
\begin{enumerate}
\item The support of $\varrho_t$ has the same number of connected components as $\varrho_t$, and the spectral edges $a_k(t)$ of $\varrho_t$ are perturbations of those of $\varrho$: $ \max_{k=1}^{2q}|a_k(t)-a_k|=\oo(1).$ In particular, it implies that $|\lambda_{\pm,t}-\lambda_{\pm}|=\oo(1)$, where $\lambda_{+,t}:=a_1(t)$ and $\lambda_{-,t}:=a_{2q}(t)$ denote the rightmost and leftmost spectral edges of $\varrho_t$, respectively.   
The number of quantiles within each component remains unchanged, that is, $n_{k}(t)$ defined as in \eqref{Nk} is constant in $t$.

\item  For any deterministic unit vectors $\bu, \bv  \in \mathbb{R}^{p+n}$, we have the anisotropic local law (recall \Cref{lem_standardaniso})
 \be\label{TAZ}
  \left| \bu^\top{G}_{t}(z) \bv -    \bu^\top \Pi_t(z) \bv  
  \right|\prec \Psi_t(z).
 \ee
%The following anisotropic and averaged local laws For all $z\in  \mathbf{D}$ in (\ref{eq_generalspectraldomain}),  we have 
 % \be\label{TAZ}
 %  \left| \bv_1^\top\mathcal{G}_{1,t}(z) \bv_2-    \bv_1^\top \Big(\frac{-1}{z(1+m_t(z) D_t)} \Big) \bv_2  
 %  \right|\prec \Psi_t(z),
 % \ee
For $\sm_{n,t}:=n^{-1}\tr \cal G_{2,t}$, the following averaged local law holds:
 \begin{equation}\label{eq_averagerightblock}
|\sm_{n,t}(z)-m_t(z)|\prec \frac{1}{n \eta}. 
 \end{equation}
 
\item For $\gamma_j(t) \in [a_{2k}(t),a_{2k-1}(t)]$, we have that (recall \Cref{lem_rigidity})
%we have {\cob rigidity}
  \be\label{TAZ2}
 \left|\lambda_j(t)-\gamma_j(t)\right|\prec [(n_{k}(t)+1-j)\wedge (j-n_{k-1}(t))]^{-1/3}n^{-{2}/{3}}.
% n^{-2/3} \left( \min\{k, \mathsf{K}-k+1\} \right)^{-1/3}+\mathfrak{g}(\gamma_k(t)), 
% , \quad   0\le t\le T 
 \ee
 % where $\mathfrak{g}(x) \equiv \mathfrak{g}_t(x)$ is defined that 
 % \begin{equation*}
 % n\mathfrak{g}_t(x) \left(t+\sqrt{\kappa(x)+\mathfrak{g}_t(x)} \right)=1, \quad \kappa(x)=\operatorname{dist}(x,\supp(\varrho_t)). 
 % \end{equation*} 

\item For $k \in \qq{\sfK}$ and any deterministic unit vector $\bv\in \R^{p}$, we have that (recall \Cref{lem_bulkproperty})
  \be\label{TAZ3}
\left|\avg{\bu_k(t),\bv}\right|^2 \prec n^{-1}. %\quad \|U_0(t)^\top \bv \|_2^2 \prec 1.
  \ee
  \end{enumerate}
  \end{lemma}
\begin{proof}
Under the edge regularity condition \eqref{eq:edgeregular}, it is easy to see that $a_k(t)$ and $n_k(t)$ are continuous in $t$ when $t=\oo(1)$. This establishes part (i) of the lemma, taking into account the fact that $n_k(t)$ takes integer values. For the other parts, we notice that $W_0(t)$ can be expressed as 
    \begin{equation}\label{eq_wtgeneralform}
W_0(t)= \Lambda_t^{1/2} \mathbf U  \begin{pmatrix}
X  \\
 X^G  
\end{pmatrix},
\end{equation}
where $\mathbf{U} \in \mathbb{R}^{p \times 2p}$ is a semi-orthogonal matrix defined as
\begin{equation*}
\mathbf{U}:=  \begin{pmatrix}
(\Lambda_0/\Lambda_t)^{1/2} V^\top,   (t/\Lambda_t)^{1/2}I_p 
\end{pmatrix} . 
\end{equation*}
This is the random matrix model studied in \cite{MR3704770}. Furthermore, \eqref{eq:perturbrho} and part (i) together validate the regularity conditions in (\ref{eq:edgeregular}) and \eqref{eq:bulkregular} for $\Lambda_t$. Then, for each fixed $t$, (ii) is proved in Theorems 3.6 and 3.7 of \cite{MR3704770}, (iii) is proved in Theorem 3.12 of \cite{MR3704770}, and (iv) follows easily from (ii) and (iii) as explained, for example, in \cite[Theorem 2.8]{isotropic}. 
Finally, by utilizing a standard $\epsilon$-net argument, we can extend these results uniformly to all $t\in [0,T]$.
\end{proof}

\begin{remark}\label{rmk_m1m2law}
The estimate (\ref{eq_averagerightblock}) provides an averaged local law for $\sm_{n,t}.$ In the following proof, we will also need an averaged local law for $\underline \sm_{n,t}(z):=p^{-1}\tr \cal G_{1,t}$, whose classical limit is given by
$$ m_{1,t}(z):= -\frac{1}{p}\tr\frac{1}{z(1+m_{t}(z)\Lambda_t)}.$$
From \eqref{koux}, we can observe that 
\begin{equation}\label{eq:m1tz}
m_{1,t}(z)=\frac{c_n^{-1}-1}{z}+c_n^{-1} m_t(z). 
\end{equation}
On the other hand, since $W_0(t)W_0(t)^\top$ share the non-zero same eigenvalues as $W_0(t)^\top W_0(t)$, we have a similar relation between $\underline \sm_{n,t}(z)$ and $\sm_{n,t}(z)$: 
\begin{equation}\label{eq:m1tz2}
\underline\sm_{n,t}(z)=\frac{c_n^{-1}-1}{z}+c_n^{-1} \sm_{n,t}(z). 
\end{equation}
Under these two equations, we can directly deduce from \eqref{eq_averagerightblock} that 
\be\label{eq_average2}
|\underline\sm_{n,t}(z)-m_{1,t}(z)|\prec (n\eta)^{-1}.
\ee
% Moreover, according to \eqref{koux}, 
% % (\ref{TAZ}), 
% we also have the following relation
% \begin{equation*}
% m_{1,t}(z)=\frac{-1}{zp} \operatorname{Tr} \frac{1}{1+m_t D_t}.  
% \end{equation*}
\end{remark} 

Next, we consider the rectangular matrix DBM of the {\bf spiked} model:
$$ W(t):=W+\sqrt{t} X^G,\quad \text{with}\quad W:=\Lambda^{1/2}V^\top X,$$
where $\Lambda$ contains a non-spiked component $\Lambda_0$ and $r$ spikes as in \eqref{eq_defnspikes}. Let $\wt\Lambda_t:=\Lambda+t$, and we again denote $\Lambda_t=\Lambda_0+t$ and $m_t(z)=m(z,\Lambda_t)$. 
Then, we define \smash{$\wt G_t$ and $\wt \Pi_t$} as described in \Cref{rmk_keymodification}, with the replacements of $W_0$, $\Lambda_0$, and $m(z)$ by $W(t)$, \smash{$\wt\Lambda_t$}, and $m_t(z)$, respectively.   
Let \smash{$\{\wt\lambda_k(t)\}_{k=1}^p$} denote the eigenvalues of $W(t)W(t)^\top$. We define a smaller spectral parameter domain than \eqref{eq_generalspectraldomain} for arbitrarily small constants $c,\tau\in (0,1)$:
\begin{equation}\label{eq_generalspectraldomain_small}
\mathbf{D}' \equiv \mathbf{D}'(c, \tau, t, n):=\left\{z=E+\ri \eta \in \mathbb{C}_+: \lambda_- - c\leq E \leq \lambda_+ +c, n^{-1+\tau} \leq \eta \leq \tau^{-1}    \right\},
\end{equation}
We show that a similar local law as in \eqref{TAZ} holds for $\wt G_t-\wt\Pi_t$.
\begin{lemma} \label{sare2}
Suppose Assumption \ref{main_assumption} holds. Fix a time $T \asymp n^{-1/3+\mathsf{c}}$ for some constant $ \mathsf{c} \in (0,1/6).$  Then, there exists a constant $c>0$ such that the following local law holds uniformly for $z\in  \mathbf{D}'$ and $ t\in [0,T]$,  considering any deterministic unit vectors $\bu, \bv  \in \mathbb{R}^{p+n}$:
 \be\label{wtTAZ}
  \big| \bu^\top\wt{G}_{t}(z) \bv -    \bu^\top \wt\Pi_t(z) \bv  
  \big|\prec \Psi_t(z).
 \ee
Moreover, for $ i \in \qq{\sfK-r}$, we have the eigenvalue sticking estimate:
\begin{equation}\label{eq_eigenvaluesticking2}
|\wt\lambda_{i+r}(t)-\lambda_i(t)|\prec n^{-1}.
\end{equation}
\end{lemma}
\begin{proof}
For the proof of \eqref{wtTAZ}, the only concern arises for $z$ around the $r$ outlier locations. However, by confining ourselves to a narrower spectral domain $\mathbf{D}'$ around $\supp(\varrho_t)$, we ensure that $z$ stays at a distance of order $\Omega(1)$ from the outliers. As a result, \eqref{wtTAZ} is essentially a consequence of the local law \eqref{TAZ}. The proof of \eqref{eq_eigenvaluesticking2} is the same as that of \eqref{eq_eigenvaluesticking}. We omit the details.
%But, by restricting to a smaller spectral domain $\mathbf D'$ around $\supp(\varrho_t)$, we are away from the outliers by a distance $\gtrsim 1$. Then, \eqref{TAZ} essentially follows directly from \eqref{wtTAZ}.  The proof of \eqref{eq_eigenvaluesticking2} is the same as that of \eqref{eq_eigenvaluesticking}. 
\end{proof}

\section{Proof of the main results} 
%This section establishes the main results discussed in Section \ref{s:main}. 
 
\subsection{Proof of the results in Section \ref{sec_statisticsmainresults}}

% prove the results of Section \ref{sec_statisticsmainresults}, i.e.,  
In this subsection, we provide the proofs of Theorem  \ref{thm_shrinkerestimate}, \Cref{coro_simplifiedformula}, and \Cref{coro_lowbound} using the local laws presented in Appendix \ref{appendix_preliminary} and \Cref{corollary_que}. 

\begin{proof}[\bf{Proof of Theorem \ref{thm_shrinkerestimate}}]
%We first prove part (i). 
Applying \Cref{corollary_que} with $\omega_j=\ell(\sigma_j)$ and noticing $p^{-1}\sum_{j=1}^{r} \ell(\sigma_j) \phi(\bm{v}_j, \bm{v}_j, \gamma_{i-r}) =\OO(p^{-1}),$ we can deduce that  
$$ \P\left(\left| \bm{u}_i^\top \ell(\Sigma) \bm{u}_i- \vartheta(\ell,\gamma_{i-r}) \right|>n^{-\fd/4}\right)\le n^{-\fd/2},\quad i\in\qq{r+1,\sfK}, $$
for a constant $\fd>0$. By combining this estimate with the trivial bound $\left| \bm{u}_i^\top \ell(\Sigma) \bm{u}_i- \vartheta(\ell,\gamma_{i-r}) \right|\lesssim 1$, we conclude \eqref{eq_nonspikedshrinkeranalytical}. 
To show \eqref{eq_nonspikedshrinkeranalytical0}, let $\Gamma_0\equiv\Gamma_0(n)$ be the contour centered at $0$ with radius $\rho_n=(\log n)^{-1}/2$. By the rigidity estimate \eqref{rigidity2}, the eigenvalue sticking \eqref{eq_eigenvaluesticking}, and the lower bound \eqref{eq:lambda-}, we know that with high probability, $\Gamma_0$ only encloses the $p-n$ zero eigenvalues of $\cal Q_1$. 
%according to the condition \eqref{eq_ratioassumption2} we are allowed to treat $0$ as an outlier in the spectrum of the sample covariance matrix $\mathcal{Q}_1$. Let $\Gamma_0\equiv\Gamma_0(n)$ be the contour centered at $0$ with radius $0<\rho_n<(\log n)^{-1}$. 
Then, with the spectral decomposition of $\wt\sigG_1$ (recall \eqref{eq_originalG}), we obtain that with high probability,
\begin{align*}
    \tr\left[U_0^\top \ell(\Sigma) U_0\right] = -\frac{1}{2\pi \ri}\oint_{\Gamma_0} \tr\big[\wt\sigG_1(z)\ell(\Sigma)\big] \dd z.
\end{align*} 
%where we recall that $\mathsf{G}_1(z)=(\mathcal{Q}_1-z)^{-1}$. 
Applying the anisotropic law in \Cref{lem_standardaniso0} and using Cauchy’s integral theorem at $z=0$, we obtain that 
\begin{align*}
     \tr\left[U_0^\top \ell(\Sigma) U_0\right] &= \frac{1}{2\pi \ri}\oint_{\Gamma_0} \tr\left[(1+m(z)\Sigma)^{-1}\ell(\Sigma)\right]\frac{\dd z}{z} + \OO_\prec(n^{-1/2}) \\
     &= \sum_{j=1}^p \frac{\ell(\wt\sigma_j)}{1+m(0)\wt\sigma_j}+ \OO_\prec(n^{-1/2}) .
\end{align*}
This yields \eqref{eq_nonspikedshrinkeranalytical0} by noticing that
%is completed by an application of the Markov's inequality with the definition that
\begin{align*}
    \vartheta(\ell,0) = \frac{1}{p-n}\sum_{j=r+1}^p \frac{\ell(\wt\sigma_j)}{1+m(0)\wt\sigma_j}=\frac{1}{p-n}\sum_{j=1}^p \frac{\ell(\wt\sigma_j)}{1+m(0)\wt\sigma_j}+\OO(n^{-1}).
\end{align*}
%\end{proof}

%Next, we focus on part (ii).
%\begin{proof}[\bf Proof of Part (ii)]  
Now, we focus on the proof of \eqref{eq_outliercase}. 
Recall the decomposition \eqref{eq_decomposition}: 
\begin{align}\label{eq_twoequationsseperate}
\bm{u}_i^\top \ell(\Sigma) \bm{u}_i &=\sum_{j=1}^r \ell(\widetilde{\sigma}_j) |\avg{\bm{u}_i, \bm{v}_j}|^2 +\sum_{j=r+1}^p \ell(\widetilde{\sigma}_j) |\avg{\bm{u}_i, \bm{v}_j}|^2:=\mathsf{w}_{1i}+\mathsf{w}_{2i},\quad i \in \qq{r}.   
\end{align}
%When $i\ge r+1$, we can establish that $\mathsf{w}_{1i} \prec n^{-1}$ based on the delocalization bound \eqref{eq_eigenvectordelocalization}. For $\mathsf{w}_{2i},$ we can apply the same approach as in part (i) by utilizing \Cref{corollary_que} and the anisotropic law from \Cref{lem_standardaniso0}. This concludes \eqref{eq_nonspikedshrinkeranalytical} and \eqref{eq_nonspikedshrinkeranalytical0} with $\vartheta_i\equiv \vartheta_i(\ell, r+1)$.  
%It remains to prove \eqref{eq_outliercase} for $i \in \qq{r}$. 
Using \eqref{eq:outlier_proj}, we obtain that
%We partition the proof into two cases based on whether $1 \leq i \leq r$ or $r+1 \leq i \leq p$.
%\medskip 
%\noindent {\bf Case I: $i \in \qq{r}$.} 
%since $r$ is bounded, we conclude that 
\begin{equation}\label{w1i}
\mathsf{w}_{1i}=\frac{\ell(\widetilde{\sigma}_i)}{\widetilde{\sigma}_i} \mathfrak{b}_i+\OO_{\prec}( n^{-1/2}). 
\end{equation}
The remaining part of the proof is devoted to estimating $\mathsf{w}_{2i}.$ Since $\widetilde{\sigma}_j=\sigma_j$ for $j \geq r$, by (\ref{eq_intermediateresult}), we have  
\begin{align}
\mathsf{w}_{2i}&=\mathfrak{c}_i \sum_{j=r+1}^p \frac{\ell(\sigma_j)}{\sigma_j} G_{ij}(\mathfrak{a}_i) G_{ji} (\mathfrak{a}_i)+\OO_{\prec}(n^{-1/2})\nonumber \\   
&=\mathfrak{c}_i \sum_{j=1}^p \frac{\ell(\sigma_j)}{\sigma_j} G_{ij}(\mathfrak{a}_i) G_{ji} (\mathfrak{a}_i) - \mathfrak{c}_i \sum_{j=1}^r \frac{\ell(\sigma_j)}{\sigma_j} G_{ij}(\mathfrak{a}_i) G_{ji} (\mathfrak{a}_i) +\OO_{\prec}(n^{-1/2}).\nonumber
\end{align}
Applying \eqref{eq:aniso_out} to estimate the second sum on the RHS and recalling the notations in \Cref{rmk_keymodification}, we can rewrite the above equation as 
%Since $r$ is finite, using the local laws in Lemma \ref{lem_standardaniso}, \eqref{aniso} and Remark \ref{rmk_keymodification}, we further have that 
\begin{equation}\label{eq_w2iform}
\mathsf{w}_{2i}=\mathfrak{c}_i \mathfrak{a}_i^2 \sigma_i \sum_{j=1}^p \ell(\sigma_j) (\mathcal{G}_1)_{ij}(\mathfrak{a}_i) (\mathcal{G}_1)_{ji} (\mathfrak{a}_i)- \frac{\mathfrak{c}_i  \sigma_i\ell(\sigma_i)}{[1+m(\mathfrak{a}_i) \sigma_i]^2} +\OO_{\prec}(n^{-1/2}). 
\end{equation}

It remains to estimate the quantity
\begin{equation}\label{eq_defLikeycalculations}
\mathcal{L}_i=\sum_{j=1}^p \ell(\sigma_j) (\mathcal{G}_1)_{ij}(\mathfrak{a}_i) (\mathcal{G}_1)_{ji} (\mathfrak{a}_i).
\end{equation}
%where $\mathfrak{a}_i$ is outside the bulk of $\varrho$ by (iv) of Assumption \ref{main_assumption}. 
For a small $\mathsf{t}>0,$ we define a new resolvent:
\begin{equation}\label{eq:Gstz}
\mathcal{G}(\mathsf{t},z):=\left(\Lambda_0^{1/2}V^\top X X^\top V \Lambda_0^{1/2}-z-\mathsf{t} L\right)^{-1},  
\end{equation} 
where $L$ is a diagonal matrix defined as $L=\operatorname{diag}\{\ell(\sigma_1), \cdots, \ell(\sigma_p)\}.$ 
% By the chain rule, we have
% \begin{equation*}
% \frac{\mathrm{d} \mathcal{G}(\mathsf{t})}{\mathrm{d} \mathsf{t}}=\mathcal{G}(\mathsf{t}) L \mathcal{G}(\mathsf{t}), 
% \end{equation*}
Under this definition, we notice that 
% In view of (\ref{eq_defLikeycalculations}), using the continuity of $\mathcal{G}(\mathsf{t})$ with respect to $\mathsf{t},$ it directly yields that 
\begin{equation}\label{eq_Lireducedform}
\mathcal{L}_i= \left. \frac{\partial \mathcal{G}_{ii}(\st,\fa_i)}{\partial \mathsf{t}} \right|_{\mathsf{t}=0} .
\end{equation}
%Since $\mathfrak{a}_i$ is outside the bulk of $\varrho,$ as in (4.17) of \cite{bao2021singular}, it suffices to establish the local laws for $\mathcal{G}(\sf{t})$ with small $\mathsf{t}>0.$ 
To estimate \eqref{eq_Lireducedform}, it suffices to establish a local law on $\mathcal{G}_{ii}(\mathsf{t},\fa_i)$ for any small $\mathsf{t}>0$. We rewrite $\mathcal{G}(\mathsf{t},\fa_i)$ as 
\begin{align*}
\mathcal{G}(\mathsf{t},\fa_i)&=A(\mathsf{t})^{-1/2}\mathfrak{G}(\mathsf{t},\fa_i) A(\mathsf{t})^{-1/2},
\end{align*}
where $A(\mathsf{t}):=I+\mathsf{t}L/\fa_i$ and 
$$\mathfrak{G}(\mathsf{t},z):=\left( \Lambda(\mathsf{t})^{1/2} V^\top X X^\top V \Lambda(\mathsf{t})^{1/2}-z \right)^{-1}\quad \text{with}\quad \Lambda(\mathsf{t}):= A(\mathsf{t})^{-1} \Lambda_0 .$$
% \begin{align*}
% \mathcal{G}(\mathsf{t},\fa_i)& =A(\mathsf{t})^{-1/2}\left( \Lambda(\mathsf{t})^{1/2} V^\top X X^\top V \Lambda(\mathsf{t})^{1/2}-z \right)^{-1}A(\mathsf{t})^{-1/2} \\ 
% & :=A^{-1/2}(\mathsf{t})\mathfrak{G}(\mathsf{t}) A^{-1/2}(\mathsf{t}),
% \end{align*}  
% where $\Lambda(\mathsf{t})$ and $A(\mathsf{t})$ are diagonal matrices that are defined as 
% \begin{equation*}
% A(\mathsf{t}):=I+\frac{\mathsf{t}}{z} L, \  \Lambda(\mathsf{t})= A^{-1}(\mathsf{t}) \Lambda_0.  
% \end{equation*}
Now, define $\widetilde m_\st(z,\fa_i)\equiv m(z,\Lambda(\st))$ as in (\ref{eq_selfconsistent_mt}) with $x=\fa_i$. 
As $\mathsf{t}\to 0$, we can utilize the local law \eqref{eq:aniso_out} to $\mathfrak{G}(\mathsf{t},z)$ at $z=\fa_i$, and substitute $\Lambda_0$ and $m(z)$ with $\Lambda(\sf{t})$ and $\widetilde m_\st(z,\fa_i)$, respectively. This leads to that  
%sufficiently small we can apply the discussions in Remark \ref{rmk_keymodification} to $\mathcal{H}(\mathsf{t})$ by replacing $\Lambda_0$ with $\Lambda(\sf{t})$ and $m(z)$ with $\widetilde m_\st(z)$. For $z=\mathfrak{a}_i,$ this results in 
% \begin{equation*}
% \mathcal{H}_{ii}(\st)=-\frac{1}{z(1+\widetilde m_\st \Lambda_{ii}(\st))}+\OO_{\prec}(n^{-1/2}).  
% \end{equation*}
% This further implies that 
\begin{align}\label{eq_keycalculations}
\mathcal{G}_{ii}(\st,\fa_i)& =- \frac{A_{ii}(\st)^{-1}}{\fa_i [1+\widetilde m_\st(\fa_i,\fa_i) \Lambda_{ii}(\st)]}+\OO_{\prec}(n^{-1/2})
% \nonumber \\ 
% &=-\frac{(1+\frac{\st}{z} \ell(\sigma_i))^{-1}}{z(1+\widetilde m_\st \sigma_i (1+\frac{\st}{z}\ell(\sigma_i))^{-1})}+\OO_{\prec}(n^{-1/2}). 
\end{align}
for all $\st=\oo(1)$. 
%Using the continuity of $h(\cdot)$ in (\ref{eq_defnmc}) and inverse function theorem, we find that {\color{red}[this is not correct in general]}
%\begin{equation*}
%m^{'}_\st \equiv \frac{\mathrm{d} \widetilde m_\st(z)}{\mathrm{d} \st} \Big|_{\st=0}=m'(z) \equiv \frac{\mathrm{d} m(z)}{\mathrm{d} z}. 
%\end{equation*}
We now take $\st=n^{-1/4}$. Using the eigenvalue rigidity \eqref{rigidity2} for $\lambda_1$ and the fact that $\fa_i$ is an outlier of $\varrho$, we get from the definition \eqref{eq:Gstz} that $ \|\mathcal{G}(\st,\fa_i)\|\lesssim 1$ with high probability. With this bound, we readily derive that with high probability,
\be\label{eq_keycalculations2}
\cal L_i - \frac{\mathcal{G}_{ii}(\st,\fa_i)-\mathcal{G}_{ii}(0,\fa_i)}{\st}=\OO(\st).
\ee
Combining \eqref{eq_Lireducedform}, \eqref{eq_keycalculations} and \eqref{eq_keycalculations2}, we obtain that 
%By differentiating with respect $\st$ and applying Cauchy’s integral formula, we obtain from (\ref{eq_Lireducedform}) that 
\begin{align*}
\mathcal{L}_i&=\left. - \frac{\partial}{\partial \st}\frac{A_{ii}(\st)^{-1}}{\fa_i [1+\widetilde m_\st(\fa_i,\fa_i) \Lambda_{ii}(\st)]}\right|_{\st=0}+\OO_\prec (\st+n^{-1/2}/\st )\\
&=\frac{1}{[1+m(\mathfrak{a}_i) \sigma_i]^2} \left( \frac{\dot m_0(\mathfrak{a}_i)\sigma_i}{\mathfrak{a}_i}+\frac{\ell(\sigma_i)}{\mathfrak{a}_i^2} \right) +\OO_{\prec} (n^{-1/4} ).
\end{align*} 
%where with a bit of abuse of notation we denote $\widetilde m'_0(\mathfrak{a}_i):=\frac{\partial }{\partial \st}\widetilde m_\st(\mathfrak{a}_i)\Big|_{\st =0}$ as in \eqref{eq_mtderivative}.
Plugging it back into (\ref{eq_w2iform}) gives an estimate for $\mathsf{w}_{2i}$. Combining this estimate with \eqref{w1i}, we obtain that 
%\eqref{eq_outliercase} as desired.
%resulting expression together with recalling \eqref{eq_twoequationsseperate} yields \eqref{eq_outliercase} as desired. 
\begin{align*}
    \bm{u}_i^\top \ell(\Sigma) \bm{u}_i = \mathfrak{b}_i\left(\frac{\ell(\widetilde{\sigma}_i)}{\widetilde{\sigma}_i} +\mathfrak{a}_i \dot m_0(\mathfrak{a}_i)\right)+\OO_{\prec} ( n^{-1/4} ),
\end{align*}
where we used that $\mathfrak{c}_i= {d_i^2\mathfrak{b}_i}/{\widetilde{\sigma}_i^2} $ and $m(\fa_i)=-\wt\sigma_i^{-1}$. This concludes \eqref{eq_outliercase}.
% Together with (\ref{eq_w2iform}), we can conclude our proof.
\end{proof}

\begin{proof}[\bf Proof of Corollary \ref{coro_simplifiedformula}]  
By \Cref{thm_shrinkerestimate}, it suffices to simplify the formulas of $\vartheta(x)\equiv \vartheta(\ell,x)$ and $\psi_i\equiv \psi_i(\ell)$ when $\ell(x)=x$. When $x=0$, we notice that %when $i\in\qq{\sfK+1,q}$,   
\begin{equation*}
\vartheta(0)=\frac{1}{p (1-c_n^{-1})} \sum_{j=r+1}^p \frac{\wt\sigma_j}{1+m(0) \wt\sigma_j}=\frac{1}{c_n-1} \frac{1}{n} \sum_{j=1}^p \frac{\sigma_j}{1+m(0) \sigma_j} + \OO(n^{-1}).
\end{equation*}
By \eqref{eq_defnmc}, we have the identity
\begin{equation*}
\frac{1}{m(0)}=\frac{1}{n} \sum_{j=1}^p \frac{\sigma_j}{1+m(0) \sigma_j}.
\end{equation*}
This proves \eqref{eq_nonspikeduniform0}. 
% According to (\ref{eq_defnmc}), under Assumption \ref{main_assumption}, we have that 
% \begin{equation*}
% \frac{1}{m(0)}=\frac{1}{n} \sum_{j=1}^p \frac{\sigma_j}{1+m(0) \sigma_j}.
% \end{equation*}
% This implies that when $\gamma_i=0$ (or $n+1 \leq i \leq p$ for $p>n$), we have that
% \begin{equation*}
% \vartheta_i=\frac{1}{(c_n-1) m(0)}. 
% \end{equation*}
Next, when $i\in\qq{r+1,\sfK}$, we have $\gamma_{i-r}>0$ and
\begin{equation}\label{eq_genegenralgenral}
\vartheta(\gamma_{i-r})=\frac{c_n}{p \gamma_{i-r}} \sum_{j=r+1}^p \frac{\wt\sigma^2_j}{|1+m(\gamma_{i-r}) \wt\sigma_j|^2}= \frac{1}{n\gamma_{i-r}}\sum_{j=1}^p \frac{\sigma^2_j}{|1+m(\gamma_{i-r}) \sigma_j|^2}+ \OO(n^{-1}).  
\end{equation}
By taking the imaginary parts of both sides of (\ref{eq_defnmc}), we obtain that
% \begin{equation*}
% \frac{1}{n}\sum_{j=1}^p \frac{\sigma^2_j}{|1+m(\gamma_i+\ri \eta) \sigma_j|^2}=\frac{1}{|m(\gamma_i+\ri \eta)|^2}-\frac{\eta}{\im m(\gamma_i+\ri \eta)}. 
% \end{equation*}
% Together with (\ref{remark_controlmplaw}), sending $\eta \downarrow 0$ gives that
\begin{equation*}
\frac{1}{n}\sum_{j=1}^p \frac{\sigma^2_j}{|1+m(\gamma_{i-r}) \sigma_j|^2}=\frac{1}{|m(\gamma_{i-r})|^2}. 
\end{equation*}
% Together with (\ref{eq_genegenralgenral}), we have shown that when $1 \leq i \leq \mathsf{K}$
% \begin{equation*}
% \vartheta_i=\frac{1}{\gamma_i|m(\gamma_i)|^2}. 
% \end{equation*} 
Plugging it into \eqref{eq_genegenralgenral} concludes the proof of \eqref{eq_nonspikeduniform}.
It remains to prove \eqref{eq_spikedform}. 
%For $\psi_i$, by Theorem \ref{thm_shrinkerestimate} 
We write $\psi_i$ as
\begin{equation}\label{eq_psi_asymp}
    \psi_i = \mathfrak{b}_i\big(1+\mathfrak{a}_i  \dot m_0(\mathfrak{a}_i)\big). %\quad \dot m_0(\mathfrak{a}_i) :={\partial_\st}\widetilde m_\st(\mathfrak{a}_i)\Big|_{\st =0},
\end{equation}
By setting $\ell(\sigma_i)=\sigma_i$ and $z=x=\fa_i$, we can simplify \eqref{eq_formofm0} using the first equation in \eqref{eq_m'm0}, which yields that 
\begin{equation*}
    \dot m_0(\mathfrak{a}_i) = \frac{1}{\mathfrak{a}_i}\left(\frac{m'(\mathfrak{a}_i)}{|m(\mathfrak{a}_i)|^2}-1\right).
\end{equation*}
Substituting this expression into \eqref{eq_psi_asymp} concludes \eqref{eq_spikedform}.
% and recalling \eqref{eq_importantdefinition} 
%yields the desired asymptotic limit \eqref{eq_spikedform}.
% and recalling \eqref{eq_importantdefinition}.
% \begin{equation*}
%     \psi_i = \frac{1}{\mathfrak{a}_i|m(\mathfrak{a}_i)|^2}.
% \end{equation*}
\end{proof}
\begin{proof}[\bf Proof of Corollary \ref{coro_lowbound}]
This corollary follows directly from Lemma \ref{lem_explicityformula}, \Cref{lem_lossdecomposition}, and Theorem \ref{thm_shrinkerestimate}. 
\end{proof}

%The proof of this lemma will be given in section \ref{sec: huowu}. 
%\bigskip      
%       First we recall the following results:
    
% Now we start our main proof.   {\cor Some explanation will be given at here.} 

\subsection{Proof of the results in Section \ref{sec_mainresultque}}\label{appendix_proof_ev}

%First of all, similar to the arguments around (\ref{eq_keyconnection}), it suffices to work with the left singular vectors of $\Lambda_0^{1/2}V^\top X$ or $\Lambda^{1/2}V^\top X.$ {\color{red}[should continue from here]}

%\subsubsection{Proof of Theorem \ref{Thm: EE}}

First, we provide the proof of \Cref{Thm: EE} by using \Cref{lem: moment flow} and \Cref{lemma_universal}. 

\begin{proof}[\bf Proof of Theorem \ref{Thm: EE}] 
Recall that for $D_0=\Lambda-\ft$ and $D_t=D_0+t$, we choose their non-spiked parts as
$D_{00}=\Lambda_0-\ft$ and $D_{t0}=D_{00}+t$, respectively. Then, we define $m_t=m(z,D_{t0})$ as in \eqref{koux} and denote the density function associated with $m_t$ as $\varrho_t$. The quantiles $\gamma_k(t)$ of  $\varrho_t$ are defined as in \Cref{defn_generaleigenvaluelocation}. Given that $\ft=\oo(1)$, it is straightforward to show that the eigenvalues of $D_{00}$ and $D_0$ satisfy (iii) and (iv) in \Cref{main_assumption}. (In particular, the regularity conditions in (\ref{eq:edgeregular}) and \eqref{eq:bulkregular} can be verified using \eqref{eq:perturbrho} and part (i) of \Cref{sare}.) Then, combining \Cref{lem: moment flow} and \Cref{lemma_universal}, we obtain that 
\be\label{eq:midstep_ft}
\sqrt{p} 
\begin{pmatrix}
\xi_1 \avg{\bm{v},\bm{u}_{i_1+r}} \\
 \vdots \\
\xi_L \avg{\bm{v}, \bm{u}_{i_L+r}}
\end{pmatrix}
\simeq \mathcal{N}(\bm{0}, \Xi_L(\ft)) ,
\ee
where $\Xi_L(\ft)$ is defined in \eqref{eq:XiL} with $\bv=V^\top \bm{v}$ (recall that $\bu_i=V^\top \bm{u}_i$). Since $D_{\ft}=\Lambda$, by comparing equations \eqref{koux} and \eqref{eq_defnmc}, we see that $\varrho_{\ft}=\varrho$. Hence, we have  $\gamma_{i_k}(\ft)=\gamma_{i_k}$ for $k\in \qq{L}$. Moreover, using the stability of the self-consistent equation \eqref{eq_defnmc} (see Appendix A of \cite{MR3704770}), it is easy to check that \be\label{eq:stabm}
m(x,D_{00})=m(x)+\OO(\sqrt{\ft}).
\ee
Then, using the definition \eqref{defvarphi1}, we conclude that 
$$\|\Xi_L(\ft)-\Xi_L\| \lesssim \sqrt{\ft}.$$ 
This establishes \eqref{eq:conv_indis2} together with \eqref{eq:midstep_ft}.
%\Cref{Thm: EE} is an immediate consequence of {\cor still need to show that the variances are equal...}
%and {\color{red} Lemma \ref{thm_localdensityestimate} below}. %and the arguments around (\ref{eq_keyconnection}). 
\end{proof}
 
Recall that the proof of \Cref{lem: moment flow} is based on \Cref{main result: Jia} and \Cref{thm_localdensityestimate0}. Using these two results, we can show that the LHS of \eqref{eq_lemma41equation} converge to the RHS also in the sense of convergence of moments. Then, by applying \Cref{lemma_universal} and following the above proof of Theorem \ref{Thm: EE}, we can deduce that  \eqref{eq:conv_indis2} holds in terms of matching moments, that is, for any fixed $\mathbf k =(k_1,\ldots, k_L)\in \N^{L}$, there exists a constant $c>0$ such that 
\be\label{eq:conv_indis3}
\mathbb E\prod_{j=1}^L \left(p|\avg{\bm{v},\bm{u}_{i_j+r}}|^2\right)^{k_j} =  \prod_{j=1}^L \E\left( |\cal N_j|^{2k_j}\right)+\OO(n^{-c}).
\ee
Now, we proceed to establish \Cref{corollary_que} by utilizing this estimate. 
%the results presented in Sections \ref{sec_maintwolemmas} and \ref{subsec_functionalrepresentation}, along with Markov's inequality and the rigidity estimate (\ref{TAZ2}). 
Our approach follows a similar argument employed for Wigner matrices in \cite{benigni2020eigenvectors,MR3606475}.

% {\it Proof of Theorem \ref{Thm: EE}} For edge case, first we note that with Lemma \ref{compedge}, we only need to study Gaussian case.  For Gaussian case, we know
% %$$W^t= W(t), \quad \hbox{in distribution}$$
%% i.e, 
%$$W(0) = W_{-t}(t), \quad \hbox{in distribution}$$
%Since for $t\ll 1$, then the limit of $$
%  \cal N\left(0,   \varphi_0\left(\lim_{N}\gamma_{k_N}^{(N)}\right)\right)\times \cal N\left(0,   \varphi_0\left(\lim_{N}\gamma_{l_N}^{(N)}\right)\right)
%$$
% does not change if one replace $D $ with $D +t$.  Then with Lemma \ref{lem: moment flow}, we prove the   edge case, i.e., Theorem \ref{Thm: EE}.
%
%For bulk case, similarly we have
%$$
%W =(W_{-t})_t
%$$
%With Lemma \ref{compbulk}, we know the singular vector of $(W_{-t})_t$ is same as $(W_{-t})(t)$ if $t\sim N^{-9/10}$.  Then with the second part of Lemma \ref{lem: moment flow}, we have proved the bulk case, i.e., Thm.  \ref{Thm: BE}. 
% \qed

%\subsubsection{Proof of Corollary \ref{corollary_que}}

%Next, we prove \Cref{corollary_que} using the results provided in Sections \ref{sec_maintwolemmas} and \ref{subsec_functionalrepresentation}, together with Markov's inequality and the rigidity estimate (\ref{TAZ2}). We follow a similar argument for Wigner matrices in \cite{benigni2020eigenvectors,MR3606475}.

\begin{proof}[\bf Proof of  \Cref{corollary_que}] 
Due to (\ref{eq_keyconnection}), it suffices to prove the following estimate for $\bu_i$:
\begin{equation}\label{eq:weakQUE2}
\P\bigg( \Big|\sum_{j=1}^{p} w_j \left|\bu_i(j) \right|^2-p^{-1}\sum_{j=1}^{p} w_j \phi_j(\gamma_{i-r})  \Big|>\e \bigg) \leq  n^{-\mathfrak d}/ \e^{2},
\end{equation}
where the function $\phi_j(x)$ is defined as
\begin{equation*}
\phi_j(x):=\phi(\bm{v}_j, \bm{v}_j, x)=\frac{c_n \wt\sigma_j}{x|1+m(x)\wt\sigma_j|^2},\quad x>0. 
\end{equation*}
Applying Markov's inequality, we can derive \eqref{eq:weakQUE2} from the following estimate: 
\be\label{eq:2moment_est}
p^{-2} \E\Big|\sum_{j=1}^{p} w_j \left(p\left|\bu_i(j) \right|^2- \phi_j(\gamma_{i-r})\right) \Big|^2 \le n^{-\fd}. 
 \ee
We observe that the LHS of \eqref{eq:2moment_est} can be bounded by  
\be\label{eq:twoparts}
\max_{j\ne k\in \qq{p}} \E \left(p\left|\bu_i(j) \right|^2- \phi_j(\gamma_{i-r})\right)\left(p\left|\bu_i(k) \right|^2- \phi_k(\gamma_{i-r})\right)+p^{-1}\max_{j=1}^{p} \E\left(p\left|\bu_i(j) \right|^2- \phi_j(\gamma_{i-r})\right)^2. 
\ee
Using \eqref{eq:conv_indis3} with $\bm{v}=\bm{v}_j$ and $k_j\in \{1,2\}$, we can conclude that 
\be\label{eq:secondmomentui}
p\E  \left|\bu_i(j) \right|^2 =\phi_j(\gamma_{i-r})+\OO(n^{-c}),\quad p^2 \E  \left|\bu_i(j) \right|^4=3\phi_j(\gamma_{i-r})^2+\OO(n^{-c}).
\ee
With these two estimates, we can bound the second term in \eqref{eq:twoparts} as 
\be\label{eq:4thmoment}
p^{-1} \E\left(p\left|\bu_i(j) \right|^2- \phi_j(\gamma_{i-r})\right)^2=\OO(p^{-1}),\quad  j \in \qq{p}.
\ee
It remains to control the first term in \eqref{eq:twoparts}. 
We now apply \eqref{eq:conv_indis3} with $L=1$, $i_1=i-r$, $k_1=2$, and $\bm{v}=(\xi_1\bm{v}_j + \xi_2\bm{v}_k)/\sqrt{2}$, where $\xi_1$ and $\xi_2$ are two uniformly random signs independent of $X$. This gives that
\begin{align*}
&p^2 \E \left|\bu_i(j) \right|^4+ p^2 \E \left|\bu_i(k) \right|^4 + 6p^2\E\left( \left|\bu_i(j) \right|^2 \left|\bu_i(k) \right|^2\right) =3\left(\phi_j(\gamma_{i-r})^2 + \phi_k(\gamma_{i-r})^2 + 2 \phi_j(\gamma_{i-r})\phi_k(\gamma_{i-r})\right).
\end{align*}
Together with \eqref{eq:secondmomentui}, this equation implies that 
\be\label{eq:2covariance}
p^2\E\left( \left|\bu_i(j) \right|^2 \left|\bu_i(k) \right|^2\right)=\phi_j(\gamma_{i-r})\phi_k(\gamma_{i-r})+\OO(n^{-c}).
\ee
Combining \eqref{eq:secondmomentui} and \eqref{eq:2covariance}, we conclude that   
\be\label{eq:2moment_match}
 \E \left(p\left|\bu_i(j) \right|^2- \phi_j(\gamma_{i-r})\right)\left(p\left|\bu_i(k) \right|^2- \phi_k(\gamma_{i-r})\right) \lesssim n^{-c}.
\ee
Plugging \eqref{eq:4thmoment} and \eqref{eq:2moment_match} into \eqref{eq:twoparts}, we obtain \eqref{eq:2moment_est}, which further concludes \eqref{eq:weakQUE2}.
\end{proof}

\subsection{Proof of the results in Section \ref{sec_statisticalestimation}}

\begin{proof}[\bf Proof of Lemma \ref{lem_usefulquantitiesconsistentestimator}]
The estimate \eqref{eq_consistentestimatorsigma} can be derived from \cite[Section 5]{NEK} or \cite[Section 3]{KV}, along with the eigenvalue sticking estimate (\ref{eq_eigenvaluesticking}). 
The first estimate in (\ref{eq_consistenestimatorotherquantities}) follows directly from \eqref{eq:outlier_location}.
%Lemma \ref{lemma_spikedlocation}. 

To prove the second estimate in (\ref{eq_consistenestimatorotherquantities}), we apply the inverse function theorem to (\ref{eq_defnmc}) and use that $m(\mathfrak{a}_i)=-\widetilde{\sigma}_i^{-1}$, which yields the relation
\be\label{eq:m'a0}
h'(-\widetilde{\sigma}_i^{-1}) =\left[m'(\mathfrak{a}_i)\right]^{-1}.
\ee
Next, considering the contour $\Gamma_i$ defined below \eqref{eq_keyconnection}, we apply Cauchy's integral formula to obtain that
\be\label{eq:m'a} 
m'(\fa_i)= \frac{1}{2\pi\ri}\oint_{\Gamma_i}\frac{m(z)}{(z-\fa_i)^2}\dd z = \frac{1}{2\pi\ri}\oint_{\Gamma_i}\frac{\sm_n(z)}{(z-\fa_i)^2}\dd z  +\OO_\prec(n^{-1/2})=\sm_n'(\mathfrak{a}_i)+\OO_\prec(n^{-1/2}).
\ee
In the second step, we utilized the local law \eqref{eq:aniso_out}, and in the third step, we used the fact that, with high probability, $\sm_n$ has no pole inside $\Gamma_i$ due to the rigidity estimate \eqref{rigidity2}. 
% According to the local law in Lemma \ref{lem_standardaniso}, by (iv) of Assumption \ref{main_assumption} , we have from Cauchy's differentiation formula that  
% \begin{equation*}
% m'_n(\mathfrak{a}_i)=m'(\mathfrak{a}_i)+\OO_{\prec}(n^{-1}).
% \end{equation*}
Since $r$ is fixed, using \eqref{eq:outlier_location} and \eqref{eq_eigenvaluesticking}, we get that
\begin{equation}\label{eq:m'a1} 
\sm_n'(\mathfrak{a}_i)= \frac{1}{n} \sum_{j=1}^n \frac{1}{(\lambda_j -{\mathfrak{a}}_i)^2} = \frac{1}{n} \sum_{j=r+1}^n \frac{1}{(\wt \lambda_j -{\mathfrak{a}}_i)^2} +\OO_\prec(n^{-1}) =\wh m_i' +\OO_{\prec}(n^{-1/2}). 
\end{equation} 
Plugging \eqref{eq:m'a1} and \eqref{eq:m'a} into \eqref{eq:m'a0} concludes the second estimate in (\ref{eq_consistenestimatorotherquantities}).
With a similar argument, employing \eqref{eq:aniso_out} and \eqref{eq:outlier_location}, we can deduce that
\be\label{eq:mai}
-\wt\sigma_i^{-1}=m(\fa_i)=\sm_n(\fa_i)+\OO_\prec(n^{-1/2})=\wh m_i  +\OO_{\prec}(n^{-1/2}).
\ee
This verifies the third estimate in (\ref{eq_consistenestimatorotherquantities}). 
%is similar with the aid of Lemma \ref{lem_standardaniso} and the fact $\wt\sigma_i=m(\fa_i).$ 

For the last estimate in (\ref{eq_consistenestimatorotherquantities}), we first notice that by \eqref{eq_evassumptionone} and \eqref{eq_consistentestimatorsigma}, 
\be\label{eq:k+e}K_\e^+=p-\oo(p)\ee
for small enough $\e>0$. Since $m$ is a real increasing function outside $\supp(\varrho)$, we have
$b_1= m(\lambda_+)  < m(\fa_i) <0.$ Combining this with condition \eqref{eq_outlierassumption}, we obtain $\min_{j=r+1}^p|1+m(\fa_i)\sigma_j| \gtrsim 1,$
which, together with \eqref{eq:mai}, implies 
$\min_{j=r+1}^p|1+\wh m_i\sigma_j| \gtrsim 1$ with high probability. 
Then, from \eqref{eq_consistentestimatorsigma}, we see that \be\label{eq:k-e}K_{\e}^-=\oo(p).\ee 
On the other hand, by employing the weak convergence in \eqref{eq_consistentestimatorsigma}, along with \eqref{eq:denominator} and \eqref{eq:mai}--\eqref{eq:k-e}, we obtain that for small $\delta>0$,
\begin{align}
\frac{1}{n} \sum_{j=1}^p \frac{\ell(\sigma_j)\sigma_j}{[1+m(\fa_i)\sigma_j]^2}&= \frac{1}{n}\sum_{j=r+1}^{\lceil(1-\delta) p\rceil} \frac{\ell(\sigma_j)\sigma_j}{(1+m(\fa_i) \sigma_j)^2+\delta}+\OO(\delta)\nonumber\\
&=\frac{1}{n}\sum_{j=r+1}^{\lceil(1-\delta) p\rceil} \frac{\ell(\wh \sigma_j)\wh \sigma_j}{(1+\wh m_i \wh\sigma_j)^2+\delta} +\oo_\P(1)+\OO(\delta)\nonumber\\
&=\frac{1}{n}\sum_{j=(r+1)\vee K_\e^-}^{p\wedge K_\e^+}  \frac{\ell(\wh \sigma_j)\wh \sigma_j}{(1+\wh m_i \wh\sigma_j)^2+\delta} +\oo_\P(1)+\OO(\delta) \nonumber\\
&=\frac{1}{n}\sum_{j=(r+1)\vee K_\e^-}^{p\wedge K_\e^+}  \frac{\ell(\wh \sigma_j)\wh \sigma_j}{(1+\wh m_i \wh\sigma_j)^2} +\oo_\P(1)+\OO(\delta).\nonumber
\end{align} 
As $\delta$ is arbitrary, we conclude from this estimate that 
\be\label{eq:consist_hard}
\frac{1}{n} \sum_{j=1}^p \frac{\ell(\sigma_j)\sigma_j}{[1+m(\fa_i)\sigma_j]^2}=\frac{1}{n}\sum_{j=(r+1)\vee K_\e^-}^{p\wedge K_\e^+}  \frac{\ell(\wh \sigma_j)\wh \sigma_j}{(1+\wh m_i \wh\sigma_j)^2} +\oo_\P(1) .
\ee
Now, plugging \eqref{eq:outlier_location}, \eqref{eq:m'a}--\eqref{eq:mai}, and \eqref{eq:consist_hard} into \eqref{eq_formofm0}, we obtain that  
\be\label{eq:dotm0}
\dot m_0(\fa_i)=\frac{m'(\fa_i)}{n\fa_i} \sum_{j=1}^p \frac{\ell(\sigma_j)\sigma_j}{[1+m(\fa_i)\sigma_j]^2}=\frac{\wh m_i'}{n\wh \fa_i}\sum_{j=(r+1)\vee K_\e^-}^{p\wedge K_\e^+}  \frac{\ell(\wh \sigma_j)\wh \sigma_j}{(1+\wh m_i \wh\sigma_j)^2} +\oo_\P(1) .
\ee
This concludes the last estimate in \eqref{eq_consistenestimatorotherquantities}.
% Finally, we prove the last estimate in (\ref{eq_consistenestimatorotherquantities}). We denote $\wh z_i=\wh\fa_i+\ii \st$, $\wt\Lambda_{\st}=\diag\{\sigma_{1,\st},\ldots, \sigma_{p,\st}\}$, and $\wh\Lambda_{\st}=\diag\{\wh\sigma_{1,\st},\ldots, \wh\sigma_{p,\st}\}.$ Then, the self-consistent equations for $\wt m_{\st}(\wh z_i)=m(\wh z_i,\wt\Lambda_{\st})$ and $\wh m_{\st}(\wh z_i):=m(\wh z_i ,\wh\Lambda_{\st})$ can be written as 
% $$ \wh z_i=-\frac{1}{\wt m_{\st}(\wh z_i)}+c_N \int \frac{x}{1+x\wt m_{\st}(\wh z_i)}\mu_{\wt\Lambda_\st}(\dd x),\quad \wh z_i=-\frac{1}{\wt m_{\st}(\wh z_i)}+c_N \int \frac{x}{1+x\wh m_{\st}(\wh z_i)}\mu_{\wh\Lambda_\st}(\dd x).$$
% By the estimate \eqref{eq_consistentestimatorsigma}, we have 
% Applying Lemma \ref{lem_standardaniso} with $\sigma_i=\wh\sigma_{i,\st}$ and $X=X^G$ gives that 
% Finally, we notice that for small $\st>0,$ applying Lemma \ref{lem_standardaniso} with $\sigma_i=\wh\sigma_{i,\st}$ and $X=X^G$ gives that 
% $$ ,\quad \wh\Lambda_{\st}=\diag\{\wh\sigma_{1,\st},\ldots, \wh\sigma_{p,\st}\}.$$
% Then, we have that from (\ref{lem_usefulquantitiesconsistentestimator})
% \begin{equation*}
% \frac{1}{n} \sum_{j=1}^n \frac{1}{\lambda_j^\st-\mathfrak{a}_i}=m_\st(\mathfrak{a}_i)+\mathrm{o}_{\mathbb{P}}(1).
% \end{equation*}
% for the last estimate in (\ref{eq_consistenestimatorotherquantities}), 
%Then the proof follows Lemma \ref{lem_standardaniso} and the first estimate.
\end{proof}

\begin{proof}[\bf Proof of Theorem \ref{thm_consistentestimators}] 
The estimate (\ref{eq_spikedshrinkerestimationandconvergence}) follows immediately from Lemma \ref{lem_usefulquantitiesconsistentestimator}. For the estimate (\ref{eq_nonspikedshrinkerestimationandconvergence}), utilizing the eigenvalue rigidity estimate \eqref{rigidity2} and the eigenvalue sticking estimate \eqref{eq_eigenvaluesticking}, we obtain that 
\be\label{eq:est222}|\wt\lambda_i-\gamma_{i-r}|\prec n^{-2/3},\quad i \in \qq{r+1,\sfK}.\ee
Using \eqref{eq_eigenvaluesticking} again, we can derive the following expression for $\wh m(\wt\lambda_i)$ when $\eta=n^{-1/2}$:
\begin{align}
    \wh m(\wt\lambda_i) %&= \frac{1}{n}  \sum_{j=r+1}^n \frac{1}{ \wt{\lambda}_j-\wt\lambda_i-\mathrm{i} \eta}=  \frac{1}{n}  \sum_{j=1}^n \frac{1}{\wt{\lambda}_j-\wt\lambda_i-\mathrm{i} \eta}+\OO_\prec(n^{-1/2})\\
    &= \frac{1}{n}  \sum_{j=r+1:|j-i|\ge n^{1/4}}^n \frac{1}{\wt{\lambda}_j-\wt\lambda_i-\mathrm{i} \eta}+ \frac{1}{n}  \sum_{j=r+1:|j-i|< n^{1/4}}^n \frac{1}{\wt{\lambda}_j-\wt\lambda_i-\mathrm{i} \eta} \nonumber\\
    &=\frac{1}{n}  \sum_{j=1:|j+r-i|\ge n^{1/4}}^{n-r} \frac{1}{{\lambda}_j-\wt\lambda_i-\mathrm{i} \eta}+ \frac{1}{n}  \sum_{j=1:|j+r-i|< n^{1/4}}^{n-r} \frac{1}{{\lambda}_j-\wt\lambda_i-\mathrm{i} \eta}+\OO_\prec(n^{-1/4}) \nonumber\\
    &=\frac{1}{n}  \sum_{j=1}^{n-r} \frac{1}{{\lambda}_j-\wt\lambda_i-\mathrm{i} \eta} +\OO_\prec(n^{-1/4}) = \sm_n(\wt\lambda_i + \ii \eta)+\OO_\prec(n^{-1/4}).\label{eq:est111}
\end{align} 
In the second step, we utilize the trivial bound $|{\lambda}_j-\wt\lambda_i-\mathrm{i} \eta|^{-1}\le \eta^{-1}$ when $|j+r -i|< n^{1/4}$, and when $|j+r -i|\ge n^{1/4}$, we have \smash{$|{\lambda}_j-\wt\lambda_i|\gtrsim n^{-1/4}$} with high probability due to the eigenvalue rigidity \eqref{rigidity2}.
When combined with the averaged local law \eqref{eq:aver}, the two estimates \eqref{eq:est222} and \eqref{eq:est111} imply that 
\be\label{eq:whmlambda} \wh m(\wt\lambda_i) = m(\wt\lambda_i + \ii \eta) +\OO_\prec(n^{-1/4}) = m(\gamma_{i-r}) +\OO_\prec(n^{-1/4}),\ee
where in the second step we used the following bound for any constant $C>0$:
\be\label{eq:changem}
|m(z_1)-m(z_2)|\lesssim |z_1-z_2|^{1/2} \quad \forall \ \ z_1,z_2\in \C_+, \ C^{-1}\le |z_1|,|z_2|\le C.
\ee
(This estimate can be proven by employing the Stieltjes transform form of $m(z)$ and the square root behaviors of $\varrho$ around the spectral edges.) Similarly, using \Cref{lem_standardaniso0}, we can derive that
\be\label{eq:whmlambda0} \wh m_0 = \re m(\ii \eta) +\OO_\prec(n^{-1/2}) = m(0) +\OO_\prec(n^{-1/2}).\ee
With the estimates \eqref{eq:est222}, \eqref{eq:whmlambda}, and \eqref{eq:whmlambda0} at hand, the proof of \eqref{eq_nonspikedshrinkerestimationandconvergence} is then similar to that of the last estimate in \eqref{eq_consistenestimatorotherquantities}, and we omit the details. 
% by an argument similar to the proof of (\ref{eq_consistenestimatorotherquantities}), for $j \geq r+1,$ we have that
% \begin{equation*}
% \widehat{\phi}_j=\phi(\bm{v}_j, \bm{v}_j, \gamma_j)+\mathrm{o}_{\mathbb{P}}(1). 
% \end{equation*}
% Then the proof follows from (\ref{eq_consistentestimatorsigma}) and the definition (\ref{eq_keydefinition1111}). 
\end{proof}

\begin{proof}[\bf Proof of \Cref{rem_otherestimators}]
The estimate \eqref{eq_nonspikedshrinkerestimationandconvergence1} follows from \eqref{eq_consistenestimatorotherquantities} and \eqref{eq:m'a}--\eqref{eq:mai}. The estimate \eqref{eq_nonspikedshrinkerestimationandconvergence2} can be derived from \eqref{eq:est222}, \eqref{eq:whmlambda}, and \eqref{eq:whmlambda0}.
\end{proof}

\section{Eigenvector moment flow and proof of Lemma \ref{lem: dyn}}\label{appendix_flowequation}

In this section, we study the dynamics of the eigenvalues and eigenvectors of the rectangular DBM and establish Lemma \ref{lem: dyn}. We first introduce some new notations. We denote the processes of eigenvalues and eigenvectors as $\bm{\lambda}(\cdot) \deq (\bm{\lambda}(t))_{t \geq 0}$ and $U(\cdot)=(U(t))_{t \geq 0}$, respectively, where $\bm{\lambda}(t)=(\lambda_1(t),\ldots, \lambda_p(t))$ and $U(t)=(\bu_1(t),\ldots, \bu_p(t))$. We further denote by {$\E^{U(0)}_{\bm{\lambda}(\cdot)} (\cdot)$} the expectation with respect to the process $U(\cdot)$ conditioned on the eigenvalue process $\bm{\lambda}(\cdot)$ and the initial state $U(0)$ of the eigenvector process. 
%with respect to the probability measure on $(Q(t))_{t \geq 0}$ 
Let $\mathsf{L}(t)$ be the generator associated with the process $U(t)$ (or equivalently the dynamics given by (\ref{dyson_evector})) conditioned on $\bm{\lambda}(\cdot)$ and $U(0)$. In other words, for any smooth functions $f: \mathbb{R}^{p^2} \rightarrow \mathbb{R}$, we have the equation
\begin{equation}\label{eq_generatordefinition}
\frac{\dd }{\dd t} \E_{\bm{\lambda}(\cdot)}^{U(0)} f \pb{U(t)}  \;=\; \E_{\bm{\lambda}(\cdot)}^{U(0)} (\mathsf{L}(t) f)\pb{U(t)}\, ,
\end{equation}
where $(\mathsf{L}(t) f)(U(t))$ represents the action of the generator on the function $f$ evaluated at $U(t)$.
Finally, we define the $\bm{\lambda}(\cdot)$-measurable $p \times p$ matrices $\Upsilon(t)=(\Upsilon_{kl}(t))$ with entries introduced in \eqref{eq_Upsilonkldefinition}. 
% as
% \begin{equation*}
% \Upsilon_{kl}(t) \;\deq\; \ind{k \not\sim l} \frac{\lambda_k(t) + \lambda_l(t)}{2 n (\lambda_k(t) - \lambda_l(t))^2}\,.
% \end{equation*}
% \textcolor{red}{$\Theta_{kl}$ in Lemma \ref{lem: dyn}}
%As in \cite{??}, the following result is a simple calculation using Lemma \ref{lem:dyson}.
\begin{lemma}\label{eq_generatorlemma}
The generator $\mathsf{L}(t)$ associated with the eigenvector flow \eqref{dyson_evector} conditioned on $\bm{\lambda}(\cdot)$ and $U(0)$ is 
\begin{equation}\label{eq_ltfinalform}
\mathsf{L}(t)=\sum_{k,l = 1}^p \frac{1}{2} \Upsilon_{kl}(t) (\mathsf X_{kl})^2\,,
\end{equation}
where $\mathsf X_{kl}$ is defined as
\begin{equation}\label{eq_defnX}
\mathsf X_{kl} \;\deq\; %\b u_k \cdot \frac{\partial}{\partial \b u_l} - \b u_l \cdot \frac{\partial}{\partial \b u_k} \;=\;
\sum_{j = 1}^p \pbb{\bu_k(j) \frac{\partial}{\partial \bu_l(j)} - \bu_l(j) \frac{\partial}{\partial  \bu_k(j)}}\,.
\end{equation}
%This means that for any smooth function $f$ we have 
\end{lemma}
%where we use the notation $\Lambda(\cdot) \deq (\Lambda(t))_{t \in [0,1]}$ and abbreviate $\E_{\cal F}(\,\cdot\,) \deq \E [\,\cdot\, \vert \cal F]$.
\begin{proof}
This result can be proved in the same way as Lemma 2.4 in \cite{MR3606475}, employing (\ref{dyson_evector}) and It\^o's formula. The details are omitted.
\end{proof}

%for some suitably chosen polynomial functions (c.f. (\ref{eq_polynomialdefinition})), they are stable under the action of the generator $\mathsf{L}(t)$ so that we can interpret them as multi-particle random walks in random environments. In addition, these polynomials  can also be understood as the moments of eigenvectors so that we can essentially obtain the joint eigenvector moment flow. 

Now, we will follow the argument presented in \cite[Section 3.1]{MR3606475} to show that $\mathsf{L}(t)$ acting on $f_t(\xi)$ in \eqref{eq_momentflowflowequation2} can be interpreted as a multi-particle random walk in random environments described by the DBM of eigenvalues. This establishes the desired \emph{eigenvector moment flow} in \Cref{lem: dyn}. Now, we will derive a more general form of EMF that encompasses \Cref{lem: dyn} as a special case. 
Let $\b v_1, \dots, \b v_M$ be a deterministic sequence of orthonormal vectors in $\R^p$, where $M$ is a fixed integer. For $(k,i) \in \qq{p} \times \qq{M},$ we define the variables $z_k^i \;\deq\; \sqrt{p} \scalar{\b v_i}{\b u_k}$ and abbreviate $Z = (z_k^i \col (k,i) \in \qq{p} \times \qq{M})$. We further define the collection of ``particle configurations"
\begin{equation}\label{eq_sndefinition}
\cal S_n \;\deq\; \hbb{\xi \in \Z^{\qq{p} \times \qq{M}} \col \sum_{i = 1}^M \xi_k^i \in 2 \Z \txt{ for all $k \in \qq{p}$}}\,.
\end{equation}
%Let $e_k^i$ with $(e_k^i)_l^j \deq \delta_{kl} \delta_{ij}$ represent the special basis elements for the configurations in $\cal S$. 
The special basis elements $e_k^i$, $(k,i) \in \qq{p} \times \qq{M},$ for the configurations in $\mathcal{S}_n$ are defined as $(e_k^i)_l^j \deq \delta_{kl} \delta_{ij}$.
For each $\xi\in\mathcal{S}_n$, we assign a polynomial $P_\xi$ of the variables in $Z$, defined as follows:
\begin{equation}\label{eq_polynomialdefinition}
P_\xi \;\deq\; \prod_{k = 1}^p \prod_{i = 1}^M (z_k^i)^{\xi_k^i}
\end{equation}
if $\xi_k^i \geq 0$ for all $(k,i)\in \qq{p} \times \qq{M}$, and $P_\xi = 0$ otherwise. With the above notations, we can derive the following key result through direct calculations. %Recall (\ref{eq_Upsilonkldefinition}) and (\ref{eq_ltfinalform}).  

\begin{lemma} \label{lem:jemf}
For $k,l \in \qq{p}$ and $i \in \qq{M}$, define the operator $A_{kl}^i \col \cal S_n \to \cal S_n$ through
\begin{equation}\label{eq_operatorA}
A_{kl}^i \xi \;\deq\; \xi - e_k^i + e_l^i\,.
\end{equation}
Then, we have that 
\begin{align} \label{L_Pxi}
\mathsf{L}(t) P_\xi
\;=\; &\sum_{k,l} \Upsilon_{kl}(t)  \sum_i \qB{\xi_k^i (\xi_k^i - 1) P_{A_{kl}^i A_{kl}^i\xi}  - \xi_k^i (\xi_l^i + 1) P_\xi} \nonumber \\
+& \sum_{k,l} \Upsilon_{kl}(t) \sum_{i \neq j} \qB{ \xi_k^i \xi_k^j P_{A_{kl}^i A_{kl}^j \xi}
- \xi_k^i \xi_l^j P_{A_{kl}^i A_{lk}^j \xi}}\,.
\end{align}
\end{lemma}

\begin{proof}
For any differentiable function $f(Z)$ of $Z$, by applying the chain rule to $z_k^i=\sqrt{p} \sum_{j=1}^p \b v_i(j) \b u_k(j)$, we obtain that 
\begin{equation}\label{eq_chain}
\sum_{j = 1}^p \bu_k(j) \frac{\partial}{\partial \bu_l(j)} f \;=\; \sum_{i = 1}^M z_k^i \frac{\partial}{\partial z_l^i} f\,.
\end{equation}
On the other hand, using the definition in (\ref{eq_polynomialdefinition}), we obtain that  
\begin{equation}\label{eq_pderivative}
z_k^i \frac{\partial}{\partial z_l^i} P_\xi \;=\; \xi_l^i P_{\xi + e_k^i - e_l^i}\,.
\end{equation}
Now, using the above two identities, we can derive the following equation for $k \not\sim l$:
\begin{align*}
(\mathsf X_{kl})^2 P_\xi \;=\;& \sum_i \Bigl[\xi_k^i (\xi_k^i - 1) P_{\xi + 2 e_l^i - 2 e_k^i} + \xi_l^i (\xi_l^i - 1) P_{\xi + 2e_k^i - 2e_l^i}
- \pb{\xi_k^i (\xi_l^i + 1) + \xi_l^i (\xi_k^i + 1)} P_\xi \Bigr]
\\
+& \sum_{i \neq j} \Bigl[ \xi_l^i \xi_l^j P_{\xi + e_k^i + e_k^j - e_l^i - e_l^j}+
\xi_k^i \xi_k^j P_{\xi + e_l^i + e_l^j - e_k^i - e_k^j}
- \xi_l^i \xi_k^j P_{\xi + e_k^i + e_l^j - e_l^i - e_k^j}
- \xi_k^i \xi_l^j P_{\xi + e_l^i + e_k^j - e_k^i - e_l^j}
\Bigr]\,.
\end{align*}
%which yields by symmetry of $\Upsilon_{kl}(t)$ and (\ref{eq_ltfinalform}) that
Finally, utilizing the symmetry $\Upsilon_{kl}(t)=\Upsilon_{lk}(t)$, this equation yields that
\begin{align*}
%\;=\; \sum_{k,l = 1}^p \frac{1}{2} \Upsilon_{kl}(t) (X_{kl})^2 P_\xi
\mathsf{L}(t) P_\xi \;=\; &\sum_{k,l} \Upsilon_{kl}(t) \sum_i \qB{\xi_k^i (\xi_k^i - 1) P_{\xi + 2 e_l^i - 2 e_k^i}  - \xi_k^i (\xi_l^i + 1) P_\xi}
\\
+ &\sum_{k,l} \Upsilon_{kl}(t) \sum_{i \neq j} \qB{ \xi_k^i \xi_k^j P_{\xi + e_l^i + e_l^j - e_k^i - e_k^j}
- \xi_k^i \xi_l^j P_{\xi + e_l^i + e_k^j - e_k^i - e_l^j}}\,.
\end{align*}
We then complete the proof by substituting the definition (\ref{eq_operatorA}) into this equation. 
\end{proof}
\begin{remark}\label{rem:color}
In light of Lemma \ref{lem:jemf}, the generator $\mathsf{L}(t)$ acting on $P_\xi$ can interpreted as a colored multi-particle random walk in random environments. In this process, particles move on the lattice $\qq{p}$, with each particle carrying a color from the set $\qq{M}$. The particles at each site in $\qq{p}$ are unordered.
%The particles move on the lattice $\qq{1,p}$, each particle carries a color in $\qq{1,M}$, and the particles at each site of $\qq{1,p}$ are unordered. 
The particle configuration is encoded by $\xi \in \cal S_n$, where $\xi_k^i$ represents the number of particles of color $i$ at site $k$. The operator $A_{kl}^i$ describes the jump of a particle with color $i$ from site $k$ to site $l$. 
Thus, on the RHS of \eqref{L_Pxi}, the first term corresponds to moving two particles of the same color from site $k$ to site $l$, the third term represents moving two particles of different colors from site $k$ to site $l$, and the last term involves exchanging two particles of different colors between site $k$ and site $l$. The second term on the right-hand side of \eqref{L_Pxi} ensures that $\mathsf L$ acts as the generator of a continuous-time Markov process. Note that under the dynamics described by \eqref{L_Pxi}, every site $k$ always contains an even number of particles, and the total number of particles of a given color remains conserved. 
By applying the techniques developed in \Cref{appendix_maximumestimation} below, we can extend our analysis to the more intricate case of a colored EMF, leading to the derivation of the more general results presented in \eqref{eq_coloredresult}. However, in the interest of simplicity, we do not delve into these details within the scope of this paper and defer their exploration to future work.

%Analyzing this more complicated colored multi-particle random walk in random environments with the arguments developed in \Cref{appendix_maximumestimation} below, we can derive the more general results in \eqref{eq_coloredresult}. For simplicity of presentation, however, we do not pursue such details in this paper and postpone it to a future work. 
%Hence, on the RHS of \eqref{L_Pxi}, the first term moves two particles of the same color from site $k$ to site $l$, the third term moves two particles of different colors from $k$ to $l$, and the last term exchanges two particles of different colors between $k$ and $l$. The second term on the RHS of \eqref{L_Pxi} ensures that $L$ is a generator of the continuous-time Markov process. 
%Note that under the dynamics \eqref{L_Pxi}, every site $k$ always has an even number of particles, and the total number of particles of a given color is conserved. 
% Based on the above heuristic discussions, we find it is more beneficial to  rewrite \eqref{L_Pxi} as
% \begin{equation*}
% \mathsf{L}(t) P_\xi
% \;=\; \sum_{k,l} \Upsilon_{kl}(t) \sum_{i,j} \xi_k^i (\xi_k^j - \delta_{ij}) P_{A_{kl}^i A_{kl}^j \xi}  - \sum_{k,l} \Upsilon_{kl}(t) \sum_{i} \xi_k^i (\xi_l^i + 1) P_\xi
% - \sum_{k,l} \Upsilon_{kl}(t) \sum_{i \neq j} \xi_k^i \xi_l^j P_{A_{kl}^i A_{lk}^j \xi}\,.
% \end{equation*}
\end{remark}

%Finally, we prove (\ref{moment_flow}) using Lemma \ref{lem:jemf}. 

\begin{proof}[\bf Proof of \Cref{lem: dyn}] 
Note that $f(\xi)$ can be written as 
$
f_t(\xi) \;\deq\; \E_{\bm{\lambda}(\cdot)}^{U(0)} Q_\xi(U(t)) \,,
$
where $Q_\xi$ is defined as
\begin{equation*}
Q_\xi \;\deq\; P_\xi \prod_{k = i_0}^{\sfK} \frac{1}{(2\xi(k)-1)!!} %\;=\; \prod_{k = 1}^p \frac{z_k(t)^{2 \xi(k)}}{(2 \xi(k)-1)!!}\,
,\quad  \text{with}\quad P_\xi \;=\; \prod_{k = i_0}^{\sfK} z_k(t)^{2 \xi(k)}\,.
\end{equation*} 
Then, \Cref{lem: dyn} is a simple consequence of Lemmas \ref{eq_generatorlemma} and \ref{lem:jemf} by taking $M=1$ and $\xi_k^1=2\xi(k)$.  
\end{proof}

\section{Proof of Theorem \ref{main result: Jia}}\label{appendix_maximumestimation}

%This section is devoted to the proof of one of the most technical results of this paper---Theorem \ref{main result: Jia}. 

This section is dedicated to proving a key technical result of this paper---Theorem \ref{main result: Jia}. As demonstrated in Lemma \ref{lem: dyn} and \Cref{appendix_flowequation}, the EMF $f_t(\xi)$ can be viewed as a multi-particle random walk in a random environment with generator \eqref{eq_ltfinalform}. 
Through a careful analysis of the EMF, we will show that $f_t(\xi)$ relaxes to the ``equilibrium state" $g_t(\xi,\e_1)$ on the time scale $t\gg n^{-1/3}$. Towards the end of this section, we will also briefly discuss how to extend our arguments to the scenario where the eigenvectors are projected onto multiple distinct directions.

The following proposition is an analog of Proposition 3.2 in \cite{MR3606475}. Recall that a measure $\pi$ on the configuration space is said to be \emph{reversible} with respect to a generator $\mathsf{L}$ if, for any functions $\mathrm{f}$ and $\mathrm{g}$, the following equality holds:
\begin{equation*}
\sum_{\xi} \pi(\xi) \mathrm{g}(\xi) \mathsf{L} \mathrm{f}(\xi)=\sum_{\xi} \pi(\xi) \mathrm{f}(\xi)\mathsf{L} \mathrm{g}(\xi).
\end{equation*} 
We then define the Dirichlet form with respect to the reversible measure $\pi$ as
\begin{equation*} 
\mathsf{D}^{\pi}(\mathrm{f})=-\sum_{\xi} \pi(\xi) \mathrm{f}(\xi) \mathsf{L} \mathrm{f}(\xi). 
\end{equation*}
\begin{proposition}[Proposition 3.2 of \cite{MR3606475}]\label{prop_revibility} Define the measure $\pi$ on the configuration space by assigning the following weight to each $\xi=(\xi_1, \cdots, \xi_p)$: 
\begin{equation}\label{eq_defnpixi}
\pi(\xi)= \prod_{i=1}^p \mathfrak{L}(\xi_i), \quad \mathfrak{L}(k):=\prod_{i=1}^k (1-(2i)^{-1}).
\end{equation}
Then, $\pi$ is a reversible measure for $\mathscr B$ defined in (\ref{eq_generatorb}), yielding that
$$ -\sum_{\xi} \pi(\xi) \mathrm{g}(\xi) \mathscr B\mathrm{f}(\xi)=\frac{1}{2}\sum_{\xi}\pi(\xi)\sum_{l \not\sim k} \Upsilon_{kl}\xi(k) (1 + 2 \xi(l)) \pb{\mathrm f(\xi^{k \to l}) - \mathrm f(\xi)}\pb{\mathrm g(\xi^{k \to l}) - \mathrm g(\xi)}.$$
In addition, the reversibility of $\pi$ remains valid for the generator $\mathscr B'$ obtained by replacing $\Upsilon$ in $\mathscr B$ with any other symmetric matrix $\Upsilon'$, i.e., $\Upsilon_{kl}'=\Upsilon_{lk}'$. 
\end{proposition}
% \begin{proof}
% See Proposition 3.2 of \cite{MR3606475}. 
% \end{proof}

% To prove Theorem \ref{main result: Jia}, instead of directly comparing $f_t$ with $g_t,$ inspired by \cite{MR3606475}, we will work on an auxiliary quantity $F_t$ in  (\ref{eq_Ft}) whose dynamics can be controlled more easily.
% We prepare some notations for the proof of Theorem \ref{main result: Jia}. 
%We first reduce the proof of Theorem \ref{main result: Jia} to simpler cases. 
Without loss of generality, in the following proof, we will focus on the case where
\be\label{eq_simplification}
i_0=1\quad \text{and}\quad c_n<1. 
\ee
%In other words, we assume our model does not contain outliers or trivial eigenvalues at 0.
In simpler terms, we assume that our model is free from outliers or trivial eigenvalues at 0. 
In fact, under \Cref{main_assumption}, the outlier eigenvalues $\wt\lambda_i$, $i\in \qq{r}$, and the trivial eigenvalues at 0 are away from \smash{$\{\wt\lambda_i\}_{i=i_0}^{\sfK}$} with a distance of order 1. Leveraging the finite speed of propagation estimate established in Lemma C.14 of \cite{9779233}, we can show that the influence of the eigenvalues/eigenvectors corresponding to the outliers and the trivial zero eigenvalues on $f_t(\xi)$ is exponentially small. Consequently, the proof of the general case without assuming \eqref{eq_simplemodel} can be concluded by following the arguments outlined in Remarks B.15 and C.2 of \cite{9779233}.

%Then, using the finite speed of propagation estimate established in Lemma C.14 of \cite{ding2022edge}, we can show that the effect of the eigenvalues/eigenvectors corresponding to the outliers and the trivial zero eigenvalues is exponentially small on $f_t(\xi)$. Then, the proof of the general case can be completed by following the arguments described in Remarks B.15 and C.2 of \cite{9779233}.

\iffalse
Moreover, as observed in \cite{DRMTA, MR3704770}, 
even though the deformed MP law $\varrho$ has $q$ bulk components as stated in Lemma \ref{lem_property}, we are allowed to analyze each of them individually, thanks to (iii) of Assumption \ref{main_assumption}. For simplicity, we focus on the case when $q=1$, and refer to Remark \ref{remark_loweredgecase} for a discussion of the general case. 

Till the end of this section, we denote 
\begin{equation}\label{eq_lambdanotations}
\lambda_+:=a_1, \quad \lambda_-:=a_2. 
\end{equation}
\fi

Given a small constant $\delta>0,$ we introduce non-negative cutoff functions $\theta_{-,\delta}$ supported on $[\lambda_-/3, \infty)$ and  ${\theta}_{+,\delta}$ supported on $(-\infty, 3 \lambda_+]$, satisfying the following properties:
\begin{equation}\label{eq_propertyonestart}
\theta_{-,\delta}(x)=
\begin{cases}
0, &  x<\lambda_-/3 \\
1, & x\geq 2\lambda_-/3
\end{cases},\qquad \theta_{+,\delta}=
\begin{cases}
0, & x>3 \lambda_+ \\
1, & x \leq 2 \lambda_+
\end{cases};
\end{equation}
the derivatives of $\theta_{-}$ and $\theta_{+}$ satisfy that for all $x \in \mathbb{R}$,
\begin{equation} \label{eq_propertyoneend}
\begin{aligned}
 0\le \theta_{-,\delta}'(x)\lesssim n^{\delta}\theta_{-,\delta}(x)+\exp(-n^{\delta/2}), \quad    |\theta_{-,\delta}''(x)|\lesssim n^{2\delta}\theta_{-,\delta}(x)+\exp(-n^{\delta/2}),\\
 0\le -\theta_{+,\delta}'(x)\lesssim n^{\delta}\theta_{+,\delta}(x)+\exp(-n^{\delta/2}), \quad    |\theta_{+,\delta}''(x)|\lesssim n^{2\delta}\theta_{+,\delta}(x)+\exp(-n^{\delta/2}). 
\end{aligned}
\end{equation}
As an example, we can consider the function $\theta_{-,\delta}$ defined as
$$ \theta_{-,\delta}(x)=\frac{1}{2}\exp\left( n^\delta(x-\lambda_-/2)\right) \quad \text{for}\quad \lambda_-/2-n^{-\delta/3} \le x \le \lambda_-/2.$$
Then, we properly choose the value of $\theta_{-,\delta}$ on $[\lambda_-/3,\lambda_-/2-n^{-\delta/3}]$ and $[\lambda_-/2,2\lambda_-/3]$ such that \eqref{eq_propertyonestart} and \eqref{eq_propertyoneend} hold. We can choose the function $\theta_{+,\delta}$ in a similar way. We are now prepared to prove Theorem \ref{main result: Jia}.

% a cutoff function $\theta_{-}:=\theta_{-,\delta}$ supported on $[\lambda_-/3, \infty)$ so that
% \begin{equation}\label{eq_propertyonestart}
% \theta_{-}(x):=
% \begin{cases}
% 0, &  x<\lambda_-/3 \\
% 1, & x\ge 2\lambda_-/3
% \end{cases}.
% \end{equation}
% In addition, we define another cutoff function ${\theta}_{+}:={\theta}_{+,\delta}$ supported on $(-\infty, 3 \lambda_+]$ so that 
% \begin{equation*}
% \theta_{+}:=
% \begin{cases}
% 0, & x>3 \lambda_+ \\
% 1, & x \leq 2 \lambda_+.
% \end{cases}
% \end{equation*}
% Moreover, we assume that the derivatives of $\theta_{+}$ satisfy that for all $x \in \mathbb{R}$
% \begin{equation}\label{eq_propertyoneend}
%  0\le -\theta_{+}'(x)\le n^{\delta}\theta_{+}(x)+\exp(-n^{\delta/2}), \quad    |\theta_{+}''(x)|\le n^{\delta}\theta_{+}(x)+\exp(-n^{\delta/2}). 
% \end{equation}

\begin{proof}[\bf Proof of Theorem \ref{main result: Jia}]
%For the parameter $\epsilon$ in the definition of (\ref{eq_onecolorgt}) via (\ref{eq_defnphiepsilon}) and the parameter $\delta_{-}$ in the definitions of the $\theta$ and $\theta_{+}$ in the above equations (\ref{eq_propertyonestart}) and (\ref{eq_propertyoneend}), we set $\epsilon=\delta.$ 

Let $\e>0$ be a sufficiently small constant. We choose $g_t(\xi,\e)$ as defined in (\ref{eq_onecolorgt}), along with the cutoff functions $\theta_{\pm,\e}$. Next, we introduce the function 
  \begin{equation}\label{eq_Ft}
  F_t\equiv F_t(L,\mathsf q, \e):=   \theta_{-,\e}(\lambda_{\mathsf K}(t))\theta_{+,\e}(\lambda_1(t))\sum_{|\xi|=L}  \pi (\xi) \wt f_t(\xi,\e)^{\mathsf{q}},
 \end{equation}
  where $\mathsf{q} \in 2\N$ is an even integer, $\pi$ is defined in (\ref{eq_defnpixi}), and 
  \begin{equation*}
  \wt f_t(\xi,\e): = f_t (\xi)-g_t(\xi,\e) . 
  \end{equation*}
To conclude the proof, it suffices to show that for any fixed $K\in \N$ and $\mathsf{p}\in 2\N$, there exist constants $0<\e<\tilde \e$ such that %for any constant $\e\in (0,\e_0)$, 
the following estimate holds with high probability:
\be\label{xiangqilai}
\sup_{t\in [T_1/2, T_2]}F_t(K,\mathsf p,\tilde \e)\lesssim n^{-\e \mathsf{p}+K} . 
\ee 
Given constants $0<\e'<\e$, $L\in \N$, and $\mathsf q\in 2\N$, we define the function $\cal A_t(L,\mathsf q, \e,\e'):= \log (F_t(L,\mathsf q, \e)+n^{-\e' \mathsf{q}+L})$. 
Then, the estimate (\ref{xiangqilai}) can be derived from the following technical lemma on $\cal A_t$, the proof of which is deferred to Section \ref{sec_keycontrollemmasection}. 

% It is direct to check that $F_t$ is $\cal F_t$-measurable, where $\cal F_t$ is defined in (\ref{eq_mathcalft}).
%Recall $\cal F_t$ in (\ref{eq_mathcalft}). It is direct to see that $F_t$ is $\cal F_t$-measurable. 
% for any $\e'<\e$ and fixed $T_1, T_2 \asymp n^{-1/3+\mathsf{c}}$, $T_2/T_1\ge 1+\mathsf{c}$, $0<\mathsf{c}\le 1/6,$ we have that for large $\nu>0$

%{\cob Note: $m$ is the number of particles and $n$ is the number of color } Together with  \eqref{TAZ2}, with large enough $p$, it clearly implies Lemma \ref{main result: Jia}.

% To show \eqref{xiangqilai}, we work with $ \cal A_t= \log (F_t+n^{-\e' \mathsf{q}+m}).$ 

% which controls the dynamics of   $\mathcal{A}_t$. 
% control of the dynamics of $\mathcal{A}_t$ which is summarized in the following lemma. 
% The proof of Lemma \ref{lem_keycontrollemma} 

\begin{lemma}\label{lem_keycontrollemma} 
Fix any $L\in \N$. Under the assumptions of Theorem \ref{main result: Jia}, suppose there exist constants $\tilde \e>0$, $\e\in (0,\min\{\tilde \e, 1/12\})$, and a fixed integer $\mathsf p_0\in 2\N$, such that \eqref{xiangqilai} holds with high probability, uniformly in $t \in [T_1/2, T_2]$, for $K= L-1$ and any fixed even integer $\mathsf p\ge \mathsf p_0$. Then, there exists a constant $\e'' \in (0,\min\{\e,3\mathsf c/8\})$ and a fixed $\mathsf q_0\in 2\N$ such that the function $\cal A_t(L,\mathsf q, \e,\e')$ satisfies the following SDE for any fixed $0<\e'\le \e''$ and $\mathsf q\ge \mathsf q_0$:
\begin{equation}\label{eq:A_sde}
\rd  \cal A_t(L,\mathsf q, \e,\e')= S(t) \dd t +  \sigma (t)\rd B(t),
\end{equation}
where $B(t)$ denotes a standard Brownian motion, $S(t)$ can be expressed as
\begin{equation*}
S(t)=-C(t) \frac{F_t}{F_t+n^{-\e' \mathsf{q}+L}} - s_+(t) + \cal E(t), %\quad {\cor \alpha:=n^{1/3-\e'},} 
\end{equation*}
and the following estimates hold with high probability, uniformly in $t\in [T_1/2, T_2]$: $\sigma(t)=\OO(n^{-1/2+2\e})$, $C(t)\gtrsim n^{1/3-\mathsf c/4}$, $s_+(t)\ge 0$, and $\cal E(t)=\OO(n^{2\e})$.
\end{lemma}

%there exists a constant $\e>0$ such that for any constant $\e'\in (0,\e)$, the following result holds: if , then the following equation holds for $\cal A_t(L,\mathsf q, \e):= \log (F_t(L,\mathsf q, \e)+n^{-\e' \mathsf{q}+L})$:
% For small fixed constant $\epsilon>0$ and $\epsilon'<\e,$ Then with high probability for all $T_1/2 \leq t \leq T_2,$ we have 
% Here we denote by $\OO_{-}$ a negative term and $\Theta(1)$ a positive bounded term.  
% % for $\alpha=n^{2/3-\mathsf{c}/2}$ we denoted that 
% \begin{equation*}
% S(t)=- \alpha \Theta(1) \frac{F_t  }{F_t+n^{-\e' \mathsf{q}+m}} +\OO(n^{1/5})+\OO_-, \quad \text{with}\quad \alpha:=n^{1/3-\mathsf{c}/4}.  %\sigma(t)=\oo(1).
% \end{equation*}
% Here we denote by $\OO_{-}$ a negative term and $\Theta(1)$ a positive bounded term.  

%We extend the process a bit to $T_1/2 \leq t \leq T_2$.
%We extend the process a little bit that for $T_1/2 \leq t \leq T_2$ and denote $\wt  {\cal A}_t$ as
%\begin{equation*}
%\quad \rd \wt {\cal A}_t=S(t) \dd t +  \frac{\sigma (t)}{\max {(1, |\sigma(t)|)}}\rd B(t),
%\end{equation*}
%with the initial condition $\wt {\cal A}_{T_1/2}=\cal A_{T_1/2}.$ 
%On the one hand, according to 

We now proceed to prove (\ref{xiangqilai}) by utilizing Lemma \ref{lem_keycontrollemma}. First, we note that \eqref{xiangqilai} trivially holds for $K=0$ since $F_t(0,\mathsf p,\e_0)=0$. Now, suppose \eqref{xiangqilai} holds for $K=L-1$ and any fixed even integer $\mathsf p\ge \mathsf p_0$. Then, by applying \Cref{lem_keycontrollemma}, we can deduce that the equation \eqref{eq:A_sde} holds for some constant $\e'>0$ and $\mathsf q\ge \mathsf q_0$. 
%We choose $\mathsf q$ to be sufficiently large so that $-\e'\mathsf q+L<0$.  
By applying the Burkholder-Davis-Gundy inequality to the diffusion term in \eqref{eq:A_sde}, we obtain that the following event holds with high probability:
\be\label{eq:event_integration0}
\left\{\cal A_t  \le  {\cal A}_{s} -\al \int_{s}^t \frac{F_{t'}  }{F_{t'} + n^{-\e' \mathsf{q}+L}}\rd t' + \OO\left(n^{-{1}/{3}+2\e+\mathsf c }\right) \   \forall t,s\in [T_1/2, T_2]\right\},
\ee
where $\al=cn^{1/3-\mathsf c/4}$ for a small constant $c>0$. Note that, according to the definition of $\cal A_t$, we have the deterministic rough bound $\cal A_{s}\le C\log n$ for a constant $C>0$. Therefore, the above event implies that
\be\label{eq:event_integration}
\left\{  {\cal A}_t  \le \min ( {\cal A}_{T_1/2}, C \log n)-\al \int_{T_1/2}^t \frac{F_s }{F_s + n^{-\e' \mathsf{q}+L}} \rd s +\oo(1) \ \forall t\in [T_1/2, T_2]\right\}
\ee 
holds with high probability. Now, we define the following stopping time (with the convention $\inf \emptyset=T_2$):
$$\tilde t:=(T_1/2)\vee \inf\{ t\in [T_1/2, T_2]:\cal A_t\le \log(2n^{-\e'\mathsf{q}+L})\}.$$ 
Noticing that ${F_t }/(F_t + n^{-\e' \mathsf{q}+L})\ge 1/2$ for $t\in [T_1/2,\tilde t]$, on the event \eqref{eq:event_integration}, we have  
\begin{equation*}
    \cal A_{\tilde t}\le \min(\cal A_{T_1/2},C\log n)-\frac{\alpha}{2} (\tilde t -T_1/2) + \oo(1)\quad \text{w.h.p.}
\end{equation*}
%Recall that $T_1\asymp n^{-1/3+\mathsf{c}}$ for $\mathsf{c}\le 1/6$, 
Since $\alpha T_1/2\gtrsim n^{3\mathsf c/4}\gg \log n$, this estimate implies that 
\be\label{eq:wtt}
\tilde t \le T_1 \quad \text{w.h.p.} 
\ee
Now, we apply this fact to \eqref{eq:event_integration0} with $s=\tilde t$ and conclude that 
\begin{equation}\label{eq_boundattransform}
\sup_{t\in [T_1, T_2]} {\cal A}_t \le {\cal A}_{\tilde t} + 1 \quad \text{w.h.p.} 
\end{equation}
Together with the definition of $F_t$, this implies that with high probability,  
$\sup_{t\in [T_1, T_2]}F_t \le (2e-1) n^{-\e' \mathsf{q}+L},$
which completes the induction argument and establishes \eqref{xiangqilai}. 

Now, given \eqref{xiangqilai} with $K=L$ and $0<\e<\tilde \e$, we choose $\e_1=\tilde \e$ and $\e_0=\e/2$ and let $\mathsf p$ be sufficiently large so that $-\e \mathsf{p}+L< -\e_0\mathsf p$. 
This allows us to obtain the following inequality:
$$F_t(L,\mathsf p, \e_1)=   \theta_{-,\e_1}(\lambda_{\mathsf K}(t))\theta_{+,\e_1}(\lambda_1(t))\sum_{|\xi|=L}  \pi (\xi) \wt f_t(\xi,\e_1)^{\mathsf{p}} \ll n^{-\e_0\mathsf p} \quad \text{w.h.p.}$$
By utilizing the rigidity estimate \eqref{TAZ2} for $\lambda_1(t)$ and $\lambda_{\sfK}(t)$, and considering the definitions of $\theta_{\pm,\e_1}$, we can deduce that $\theta_{-,\e_1}(\lambda_{\mathsf K}(t))=\theta_{+,\e_1}(\lambda_1(t))=1$ with high probability. Consequently, taking into account the fact that $\pi (\xi)$ are all of order 1, the above estimate implies that 
$$\wt f_t(\xi,\e_1)^{\mathsf{p}}\le n^{-\e_0\mathsf p} \quad \text{w.h.p.}$$
for any configuration $\xi$ with $|\xi|=L$. This concludes the proof of \Cref{main result: Jia}.

\end{proof}

%\subsection{The non-spiked model} For the results non-spiked model, i.e., part (i) of Theorem \ref{thm_shrinkerestimate}, the results follow directly from  {\color{blue} finish here}
%
%
%\subsection{The spiked model} For the spiked model, the discussion for the spiked shrinkers are different from 
\subsection{Proof of Lemma \ref{lem_keycontrollemma}}\label{sec_keycontrollemmasection}

In this subsection, we present the proof of Lemma \ref{lem_keycontrollemma}. By It\^o's formula, we have 
\[
d{\cal A}_t = \frac{\rd F_t}{F_t+n^{-\e'\mathsf{q}+L}} - \frac{\rd [F]_t}{2(F_t+n^{-\e'\mathsf{q}+L})^2},
\]
where the second term on the RHS is negative. Hence, it suffices to prove that the following equation holds for $t \in [T_1/2, T_2]$: 
\be\label{cl1}
\frac{\rd F_t}{F_t+n^{-\e' \mathsf{q}+L}}= -  C(t)\frac{F_t  }{F_t+n^{-\e' \mathsf{q} +L}} \dd t - s_+(t)\rd t+\cal E(t)\rd t+ \sigma(t)\rd B(t), 
\ee
where $s_+(t)\ge 0$ and $C(t)$, $\cal E(t)$, $\sigma(t)$ satisfy the desired estimates stated in \Cref{lem_keycontrollemma}. 

%Recall (\ref{eq_onecolorgt}) and the SDE of $ \lambda_k(t)$ in \eqref{sLa}. Till the end of this subsection, 

For simplicity of notation, in the following proof, we will abbreviate $F \equiv F_t$, $f \equiv f_t$, $g \equiv g_t$, $\wt f\equiv \wt f_t$, $ \lambda_k \equiv \lambda_k(t)$, $\bu_k \equiv \bu_k(t)$, $\phi(\lambda_k(t),\e)\equiv \phi(\lambda_k)$, $\theta_-\equiv \theta_{-,\e}(\lambda_{\mathsf K})$, $\theta_+\equiv \theta_{+,\e}(\lambda_{1})$, and $\sm_{n}\equiv \sm_{n,t}$, $\underline\sm_{n}\equiv \underline\sm_{n,t}$ (where we adopt the notations in \Cref{sec: iso of Wt} with $\Lambda_0=D_{00}$ and $\Lambda_t=D_{00}+t$). Using \eqref{sLa} and It\^{o}'s formula, we get that 
\begin{align}\label{eq_sde1}
&\rd \phi(\lambda_k )= \phi'(\lambda_k)\rd \lambda_k+ \phi''(\lambda_k)\frac{2\lambda_k}{n} \rd t,\\
&\rd g(\xi)  = \sum_k \Big( \partial_{\lambda_k} g(\xi) \Big) \rd \lambda_k+ \sum_k\Big( \partial^2_{\lambda_k} g(\xi) \Big)\frac{2\lambda_k}{n}\rd  t,\label{eq_sde2}\\
&\rd \theta_{-}= \theta_{-}' \rd \lambda_{\mathsf K}+ \theta_{-}'' \frac{2\lambda_{\mathsf K}}{n}\rd t, \quad 
\rd  \theta_{+}= \theta_{+}' \rd \lambda_1+  \theta_{+}'' \frac{2\lambda_1}{n}\rd t. \label{eq_sde3}
\end{align}
Note that by the definition of $g$ in (\ref{eq_onecolorgt}) and the estimate \eqref{xiaoy}, we have   
 \begin{equation}\label{rough_control}
  \theta_{-}\theta_{+} \partial_{\lambda_k} g(\xi) =\OO(n^{\e}), \quad    \theta_{-}\theta_{+} \partial^2_{\lambda_k} g(\xi) =\OO(n^{2\e}).
   \end{equation}
%{\cob For following calculation, I will use the notations $\e_g$ and $\e_\theta$, but we choose them as
%$$
%\e_g=\e_\theta=\e_0
%,\quad  \e<\e_0
%$$
%There is a parameter $\tau$ for ridigity, anisotropic law and delocalization, which is almost as small as possible. 
%} 
%Now we proceed to the proof of Lemma \ref{lem_keycontrollemma}, or equivalently (\ref{cl1}). 
%\begin{proof}[\bf Proof of Lemma \ref{lem_keycontrollemma}] For notational simplicity, till the end of the proof, we abbreviate $F \equiv F_t, f \equiv f_t, g \equiv g_t, m_{1}(z) \equiv m_{1,t}(z), \lambda_k \equiv \lambda_k(t)$ and $\bu_k \equiv \bu_k(t).$  
Now, applying It\^{o}'s formula to $F$ defined in (\ref{eq_Ft}) and using (\ref{eq_sde1})--(\ref{eq_sde3}), we can derive that 
\begin{align}
\rd F =     &
      \sum_{\xi } \pi (\xi)
    \wt  f(\xi)   ^{\mathsf{q} } 
   \left( \theta_{-} \theta_{+}'  \rd \lambda_1+\theta_{-} \theta_{+} ''   \frac{2\lambda_1}{n}\rd t
   +  \theta_{-}'\theta_{+}  \rd \lambda_{\mathsf K} 
   + \theta_{-}''    \theta_{+}  \frac{2\lambda_{\mathsf K}}{n} \rd t \right) \label{eq_controlone}
    \\
& + \theta_{-}\theta_{+} 
   \sum_{\xi } \pi (\xi)
   \left(\mathsf{q}  \wt  f(\xi)   ^{\mathsf{q}-1 }(\rd f(\xi)-\rd g(\xi))+  \mathsf{q}(\mathsf{q}-1)  \wt  f(\xi)   ^{\mathsf{q}-2 } \sum_{k}\big(\partial_{\lambda_k} g(\xi) \big)^2\frac{2\lambda_k}{n}\rd t \right)  \label{eq_controltwo}
   \\\label{gongzuo}
 & - \sum_{\xi  } \pi (\xi)
   \left(\mathsf{q}  \wt  f(\xi)   ^{\mathsf{q}-1 } \right)\left(\partial_{\lambda_{\mathsf K}} g(\xi)  \theta_{-}' \theta_{+} \frac{4\lambda_{\mathsf K}}{n}+
 \partial_{\lambda_1} g(\xi) \theta_{-}\theta_{+}' \frac{4\lambda_1}{n}  \right)\rd t.  
   \end{align}
Here, we have used the convention that $\sum_{\xi }\equiv \sum_{\xi: |\xi|=L}.$ Next, we will control these terms one by one. 

First, considering the term (\ref{eq_controlone}), we utilize (\ref{eq_propertyonestart}) and (\ref{eq_propertyoneend}) with $\delta=\e$, along with the SDE \eqref{sLa}, to obtain that  
 \begin{equation}\label{eq_keyexamplecontrolexampleone}
 \begin{aligned}
\theta_{-}\theta_{+} '  \rd \lambda_1+\theta_{-}\theta_{+}'' \frac{2\lambda_1}{n}\rd t
= &~ \theta_{-}\theta_{+}' 2 \sqrt{\lambda_1}\frac{ \rd B_{11}}{\sqrt{n}}+ \theta_{-}\theta_{+}' \bigg(1+\frac{1}{n}\sum_{k \neq 1}\frac{\lambda_1 +\lambda_k}{\lambda_1-\lambda_k}\bigg)\rd t\\
 &~ + \OO \left(n^{-1+2\e} \theta_{-}\theta_{+} +\exp(-n^{ \e/2})\right) \rd t  .
 \end{aligned}
 \end{equation}
It then follows from \eqref{eq_keyexamplecontrolexampleone} and the bound \eqref{eq_propertyoneend} that    
%the definition of $F_t$ in (\ref{eq_Ft}) and the fact $\theta_{-}$ and $\theta_{+}$ are bounded that with high probability that   
% It implies that  {\cob (Note: it holds a.s.)}
 \begin{equation}\label{eq:term11}
 \sum_{\xi} \frac{\pi (\xi)
 \wt  f(\xi)   ^{\mathsf{q}}}{F+n^{-\e' \mathsf{q}+L}} \left( \theta_{-}\theta_{+} '  \rd \lambda_1+\theta_{-}\theta_{+}'' \frac{2\lambda_1}{n}\rd t\right)
   =\OO(n^{-1+2\e})\rd t+ \OO(n^{-1/2+\e})\rd B_{11}-s_+^{(1)}(t)\rd t,
 \end{equation}
where $s_+^{(1)}(t)$ represents a positive term arising from the second term on the RHS of \eqref{eq_keyexamplecontrolexampleone}. Using a similar argument, we can also show that 
\begin{align*}
 &\sum_{\xi} \frac{\pi (\xi)
 \wt  f(\xi)   ^{\mathsf{q}}}{F+n^{-\e' \mathsf{q}+L}} \left( \theta_{-}'\theta_{+}  \rd \lambda_{\mathsf K} 
   + \theta_{-}''\theta_{+}  \frac{2\lambda_{\mathsf{K}}}{n} \rd t\right)\\
  & =\OO(n^{-1+2\e})\rd t+ \OO(n^{-1/2+\e})\rd B_{\sfK\sfK} 
   +\frac{\theta_{-}'\theta_{+}\sum_{\xi} \pi (\xi)
 \wt  f(\xi)^{\mathsf{q}}}{F+n^{-\e' \mathsf{q}+L}}\bigg(1+\frac{1}{n}\sum_{k \neq \sfK}\frac{\lambda_\sfK +\lambda_k}{\lambda_\sfK-\lambda_k}\bigg)\rd t\\
 & =\OO(n^{-1+2\e})\rd t+ \OO(n^{-1/2+\e})\rd B_{\sfK\sfK} -s_+^{(2)}(t)\rd t
   +\left( 1+\frac{p-\sfK}{n}\right)\frac{\theta_{-}'\theta_{+}\sum_{\xi} \pi (\xi)
 \wt  f(\xi)^{\mathsf{q}} }{F+n^{-\e' \mathsf{q}+L}}\rd t,
 \end{align*}
 where $s_+^{(2)}(t)$ represents another positive term.
 We can exploit the rigidity of $\lambda_{\sfK}$ given by \eqref{TAZ2} and the fact that $|\lambda_{-,t}-\lambda_-|=\oo(1)$ (as stated in (i) of \Cref{sare}) to conclude that $\theta_{-}'=0$ with high probability, uniformly in $t\in [T_1/2,T_2]$. Consequently, we obtain a bound similar to \eqref{eq:term11}: 
% Thus, we get a similar bound as \eqref{eq:term11}: with high probability,
\begin{equation}\label{eq:term22}
 \sum_{\xi} \frac{\pi (\xi)
 \wt  f(\xi)   ^{\mathsf{q}}}{F+n^{-\e' \mathsf{q}+L}} \left( \theta_{-}'\theta_{+}  \rd \lambda_{\mathsf K} 
   + \theta_{-}''\theta_{+}  \frac{2\lambda_{\mathsf{K}}}{n} \rd t\right)
   =\cal E^{(2)}(t)\rd t+ \OO(n^{-1/2+\e})\rd B_{\sfK\sfK}-s_+^{(2)}(t)\rd t,
 \end{equation}
where $\cal E^{(2)}(t)$ is a variable satisfying $\cal E^{(2)}(t)=\OO(n^{-1+2\e})$ with high probability, uniformly in $t\in [T_1/2,T_2]$.
 
% the part with $\theta_{-}'\theta_{+}  \rd \lambda_{\mathsf K} 
%    + \theta_{-}''\theta_{+}  \frac{2\lambda_{\mathsf{K}}}{n} \rd t$ satisfies a similar estimate. In this case, the second term on the RHS of \eqref{eq_keyexamplecontrolexampleone} becomes 
%    $$ \theta_{-}'\theta_{+} \bigg(1+\frac{1}{n}\sum_{k \neq \sfK}\frac{\lambda_\sfK +\lambda_k}{\lambda_\sfK-\lambda_k}\bigg)\rd t=-\theta_{-}'\theta_{+} \frac{1}{n}\sum_{k < \sfK}\frac{2\lambda_\sfK +\lambda_k}{\lambda_k-\lambda_\sfK}\rd t, $$
 
Next, we control \eqref{gongzuo}. For any $0<l<\mathsf{q}$, we have the trivial bound 
\be\label{eq_trivialf}
|\wt  f(\xi)|^{\mathsf{q}-l}\lesssim n^{l\e}\wt  f(\xi)^\mathsf{q}+n^{-(\mathsf{q}-l)\e}. 
\ee
Utilizing \eqref{eq_trivialf} and the estimates \eqref{rough_control} and \eqref{eq_propertyoneend}, we get that  
 \begin{equation}\label{eq:term33}
  \sum_{\xi  } \frac{\pi (\xi)
   \mathsf q\wt  f(\xi)   ^{\mathsf{q}-1 } }{F+n^{-\e'\mathsf{q} +L}} \left(\partial_{\lambda_{\mathsf K}} g(\xi)  \theta_{-}' \theta_{+} \frac{4\lambda_{\mathsf K}}{n}+
\partial_{\lambda_1} g(\xi) \theta_{+}' \theta_{-} \frac{4\lambda_1}{n}  \right) \rd t 
=\OO(n^{-1+3\epsilon})\rd t.
 \end{equation}
Similarly, the second term of \eqref{eq_controltwo} can be bounded by
 \begin{equation}\label{eq:term44}
 \mathsf{q}(\mathsf{q}-1) \theta_{-}\theta_{+} 
   \sum_{\xi } \frac{\pi (\xi) \wt  f(\xi)   ^{\mathsf{q}-2 }}{F+n^{-\e' \mathsf{q}+L}}  \sum_{k}\left(\partial_{\lambda_k} g(\xi) \right)^2\frac{2\lambda_k}{n}\rd t =\OO(n^{-1+4\epsilon})\rd t. 
\end{equation}
Combing \eqref{eq:term11}, \eqref{eq:term22}, \eqref{eq:term33}, and \eqref{eq:term44}, we obtain that
\be
\begin{aligned}\label{tuodi}
  \frac{\rd F}{F+n^{-\e' \mathsf{q}+L}} 
  = &~  \frac{\mathsf{q} \theta_{-}\theta_{+} }{F+n^{-\e' \mathsf{q}+L}}
   \sum_{\xi } \pi (\xi) \wt  f(\xi)^{\mathsf{q}-1 }(\rd f(\xi)-\rd g(\xi)) 
   \\
  &~ + \cal E_1(t)\rd t+ \OO(n^{-1/2+\e})\rd B_{11}+ \OO(n^{-1/2+\e})\rd B_{\mathsf{K} \mathsf{K}}-(s_+^{(1)}(t)+s_+^{(2)}(t))\rd t,
\end{aligned}
\ee
where $\cal E_1(t)$ is a variable satisfying $\cal E_1(t)=\OO(n^{-1+4\e})$ with high probability, uniformly in $t\in [T_1/2,T_2]$.

%Since we only have one color here, we denote the operator $A_{kl} $ as in (\ref{eq_operatorA}) by removing the color index $i.$ 
%   Similar to the discussion of (\ref{eq_generatorb}), we have that    
% \begin{equation} \label{hha}
% \partial_t f(\xi) 
%\;=\; \sum_{k,l} \Theta_{kl} \qB{  f({A_{kl}\xi})  - f(\xi)}, 
%\end{equation}
%where we recall $\Theta_{kl}$ in (\ref{eq_generatorb}). Using \eqref{hha}, we immediately see that  
Now, to conclude (\ref{cl1}), it remains to estimate the first term on the RHS of \eqref{tuodi}.  
% We define $A_{kl}\equiv A_{kl}^1$ as in  (\ref{eq_operatorA}). 
By (\ref{moment_flow}), we express the term as 
\begin{align}\label{eq_changechangechange}
   \theta_{-}\theta_{+}  \sum_{\xi } \pi (\xi)
  \wt  f(\xi)   ^{\mathsf{q}-1 } \rd f(\xi)= 
  \theta_{-}\theta_{+}  \sum_{\xi } \pi (\xi)  \wt  f(\xi)   ^{\mathsf{q}-1 }    \sum_{k \not\sim l} \Upsilon_{kl}(t) \, \xi(k) (1 + 2 \xi(l)) \pb{f(\xi^{k \to l}) - f(\xi)}\dd t.
   \end{align}
On the other hand, using \eqref{sLa} and (\ref{eq_sde2}), we obtain:
    \begin{align}
  & \theta_{-}\theta_{+}  \sum_{\xi} \pi (\xi)
  \wt  f(\xi)   ^{\mathsf{q}-1 } \rd g(\xi)=\theta_{-}\theta_{+}  \sum_{\xi } \pi (\xi) 
  \wt  f(\xi)   ^{\mathsf{q}-1} \label{eq_changechangechange2}\\ 
  &\quad \times \sum_{k}   \bigg[ \partial_{\lambda_k} g(\xi)  
  \sum_{l:l \not\sim k} \frac{ \lambda_k + \lambda_l}{n(\lambda_k-\lambda_l)} \rd t + \partial_{\lambda_k} g(\xi)\dd t
  +  \partial^2_{\lambda_k} g(\xi)\frac{2\lambda_k}{n}\rd   t
  +  (\partial_{\lambda_k} g(\xi))  2\sqrt \lambda_k \frac{\rd B_{kk}}{\sqrt n}   
\bigg]. \nonumber
   \end{align}
% With Definition \ref{defn_asymptoticweakconvergence}, we can control 
The diffusion part of the above equation has the same law as  
$$\theta_{-}\theta_{+}  \bigg[\sum_k \bigg| \sum_{\xi}\Big(\pi(\xi) \wt  f(\xi)   ^{\mathsf{q}-1 }  \left( \partial_{\lambda_k} g(\xi) \right)  \frac{2\sqrt \lambda_k}{\sqrt n}  \Big)\bigg|^2\bigg]^{1/2}\rd B,$$
where $B$ is a standard Brownian motion.
Note that the coefficient satisfies that
 \begin{align}\label{eq:control_coe}
  \theta_{-}\theta_{+}  \bigg[\sum_k \bigg| \sum_{\xi}\pi(\xi) \wt  f(\xi)   ^{\mathsf{q}-1 }  \left( \partial_{\lambda_k} g(\xi) \right)  \frac{2\sqrt \lambda_k}{\sqrt n}  \bigg|^2\bigg]^{1/2}= \OO (  n^{-1/2+\e}) \cdot  \theta_{-}\theta_{+} \sum_{\xi }  \pi(\xi)| \wt  f(\xi)|   ^{\mathsf q-1 }.
   \end{align}
Here, we have used (\ref{rough_control}) and the facts that $\pi(\xi)\asymp 1$ and $\partial_{\lambda_k} g(\xi)$ is non-zero for at most $L$ many $k$'s. Plugging \eqref{eq_changechangechange}--\eqref{eq:control_coe} into (\ref{tuodi}) and using \eqref{eq_trivialf} and \eqref{rough_control}, we can deduce that
  \begin{align}\label{332}
  \frac{\rd F}{F+n^{-\e' \mathsf{q}+L}}    = &~\sum_{\xi}\frac{\mathsf q\theta_{-}\theta_{+} \pi(\xi)\wt  f(\xi)   ^{\mathsf{q}-1 } }{F+n^{-\e' \mathsf{q}+L}}
  \sum_{k \not\sim l}   \Bigg( \Upsilon_{kl}  \,  \xi(k) (1 + 2 \xi(l)) \pb{f(\xi^{k \to l}) - f(\xi)}   - \frac{ (\lambda_k +\lambda_l)\partial_{\lambda_k} g(\xi) }{n(\lambda_k-\lambda_l)} 
  \Bigg)
   \rd t \nonumber \\
  &~ +\cal E_2(t)\rd t+\OO(n^{-1/2+2\e})\rd \wt B  -(s_+^{(1)}(t)+s_+^{(2)}(t))\rd t,
   \end{align}
where $\widetilde{B}$ is another standard Brownian motion and $\cal E_2(t)$ is a variable satisfying $\cal E_2(t)=\OO(n^{2\e})$ with high probability, uniformly in $t\in [T_1/2,T_2]$.
% Inserting the above control back into (\ref{tuodi}) and using $|x|^{\mathsf{q}-l}\le n^{l\e}(|x|^\mathsf{q}+n^{-\e \mathsf{q}})$ for $l<\mathsf{q}$ again, we find that with high probability for some constant $C>0$      
% %Here we used: for any fixed $\xi$, $\|\xi\|= m$.   Therefore, with $|x|^{p-m}\le N^{m\e}(|x|^p+N^{-\e p})$,
% %    so far we have proved that  {\cob (Note: it holds a.s.)}
%   \begin{align}\label{332}
%   \frac{\rd F}{F+n^{-\e \mathsf{q}+L}}    = &\frac{\theta_{-}\theta_{+} }{F+n^{-\e \mathsf{q}+L}}\sum_{\xi}\pi(\xi)
%   \sum_{k \not\sim l}   \Bigg(\tfrac12 \Upsilon_{kl}  \, 2 \xi(k) (1 + 2 \xi(l)) \pb{f(\xi^{k \to l}) - f(\xi)}   - \frac{ \partial_{\lambda_k} g(\xi) 2\lambda_k }{(\lambda_k-\lambda_l)n} 
%   \Bigg)
%   \wt  f(\xi)   ^{\mathsf{q}-1 }  \rd t \nonumber \\
%   &+ \OO(n^{-1+C\e})\rd t+\OO(n^{-1/2+C \e})\rd \wt B  +\OO_-\rd t,
%    \end{align}
% where $\widetilde{B}$ is some standard Brownian motion. 

Taking $\eta=n^{-2/3+\mathsf c/2}$, we define 
$$\mathrm{T}_{kl}\equiv \mathrm T_{kl}(\eta) \deq \mathbf 1(k\not\sim l)\frac{(\lambda_k- \lambda_l)^2}{(\lambda_k- \lambda_l)^2 + \eta^2},\quad \mathrm{T}_{kl}^c:=\mathbf 1(k\not\sim l)-\mathrm{T}_{kl}.$$
%For $\eta \geq n^{-1}$ and all $1 \leq k \not\sim l \leq p,$ we introduce $\mathrm{T}_{kl} \;\deq\; \frac{(\lambda_l- \lambda_k)^2}{(\lambda_l- \lambda_k)^2 + \eta^2}\,$ and by convention we write $\mathrm{T}_{kl}^c:=1-\mathrm{T}_{kl}$. 
Then, we decompose the first term on the RHS of (\ref{332}) as 
\begin{equation}\label{eq_importantdecomposition}
\mathsf q(F+n^{-\e' \mathsf{q}+L})^{-1}\sum_{\xi }\pi(\xi) \wt  f(\xi)^{\mathsf{q}-1 } \left(A_1+A_2+A_3\right)  \rd t,
\end{equation}
 where $A_1$, $A_2$, and $A_3$ are defined as 
  \begin{align}
 & A_1:= \theta_{-}\theta_{+} 
  \sum_{k \not\sim l}  \Upsilon_{kl}  \mathrm{T}_{kl}  \xi(k) (1 + 2 \xi(l)) \pb{f(\xi^{k \to l}) - f(\xi)}, \label{eq_termonecontrol} \\
 & A_2:= -\theta_{-}\theta_{+} 
  \sum_{k \not\sim l}  \mathrm{T}_{kl} \frac{(\lambda_k+\lambda_l)  \partial_{\lambda_k} g(\xi) }{n(\lambda_k-\lambda_l)},
\label{eq_termtwocontrol} \\
&  A_3:=  \theta_{-}\theta_{+} 
  \sum_{k \not\sim l}   \mathrm{T}_{kl}^c \left[\Upsilon_{kl} \xi(k) (1 + 2 \xi(l)) \pb{f(\xi^{k \to l}) - f(\xi)}- \frac{(\lambda_k+\lambda_l) \partial_{\lambda_k} g(\xi) }{n(\lambda_k-\lambda_l)} \right] \label{eq_termthreecontrol}.
\end{align}
The remaining part of the proof focuses on controlling the above three terms one by one. Specifically, we will demonstrate that $A_1$ contributes to the dominant drift term in \eqref{cl1}.

%The rest of the proof is devoted to controlling the above terms one by one. In particular, we will show that $A_1$ contributes to the dominant drift term in \eqref{cl1}. %while $A_2$ and $A_3$ will be properly absorbed into the other terms in \eqref{332}.

\medskip
\noindent{\bf The term $A_2$}. 
We start with the term $A_2$:
\begin{align*}
A_2&=\theta_{-} \theta_{+} \sum_{k \not\sim l} \frac{\partial_{\lambda_k} g(\xi)}{n} \frac{\lambda_l^2-\lambda_k^2}{(\lambda_l-\lambda_k)^2+\eta^2}=\theta_{-} \theta_{+}\left[\sum_{k} 2c_n\lambda_k \partial_{\lambda_k} g(\xi) \cdot \re \underline\sm_{n} (\lambda_k + \ii\eta)+ \sum_{k \not\sim l} \mathrm{T}_{kl}\frac{\partial_{\lambda_k} g(\xi)}{n}\right],
\end{align*}
where in the second step we used the spectral decomposition of $\underline\sm_{n}$:
%Note that by the spectral decomposition of $m_n(z)$ in (\ref{eq_mnzdefinition}), for any $1 \leq l \leq \mathsf{K}$ we have 
\begin{equation*}
\re \underline\sm_{n}(\lambda_k+\ri \eta)=\re\frac{1}{p} \sum_{l}\frac{1}{\lambda_l-(\lambda_k+\ii\eta)} =\frac{1}{p} \sum_{l:l \not\sim k} \frac{\lambda_l-\lambda_k}{(\lambda_l-\lambda_k)^2+\eta^2}.
\end{equation*}
Since $\partial_{\lambda_k} g(\xi)$ is non-zero for at most $L$ many $k$'s, using (\ref{rough_control}), the rigidity of $\lambda_k$ in \eqref{TAZ2}, the trivial bound $\mathrm{T}_{kl}<1$ and the local law \eqref{TAZ}, we obtain that $A_2=\OO(n^\e)$ with high probability. Together with \eqref{eq_trivialf}, it implies that 
\be\label{shsan}
    \frac{\mathsf q }{F+n^{-\e' \mathsf{q}+L}}\sum_{\xi}\pi(\xi) \wt  f(\xi)   ^{\mathsf{q}-1 }A_2   =\OO(n^{2\e}  )\quad \text{w.h.p.}
\ee

%and delocalization results in Lemma \ref{sare}, for small enough $\e$ and some absolute constant $C$ we have with high probability that  
% we see that when $\epsilon$ is sufficiently small, for some constant $C>0,$ with high probability, we have that $A_2=\OO(n^{C\epsilon}).$  According to the definition of $f(\xi)$ and $g(\xi),$
% by the rigidity and delocalization results in Lemma \ref{sare}, we know that with high probability for $0 \leq t \leq T_2$ and some constant $C>0$
% \be\label{shsan}
%     \frac{1 }{F+n^{-\e' \mathsf{q}+L}}\sum_{\xi}\pi(\xi) A_2  \wt  f(\xi)   ^{\mathsf{q}-1 } =\OO(n^{C\e}  ).
% \ee
% where $\tau$ is rigidity parameter. 

%\medskip
\noindent{\bf The term $A_3$}. For the term $A_3$, we will establish that 
  \be\label{AB2}
   \frac{\mathsf q}{F+N^{-\e' \mathsf{q}+L}}\sum_{\xi}\pi(\xi) \wt  f(\xi)   ^{\mathsf{q}-1 }  A_3 = \cal E_3(t)-s_+^{(3)}(t),
\ee
where $s_+^{(3)}(t)\ge 0$ and $\cal E_3(t)$ represents a variable that satisfies $\cal E_3(t)=\OO(n^{-2/3+4\e+\mathsf c})$ with high probability, uniformly in $t\in [T_1/2,T_2]$. By utilizing \Cref{prop_revibility}, we can express the LHS of \eqref{AB2} as 
%Using the reversibility of $\pi$ as shown in Proposition \ref{prop_revibility}, we can rewrite the LHS of (\ref{AB2}) as follows:
\begin{align} 
  \frac{\mathsf q\theta_{-}\theta_{+} }{F+n^{-\e' \mathsf{q}+L}}\sum_{\xi }\pi(\xi) \bigg[& \frac{1}{2}\sum_{k \not\sim l}  \mathrm{T}_{kl}^c \Upsilon_{kl} \xi(k) (1 + 2 \xi(l)) \pb{f(\xi^{k \to l}) - f(\xi)} \pb{\wt f(\xi)^{\mathsf{q}-1}-\wt f(\xi^{k \to l})^{\mathsf{q}-1}} \nonumber \\
 &  -\wt f(\xi)^{\mathsf{q}-1} \sum_{k \not\sim l}  \mathrm{T}_{kl}^c \frac{ (\lambda_k+\lambda_l) \partial_{\lambda_k} g(\xi) }{n(\lambda_k-\lambda_l)} \bigg]. \label{hha3}
  \end{align}
We proceed to control the two terms in \eqref{hha3}. To simplify the presentation, for $k\not\sim l$ and a configuration $\xi$, we define the equivalence class of configurations $[\xi]\equiv [\xi]_{k,l}$ as follows: $\xi'\in [\xi]$ if and only if $|\xi'|=|\xi|$ and $  \supp (\xi'-\xi)\subset \{\b e_k, \b e_l\}$. 
% that $[\xi]\sim [\xi']$ if and only if $|\xi|=|\xi'|$
%  and $
%   \supp (\xi-\xi')=\{\b e_k, \b e_l\}
% $. 
% Therefore,  $A_{2,1}$ can be written as $\E _{\b \lambda} \theta_{-}\theta_{+} \sum_{\al\ne  \beta} \sum_{[\eta]}\sum_{\eta'\sim \eta} *$. 
Given $[\xi]$ and $0\le \ell \le \xi(k)+\xi(l)$, we define $\xi^\ell\in [\xi]$ to be the configuration with $\xi^\ell(k)=\ell$. 
%\textcolor{blue}{(Notation: maybe change $\ell$ to another letter, say $i$?)}
Moreover, for any $\ell\in \qq{\xi(k)+\xi(l)}$, we define the following functions: 
\begin{align*}
 h_{\xi^\ell, k,l } &:=\pi(\xi^\ell) \xi^\ell(k) (1+2\xi^\ell(l))=\pi(\xi^{\ell-1}) \xi^{\ell-1}(l) (1+2\xi^{\ell-1}(k)), \\
 \quad \wh h_{\xi^\ell, k, l}&:= \pi(\xi^\ell) \xi ^\ell(k)=\pi(\xi^{\hat \ell}) \xi ^{\hat \ell}(l), \quad \text{with}\quad \hat \ell=\xi(k)+\xi(l)-\ell.
\end{align*}
 %Then by the above symmetry and  (\ref{eq_onecolorgt}), 
Then, using the definition  (\ref{eq_onecolorgt}) and the anti-symmetry of the coefficient ${(\lambda_k+\lambda_l)}/{(\lambda_k-\lambda_l)}$, we can further rewrite (\ref{hha3}) as 
\begin{align}  
\frac{\mathsf q\theta_{-}\theta_{+} }{F+n^{-\e' \mathsf{q}+L}}&\sum_{k\not\sim  l} \frac12 \mathrm{T}^c_{kl}
\sum_{[\xi]}
 \sum_{\ell=1}^{ \xi(k)+\xi(l)} \bigg[ \Upsilon_{kl} h_{\xi^\ell, k,l}   \big(f(\xi^{\ell-1})- f(\xi^{\ell})\big)
\big( \wt f(\xi^\ell)^{\mathsf{q}-1}-\wt f(\xi^{\ell-1})^{\mathsf{q}-1}\big) \nonumber\\
&- \wh h_{\xi^\ell, k, l}\frac{ \lambda_k+\lambda_l }{n(\lambda_k -\lambda_l)}  
\left(g(\xi^{ \ell })\frac{ \phi'(\lambda_k,\e)}{\phi(\lambda_k,\e)}\wt  f(\xi^{\ell})^{\mathsf{q}-1}
-g(\xi^{\hat \ell })\frac{ \phi'(\lambda_l,\e)}{\phi(\lambda_l,\e)}\wt f(\xi^{\hat\ell})^{\mathsf{q}-1}\right)
 \bigg]. \label{hha4}
 \end{align}
Using the definitions of $\theta_{\pm}$ and  \eqref{xiaoy}, we obtain that
\begin{align*}
 \theta_{-}\theta_{+} g(\xi^{ \ell })\frac{ \phi'(\lambda_k,\e)}{\phi(\lambda_k,\e)}
 &=  \theta_{-}\theta_{+} \frac{g(\xi^{ \ell })}{\phi(\lambda_k,\e)}\frac{ \phi (\lambda_k,\e)-\phi (\lambda_l,\e)}{ \lambda_k-\lambda_l} +\OO\left(n^{2\e}(\lambda_{k}-\lambda_l)\right) \\
 &=
 \theta_{-}\theta_{+} g(\xi^{\hat \ell })\frac{  \phi'(\lambda_l,\e)}{ \phi(\lambda_l,\e)}+ \OO\left(n^{2\e}(\lambda_{k}-\lambda_l)\right) . 
\end{align*}  
By substituting this estimate into (\ref{hha4}) and using \eqref{eq_trivialf} again, we deduce that 
%for some constant $C>0$ we can write   
\begin{align*}
\eqref{hha4}=  \frac{\mathsf q\theta_{-}\theta_{+} }{F+n^{-\e' \mathsf{q}+L}} &\sum_{k\not\sim  l} \frac12\mathrm{T}^c_{kl}\Upsilon_{kl} 
\sum_{[\xi]}
 \sum_{\ell=1}^{ \xi(k)+\xi(l)} 
 \bigg[ h_{\xi^\ell, k,l}   \big(f(\xi^{\ell-1})- f(\xi^{\ell})\big)
\big( \wt f(\xi^\ell)^{\mathsf{q}-1}-\wt f(\xi^{\ell-1})^{\mathsf{q}-1}\big) \\
& - 2\wh h_{\xi^\ell, k, l} \frac{g(\xi^{ \ell })}{\phi(\lambda_k,\e)} \big(\phi(\lambda_k,\e)-\phi(\lambda_l,\e)\big)
\big( \wt  f(\xi^{\ell})^{\mathsf{q}-1}
 - \wt f(\xi^{\hat\ell})^{\mathsf{q}-1} \big)\bigg]
+\OO(n^{-1+3\e}).
\end{align*}
Note \smash{$ \wt  f(\xi^{\ell})^{\mathsf{q}-1}- \wt f(\xi^{\hat\ell})^{\mathsf{q}-1}$} can be expressed as a telescoping sum of terms of the form \smash{$\wt f(\xi^{\ell'})^{\mathsf{q}-1} - \wt f(\xi^{ \ell'- 1})^{\mathsf{q}-1} $} 
for some $\ell'$ between $\ell$ and $\hat \ell$.  Therefore, there exist a sequence of deterministic coefficients $c_{\xi,\ell, k,l}$ of order $\OO(1)$ such that we can rewrite the above equation as  
\begin{align}
\eqref{hha4}= \frac{\mathsf q\theta_{-}\theta_{+} }{F+n^{-\e' \mathsf{q}+L}}
&\sum_{k\not\sim  l}  \frac12\mathrm{T}^c_{kl} \Upsilon_{kl}
\sum_{[\xi]}
\sum_{\ell=1}^{\xi(k)+\xi(l)} h_{\xi^\ell, k,l}  
  \Big[ \big(f(\xi^{\ell-1})- f(\xi^{\ell})\big)+c_{\xi,\ell, k,l}\frac{g(\xi^{ \ell })}{\phi(\lambda_k,\e)}\big(\phi(\lambda_l,\e)-\phi(\lambda_k,\e)\big)\Big] \nonumber \\
&  \times \left( \wt f(\xi^\ell)^{\mathsf q-1}-\wt f(\xi^{\ell-1})^{\mathsf{q}-1}\right)+\OO(n^{-1+3\e})\nonumber \\
=\frac{\mathsf q\theta_{-}\theta_{+} }{F+n^{-\e' \mathsf{q}+L}}
&\sum_{k\not\sim  l}  \frac{1}{2}\mathrm{T}^c_{kl}\Upsilon_{kl}
\sum_{[\xi]} \sum_{\ell=1}^{\xi(k)+\xi(l)} h_{\xi^\ell, k,l}  
  \Big[ \big(f(\xi^{\ell-1})- f(\xi^{\ell})\big)+c_{\xi,\ell, k,l}\big(g(\xi^{\ell-1})-g(\xi^{\ell})\big)\Big] \nonumber \\
& \times \left( \wt f(\xi^\ell)^{\mathsf q-1}-\wt f(\xi^{\ell-1})^{\mathsf q-1}\right)+\OO(n^{-1+3\e}).\label{eq_finalreducedform}
\end{align} 
 %For the last line, we used the fact that if  $\theta_{-}\theta_{+}\ne 0$ then  
% $$   g(\xi^{\ell-1})-g(\xi^{\ell}) = g\left((\xi^{\ell-1})^{(l)}\right)\left(\phi(\lambda_l)-\phi(\lambda_k) \right)=O(N^{\e_g})(\lambda_k-\lambda_l)$$  
By definition, we have $f(\xi^{\ell-1})- f(\xi^{\ell}) =g(\xi^{\ell-1})-g(\xi^{\ell})+ \wt f(\xi^{\ell-1})-\wt f(\xi^{\ell}).$
Observe that if $\abs{ f(\xi^{\ell-1})-  f(\xi^{\ell})}\ge \abs{ g(\xi^{\ell-1})-g(\xi^{\ell})}\log n$, then the corresponding term in (\ref{eq_finalreducedform}) is negative, i.e.,
$$  \Big[ \big(f(\xi^{\ell-1})- f(\xi^{\ell})\big)+c_{\xi,\ell, k,l}\big(g(\xi^{\ell-1})-g(\xi^{\ell})\big)\Big] \left( \wt f(\xi^\ell)^{\mathsf q-1}-\wt f(\xi^{\ell-1})^{\mathsf q-1}\right) < 0. $$
Otherwise, based on \eqref{rough_control}, we can get that
\begin{align}
& \theta_{-}\theta_{+} \Upsilon_{kl}\Big[ \big(f(\xi^{\ell-1})- f(\xi^{\ell})\big)+c_{\xi,\ell, k,l}\big(g(\xi^{\ell-1})-g(\xi^{\ell})\big)\Big] \left( \wt f(\xi^\ell)^{\mathsf q-1}-\wt f(\xi^{\ell-1})^{\mathsf q-1}\right)  \nonumber \\
&=\OO\left(n^{-1+2\e}(\log n)^2\right)\cdot \theta_{-}\theta_{+} \left( \wt f(\xi^\ell)^{\mathsf{q}-2}+\wt f(\xi^{\ell-1})^{\mathsf{q}-2}\right). \label{eq_whygtrandom}
  \end{align}
(We remark that this is the step where it is convenient to define $g_t$ in terms of $\lambda_k(t)$ as in \eqref{eq_onecolorgt}.) 
%  Therefore, using the definition of $\widetilde{f}$ and Lemma \ref{sare}, we obtain that with high probability, for some constant $C>0$ 
By combining these facts, we find that 
$$\eqref{eq_finalreducedform} = \cal E_3(t)-s_+^{(3)}(t),$$
where $s_+^{(3)}(t)\ge 0$ and the variable $\cal E_3(t)$ satisfies the following bound: 
\begin{align}\label{shsan2}
|\cal E_3(t)|&\lesssim  n^{-1+3\e}+  \frac{\mathsf q\theta_{-}\theta_{+} }{F+n^{-\e' \mathsf{q}+L}}
\sum_{\xi} \pi(\xi)\wt f(\xi)^{\mathsf{q}-2} \frac{n^{2\e}(\log n)^2}{n}\sum_{k\not\sim l} \frac{\xi(k)\eta^2}{(\lambda_k-\lambda_l)^2+\eta^2} \nonumber\\
&= n^{-1+3\e}+ \frac{\mathsf q\theta_{-}\theta_{+} }{F+n^{-\e' \mathsf{q}+L}}
\sum_{\xi} \pi(\xi) \wt f(\xi)^{\mathsf{q}-2}  n^{2\e}(\log n)^2 \sum_k \xi(k)\eta\im \underline\sm_{n}(\lambda_k + \ii \eta)\nonumber\\
&\lesssim n^{-1+3\e} +  n^{4\e}(\log n)^2\eta  \le  n^{-2/3+4\e+\mathsf c}\quad \text{w.h.p.}
\end{align}
Here, in the second step, we used the spectral decomposition of $\underline\sm_{n}$, and in the third step, we applied the local law \eqref{TAZ} to $\underline\sm_{n}(\lambda_k + \ii \eta)$ along with the rigidity of $\lambda_k$ given by \eqref{TAZ2}. This concludes \eqref{AB2}.

% \begin{align}\label{shsan2}
% &\frac{1 }{F+n^{-\e' \mathsf{q}+L}}\sum_{\xi }\pi(\xi) A_3 \wt  f(\xi)   ^{\mathsf{q}-1 } 
% = \OO(n^{-1+C\e})+\OO_-,
% \end{align}
% which implies \eqref{AB2}. 

%First we study the case $n=1$, $\xi=2$, in this case with \eqref{hha}, we have 

\medskip
\noindent{\bf The term $A_1$}. 
Finally, we prove the following inequality for sufficiently large $\mathsf q$:
   \be\label{AB1}
     \frac{\mathsf q}{F+n^{-\e' \mathsf{q}+L}}\sum_{\xi}\pi(\xi) \wt  f(\xi)   ^{\mathsf{q}-1 } A_1  \le    - C(t) \frac{ F}{F+n^{-\e' \mathsf{q}+L}} +\OO_\prec(1), \ee
     where $C(t)$ is a positive variable satisfying $C(t)\gtrsim \eta^{-1/2}$ with high probability. With the definition of $A_1$ in (\ref{eq_termonecontrol}), we can express the LHS of \eqref{AB1} as follows: 
\begin{align}
   &~\frac{\mathsf q}{F+n^{-\e' \mathsf{q}+L}}\sum_{\xi}\pi(\xi) \wt  f(\xi)   ^{\mathsf{q}-1 } A_1\nonumber\\
   =&~\frac{\mathsf q\theta_{-}\theta_{+}}{2n(F+n^{-\e'\mathsf{q}+L})} 
 \sum_{\xi}\pi(\xi) \wt  f(\xi)   ^{\mathsf{q}-1 } \sum_{k \not\sim l} \frac{(\lambda_k+\lambda_l)\xi(k) (1 + 2 \xi(l))}{(\lambda_k-\lambda_l)^2+\eta^2}   \big(f(\xi^{k \to l}) - f(\xi)\big). \label{eq_finalparttobeestimated}  
\end{align}
Recalling the definitions of $f$ and $\xi^{k\to l}$ in \eqref{eq_momentflowflowequation2} and \eqref{eq_simplifiedcomeandgo}, we can express $ f(\xi^{k \to l}) $ as 
\begin{align}\label{eq:ftxikl}
 f(\xi^{k \to l})  = \E\bigg[ \frac{p|\avg{\bv,\bu_l}|^2}{1+2\xi(l)}
 \prod_{j }\frac{(p|\bv^\top \bu_j|^{2})^{\xi(j)-\mathbf{1}_{j=k}}}{ a_{\xi(j)-\mathbf{1}_{j=k}}}\bigg| \cal F_t\bigg],
\end{align}
where $a_n:=(2n-1)!!$ for $n\in \N$. Utilizing the spectral decomposition of $\cal G_{1,t}(z)= (W(t)W(t)^\top-z)^{-1}$ with $z=z_k:=\lambda_k+\ri \eta$, we have
\begin{align*}
\sum_{l: l \not\sim k } \frac{ (\lambda_k+\lambda_l)|\avg{\bv, \bu_l}|^2}{(\lambda_l-\lambda_k)^2+\eta^2} =&~ \sum_{l} \frac{2\lambda_k|\avg{\bv, \bu_l}|^2}{(\lambda_l-\lambda_k)^2+\eta^2} -\frac{ 2\lambda_k  |\avg{\bv, \bu_k}|^2}{ \eta^2}+\sum_{l: l \not\sim k } \frac{ (\lambda_l-\lambda_k)|\avg{\bv, \bu_l}|^2}{(\lambda_l-\lambda_k)^2+\eta^2}\\
%+\sum_{l: l \not\sim k } \frac{ (\lambda_k+\lambda_l)2\xi(l) |\avg{\bv, \bu_l}|^2}{(\lambda_l-\lambda_k)^2+\eta^2}\\
%=&~\sum_{l   } \frac{  |\avg{\bv, \bu_l}|^2}{(\lambda_l-\lambda_k)^2+\eta^2} 
%-\sum_\beta\frac{2\eta_\al (2\eta_\beta )  |u_\beta^2|}{(\lambda_\al-\lambda_\beta)^2+\eta^2}-\frac{   |\avg{\bv, \bu_l}|^2}{ \eta^2} \\
=&~\frac{2\lambda_k}\eta\im \left(\bv^\top \cal G_{1,t}(z_k) \bv\right)+\re \left(\bv^\top \cal G_{1,t}(z_k) \bv\right) +\OO_\prec (n^{-1}\eta^{-2}) \\
=&~ \frac{2\lambda_k}\eta\im \left(\bv^\top \wt\Pi_t(z_k)\bv\right)+\OO_\prec \left(\eta^{-1}\Psi_t(z_k)\right) .
%+\sum_{l: l \not\sim k } \frac{ (\lambda_k+\lambda_l)2\xi(l) |\avg{\bv, \bu_l}|^2}{(\lambda_l-\lambda_k)^2+\eta^2} +\OO_\prec (n^{-1}\eta^{-2}) ,
\end{align*} 
Here, in the second step, we applied the delocalization estimate \eqref{TAZ3} to $|\avg{\bv, \bu_k}|^2$, and in the third step, we applied the local law \eqref{TAZ} to $\bv^\top \cal G_{1,t}(z_k) \bv$ with $\wt\Pi_t(z):= - [z(1+m_t(z))\Lambda_t]^{-1}$. 
Similarly, by observing that $\xi(l)$ is nonzero for at most $L$ many $l$'s, and using the spectral decomposition of $\underline\sm_{n}$ and the local law \eqref{eq_average2}, we obtain that 
\begin{equation*}
\frac1p \sum_{l: l \not\sim k } \frac{(\lambda_k+\lambda_l)(1+2\xi(l))}{(\lambda_l-\lambda_k)^2+\eta^2} = \frac{2\lambda_k}{\eta} \im m_{1,t}(z_k) + \OO_\prec(n^{-1}\eta^{-2}).
%=\frac{1}{n\eta} \tr\im (G_t(\lambda_k+\ri \eta )) +\OO(n^{-1+\e}\eta^{-2}). 
\end{equation*}
%   \begin{align}\label{eq_finalparttobeestimated}
%&     \frac{1}{F+n^{-\e' \mathsf{q}+m}}  \sum_{\xi}\pi(\xi) A_1 \wt  f(\xi)   ^{\mathsf{q}-1 }  \rd t \\
%&=   \frac{\theta_{-}\theta_{+}}{4n(F+n^{-\e' \mathsf{q}+m})} \sum_{\xi} \pi (\xi) 
%\sum_{k  }\sum_{l\not\sim k } \frac{2\xi(k) (1+2\xi(l)) }{(\lambda_k-\lambda_l)^2+\eta^2}
% \pb{f(\xi^{k \to l}) - f(\xi)}
%    \wt f(\xi)^{\mathsf{q}-1} \nonumber \\
%&= \frac{\theta_{-}\theta_{+}}{4n(F+n^{-\e' \mathsf{q} +m})} \sum_{\xi} \pi (\xi) 
%\sum_{k  }\sum_{l: l\not\sim k } \frac{2\xi(k) (1+2\xi(l)) }{(\lambda_k-\lambda_l)^2+\eta^2} \left( \E\left( \frac{p|\bv^\top \bu_l|^2}{2\xi(l)+1}
% \prod_{m }\frac{|\bv^\top \bu_m|^{2(\xi(m)- \mathbf{1}_{m=k})}}{ a_{\xi(m)-\mathbf{1}_{m=k}}}\Big| \cal F_t\right)-f(\xi)\right) 
%    \wt f(\xi)^{\mathsf{q}-1}, \nonumber
%\end{align} 
Plugging the above three equations into \eqref{eq_finalparttobeestimated} and applying the delocalization estimate \eqref{TAZ3} to the factors $p|\avg{\bv,\bu_j}|^2$ in $f(\xi)$ and \eqref{eq:ftxikl}, we obtain that 
  \begin{align}
   (\ref{eq_finalparttobeestimated})=&~ \frac{c_n \mathsf q\theta_{-}\theta_{+} }{F+n^{-\e' \mathsf{q}+L}}  \sum_{\xi } \pi (\xi)  \wt f(\xi)^{\mathsf{q}-1}
\sum_{k  }\frac{\lambda_k\xi(k)}{\eta} \nonumber
\\
&\times  \left[   \im \left(\bv^\top \wt\Pi_t(z_k) \bv\right)f(\xi -{\bf e}_k)-   \im m_{1,t} (z_k)f(\xi)+\OO_\prec \left(\Psi_t(z_k)\right) \right]. \label{woB}
  \end{align}
%where we abbreviated that $\wt\Pi_t(z):= - [z(1+m_t(z))\Lambda_t]^{-1}$. 
% we have for any $0\le t\le T_2$ that
%   &\frac{\theta_{-}\theta_{+}}{F+n^{-\e' \mathsf{q}+L}}\sum_{\xi}\pi(\xi) A_1 \wt  f(\xi)   ^{\mathsf{q}-1 } 
% \\\nonumber
%   \begin{align}\label{chenh} &(\ref{eq_finalparttobeestimated})= \frac{c_n \mathsf q\theta_{-}\theta_{+} }{F+n^{-\e' \mathsf{q}+L}}  \sum_{\xi } \pi (\xi)  \wt f(\xi)^{\mathsf{q}-1}
% \sum_{k  }\frac{\lambda_k\xi(k)}{\eta} 
% \\\nonumber
% &\times    \bigg[\E\bigg( \im \bv^\top \wt\Pi_t(z_k) \bv  +\OO_\prec \left(\frac{1}{\sqrt{n\eta}}\right) \bigg) 
%  \prod_{j}\frac{(p |\bv^\top \bu_j|^2)^{\xi(j)- \mathbf{1}_{j=k}} }{a_{\xi(j)-\mathbf{1}_{j=k}}}\bigg|\cal F_t\bigg)
%  -
%   \left( \im m_{1,t}(z_k) +\OO_\prec \left(\frac{1}{n\eta }\right)\right)f(\xi)\bigg],
%   \end{align}
% Again by Lemma \ref{sare}, with high probability we can  write \eqref{chenh} as 
%   \begin{align}\label{woB}
%   \frac{\theta_{-}\theta_{+} }{F+n^{-\e' \mathsf{q}+L}}  \sum_{\xi} \pi (\xi) 
% \sum_{k  } 
%  \frac{2\xi(k)}{\eta} \Big(   \im (\bv^\top \Pi_t(z_k) \bv)f(\xi -{\bf e}_k)-   \im m_{t} (z_k)f(\xi)+\OO (\Psi_t(z_k) n^\e) 
%   \Big) 
%     \wt f(\xi)^{\mathsf{q}-1}.
%   \end{align}
%When $\xi(k)\ge 1$, it is direct to check by induction that for any $\e'< \e$ we have $f(\xi -{\bf e}_k)- g(\xi -{\bf e}_k)=\OO(n^{-\e'})$ holds with high probability for any $T_1/2\le t\le T_2$. 
By the induction hypothesis, when $K=L-1$, the estimate \eqref{xiangqilai} holds with high probability for large even integer $\mathsf p\ge \mathsf p_0$,  implying that 
\be\label{eq:adddd1}
|\wt f(\xi-\mathbf e_k)| \lesssim n^{-\e +(L-1)/\mathsf p}\quad \text{w.h.p.} \ \Rightarrow \ |\wt f(\xi-\mathbf e_k)| \prec n^{-\e}.
\ee
On the other hand, by \eqref{eq:Imm} and \eqref{eq:m1tz}, we have 
\be\label{eq:adddd1.5} \im m_{1,t}(z_k) \asymp \im m_t(z_k)\gtrsim \min\left\{\sqrt{\kappa_k+\eta},\frac{\eta}{\sqrt{\kappa_k+\eta}}\right\}\gtrsim \sqrt{\eta}\quad \text{w.h.p.},
\ee
where $\kappa_k$ is defined as $\kappa_k:=\dist(\lambda_k,\supp(\varrho_t))$, and we have used the fact that $\kappa_k\prec n^{-2/3}$ according to the rigidity estimate \eqref{TAZ2}. With \eqref{eq:adddd1.5}, for $\eta=n^{-2/3+\mathsf c/2}$, we obtain that 
\be\label{eq:adddd2} \frac{\Psi_t(z_k)}{\im m_{1,t}(z_k)} \lesssim \left(\frac{1}{n\eta\im m_{1,t}(z_k)}\right)^{1/2}+ \frac{1}{n\eta \im m_{1,t}(z_k)} \lesssim \left(\frac{1}{n\eta^{3/2}}\right)^{1/2}+ \frac{1}{n\eta^{3/2}}\lesssim n^{-3\mathsf c/8} .
\ee
%     where the second and third step follows from  that 
%    $$ \im m_{1,t}(z_k) \sim \im m_t(z_k)\gtrsim \sqrt{\eta}$$
%As a result, with high probability, the left-hand side of \eqref{chenh} can be written as
With the estimates \eqref{eq:adddd1} and \eqref{eq:adddd2}, we can simplify \eqref{woB} as follows:
   \begin{align}\label{woB2}
 (\ref{eq_finalparttobeestimated})= \frac{c_n\mathsf q\theta_{-}\theta_{+} }{F+n^{-\e' \mathsf{q}+L}}  &\sum_{\xi} \pi (\xi)\wt f(\xi)^{\mathsf{q}-1} 
\sum_{k  }\lambda_k\xi(k)  \frac{\im m_{1,t} (z_k)}{\eta}\nonumber\\
&\times   \Big(  \varphi_t (z_k)g(\xi -{\bf e}_k) - f(\xi)+\OO_\prec\big(n^{-\e}+n^{-3\mathsf c/8}\big)\Big) 
    ,
     \end{align}
 where $\varphi_t(z_k)\equiv \varphi_t(\bv,\bv,z_k)={\im  (\bv^\top \wt \Pi_t(z_k) \bv) }/{\im m_{1,t}(z_k) }$ as defined in \eqref{eq_varphitz}. By using \eqref{eq:denominator}, \eqref{eq:stabm},  \eqref{eq:changem}, and the definition of $\phi(\lambda_k,\e)$, we can check that 
\begin{align}  \varphi_t (\bv,\bv,z_k)- \phi(\lambda_k,\e)&= \left[\varphi_t (\bv,\bv,z_k)- \varphi_0(\bv,\bv,\lambda_k)\right]+\left[\varphi_0(\bv,\bv,\lambda_k)-\phi(\lambda_k,\e)\right]\nonumber\\
&\lesssim \sqrt{t+\eta}+n^{-\e}\le 2 n^{-\e},\label{eq:addd3}
\end{align}
where in the last step, we used the fact that $\sqrt{t+\eta}\lesssim n^{-1/6+\mathsf c/2}+n^{-1/3+\mathsf c/4}\le n^{-\e}$ for $\mathsf c<1/6$ and $\e<1/12$.
Plugging \eqref{eq:addd3} into \eqref{woB2}, and in light of the definition of $g(\xi)$ in \eqref{eq_onecolorgt}, we find that 
       \begin{align*}
&(\ref{eq_finalparttobeestimated})=\frac{c_n\mathsf q\theta_{-}\theta_{+} }{F+n^{-\e' \mathsf{q}+L}}  \sum_{\xi} \pi (\xi)  
\sum_{k  }\lambda_k \xi(k) 
 \frac{\im m_{1,t} (z_k)}{\eta} \left(  g(\xi ) - f(\xi)+\OO_\prec\big(n^{-\e}+n^{-3\mathsf c/8}\big) \right)    \wt f(\xi)^{\mathsf{q}-1}  \\
&=\frac{-c_n\mathsf q\theta_{-}\theta_{+} }{F+n^{-\e' \mathsf{q}+L}}  \sum_{\xi} \pi (\xi)  
\sum_{k  }\lambda_k \xi(k) 
 \frac{\im m_{1,t} (z_k)}{\eta} \left[ \wt f(\xi)^{\mathsf{q}}+\OO_\prec\left(n^{-\tau} \big(\wt f(\xi)^{\mathsf{q}}+n^{-\mathsf q\delta}\big)\right) \right]    \\
 &\le \frac{-c\eta^{-1/2}F}{F+n^{-\e' \mathsf{q}+L}}  + \OO_\prec\left(\frac{n^{-\mathsf q\delta-\tau+L}\eta^{-1}}{F+n^{-\e' \mathsf{q}+L}}\right)=\frac{-c\eta^{-1/2}F}{F+n^{-\e' \mathsf{q}+L}}  + \OO_\prec\left(1\right)  
 %   &=   - \alpha \Theta(1) \frac{(1-\OO(n^{-\e'}))F  }{F+n^{-\e' \mathsf{q}+L}} ,
\end{align*}
for a constant $c>0$. In the second step, we used \eqref{eq_trivialf} with $\e$ replaced by a constant $\e'<\delta < \min\{\e,3\mathsf c/8\}$, and we denote $\tau:=\min\{\e,3\mathsf c/8\}-\delta$; in the third step, we used $\sqrt{\eta}\lesssim \im m_{1,t} (z_k)\lesssim 1 $ by \eqref{eq:adddd1.5} and \eqref{eq:Imm}; in the last step, we choose $\mathsf q$ to be sufficiently large (depending on $\delta-\e'$) such that $n^{-\mathsf q\delta-\tau}\eta^{-1} \le n^{-\e' \mathsf{q}}$. This proves \eqref{AB1} and completes the proof of Lemma \ref{lem_keycontrollemma}.

\section{Proof of the results in Section \ref{subsec_functionalrepresentation}}\label{appendix_finalone}

\subsection{Proof of Lemma \ref{7.2}}\label{sec_proofof72}
 
In this subsection, we prove Lemma \ref{7.2} using the strategies outlined in \cite{bloemendal2016principal,knowles2013eigenvector}. A key component of the proof relies on the level repulsion estimate. In the existing literature, the level repulsion estimate for sample covariance matrices has only been established around the edge \cite{bloemendal2016principal} or inside the bulk via the universality of bulk eigenvalue statistics \cite{PY}. However, for our specific applications, we require the level repulsion estimate to hold over the entire spectrum. 
Recall the definition of $\cal W_t$ from \eqref{eq_mathcalwt}, and denote the eigenvalues of $\mathcal{Q}_t=\mathcal{W}_t \mathcal{W}_t^\top$ as $\{\bar \lambda_k(t)\}$. Given an interval $I\subset \R_+$, we define the counting function
\[ \overline {\mathsf{N}}_t(I)=|\{i \in \qq{1, \mathsf{K}}: \bar\lambda_i(t) \in I\}|.\] 
Similarly, we define the counting function ${\mathsf{N}}_t$ for the eigenvalues $\{ \lambda_k(t)\}$ of the matrix DBM defined in \eqref{eq_qt}:
\[  {\mathsf{N}}_t(I)=|\{i \in \qq{1, \mathsf{K}}:  \lambda_i(t) \in I\}|.\]
Recall that $\gamma_k(t)$ represents the quantiles of $\varrho_t$ as defined below \eqref{koux}, and $\Delta_k$ is defined in \eqref{eq_mathsfKdefinition}. 
 
  %Suppose Assumption \ref{main_assumption} holds. Let $\{\lambda_k(t)\}$ denote the eigenvalues of (\ref{eq_mathcalwt}) with $t=\oo(1)$ and $\{\gamma_i\}$ be the classical eigenvalue locations.  Let $\widehat{k}=\min \{k, \mathsf{K}-k+1\}.$ 
 
  \begin{proposition}[Level repulsion] \label{prop: level repulsion} 
  Under the setting of \Cref{lemma_universal}, consider a fixed $t\asymp n^{-1/3+\mathsf c}$. There exists a small constant $\delta_0\in (0,2/3)$ such that the following statements hold. Define the disjoint union of intervals $I_{\delta_0}^+(t)$ as
\[I_{\delta_0}^+(t):=\cup_{k=1}^q[a_{2k}(t)-n^{-\delta_0},a_{2k-1}(t)+n^{-\delta_0}],\]
  and define $k_E:=\mathrm{arg\,min}\{|\gamma_k(t)-E|:k\in \qq{\sfK}\}$ for $E\in I_{\delta_0}^+$.    
  Then, for every constant $\delta\in (0, \delta_0],$ we can find a constant $\epsilon \equiv \epsilon(\delta)>0$ such that %for all $1 \leq k \leq \mathsf{K}-1$
\begin{equation} \label{level repulsion}
\P \left( {\mathsf{N}}_t\left([E-\Delta_{k_E} n^{-\delta}, E+\Delta_{k_E} n^{-\delta}] \right)\geq 2\right) \leq n^{-\delta-\epsilon}, \quad \forall E\in I_{\delta_0}^+.
\end{equation}
The same estimate also applies to $\overline {\mathsf N}_t$.
%\begin{equation} \label{level repulsion}
%\P \pb{\abs{\lambda_k - \lambda_{k+1}} \leq \Delta_k \mathsf{K}^{-\epsilon}} \;\leq\; \mathsf{K}^{-\delta}\,.
%\end{equation}
\end{proposition}
\begin{proof}
An analogous level repulsion estimate inside the bulk has been established for the matrix DBM of Wigner ensembles in \cite[Theorem 5.1]{landon2017convergence} and \cite[Lemma B.1]{benigni2022optimal}, with the key inputs being the local laws and the eigenvalue rigidity estimate. With the help of our Lemmas \ref{sare} and \ref{sare2}, we can readily extend the arguments from \cite{landon2017convergence,benigni2022optimal} to our rectangular matrix DBM $Q(t)$ defined in \eqref{eq_qt}. Specifically, by following the reasoning presented in \cite[Theorem 5.1]{landon2017convergence}, we can establish that \eqref{level repulsion} holds for $E\in \cup_{k=1}^q [a_{2k}+\kappa,a_{2k-1}-\kappa]$, where $\kappa$ is an arbitrarily small positive constant.
Furthermore, as explained in the proof of \cite[Lemma B.1]{benigni2022optimal}, the argument in \cite[Theorem 5.1]{landon2017convergence} can be extended a little bit to show that there exists a sufficiently small constant $\e>0$ so that \eqref{level repulsion} holds for $E\in \cup_{k=1}^q [a_{2k}+n^{-\e},a_{2k-1}-n^{-\e}]$. Since the extension from Wigner ensembles to sample covariance ensembles is standard, we omit the complete details here. 

To establish \eqref{level repulsion} for $E$ around the spectral edges, specifically for $E\in \cup_{k=1}^q ([a_{2k}-n^{-\delta_0},a_{2k}+n^{-\e}]\cup [a_{2k-1}-n^{-\e},a_{2k-1}+n^{-\delta_0}])$, we can refer to \cite[Proposition 6.3]{bloemendal2016principal}. Although this result is stated for the special case with $\Lambda_0=I$ in \cite{bloemendal2016principal}, the same argument can be extended to our more general setting. More precisely, as explained in Lemma 6.4 of \cite{bloemendal2016principal}, we can show that \eqref{level repulsion} holds for the Gaussian ensemble when $X$ is a Gaussian random matrix. 
(The proof of \cite[Lemma 6.4]{bloemendal2016principal} relies on an analysis of the joint eigenvalue probability density, employing a method developed in the proof of \cite[Theorem 3.2]{Bourgade_edge_Beta}. While \cite[Theorem 3.2]{Bourgade_edge_Beta} is presented for the Gaussian $\beta$ ensemble, the same proof carries over almost verbatim to our setting, where the eigenvalues follow a Laguerre $\beta$ ensemble.)  
Next, as explained in the proof of \cite[Lemma 6.5]{bloemendal2016principal}, we can establish the level repulsion estimate \eqref{level repulsion} for $Q(t)$ through a Green's function comparison argument between the case with Gaussian $X$ and that with a non-Gaussian $X$ in $W(0)$. Once again, since this extension is standard and straightforward, we omit the complete details here. This concludes the proof of \eqref{level repulsion} for $E$ near the spectral edges.

%Again, since the extension is standard and straightforward, we do not present the full details here.  

By \eqref{level repulsion}, we know that the level repulsion estimate applies to the eigenvalues of $\cal Q^G_t:=\cal W^G_t(W^G_t)^\top$ with \smash{$W^G_t:= D_t^{1/2} V^\top X^G$}. 
To further prove this estimate for $\overline {\mathsf N}_t$, we can once again employ the Green's function comparison argument between $\cal Q^G_t$ and $\cal Q_t$, as explained in Lemma 6.5 of \cite{bloemendal2016principal}. We omit the details here. 
%we can show that the estimate \eqref{level repulsion} also holds for $\overline {\mathsf N}_t$.
\end{proof}
% \begin{proof}
% \cor We remark that an analogous level repulsion result has been proved for Wigner matrices in Proposition B.17 of \cite{benigni2022optimal}, where the key inputs are the corresponding level repulsion estimate for the Gaussian ensembles,  the local laws, the eigenvalue rigidity estimate, and the universality of eigenvalue statistics. 

% First, to prove (\ref{level repulsion}) for the Gaussian divisible ensembles, one can follow the machinery in Section B.1 of \cite{benigni2022optimal}, where the only difference is that we need to replace the explicit probability density for the joint eigenvalue distribution of symmetric Gaussian divisible ensembles with that for the rectangular matrix Dyson Brownian motion \eqref{eq_qt}. 

% Second, the local laws and eigenvalue rigidity have been provided in Section \ref{appendix_preliminary}. Third, as discussed in \cite{benigni2022optimal}, with the above ingredients, one can follow, for example,  \cite{Erdos2015, PY, TV} to obtain it. With these inputs, we can follow the machinery of the proof of Proposition B.17 of \cite{benigni2022optimal} to conclude the proof. We omit the details due to similarity. 
% \end{proof}

With Proposition \ref{prop: level repulsion}, we now proceed to complete the proof of Lemma \ref{7.2} by employing the method developed in \cite{bloemendal2016principal, knowles2013eigenvector}.

%The proof consists of two steps: the first step is accomplished in Lemma \ref{lem_expansionchacteristic}, while the second step is concluded by Lemma \ref{lem_smoothapproximation}. 
%The first step is relatively standard and we only sketch the proof here.   

\begin{proof}[\bf Proof of Lemma \ref{7.2}] 
\iffalse
As before, for simplicity, without loss of generality, we focus our proof on  all the eigenvectors of  the non-spiked model with $m=1$ and $q=1$. We first assume $p \leq n$ so that $\mathsf{K}=p.$ The case for $p>n$ will be discussed in the end of the proof. 
\fi
We only present the proof for $Q(t)$, while the proof for $\cal Q_t$ is the same. To simplify the presentation, we prove Lemma \ref{7.2} in the special case where $\theta=\theta(p |\avg{\bv, \bu_{i+r}(t)}|^2)$. For brevity, we will use the abbreviations $\bu_i \equiv \bu_i(t)$, $Q \equiv Q(t)$, and \smash{$\wt{\mathcal{G}}_1(z) \equiv \wt{\mathcal{G}}_{1,t}(z) := (Q(t)-z)^{-1}$} for a fixed $t\asymp n^{-1/3+\mathsf c}$. Recall the definition of $\Delta_i$ in (\ref{eq_mathsfKdefinition}) and the definition of $\eta_i$ in \eqref{eq_keyquantitiesdefinition}.
% \begin{equation}\label{eq_choiceofeta}
% \eta_i:= n^{-2\e}\Delta_i. 
% \end{equation}  
We now present the first ingredient of the proof. 
%Till the end of the proof, for notional convenience, with a little bit abuse of notations, throughout the proof, we write $\bu_i \equiv \bu_i(t), Q \equiv Q(t)$ and $\mathcal{G}_1 \equiv \mathcal{G}_{1,t}.$ The proof can be divided into two parts. First, in Lemma \ref{lem_expansionchacteristic}, we will express the eigenvectors as an integral of the resolvent $\wt{\mathcal{G}}_1$ over some properly chosen intervals using the sharp indicator function. Second, as summarized in Lemma \ref{lem_smoothapproximation}, we will replace the indicator function with the smooth functions (\ref{eq_sequecenoffunction}) in and (\ref{eq_smoothqfunction}).  We now provide the first ingredient of the proof. 

\begin{lemma}\label{lem_expansionchacteristic} 
Under the assumptions of Lemma \ref{7.2}, there exists a small constant $\e>0$ such that the following estimate holds for sufficiently small constants $\delta_1\equiv \delta_1(\e)$, $\delta_2\equiv \delta_2(\e,\delta_1)$, and $\nu>0$: 
\begin{equation*}
\mathbb{E} \theta(p |\avg{\bv, \bu_{i+r}(t)}|^2)=\E \theta \left( \frac{p}{\pi} \int_{I_i} \im \big[\bv^\top \wt{\mathcal{G}}_1 (E+\ii \eta_i ) \bv \big] \chi_i(E) \dd E \right)+\OO(n^{-\nu}), \quad i\in \qq{1,\sfK-r},
\end{equation*}
where $I_i$ is defined in \eqref{eq:Ik} and  $\chi_i(E):=\mathbf{1}(\lambda_{i+1}<E_i^- \leq \lambda_i)$. Here, $E_i^{-}=E - n^{\delta_1} \eta_i$ is defined in \eqref{eq_keyquantitiesdefinition}, and we adopt the convention $\lambda_{\sfK+1}= 0$.    
\end{lemma}
\begin{proof}
Our proof follows a similar approach to that of \cite[Lemma 3.1]{knowles2013eigenvector}, \cite[Lemma 7.1]{bloemendal2016principal}, and \cite[Lemma 3.2]{ding2019singular}. The only difference is that we utilize the level repulsion estimate from \Cref{prop: level repulsion} over the entire spectrum, whereas the aforementioned works rely on the level repulsion estimate either near the edges or inside the bulks. We omit the detailed exposition here. 
\end{proof}
\iffalse
To begin with, it is not hard to see that 
\begin{equation*}
|\bv^\top \bu_i|^2=\frac{\eta_i}{\pi} \int_{\mathbb{R}} \frac{|\bv^\top \bu_i|^2}{(E-\lambda_i)^2+\eta_i^2} \dd E.
\end{equation*}
Combining Lemma \ref{sare} and a discussion similar to (3.16) of \cite{ding2019singular} with Proposition \ref{prop: level repulsion}, we find that 
\begin{equation*}
\mathbb{E}\theta(p |\bv^\top \bu_i|^2)= \mathbb{E} \theta \left( \frac{\eta_i}{\pi} \int_{I_i} \frac{|\bv^\top \bu_i|^2}{(E-\lambda_i)^2+\eta_i^2} \chi(E) \dd E \right)+\mathrm{o}(1). 
\end{equation*} 
Furthermore, by spectral decomposition, we see that 
\begin{equation*}
\im \bv^\top \mathcal{G}_1(E+\ri \eta) \bv^\top= \sum_{j \neq i} \frac{\eta_i |\bv^\top \bu_j|^2}{(E-\lambda_j)^2+\eta_i^2}+\frac{\eta_i |\bv^\top \bu_i|^2}{(E-\lambda_i)^2+\eta_i^2}. 
\end{equation*}
Following lines of the arguments between (3.25) and (3.33) of \cite{ding2019singular}, using Proposition \ref{prop: level repulsion}, we can prove that  
\begin{equation*}
\mathbb{E} \theta \left( \frac{\eta_i}{\pi} \int_{I_i} \frac{|\bv^\top \bu_i|^2}{(E-\lambda_i)^2+\eta_i^2} \chi(E) \dd E \right)=\E \theta \left( \frac{p}{\pi} \int_{I_i} \im \bv^\top \mathcal{G}_1 (E+\ii \eta_i ) \bv  \chi(E) \dd E \right)+\mathrm{o}(1). 
\end{equation*} 
This completes our proof.
% {\color{red} [need to mention a little bit how the level repulsion is used, essentially need to control something like $\mathbb{E}(\eta^2/((\lambda_{k-1}-\lambda_k)^2+\eta^2))$] }
\end{proof}
\fi

For the second ingredient of the proof, we show that the indicator function $\chi_i(E)$ can be approximated by the smooth functions $\mathfrak{q}_i(x)$ and $\mathfrak{f}_i(x)$ defined in (\ref{eq_smoothqfunction}) and (\ref{eq_sequecenoffunction}) up to a negligible error. 

\begin{lemma}\label{lem_smoothapproximation} 
Under the assumptions of Lemma \ref{7.2}, there exists a small constant $\e>0$ such that the following estimate holds for sufficiently small constants $\delta_1\equiv \delta_1(\e)$, $\delta_2\equiv \delta_2(\e,\delta_1)$, and $\nu>0$: 
\begin{equation*}
\E \theta \left( \frac{p}{\pi} \int_{I_i} \im \big[\bv^\top \wt{\mathcal{G}}_1 (E+\ii \eta_i ) \bv \big] \chi_i(E) \dd E \right)=
\mathbb{E} \theta \left( \frac{p}{\pi} \int_{I_i} \im \big[\bv^\top \wt{\mathcal{G}}_1 (E+\ii \eta_i ) \bv \big]\mathfrak{q}_i(\tr \mathfrak{f}_i(Q)) \dd E \right) +\OO(n^{-\nu}),
\end{equation*} 
where $i\in \qq{1,\sfK-r}$.
\end{lemma} 
\begin{proof}
%Denote the number of eigenvalues inside the interval $[E_1, E_2]$ as $\mathsf{N}(E_1, E_2)\equiv \mathsf N([E_1,E_2])$. 
%Similar to the proof of (3.39) in \cite{ding2019singular}, we can obtain that with high probability  (recall (\ref{eq_keyquantitiesdefinition})) 
Using the rigidity of eigenvalues in \eqref{TAZ2}, we can obtain that with high probability,
\begin{align*}
p\int_{I_i}  \im \big[\bv^\top \wt{\mathcal{G}}_1 (E+\ii \eta_i) \bv \big] \chi_i(E) \dd E & =p \int_{I_i} \im \big[\bv^\top \wt{\mathcal{G}}_1 (E+\ii \eta_i) \bv \big] \mathbf{1}(\mathsf{N}([E^-_i, E^+])=i)  \dd E  \\
&=p \int_{I_i}  \im \big[\bv^\top \wt{\mathcal{G}}_1 (E+\ii \eta_i) \bv \big] \mathfrak{q}_i[\tr (\chi_E(Q)) ]\dd E, 
\end{align*}
where $E^+$ is defined in \eqref{eq_keyquantitiesdefinition} and we denote $\chi_E(x):=\mathbf{1}(x\in [E_i^- , E^+])$. Next, by utilizing Proposition \ref{prop: level repulsion} and following the proof of (4.11) in \cite{ding2019singular}, we can show that 
\begin{equation*}
\E \theta \left( \frac{p}{\pi} \int_{I_i} \im \big[\bv^\top \wt{\mathcal{G}}_1 (E+\ii \eta_i) \bv \big] \mathfrak{q}_i[\tr (\chi_E(Q)) ] \dd E \right)=
\mathbb{E} \theta \left( \frac{p}{\pi} \int_{I_i} \im \big[\bv^\top \wt{\mathcal{G}}_1 (E+\ii \eta_i) \bv \big] \mathfrak{q}_i(\tr \mathfrak{f}_i(Q)) \dd E \right) +\OO(n^{-\nu}).
\end{equation*} 
Combining the above two equations completes the proof.
\end{proof}

We observe that Lemmas \ref{lem_expansionchacteristic} and \ref{lem_smoothapproximation} together establish Lemma \ref{7.2} when $L=1$. The proof for the general case of $L>1$ is analogous, as explained in Section 4 of \cite{knowles2013eigenvector}. We omit the details here.
%The proof of the general case with $L>1$ is analogous, and we omit the details. 
\iffalse
Before concluding the proof, we discuss the case when $p>n$ and $n+1 \leq i \leq p.$ In fact, the arguments and discussions still hold by replacing $\eta_i$ in (\ref{eq_choiceofeta}) with $\eta_i=n^{-D}$ for some large constant $D>0,$ and using the fact that (iii) if Assumption \ref{main_assumption} implies that $\lambda_j$ is bounded from below for all $1 \leq j \leq n.$ We omit further details.

For simplicity, we focus on all the eigenvectors $\bu_i, 1 \leq i \leq p$ of the non-spiked model. However, our arguments apply to the non-outlier eigenvectors $\bu_i, r+1 \leq i \leq p$ of the spiked model straightforwardly and easily following the discussions of Remarks B.15 and C.2 of \cite{9779233}. We omit further details. 
\fi 
\end{proof}

\subsection{Proof of Lemma \ref{nanzi}}\label{sec_prooflemma410}

%Recall $\mathcal{W}_t$ in (\ref{eq_mathcalwt}) and $W(t)$ in (\ref{defn_wt}). For fixed $t$ and $0\le s\le 1$, 
In this subsection, we prove Lemma \ref{nanzi} using a continuous comparison argument developed in \cite[Section 7]{MR3704770}. 
Recall the sequence of interpolating matrices $W^s_t\in \mathbb{R}^{p\times n}$ defined in \eqref{eq_dtusx}. For the sake of simplicity, we denote \smash{$Y^s\equiv Y^s_t:=D_t^{-1/2}W^s_t$}, i.e., 
%such that $W^s_t = D_t^{1/2}Y_t^s$, then the entries of $Y^s$ can be expressed as 
\begin{equation*}
Y^s_{i\mu}=
\sum_{j=1}^{2p} 
U^1_{ij}\chi^s_{j\mu} \mathsf{X}_{j\mu}
+ \sum_{j=1}^{2p} 
U^0_{ij} (1-\chi^s_{j\mu})\mathsf{X}_{j\mu}, \quad i\in \mathcal{I}_1, \ \mu\in\mathcal{I}_2, 
\end{equation*}
where $\{\chi^s_{j\mu}\}_{j\in\qq{2p},\mu\in\qq{n}}$ are i.i.d.~Bernoulli$(s)$ random variables.
% satisfying that 
% \begin{equation*}
% \P(\chi ^s_{j\mu}=1)=s, \quad \P(\chi ^s_{j\mu}=0)=1-s. 
% \end{equation*}
% In particular, we have $Y^s=U^s\mathsf{X}$ for $s=0,1$ as defined in (\ref{eq_rewriteone}) and (\ref{eq_rewritetwo}). Extending this definition to $0<s<1$ and setting $Y^s:=U^s\mathsf{X}$ we get the formal definition of $\mathsf{Q}_t^s$ in \eqref{eq_dtusx0}. 
Alternatively, we can express $Y^s$ as
\begin{equation}\label{defZth1}
Y^s=\sum_{j=1}^{2p} \sum_{\mu=1}^{n} Z^s_{(j\mu)}, \quad \text{with}\quad Z^s_{(j\mu)}:=  \mathsf{X}_{j\mu} \Delta^s_{(j\mu)},\quad 
\Delta^s_{(j\mu)}:=\chi^s_{j\mu} \left( U^1 \mathbf{e}_{j\mu} \right)
+ (1-\chi^s_{j\mu})\left(U^0 \mathbf{e}_{j\mu} \right).
\end{equation} 
Here, $\mathbf{e}_{j\mu}:=\mathbf e_j \mathbf e_\mu^\top$ represents a $2p \times n$ matrix with only one non-zero entry at the $(j,\mu)$-th position. 
% Note for any fixed $j$ and $\mu$, $Z^s_{(j\mu)}$ is a $p \times n$ matrix satisfying that 
%\begin{equation*}
%\left(Z^s_{(j\mu)}\right)_{i\mu'}={\bf 1}_{\mu=\mu'} \mathsf{X}_{j\mu}
%\left(  \chi ^s_{j\mu}U_{1}+ (1-\chi ^s_{j\mu}) U^0\right)_{i\mu}.  
%\end{equation*}
%the sequence of interpolating matrices $\{Y^{s, \Gamma}_{(j \mu)}\}$ as 
%\begin{equation}\label{eq_interpolation}
%\left( Y^{s,\Gamma}_{(j \mu)} \right)_{i \nu}:=
%\begin{cases}
%\lambda, & \text{if} \ (i, \nu)=(j, \mu) \\
%Y^s_{i \nu}, & \text{otherwise}
%\end{cases}. 
%\end{equation}
Furthermore, given a matrix $\Gamma \in \mathbb{R}^{p\times n}$, we define 
\begin{equation}\label{eq_interpolation} Y^{s,\Gamma}_{(j \mu)} :=\sum_{(i, \nu)\ne (j, \mu)}Z_{(i \nu)}^s+\Gamma.
\end{equation}
In particular, using this notation, we have $Y^s-Y^{s, \mathbf{0}}_{(j \mu)} =  Z_{(j \mu)}^s$.
Similar to \eqref{eq_defnfinalG} and \eqref{aniso}, we will use the following linearized resolvent:
\begin{equation}\label{eq_generallinearmatrix}
G^s_t(z)  :=\begin{pmatrix}
-D_t^{-1} & Y^s \\
(Y^s)^\top &-z
\end{pmatrix} ^{-1}=\begin{pmatrix}
z D_t^{1/2} {\cal G}_{1,t}^s(z) D_t^{1/2} & D_t Y^s \mathcal{G}^s_{2,t} \\
\mathcal{G}^s_{2,t} (Y^s)^\top D_t  & \mathcal{G}^s_{2,t}
\end{pmatrix},
\end{equation}
where $ \cal G_{1,t}^s:=(D_t^{1/2}Y^s(Y^s)^\top D_t^{1/2} -z)^{-1}$ and $\mathcal{G}_{2,t}^s:=((Y^s)^\top D_t Y^s-z)^{-1}$. 
For $s\in \{0,1\}$, the local laws for the resolvent $G^s$ have already been established in \Cref{lem_standardaniso} (when $s=0$) and Lemma \ref{sare2} (when $s=1$). These local laws can also be extended to all $s\in (0,1)$, as summarized in the following lemma. 

%\quad  Y^s=Y^s_t:= U^s \mathsf{X} \in \mathbb{R}^{p \times n}.
%,\quad G_M=zD^{1/2}\left(WW^*-z\right)^{-1}D^{1/2} 
% The proof of Lemma \ref{nanzi} relies on the local laws for the interpolating matrices $G^s$ and a suitable functional integral representation of \eqref{eq_comparisonkeykeykeykey}.
% Note that in the case when $s=1$, the local laws for $G^1$ have already been established in Lemma \ref{sare}.  
%The proof of Lemma \ref{nanzi} relies on the such similar results for the interpolating matrices $G^s$ and the following functional integral representation. 
% The local laws are summarized in the following lemma. 
 
\begin{lemma}\label{lem_locallaw} 
Suppose the assumptions of Lemma \ref{nanzi} hold. Define $\wt\Pi_t$ and $\Psi_t$ as in \Cref{sec: iso of Wt} with $\Lambda=D_0$ and $\Lambda_0=D_{00}$. Then, there exists a constant $c>0$ such that the following local law holds uniformly in $z\in  \mathbf{D}'(c,\tau, t,n )$ for any constant $\tau>0$ and deterministic unit vectors $\bu, \bv  \in \mathbb{R}^{p+n}$: 
	\begin{equation}
	\left| \mathbf{u}^\top G^s_t(z) \mathbf{v}-\mathbf{u}^\top \wt\Pi_{t}(z) \mathbf{v} \right| \prec \Psi_t(z).\label{eq_localGs}
\end{equation}
%Let  $t=\mathrm{O}(n^{-1/3+\mathsf{c}})$ with some constant $0<\mathsf{c} \leq 1/6$ and $\Pi_t$ be defined in the same way as (\ref{eq_Pigz}) by replacing $m(z), \Sigma_0$ with $m_t(z), D_t,$ respectively.
%	Let $\mathbf{u}, \mathbf{v}$ be two deterministic vectors in $\mathbb{R}^{p+n}.$ Then for any $z \in \mathbf{D}_s$ in (\ref{eq_generalspectraldomain}) uniformly, we have that for all $0 \leq s \leq 1$ 
 %\sqrt{\frac{\im m_t(z)}{n \eta}}+\frac{1}{n \eta}.
 % Recall (\ref{koux}).
\end{lemma}
% \begin{proof}
% 	See Appendix \ref{sec_proofoflocallaw}.
% \end{proof}

%The local laws are summarized in the following lemma. Recall (\ref{koux}). 
%\begin{lemma}\label{lem_locallaw} Suppose the assumptions of Lemma \ref{nanzi} hold.  Let  $t=\mathrm{O}(n^{-1/3+\mathsf{c}})$ with some constant $0<\mathsf{c} \leq 1/6$ and $\Pi_t$ be defined in the same way as (\ref{eq_Pigz}) by replacing $m(z), \Sigma_0$ with $m_t(z), D_t,$ respectively.
%Let $\mathbf{u}, \mathbf{v}$ be two deterministic vectors in $\mathbb{R}^{p+n}.$ Then for any $z \in \mathbf{D}$ in (\ref{eq_generalspectraldomain}) uniformly, we have that for all $0 \leq s \leq 1$ 
%\begin{equation*}
%\left| \mathbf{u}^\top G^s(z) \mathbf{v}-\mathbf{u}^\top \Pi_{t}(z) \mathbf{v} \right| \prec \sqrt{\frac{\im m_t(z)}{n \eta}}+\frac{1}{n \eta}.
%\end{equation*}
%\end{lemma}
%\begin{proof}
%See Appendix \ref{sec_proofoflocallaw}. 
%\end{proof}

%\subsubsection{Proof of Lemma \ref{lem_locallaw}}\label{sec_proofoflocallaw}

We now prove this lemma using the continuous comparison argument presented in \cite{MR3704770,10.1214/19-EJP381}. In the interest of brevity, we will provide a sketch of the key points, highlighting the main differences in the argument. For a more comprehensive understanding, readers can refer to Section 6 of \cite{10.1214/19-EJP381}. 

%Due to similarity, we only sketch the key points with a focus on the main differences in the argument. For more details, the readers can refer to Section 6 of \cite{10.1214/19-EJP381}. 
% For simplicity, till the end of this section, we denote 
% \begin{equation*}
% \Psi \equiv \Psi_t(z):=\sqrt{\frac{\im m_t(z)}{n \eta}}+\frac{1}{n \eta}.
% \end{equation*}

\begin{proof}[\bf Proof of Lemma \ref{lem_locallaw}]
For any $\eta\ge N^{-1+\tau}$ and a sufficiently small constant $\delta\in (0,1/100)$, we define
\begin{equation}\label{eq_comp_eta}
\eta_l:=\eta N^{\delta l} \text{ for } \ l=0,...,K-1,\ \ \ \eta_K:=1,
\end{equation}
where $K\equiv K(\eta):=\max\left\{l\in\mathbb N: \eta N^{\delta(l-1)}<1\right\}.$ Since $z\mapsto G_t^s(z)- \wt\Pi_t(z)$ is Lipschitz continuous in $\mathbf D'\equiv  \mathbf{D}'(c,\tau, t,n )$ with Lipschitz constant of order $\OO(\eta^{-2})=\OO(n^2)$, it suffices to show that \eqref{eq_localGs} holds for all $z$ in a discrete but suitably dense subset ${\mathbf S} \subset \mathbf D'$. Specifically, we choose $\mathbf S$ to be an $n^{-10}$-net of $\mathbf D'$ such that $ |\mathbf S |\lesssim n^{20}$ and
\[E+\ii\eta\in\mathbf S\Rightarrow E+\ii\eta_l\in\mathbf S\text{ for }l=1,\ldots,K(\eta).\]
%Let $\mathbf{S}$ be an $n^{-10}$-net of $\mathbf{D}$ in (\ref{eq_generalspectraldomain}) (see Definition 6.2 of \cite{10.1214/19-EJP381} for some detailed definition). 
Similar to \cite[Section 6]{10.1214/19-EJP381}, our proof relies on an induction argument based on two scale-dependent properties, denoted as \textbf{($\b A_d$)} and \textbf{($\b C_d$)}, formulated on the subsets
%\begin{equation*}
${\b	S}_d  \deq \hb{z \in {\b S} \col \im z \geq n^{-\delta d}},$ $d\in \N$.
%\end{equation*}
%where $d \in \N$ and $\delta>0$ is some small constant.  
\begin{itemize}
\item[\textbf{($\b A_d$)}]
For all $z \in {\b S}_d$, we have that for any deterministic unit vectors $\bu, \bv  \in \mathbb{R}^{p+n}$,
\begin{equation} \label{bound_assumption}
\big| \mathbf{u}^\top G^s_t(z) \mathbf{v}-\mathbf{u}^\top \wt\Pi_{t}(z) \mathbf{v} \big| \prec 1.  
\end{equation}
%for all orthogonal $U^1, U_2$ and $X$ satisfying \eqref{cond on entries of X} and \eqref{moments of X-1}.
\item[\textbf{($\b C_d$)}]
For all $z \in {\b S}_d$, we have that for any deterministic unit vectors $\bu, \bv  \in \mathbb{R}^{p+n}$,
\begin{equation} \label{bound_goal}
\big| \mathbf{u}^\top G^s_t(z) \mathbf{v}-\mathbf{u}^\top \wt\Pi_{t}(z) \mathbf{v} \big| \prec n^{24 \delta} \Psi(z).
\end{equation}
 \end{itemize}
It is clear that property ($\mathbf{C}_d$) implies property ($\mathbf{A}_d$) if $\delta$ is chosen sufficiently small (depending on $\tau$) such that $n^{24 \delta} \Psi(z)\ll 1$. On the other hand, if we can show that property ($\mathbf{A}_{d-1}$) implies property ($\mathbf{C}_d$) for all $d \leq \delta^{-1}$, then we can conclude that (\ref{bound_goal}) holds for all $z \in \mathbf{S}$ by induction on $d$. By polarization, it suffices to prove that when property (${\b A}_{d-1}$) holds, we have
\begin{equation}\label{goal_ani2}
\big| \mathbf{v}^\top G^s_t(z) \mathbf{v}-\mathbf{v}^\top \wt\Pi_{t}(z) \mathbf{v} \big| \prec n^{24 \delta} \Psi(z),
\end{equation}
for all $z\in \mathbf S_d $ and any deterministic unit vector $ \mathbf v\in{\mathbb C}^{p+n}$. In fact, we can derive the more general bound \eqref{bound_goal} by applying (\ref{goal_ani2}) to the vectors $\mathbf u + \mathbf v$ and $\mathbf u + \ri\mathbf v$, respectively. 

By utilizing Markov's inequality, showing \eqref{goal_ani2} requires proving that for any fixed large $\mathsf{q} \in 2\mathbb{N}$ and $d \leq \delta^{-1},$ there exists a constant $C>0$ such that 
%when property (${\b A}_{d-1}$) holds, we have that for some constant $C>0$, 
\be\label{benben}
|\E F^{\bv,  \, \mathsf{q}} (Y^s)|\le C(n^{24\delta}\Psi)^{\mathsf{q}}
 \ee
uniformly for all $z\in \mathbf S_d$. %$V \in \mathbb{R}^{a \times (p+n)}$ is an arbitrary deterministic semi-orthogonal matrix, where $a$ is a fixed integer; 
Here, $ F^{\bv, \mathsf{q}}  (Y^s)$ is defined as 
\begin{equation}\label{eq_FVQdefinition}
F^{\bv,  \,\mathsf{q}}  (Y^s) \equiv F^{\bv,  \,\mathsf{q}}  (Y^s, z):=  \big| \bv^\top G_t^s \bv -\bv^\top \wt\Pi_t \bv\big| ^{\mathsf{q}}.  
\end{equation}
To establish (\ref{benben}), we follow the arguments presented in Sections 6.1 and 6.2 of \cite{10.1214/19-EJP381}. First, we see that \eqref{eq_localGs} and consequently (\ref{benben}) hold when $s=0$ by \Cref{sare2}. Second, similar to \cite[Lemma 7.9]{MR3704770}, we can employ the fundamental theorem of calculus to obtain that 
 \begin{equation}\label{eq_calculusoffundemantales}
\E F^{\bv, \mathsf{q}}(Y^1) - \E F^{\bv, \mathsf{q}}(Y^0) \;=\; \int_0^1 \dd s \, \sum_{j=1}^{2p} \sum_{\mu=1}^{n} \qB{ \E F^{\bv, \mathsf{q}} \pB{Y^{s,Z^{1}_{(j\mu)}}_{(j \mu)}} - \E F^{\bv, \mathsf{q}} \pB{Y^{s,Z^{0}_{(j\mu)}}_{(j \mu)}}}\,,
\end{equation} 
where we use the notation introduced in (\ref{eq_interpolation}) with $\Gamma=Z^{0}_{(j\mu)}$ or $Z^{1}_{(j\mu)}$. 
%for $1 \leq \alpha, \beta \leq a,$ 
Under the property (${\b A}_{d-1}$), we claim the following estimate:  
\be\label{yanglao}
 \sum_{j=1}^{2p} \sum_{\mu=1}^n 
 \qB{ \E F^{\bv,\mathsf{q}} \pB{Y^{s,Z^{1}_{(j\mu)}}_{(j \mu)}}
  - \E F^{\bv,\mathsf{q}} \pB{Y^{s,Z^{0}_{(j\mu)}}_{(j \mu)}}}
 \lesssim (n^{24\delta}\Psi)^{\mathsf{q}}+ | \E  F^{\bv, \mathsf{q}} (Y^s)|.
% +  |\E F^{I, \mathsf{q}} (Y^s)|   + | \E  F^{V, \mathsf{q}} (Y^s)|.
\ee
By substituting this estimate into \eqref{eq_calculusoffundemantales} and applying a standard Gr\"{o}nwall's argument, we conclude (\ref{benben}). 
%(see the discussions below Lemma 7.10 of \cite{MR3704770}), we can prove (\ref{benben}). 

The remainder of the proof focuses on establishing (\ref{yanglao}).  To simplify the presentation, until the end of this proof, we slightly abuse the notation and define
%similar to (\ref{eq_generallinearmatrix}), for some matrix $Y,$ we denote that
\begin{equation}\label{eq_GYdefinition}
 G(Y) \equiv G(Y,z)  :=\begin{pmatrix}
-D_t^{-1} & Y \\
Y^\top &-z
\end{pmatrix} ^{-1}
\end{equation}  
for given random matrix $Y$. 
%with a little bit of abuse of notations, 
%By the definition of the matrices $Y^{s, 0}_{(j \mu)}$ and $Y^s$, 
% we see that 
%  $$
% Y^{s, 0}_{(j \mu)}-Y^{s} = \chi^s_{j\mu} \left( U^1 {\bf e}_{j\mu} \right)\mathsf{X}_{j\mu}
%   +
%     (1-\chi ^\theta_{j\mu}) \left(U^0 {\bf e}_{j\mu} \right)\mathsf{X}_{j\mu}. 
% $$
%(\ref{eq_GYdefinition}) and resolvent expansion, we obtain that 
By the definition \eqref{eq_interpolation}, for any given $\Gamma,\Gamma' \in \mathbb{R}^{p\times n}$ and $M\in \N$, we have the following resolvent expansion:
%of the matrices $Y^{s, 0}_{(j \mu)}$ and $Y^s$, we have the following resolvent expansion:
\be
\begin{aligned}\label{6j8}
G\left(Y^{s,\Gamma'}_{(j \mu)}\right)-G\left(Y^{s, \Gamma}_{(j \mu)}\right) = &~\sum_{k=1}^M G\left(Y^{s, \Gamma}_{(j \mu)}\right)\left[\begin{pmatrix}
    0 & \Gamma-\Gamma'\\ (\Gamma-\Gamma')^\top &0
\end{pmatrix}G\left(Y^{s, \Gamma}_{(j \mu)}\right)\right]^k \\
&~+ G\left(Y^{s, \Gamma'}_{(j \mu)}\right)\left[\begin{pmatrix}
    0 & \Gamma-\Gamma'\\ (\Gamma-\Gamma')^\top &0
\end{pmatrix}G\left(Y^{s, \Gamma}_{(j \mu)}\right)\right]^{M+1}. 
\end{aligned}
\ee
% \begin{align}\label{6j8}
% G\left(Y^{s,\Gamma}_{(j \mu)}\right)-G\left(Y^{s, \Gamma'}_{(j \mu)}\right)
% =  &~ \chi^s_{j\mu} \sum_{k=1}^M G(Y^s)
%  \left[\begin{pmatrix}
% \mathbf{0}& U^1 {\bf e}_{j\mu} 
%  \\
%  \left(U^1 {\bf e}_{j\mu} \right)^\top& \mathbf{0}
%  \end{pmatrix} 
%  G(Y^s)
% \right]^k \mathsf{X}_{j\mu}^k
% \\\nonumber
% &~ +  (1-\chi^s_{j\mu}) \sum_{k=1}^M G(Y^s)
%  \left[\begin{pmatrix}
% \mathbf{0}& U^0 {\bf e}_{j\mu} 
%  \\
%  \left(U^0 {\bf e}_{j\mu} \right)^\top& \mathbf{0}
%  \end{pmatrix} 
%  G(Y^s)
% \right]^k \mathsf{X}_{j\mu}^k \nonumber\\
% &~+ G(Y^s) . 
% \end{align}
By employing this expansion with $\Gamma=Z_{(j\mu)}^s$ and choosing a sufficiently large $M$, we can utilize property $({\bf A}_{d-1})$ for $G_t^s$ and the rough bound 
$\|G(Y^{s,\Gamma'}_{(j \mu)})\|\prec \eta^{-1}$ to verify the following inequalities for any deterministic unit vectors $\bu, \bv \in \mathbb{R}^{p+n}$:
\be\label{zayy}
 \left| \bu^\top G\pB{Y^{s,Z^{1}_{(j\mu)}}_{(j \mu)}} \bv \right|\prec 1, \quad \left| \bu^\top G\pB{Y^{s,Z^{0}_{(j\mu)}}_{(j \mu)}} \bv \right|\prec 1, \quad   \left| \bu^\top G\pB{Y^{s, \mathbf{0}}_{(j \mu)}} \bv \right|\prec 1.
\ee
\iffalse
We point out that the actual calculation of (\ref{zayy}) relies on the simple structure of ${\bf e}_{j\mu}$ in the sense that (\ref{6j8}) is still very simple after being expanded. For example, for $1 \leq \ell_1, \ell_2 \leq p+n,$ we have that 
$$
\left[G(Y^s) \begin{pmatrix}
\mathbf{0}& U^0 {\bf e}_{j\mu} 
 \\
 \left(U^0 {\bf e}_{j\mu} \right)^\top& \mathbf{0}
 \end{pmatrix} 
 G(Y^s)\right]_{\ell_1 \ell_2}=\left[G(Y^s)\right]_{\ell_1 \bu^0_{ j}}\left[G(Y^s)\right]_{\mu \ell_2}+\left[G(Y^s)\right]_{\ell_1\mu}\left[G(Y^s)\right]_{\bu^0_{ j} \ell_2}.
$$
where $\bu^0_j$ is the $j$-th column of $U^0$, which we consider it as a vector in $\R^{\cal I_1}$.  \fi 
To simplify the notation, for fixed $s\in (0,1)$, we further define
\begin{equation} \label{def_fch}
f _{(j \mu)}(\lambda_0,  \lambda_1 ) \deq 
F^{\bv, \mathsf{q}}
\left(Y^{s, \lambda_0 U^0 {\bf e}_{j \mu}+ \lambda_1 U^1 {\bf e}_{j \mu}}_{(j \mu)}\right)\, , 
 \end{equation}
 and denote its derivatives as 
 $$ f_{(j \mu)}^{(k,l)}=\partial_{\lambda_0}^k \partial_{\lambda_1}^l f _{(j \mu)}(\lambda_0,  \lambda_1 ),\quad k,l\in \N.$$
%Following a similar discussion as that between (6.24)--(6.26) of \cite{10.1214/19-EJP381}, using \eqref{zayy}, (ii) of Assumption \ref{main_assumption} and Taylor expansion, 
By applying the Taylor expansion and using the estimate \eqref{zayy}, we derive that 
 \begin{align}
&  \E F^{\bv,\mathsf{q}} \pB{Y^{s,Z^{0}_{(j\mu)}}_{(j \mu)}} 
- \E F^{\bv,\mathsf{q}} \pB{Y^{s,\mathbf{0}}_{(j \mu)}}= \sum_{\ell = 2}^{4 \mathsf{q}} \frac{1}{\ell!} \E f_{(j \mu)}^{(\ell,0)}(0,0) \E (\mathsf{X}_{j \mu} )^\ell + \OO_\prec(\Psi^\mathsf{q}), \label{haunting}\\
%   &\;=\;
% \E \qb{f _{(j \mu)}\pb{\mathsf{X}_{j \mu} , 0} - f _{(j \mu)}(0, 0)}   
%   \\\nonumber
%   &\;=\; 
& \E F^{\bv,\mathsf{q}} \pB{Y^{s,Z^{1}_{(j\mu)}}_{(j \mu)}} - \E F^{\bv,\mathsf{q}} \pB{Y^{s,\mathbf{0}}_{(j \mu)}}
  \;=\;  \sum_{\ell = 2}^{4 \mathsf{q}} \frac{1}{\ell!} \E f_{(j \mu)}^{(0,\ell)}(0,0) \E (\mathsf{X}_{j \mu} )^\ell + \OO_\prec(\Psi^\mathsf{q}),\label{haunting2}
  \end{align}
  where we have utilized the mean zero condition $\mathbb{E}\mathsf{X}_{j \mu}=0$ and the independence between $Y^{s, \b0}_{(j \mu)}$ and $\mathsf{X}_{j \mu} $. Now, to establish \eqref{yanglao}, we only need to prove that 
\be\label{luo}
 \sum_{j,\mu} \E (\mathsf{X}_{j\mu})^\ell  \left( \E f_{(j \mu)}^{(\ell,0)}(0,0)- \E f_{(j \mu)}^{(0, \ell)}(0,0)\right) \lesssim (n^{24\delta}\Psi)^\mathsf{q} + | \E  F^{\bv, \mathsf{q}} (Y^s)|.  
\ee
  
Similar to \eqref{haunting} and \eqref{haunting2}, we can utilize the Taylor expansion and the estimate \eqref{zayy} to derive that 
 \begin{align}\label{fei}
 & \E f_{(j \mu)}^{(\sk,\wt \sk)}(0,0)= \E f_{(j \mu)}^{(\sk,\wt \sk)}\left(\chi^s_{j\mu} \mathsf{X}_{j \mu}, (1-\chi^s_{j\mu}) \mathsf{X}_{j \mu}\right)
 \\\nonumber
& -s \sum_{l = 2}^{4\mathsf{q}-\sk-\wt \sk} \frac{1}{l !}  \E f_{(j \mu)}^{(\sk+l,\wt \sk )}(0,0)
 \E (\mathsf{X}_{j \mu} )^{l} 
-(1-s) \sum_{\wt l = 2}^{4\mathsf{q}-\sk-\wt \sk} \frac{1}{\wt l !} \E f_{(j \mu)}^{(\sk,\wt \sk+\wt l)}(0,0) 
 \E (\mathsf{X}_{j \mu} )^{\wt l}
 +\OO_\prec (\Psi^{\mathsf{q}-\sk-\wt \sk}) 
\end{align}
 for any fixed $\sk$ and $\wt\sk$. In the derivation, we also used the fact that $\chi^s_{j\mu}(1-\chi^s_{j\mu})=0$. By repeatedly applying \eqref{fei}, we can further obtain that 
\begin{align*}
& \E f_{(j \mu)}^{(\ell,0)}(0,0)=\sum_{\sk, \wt \sk: \sk+\wt \sk=0}^{4\mathsf{q}-\ell} C_{(j\mu)}^{\sk, \wt \sk} \sum_{j,\mu}\E f_{(j \mu)}^{(\ell+\sk, \wt \sk)}\left(\chi^s_{j\mu} \mathsf{X}_{j \mu}, (1-\chi^s_{j\mu}) \mathsf{X}_{j \mu}\right)+\OO_\prec(\Psi^{\mathsf{q}-\ell}),\\
& \E f_{(j \mu)}^{(0,\ell)}(0,0)=\sum_{\sk, \wt \sk: \sk+\wt \sk=0}^{4\mathsf{q}-\ell} C_{(j\mu)}^{\sk, \wt \sk} \sum_{j,\mu}\E f_{(j \mu)}^{(\sk, \ell+\wt \sk)}\left(\chi^s_{j\mu} \mathsf{X}_{j \mu}, (1-\chi^s_{j\mu}) \mathsf{X}_{j \mu}\right)+\OO_\prec(\Psi^{\mathsf{q}-\ell}),
\end{align*}
where $C_{(j\mu)}^{ \sk , \wt \sk}$ represents deterministic coefficients satisfying that $C_{(j\mu)}^{0,0}=1$, 
%$=\OO(n^{-{(\sk+\wt \sk)}/2}) $. In addition, since $\mathbb{E}\mathsf{X}_{j \mu}=0,$ it is not hard to see that 
\be\label{curG}
C_{(j\mu)}^{ \sk , \wt \sk}=\OO(n^{-{(\sk+\wt \sk)}/2}) ,\quad \text{and}\quad   \sk=1 \ \ \text{or}\ \ \wt \sk=1 \implies C_{(j\mu)}^{ \sk , \wt \sk}=0.  
\ee
Now, to conclude \eqref{luo}, it suffices to control $\E f_{(j \mu)}^{(\ell+\sk, \wt \sk)}\left(\chi^s_{j\mu} \mathsf{X}_{j \mu}, (1-\chi^s_{j\mu}) \mathsf{X}_{j \mu}\right)-\E f_{(j \mu)}^{(\sk, \ell+\wt \sk)}\left(\chi^s_{j\mu} \mathsf{X}_{j \mu}, (1-\chi^s_{j\mu})\mathsf{X}_{j \mu}\right)$. 
For the terms in (\ref{fei}), we express them as  
\be\label{dong}
 \E f_{(j \mu)}^{(\sk,\wt \sk)}\left(\chi^s_{j\mu} \mathsf{X}_{j \mu}, (1-\chi^s_{j\mu}) \mathsf{X}_{j \mu}\right)
=  \E \Big[
 \Big(P^{U^0}_{ j\mu}\Big)^\sk
 \Big(P^{U^1}_{ j\mu}\Big)^{\wt \sk}F^{\bv,\mathsf{q}} \Big]\left(Y^{s}\right),
\ee
where we adopt the following notation given any $p\times 2p$ matrix $U$:   
\begin{equation}\label{eq_operatoroperator}
 P^{U }_{ j\mu}:=(U^\top \nabla_Y)_{j\mu}= \sum_{i=1}^p{U}_{ij}\frac{\partial}{\partial Y^s_{i\mu}} .
\end{equation}
Note that $P^{U }_{ j\mu}$ acting on any resolvent entry yields
%Note that since $F^{V, \mathsf{q}}$ is essentially a product of $G^s$'s, we find that for any $\ell_1, \ell_2\in \cal I_1 \cup \cal I_2,$  
\be\label{Pujmu}
P^{U}_ {j\mu} G_{\fa \fb}= -G_{\fa \bu_j}G_{\mu \fb}- G_{\fa \mu}G_{\bu_j \fb},\quad \fa,\fb\in \cal I,
\ee
where $\bu_j=U\mathbf e_j$ represents the $j$-th column vector of $U$. With the above notations, we observe that to establish \eqref{luo},  %using \eqref{dong}, 
it is sufficient to prove that for any $\ell,$ $\sk$,  $\wt\sk$ satisfying $\sk\ne 1$, $\wt\sk\ne 1$, $\ell\ge 2$, and $\ell+\sk+\wt \sk \le 4\mathsf{q},$ 
%\quad  \wt \ell+\sk+\wt \sk \le 4\mathsf{q},$$
\begin{align}\label{fouling}
&n^{-(\sk +\wt \sk+\ell)/2} \sum_{j,\mu}  \cal C_{(j\mu)}^\ell \E \Big\{  \Big[\Big(P^{U^0}_{ j\mu} \Big)^{ \ell+\sk} \Big(P^{U^1}_{ j\mu} \Big)^{\wt \sk} F^{\bv,\,\mathsf{q}}\Big]\left(Y^{s}\right) -  \Big[\Big(P^{U^0}_{ j\mu} \Big)^{\sk} \Big(P^{U^1}_{ j\mu} \Big)^{ \ell+  \wt \sk}F^{\bv,\,\mathsf{q}}\Big]\left(Y^{s}\right) \Big\}
\\\nonumber
& \lesssim (n^{24\delta}\Psi)^\mathsf{q}+ | \E  F^{\bv, \mathsf{q}} (Y^s)|. 
\end{align}
Here, we introduce the notation $\cal C_{(j\mu)}^\ell := n^{\ell/2}\E \left(\mathsf{X}_{j\mu}\right)^\ell$, which is of order $\OO(1)$ by \eqref{eq_momentassumption}. 
% we have that 
% \begin{align}\label{fouling}
% \left(n^{-1/2}\right)^{ \sk +\wt \sk}  \sum_{j\mu} & \left(\E \mathsf{X}_{j\mu}^\ell \right)  \Big(  \left[\left(P^{U^1}_{ j\mu} \right)^{ \ell+\sk} \left(P^{U^0}_{ j\mu} \right)^{\wt \sk} F^{V,\,\mathsf{q}}_{\alpha \beta}\right]\pB{Y^{s}} - \nonumber \\
% & \left[\left(P^{U^1}_{ j\mu} \right)^{\sk} \left(P^{U^0}_{ j\mu} \right)^{ \ell+  \wt \sk}F^{V,\,\mathsf{q}}_{\alpha \beta}\right]\pB{Y^{s}} \Big)
% \\\nonumber
% = &\; \OO_\prec (\Psi^\mathsf{q}) +  \|\E F^{I, \mathsf{q}} (Y^s)\|_{\max}
%    + \| \E  F^{V, \mathsf{q}} (Y^s)\|_{\max}. 
% \end{align}
%The rest of the proof is devoted to verifying (\ref{fouling}) case by case. 
The remaining part of the proof is dedicated to verifying (\ref{fouling}). To accomplish this, we divide the proof into two cases.

\medskip
\noindent{\bf Case 1: $\ell+\sk+\wt \sk=2.$} In this case, we must have $\ell=2, \sk=\wt \sk=0$. Using the fact that $U^0 (U^0)^{\top}=U^1(U^1)^{\top}=I$, we can easily check that 
$$
\sum_{j,\mu}\left(P^{U^1}_{ j\mu} \right)^{2} F^{\bv,\,\mathsf{q}} (Y^s)
  =\sum_{j,\mu}\left(P^{U^0}_{ j\mu} \right)^{2}F^{\bv,\,\mathsf{q}}(Y^s)=\sum_{i\mu }\frac{\partial^2}{\partial (Y^s_{i\mu})^2} F^{\bv,\,\mathsf{q}}(Y^s).
$$ 
% where we denote 
% $$\Delta_Y^s:= \sum_{i\mu }(\frac{\partial}{\partial Y^s_{i\mu}})^2.$$ With the above expression, it 
This equation clearly implies \eqref{fouling} when $\ell+\sk+\wt \sk=2.$  

\medskip
\noindent{\bf Case 2: $3\le \ell+\sk+\wt \sk \le 4\mathsf q.$} 
%In this case, we only need to consider 
%$$ \ell \ge2, \quad 2\le \sk+\wt \sk \le 4\mathsf{q}-1.$$For notional simplicity, 
We denote $\wt U=U^1-U^0.$ It is evident that in order to establish  (\ref{fouling}), it suffices to prove the following estimate for every $(k_1,k_2)\in \N^2$ satisfying $2\le k_1+k_2 = \sk+\wt\sk+\ell-1\le 4\mathsf q-1$: 
\begin{align}\label{fouling2}
& n^{-(k_1+k_2+1)/2} 
  \sum_{j,\mu} \cal C_{(j\mu)}^\ell    \E \Big|
\Big[\Big(P^{U^0}_{ j\mu} \Big)^{k_1}
\Big(P^{U^1}_{ j\mu} \Big)^{k_2} 
  P^{\wt U}_{ j\mu}  
  F^{\bv,\,\mathsf{q}} \Big]\pB{Y^{s} } \Big| \lesssim (n^{24\delta}\Psi)^\mathsf{q} + | \E  F^{\bv, \mathsf{q}} (Y^s)|. 
\end{align}
For the proof of (\ref{fouling2}), we need to delve into the detailed structure of the derivatives. To this end, in view of (\ref{eq_operatoroperator}) and \eqref{Pujmu}, we introduce the following algebraic objects, which are employed in \cite{MR3704770,10.1214/19-EJP381}. 

%similar to the arguments of \cite{10.1214/19-EJP381}, 
%The key to the proof of (\ref{fouling2}) is to study the detailed structure of the derivatives. For this purpose, in view of (\ref{eq_operatoroperator}) and \eqref{Pujmu}, we introduce the following algebraic objects that are used in \cite{MR3704770,10.1214/19-EJP381}. 
\iffalse
More specifically, we need to control quantities of the following form  
\be\label{yuzhong2}
(\partial_{Y^s_{i\mu}})^{\sk+ \wt \sk+1}F_{\alpha \beta}^{{V},\mathsf{q}} \pB{Y^{s}} , \quad i\in \cal I_1, \quad \mu \in \cal I_2.
\ee
In light of the definition (\ref{eq_FVQdefinition}), we find that it is essentially a sum of  products of $G^s$'s entries  with the "general" index ${\bf v}_\alpha$, $\b v_\beta$ and $i$, $\mu$. Similarly, the quantity  
\be\label{xieQ}
\left[
\left(P^{U^1}_{ j\mu} \right)^{\wt \sk}
 \left(P^{U^0}_{ j\mu} \right)^{\sk}
  \left(P^{\wt U}_{ j\mu} \right) 
  F^{V,\,\mathsf{q}}_{\alpha \beta}\right]\pB{Y^{s}}, 
\ee
will be analogous to \eqref{yuzhong2} except that $\sk$ of  the index $i$'s  are replaced with   $\bu^0_j$ (i.e., the $j$th column of $U^0$),  $\wt \sk$ of  the index $i$'s are replaced  with  $\bu^1_j$ (i.e., the $j$th column of $U^1$), and one of the index $i$ is replaced  with ${\bf \wt u}_j$ (i.e. the $j$th column of ${\wt U} $).  
\fi

%To address this issue, we employ the algebraic objects called \emph{words} as in \cite{MR3704770,10.1214/19-EJP381}. 

\begin{definition}[Words]\label{defn_words}  
Given $j\in \mathcal I_1$ and $\mu\in \mathcal I_2$, let $\mathcal W$ be the set consisting of words of even length formed from the four letters \smash{$\{\mathbf j_0, \mathbf j_1, \wt{\mathbf j}, \bm{\mu}\}$}. The length of a word $w\in\sW$ is denoted by $2{\bm l}(w) \in 2\mathbb N$. We use bold symbols to represent the letters in words. For instance, 
\be\label{eq_word}
w=\mathbf t_1\mathbf s_2\mathbf t_2\mathbf s_3\cdots\mathbf t_l\mathbf s_{l+1}
\ee
denotes a word of length $2l$. Define $\sW_l:=\{w\in \mathcal W: {\bm l}(w)=l\}$ as the set of words of length $2l$ such that the following property holds for a word as in \eqref{eq_word}: 
$$\mathbf t_{i}\mathbf s_{i+1}\in \{\mathbf j_0\bm{\mu},\bm{\mu}\mathbf j_0,\mathbf j_1\bm{\mu},\bm{\mu}\mathbf j_1, \wt{\mathbf j}\bm{\mu},\bm{\mu}\wt{\mathbf j}\},\quad i\in\qq{l}.$$
%satisfies that $\mathbf t_l\mathbf s_{l+1}\in\{\mathbf i\bm{\mu},\bm{\mu}\mathbf i,\wt{\mathbf i}\bm{\mu},\bm{\mu}\wt{\mathbf i}\}$ for all $1\le l\le r$.
Next, we assign a value $[\cdot]$ to each letter as follows: $[\mathbf j_0]=\bu_j^0$, $[\mathbf j_1]=\bu_j^1$, $[\wt{\mathbf j}]=\wt\bu_j$,  $[\bm {\mu}]:=\mathbf e_\mu$, where $\bu_j^0:=U^0\mathbf e_j$, $\bu_j^1:=U^1\mathbf e_j$, \smash{$\wt\bu_j:=\wt U\mathbf e_j/\|\wt U\mathbf e_j\|_2$}. % where $\mathbf u_i^0$ and $\bu^1_\mu$ are regarded as summation indices. 
It is important to distinguish between the abstract letter and its corresponding value, which is treated as a (generalized) summation index. For the word $w$ in \eqref{eq_word} and vectors $\bu,\bv\in \C^{p+n}$, we assign a random variable $A_{\mathbf u, \mathbf v, j, \mu}(w)$ as follows: if ${\bm l}(w)=0$, we define
 $$A_{\bu,\bv, j, \mu}(w):= \big( G(Y^s)-\wt\Pi_t\big)_{\bu\bv};$$
if ${\bm l}(w)\ge 1$, we define
 \begin{equation}\label{eq_comp_A(W)}
 A_{\bu,\bv, j, \mu}(w):=G_{\bu[\mathbf t_1]}(Y^s) G_{[\mathbf s_2][\mathbf t_2]}(Y^s)\cdots G_{[\mathbf s_l][\mathbf t_l]} (Y^s)G_{[\mathbf s_{l+1}]\bv}(Y^s).
 \end{equation}
Finally, given any $w\in \cal W_l$ of the form \eqref{eq_word}, we denote $\mathfrak n_0(w):=\#\{i \in \qq{l}:\mathbf t_{i}= \mathbf j_0 \ \text{or}\ \mathbf s_{i+1} =\mathbf j_0\}$, $\mathfrak n_1(w):=\#\{ i \in \qq{l}:\mathbf t_{i}= \mathbf j_1 \ \text{or}\ \mathbf s_{i+1} =\mathbf j_1\}$, and $\wt{\mathfrak n}(w):=\#\{i \in \qq{l}:\mathbf t_{i}=\wt{\mathbf j} \ \text{or}\ \mathbf s_{i+1} =\wt{\mathbf j} \}$.
\end{definition}
\iffalse
Given $k_1,k_2$ as in \eqref{fouling2}, we define $\cal W_l^*\subset \cal W_l$ such that the followings hold for every $w\in \sW_l$: there exists one $i_0\equiv i_0(w)\in \qq{l}$ such that \smash{$\mathbf t_{i_0}\mathbf s_{i_0+1}\in\{\wt{\mathbf j}\bm{\mu},\bm{\mu}\wt{\mathbf j}\}$}, and for all other $i\ne i_0$,  $\mathbf t_{i}\mathbf s_{i+1}\in \{\mathbf i\bm{\mu},\bm{\mu}\mathbf i\}$.\fi

From \eqref{6j8} and \eqref{Pujmu}, we observe that the above notations are defined such that for every $l \in \mathbb{N}$, the following equality holds:
%we see that the above notations are defined in a way that for each $l \in \mathbb{N}$,
\[\Big(P^{U^0}_{ j\mu} \Big)^{l} \big( G(Y^s)-\wt\Pi_t \big)_{\bu\bv}=(-1)^l l!\sum_{w\in \mathcal W_l: \mathfrak n_0(w)=l} A_{\bu, \bv, j, \mu}(w).\]
A similar equation holds if we replace $0$ with $1$. By denoting $\al_j:=\|\wt U\mathbf e_j\|_2$, we have that
\[ P^{\wt U}_{ j\mu} \big( G(Y^s)-\wt\Pi_t \big)_{\bu\bv}=- \al_j \sum_{w\in \{\wt{\mathbf j}\bm{\mu},\bm{\mu}\wt{\mathbf j}\}} A_{\bu, \bv, j, \mu}(w).\]
% by an argument similar to those between (6.31) and (6.32) of \cite{10.1214/19-EJP381}, we can obtain that for all $r \in \mathbb{N}$
% \[\left(\frac{\partial}{\partial Y_{i\mu}^s}\right)^r \left( VG^sV^\top-V\Pi_t V^\top\right)_{\alpha \beta}=(-1)^r r!\sum_{w\in \mathcal W_r} A_{\alpha, \beta, i, \mu}(w),\]
Using the above two identities, we can derive that
\begin{align}
    \Big(P^{U^0}_{ j\mu} \Big)^{k_1}
\Big(P^{U^1}_{ j\mu} \Big)^{k_2} 
  P^{\wt U}_{ j\mu}  
  F^{\bv,\,\mathsf{q}}(Y^s)&=  \sum_{\substack{l_1,\ldots,l_{\mathsf q} \ge 0:\\ l_1+\cdots+l_{\mathsf q}=k_1+k_2+1}}\sum_{w_1\in \cal W_{l_1},\ldots, w_{\mathsf q}\in \cal W_{l_{\mathsf q}}}^* c(w_1,\ldots, w_{\mathsf q}) \al_j \nonumber\\
&\times \prod_{\mathsf{t}=1}^{\mathsf{q}/2} \left[A_{\bv, \bv, j, \mu}(w_\mathsf{t})\overline{A_{\bv, \bv, j, \mu}(w_{\mathsf{t}+\mathsf{q}/2})} \right].\label{eq_longproduct}
\end{align}
Here, $c(w_1,\ldots, w_{\mathsf q})$ represents certain deterministic coefficients of order $\OO(1)$, and $\sum_{w_1\in \cal W_{l_1},\ldots, w_{\mathsf q}\in \cal W_{l_{\mathsf q}}}^*$ denotes the summation over a sequence of words $w_1,\ldots, w_{\mathsf q}$ satisfying the following conditions:
\be\label{eq_words_structure}
\sum_{\mathsf t=1}^{\mathsf q} \mathfrak n_0(w_{\mathsf t})=k_1, \quad \sum_{\mathsf t=1}^{\mathsf q} \mathfrak n_1(w_{\mathsf t})=k_2, \quad \sum_{\mathsf t=1}^{\mathsf q} \wt{\mathfrak n}(w_{\mathsf t})=1.
\ee
% \begin{align*}
% \left(\frac{\partial}{\partial Y^s_{i\mu}}\right)^r F_{\alpha \beta}^{V,\mathsf{p}}(Y^s)=(-1)^r & \sum_{l_1+\cdots+l_p=r}\prod_{\mathsf{t}=1}^{\mathsf{q}/2}\left(l_{\mathsf{t}}! l_{\mathsf{t}+\mathsf{q}/2}!\right) \\
% &\times \left(\sum_{w_{\mathsf{t}}\in\sW_{l_{\mathsf{t}}}}\sum_{w_{\mathsf{t}+\mathsf{q}/2}\in\sW_{l_{\mathsf{t}+\mathsf{q}/2}}}A_{\alpha, \beta, i, \mu}(w_\mathsf{t})\overline{A_{\alpha, \beta, i, \mu}(w_{\mathsf{t}+\mathsf{q}/2})}\right).
% \end{align*}
Without loss of generality, suppose there are $\mathsf p$ non-zero length words in \eqref{eq_longproduct} for some $1\le \mathsf{p} \le \mathsf q\wedge (k_1 +k_2+1)$, and these words are $w_1,\ldots, w_{\mathsf p}$. Then, we can identify and separate the words of length 0 from the product in \eqref{eq_longproduct}, denoted as $w_0$. Now, to show (\ref{fouling2}), it suffices to prove that for any $1\le \mathsf{p} \le \mathsf q\wedge (k_1 +k_2+1)$ and a sequence of words satisfying \eqref{eq_words_structure}, the following estimate holds:
 \begin{align}\label{eq_hahhahahah}
n^{-(k_1+k_2+1)/2} 
  \sum_{j,\mu} \al_j 
  \E\left|A_{\bv,\bv,j, \mu}(w_0)^{\mathsf{q} - \mathsf{p}} \prod_{r = 1}^{\mathsf{p}} A_{\bv,\bv,j, \mu}(w_r)\right| 
\lesssim (n^{24\delta}\Psi)^{\mathsf{q}}  + | \E  F^{\bv, \mathsf{q}} (Y^s)|.
 \end{align}
%Here ${\bf A}_{\alpha,\beta,i, \mu}$ is a word like $A_{s,t,i, \mu}$, except that $\sk$ of  the index $i$'s  are replaced with  $\bu^0_j$,  $\wt \sk$ of  the index $i$'s are replaced  with  $\bu^1_j$,  and one of the index $i$ replaced  with ${\bf \wt u}_j$.

The proof of (\ref{eq_hahhahahah}) follows a similar approach to that of (6.32) in \cite{10.1214/19-EJP381}.
%using assumption \eqref{tianya} that  $|{\bf \wt u}_j |\le n^{-\mathsf{c}'}.$ 
As we assume $\mathbb{E} \mathsf{X}^3_{i \mu}=0$ according to (ii) of \Cref{main_assumption}, we only need to consider the case when $\ell \geq 4$. (It is in this particular case that the vanishing third-moment condition is relevant and utilized.) 
%as in (6.33) of \cite{10.1214/19-EJP381} which relies on some crucial controls on the words $A_{\alpha,\beta, i, \mu}(w).$ 
%Let $\bv_{\alpha}$ be the $\alpha$th column of $V$ in (\ref{eq_FVQdefinition}) and $G^s_{\bv_{\alpha} \bv_\beta}=\bv_\alpha^\top G^s \bv_\beta.$ Moreover, 
For $j\in \cal I_1$ and $\mu \in \cal I_2$, we define
$$
  \mathcal{R}_j:= \big|G_{\bv\bu^0_ j}\big|+\big|G^s_{\bv\bu^1_ j}\big|+\big|G_{\bv_\alpha \wt\bu_j}\big|,\quad  \mathcal{R}_\mu:= \big|G_{\bv\mu}\big|.$$
  %+\big|G^s_{\bu^0_ j \bv_\beta}\big|+\big|G^s_{\bu^1_ j \bv_\beta}\big|+\big|G^s_{\frac{{\bf \wt u}_j }{|{\bf \wt u}_j |} \bv_\beta}\big|.  $$ 
Similar to Lemma 6.16 in \cite{10.1214/19-EJP381}, we can establish the following rough bound using property ($\mathbf{A}_{d-1}$):
  \begin{equation*}
  |A_{\bv,\bv, i, \mu}(w)|\prec n^{2\delta({\bm l}(w)+1)},
  \end{equation*}
Additionally, we can get the following bounds: for ${\bm l}(w)\ge 2$, 
  \begin{equation*}
  |A_{\bv,\bv,  i, \mu}(w)|\prec(\mathcal R_i^2+\mathcal R_\mu^2)n^{2\delta({\bm l}(w)-1)},
  \end{equation*}
and for ${\bm l}(w)=1$, 
  \begin{equation*}
  |A_{\bv,\bv, i, \mu}(w)|\prec \mathcal R_i\mathcal R_\mu.
  \end{equation*}
With these estimates, we can follow the arguments presented between (6.39) and (6.42) in \cite{10.1214/19-EJP381} to complete the proof of (\ref{eq_hahhahahah}). (It is worth noting that for this proof, the condition \eqref{tianya}, which implies $\al_j \le C n^{-\mathsf{c}'}$, is {\bf not} necessary.)
%(We remark that for the proof here, it is {\bf not} necessary to use the condition \eqref{tianya}, which implies $\max_j \al_j \le C n^{-\mathsf{c}'}$.)
Since the argument is almost the same, we omit the details here. This concludes the proof of \Cref{lem_locallaw}.  
\end{proof}

%\subsubsection{Completion of the proof}
With the above results, we proceed to complete the proof of Lemma \ref{nanzi} by using the following lemma. % The following result will be used in our proof. 

\begin{lemma}\label{lem_reducedusingfc} 
Under the assumptions of Lemma \ref{nanzi}, the following estimate holds for every $s \in [0,1]$: 
\begin{align*}
	&\E \theta \left( \mathcal U_{i_1}(\mathsf Q_t^s),  \ldots, \mathcal U_{i_L}(\mathsf Q_t^s)\right)=\E \theta \left( \overline{\mathcal U}_{i_1}(\mathsf Q_t^s),  \ldots, \overline{\mathcal U}_{i_L}(\mathsf Q_t^s)\right) +\OO( n^{-\nu}),
	% &=  \E  \theta \Big( \frac{p}{\pi}\int_{I_{i_1}} \left[\im x^s(E) \right]
	% \mathfrak q_{i_1}\left(y^s_{i_1}(E) \right)\dd E,  \cdots, \frac{p}{\pi}\int_{I_{i_m}} \left[\im x^s(E) \right]
	% \mathfrak q_{i_m}\left( y_{i_m}^s(E) \right)\dd E\Big) \nonumber
\end{align*}
for a constant $\nu>0$, where $\overline{\mathcal U}_{k}(\mathsf Q_t^s)$, $k \in \qq{\sfK-r},$ are defined as
 $$ \overline{\mathcal U}_{k}(\mathsf Q_t^s) :=\frac{p}{\pi}\int_{I_{k}} \left[\im x_k^s(E) \right]	\mathfrak q_{k}\left(y^s_{k}(E) \right)\dd E$$
Here, the functions $y^s_{k}(E)$, $k \in \qq{\sfK-r},$ are defined as
	\begin{align} 
	y^s_{i_k}(E) &\;\deq\; \frac 1{2\pi}\int_{\R^2}
	\ii \sigma \mathfrak{f}_{i_k}''(e)  \chi(\sigma)  \tr \mathcal{G}_{2,t}^s(e + \ii \sigma)\,  \indb{\abs{\sigma} \geq \tilde \eta_k n^{-C \e}} \,\dd e \, \dd \sigma
	\nonumber \\ \label{defy428}
	&\qquad \;+\; 
	\frac 1{2\pi}\int_{\R^2}
	\pb{ \ii \mathfrak{f}_{i_k}(e) \chi'(\sigma)- \sigma \mathfrak{f}_{i_k}'(e)\chi'(\sigma)} \tr \mathcal{G}_{2,t}^s(e + \ii \sigma) \, \dd e \, \dd \sigma\, ,
	\end{align}
	where $\wt \eta_k$ is defined in (\ref{eq_keyquantitiesdefinition}), $C>0$ is an absolute constant, and $\chi$ is a smooth cutoff function with support in $[-1,1]$, satisfying $\chi(\sigma) = 1$ for $|\sigma|\le 1/2$, and having bounded derivatives up to arbitrarily high order.

 % \begin{align*}
	% \E \theta \Big( & \frac{p}{\pi}\int_{I_{i_1}} \left[\im x^s(E) \right]
	% \mathfrak q_{i_1}\left(\tr \mathfrak f_{i_1}(\mathsf{Q}_t^s) \right)\dd E,  \cdots, \frac{p}{\pi}\int_{I_{i_m}} \left[\im x^s(E) \right]
	% \mathfrak q_{i_m}\left(\tr \mathfrak f_{i_m}(\mathsf{Q}_t^s) \right)\dd E\Big) \\
	% &=  \E  \theta \Big( \frac{p}{\pi}\int_{I_{i_1}} \left[\im x^s(E) \right]
	% \mathfrak q_{i_1}\left(y^s_{i_1}(E) \right)\dd E,  \cdots, \frac{p}{\pi}\int_{I_{i_m}} \left[\im x^s(E) \right]
	% \mathfrak q_{i_m}\left( y_{i_m}^s(E) \right)\dd E\Big)+\oo(1). \nonumber
	% \end{align*}
\end{lemma}
\begin{proof}
	The proof follows the same lines of the arguments presented between (5.9) and (5.11) in \cite{knowles2013eigenvector}, utilizing Lemma \ref{lem_locallaw} and the Helffer-Sj{\" o}strand formula (see \cite[Proposition 1.13.4]{benaych2017advanced}). Further details are omitted.
	%{\color{red}[from here]} In the actual proof, in order to connect the integrand with the resolvent more closely, we apply  to $\mathfrak{f}_{i_k}$ that 
	%\begin{equation*}
	%\mathfrak{f}_{i_k}(\lambda)=\frac{1}{2 \pi} \int_{\R^2} \frac{\ii \sigma \mathfrak{f}_{i_k}''(e) \mathcal{X}(\sigma)+\ii \mathfrak{f}_{i_k}(e) \mathcal{X}'(\sigma)-\sigma \mathfrak{f}'_{i_k}(e) \mathcal{X}'(e)}{\lambda-e-\ii \sigma} \dd e \dd \sigma,
	%\end{equation*}
	%where $\mathcal{X}(y)$ is a smooth-cutoff function with support on $[-1,1]$ and $\mathcal{X}(y) \equiv 1$ for $|y| \leq 1/2 $ with bounded derivatives. Moreover, denote $\mathcal{G}_{2,t}^s(z)=(\mathsf{X}^\top U_s^\top D_t U_s \mathsf{X}-z)^{-1}.$  We can further show that (see [{\color{red}cite the exact calculation equation here}]) for $s=0,1$
	%
	%where $C>0$ is some constant and $\epsilon>0$ is some small constant. The above expression is a functional representation only in terms of the resolvents so that the comparison strategy applies.  
\end{proof}

%Using the above representation,  we are able to complete the proof of Lemma \ref{nanzi} with a similar argument as the above proof of \Cref{lem_locallaw}.

\begin{proof}[\bf Proof of Lemma \ref{nanzi}]
For $s\in [0,1]$, we introduce the function 
\begin{equation}\label{eq_keyquantitytouse}
	F(Y^s)\equiv F(\mathsf Q^s_t):= \theta \left( \overline{\mathcal U}_{i_1}(\mathsf Q_t^s),  \ldots, \overline{\mathcal U}_{i_L}(\mathsf Q_t^s)\right) . 
	\end{equation} 
	% \begin{equation}\label{eq_keyquantitytouse}
	% F(Y^s)\equiv F(\mathsf Q^s_t):=  \theta\left(\frac{p}{\pi}\int_{I_k} [\im x^s (E)] \mathfrak{q}(y^s(E)) \dd E \right). 
	% \end{equation} 
	According to Lemma \ref{lem_reducedusingfc}, it suffices to show that 
	\begin{equation*}
	\mathbb{E} F(Y^1)-\mathbb{E} F(Y^0)=\OO(n^{-\nu}) 
	\end{equation*}
 for a constant $\nu>0$. Similar to \eqref{eq_calculusoffundemantales}, we have 
%Similar fundamental theorem of calculus and (\ref{eq_interpolation}), we have 
	\begin{equation*}
	\mathbb{E} F(Y^1)-\mathbb{E} F(Y^0)=\int_0^1 \mathrm{d}s \sum_{j,   \mu } 
	\qB{ \E F  \pB{Y^{s,Z^{1}_{(j\mu)}}_{(j \mu)}}
		- \E F  \pB{Y^{s,Z^{0}_{(j\mu)}}_{(j \mu)}}}. 
	\end{equation*}
	Hence, it remains to prove 
	\begin{equation}\label{eq_teles}
		\sum_{j  , \mu } 
		\qB{ \E F  \pB{Y^{s,Z^{1}_{(j\mu)}}_{(j \mu)}}
			- \E F  \pB{Y^{s,Z^{0}_{(j\mu)}}_{(j \mu)}}} =\OO(n^{-\nu}),\quad s\in[0,1].
	\end{equation}
 The proof of this estimate is similar to the above proof of \Cref{lem_locallaw}. In fact, the proof here is slightly easier since we have the local law \eqref{eq_localGs} at hand, and there is no need to use an induction argument based on the two scale-dependent properties \textbf{($\b A_d$)} and \textbf{($\b C_d$)}. Furthermore, by employing the argument presented in the proof of \cite[Lemma 15.5]{Erds2017ADA}, we can deduce the following estimate for $G(Y^s)$ from \Cref{lem_locallaw} when the value of $\eta$ is less than $n^{-1}$: for any constant $\sigma>0$ and deterministic unit vectors $\bu,\bv\in \R^{p+n}$,
 %we can derive from \Cref{lem_locallaw} the following estimate for $G(Y^s)$ when $\eta$ is below $n^{-1}$:
 \be\label{eq_local_smalleta}
 \sup_{\eta\ge n^{-1-\sigma}}\sup_{\lambda_- - c\le  E\le \lambda_+ + c} \big|(G_t^s-\wt\Pi_t)_{\bu\bv}(E+\ii \eta)\big|\prec n^{\sigma}.
 \ee

Notice that as $F^{\bv,\mathsf q}$ in \eqref{eq_FVQdefinition}, $F$ is also a function of the resolvent $G(Y^s)$. Similar to \eqref{def_fch}, we abbreviate 
 \begin{equation} \label{def_fch2}
f _{(j \mu)}(\lambda_0,  \lambda_1 ) \deq 
F\left(Y^{s, \lambda_0 U^0 {\bf e}_{j \mu}+ \lambda_1 U^1 {\bf e}_{j \mu}}_{(j \mu)}\right)\, , 
 \end{equation}
To analyze this further, we perform Taylor expansions of $f _{(j \mu)}(\mathsf X_{j\mu}, 0)$ and $f _{(j \mu)}(0,\mathsf X_{j\mu})$ around $f _{(j \mu)}(0,0)$. Then, we compare the two Taylor expansions and estimate their differences using the resolvent expansion \eqref{6j8}, the derivatives as in \eqref{Pujmu}, and the local laws \eqref{eq_localGs} and \eqref{eq_local_smalleta}. Following the argument below \eqref{def_fch} (or the argument in \cite[Section 16]{Erds2017ADA}), we can derive an estimate similar to \eqref{eq_hahhahahah}: 
\begin{align}\label{eq:f-f}
   \left| \E f _{(j \mu)}(\mathsf X_{j\mu}, 0)-\E f _{(j \mu)}(0,\mathsf X_{j\mu})\right| \lesssim  n^{-2}\sum_{j,\mu} \al_j \cal E_{j\mu} + n^{-1}.
\end{align}
Here, $\cal E_{j\mu}$ are positive variables of order $\OO(n^{C\e})$, where $C>0$ is an absolute constant that does no depend on $\e$. In the derivation, we used the vanishing third-moment condition, and the factor $n^{-2}$ arises from the fourth or higher-order moments of $\mathsf X_{j\mu}$. The factor $\al_j$ has the same origin as that in the LHS of \eqref{eq_hahhahahah}. We omit the details here. Since $\al_j \le C n^{-\mathsf{c}'}$ by the condition \eqref{tianya}, we deduce from \eqref{eq:f-f} that 
\begin{align*}
   \left| \E f _{(j \mu)}(\mathsf X_{j\mu}, 0)-\E f _{(j \mu)}(0,\mathsf X_{j\mu})\right| \lesssim  n^{-\mathsf c'+ C\e}.
\end{align*}
This concludes the proof by choosing $\e$ sufficiently small, depending on $\mathsf c'$ and $C$.  
 % {\cor here ... utilizing the condition \eqref{tianya}, which implies $\max_j \al_j \le C n^{-\mathsf{c}'}$} 
	% Using the smooth properties of $\theta, \mathfrak{f}$ and $\mathfrak{q},$ we can control the derivatives of $F^\mathsf{q}(Y^s)$ exactly the same as in Lemma 6.13 of \cite{10.1214/19-EJP381}. Therefore,  we can repeat the arguments from (\ref{yanglao}) to the end of Section \ref{sec_proofoflocallaw} to conclude the proof. Due to similarity, we omit further details. 
\end{proof}

\end{document}